\documentclass[10pt]{amsart}
\allowdisplaybreaks
\usepackage{fullpage}
\usepackage{amssymb}
\usepackage[all]{xy}
\usepackage{bbm}
\usepackage{cancel}
\usepackage{graphicx}
\usepackage{mathabx}
\usepackage{tikz-cd}
\usepackage{quiver}
\usepackage{epigraph} 
\usepackage{stmaryrd}
\usepackage{enumitem}
\usepackage[mathscr]{eucal}
\usepackage{soul}
\usepackage{mathtools}
\usepackage{xfrac}
\usepackage{verbatim}

\usepackage[utf8]{inputenc}
\usepackage[T1]{fontenc}
\usepackage[backend=biber,
            isbn=false,
            doi=false,
            maxbibnames=9,
            giveninits=true,
            style=alphabetic,
            citestyle=alphabetic]{biblatex}
\usepackage{url}
\renewbibmacro{in:}{}
\bibliography{paper.bib}

\usepackage{xcolor}
\definecolor{e-mail}{rgb}{0,.40,.80}
\definecolor{reference}{rgb}{.20,.60,.22}
\definecolor{citation}{rgb}{0,.40,.80}
\usepackage[colorlinks=true,
            linkcolor=reference,
            citecolor=citation,
            urlcolor=e-mail]{hyperref}

\usepackage{aliascnt}
\usepackage[capitalize]{cleveref}
\usepackage{adjustbox}

\theoremstyle{plain}
\newtheorem{maintheorem}{Theorem}

\newtheorem{mainconjecture}[maintheorem]{Conjecture}
\newtheorem{theorem}{Theorem}[section]

\newtheorem{corollary}[theorem]{Corollary}
\newtheorem{proposition}[theorem]{Proposition}
\newtheorem{lemma}[theorem]{Lemma}
\newaliascnt{assumption}{theorem}

\aliascntresetthe{assumption}
\crefname{assumption}{assumption}{assumptions}
\theoremstyle{definition}
\newtheorem{definition}[theorem]{Definition}
\newtheorem{example}[theorem]{Example}
\newtheorem{remark}[theorem]{Remark}

\let\oldtocsection=\tocsection
\let\oldtocsubsection=\tocsubsection
\let\oldtocsubsubsection=\tocsubsubsection
\renewcommand{\tocsection}[2]{\hspace{0em}\oldtocsection{#1}{#2}}
\renewcommand{\tocsubsection}[2]{\hspace{1em}\oldtocsubsection{#1}{#2}}
\renewcommand{\tocsubsubsection}[2]{\hspace{2em}\oldtocsubsubsection{#1}{#2}}

\newcommand{\defterm}[1]{\textbf{\emph{#1}}}

\usepackage{calc}
\newcommand{\dbtilde}[1]{\tilde{\raisebox{0pt}[0.9\height]{$\tilde{#1}$}}}

\def\textin{\quad\textup{in}\quad}
\def\and{\quad\textup{and}\quad}
\def\where{\quad\text{where}\quad }

\def\all{\mathrm{all}}

\def\BC{\mathrm{BC}}
\def\CAlg{\mathrm{CAlg}}
\def\cart{\ar@{}[rd]|{\Box}}
\def\Cat{\mathrm{Cat}}
\def\colim{\operatorname{colim}}
\def\Corr{\mathrm{Corr}}
\def\counit{\mathrm{counit}}
\def\cC{\mathscr{C}}
\def\cE{\mathscr{E}}

\def\cV{\mathscr{V}}
\def\bD{\mathbf{D}}
\def\Db{\mathbf{D}^{\mathrm{b}}}
\def\Dbc{\mathbf{D}^{\mathrm{b}}_{\mathrm{c}}}

\def\ex{\mathrm{ex}}

\def\Fun{\mathrm{Fun}}
\def\cF{\mathscr{F}}
\def\cG{\mathscr{G}}
\def\Gpd{\mathrm{Gpd}}
\def\cHom{\mathcal{H}\mathrm{om}}
\def\Hom{\mathrm{Hom}}
\def\id{\mathrm{id}}

\def\Ind{\mathrm{Ind}}
\def\Map{\mathrm{Map}}
\def\Mod{\mathrm{Mod}}
\def\pr{\mathrm{pr}}
\def\Perf{\mathrm{Perf}}
\def\PrSt{\mathrm{Pr}^{\mathrm{St}}}
\def\Set{\mathrm{Set}}

\def\bu{\mathbf{1}}
\def\unit{\mathrm{unit}}

\def\lch{\mathrm{lch}}

\def\Top{\mathrm{Top}}

\def\bD{\mathbf{D}}
\def\bbD{\mathbb{D}}

\def\Ex{\mathrm{Ex}}
\def\scF{\mathscr{F}}
\def\cF{\scF}
\def\fgsp{\mathrm{fgsp}}

\def\scG{\mathscr{G}}

\def\cK{\mathscr{K}}
\def\Loc{\mathrm{Loc}}

\def\Perv{\mathrm{Perv}}
\def\bPerv{\mathbb{P}\mathrm{erv}}
\renewcommand{\Pr}{\mathrm{Pr}}
\def\cP{\mathscr{P}}
\def\pur{\mathrm{pur}}
\def\cQ{\mathscr{Q}}
\def\Shv{\mathrm{Shv}}

\def\TS{\mathrm{TS}}

\def\rH{\mathrm{H}}
\def\HBM{\mathrm{H}^{\mathrm{BM}}}

\def\Ker{\mathrm{Ker}}
\def\op{\mathrm{op}}
\def\Tot{\mathrm{Tot}}

\def\C{\mathbb{C}}

\def\Q{\mathbb{Q}}
\def\R{\mathbb{R}}
\def\Z{\mathbb{Z}}

\def\dual{^{\vee}}

\def\rO{\mathrm{O}}
\def\q{\mathfrak{q}}
\def\rank{\mathrm{rk}}
\renewcommand{\Re}{\operatorname{Re}}

\def\SO{\mathrm{SO}}

\def\vol{\mathrm{vol}}

\def\bA{\mathbb{A}}

\def\an{\mathrm{an}}
\def\Art{\mathrm{Art}}
\def\B{\mathrm{B}}

\def\cl{\mathrm{cl}}

\def\ev{\mathrm{ev}}
\def\ft{\mathrm{ft}}

\def\geometric{\mathrm{geometric}}

\def\Gm{\mathbb{G}_{\mathrm{m}}}
\def\gr{\mathrm{gr}}
\def\cI{\mathcal{I}} 
\def\I{\cI}
\def\bL{\mathbb{L}}
\def\cL{\mathscr{L}}
\def\ft{\mathrm{ft}}
\def\lft{\mathrm{lft}}

\def\lis{{\triangleleft}}

\def\rN{\mathrm{N}}
\def\cO{\mathscr{O}} 
\def\O{\cO}
\def\OG{\mathrm{OG}}

\def\Pic{\mathrm{Pic}}
\def\pt{\mathrm{pt}}
\def\QCoh{\mathrm{QCoh}}

\def\red{\mathrm{red}}

\def\Sch{\mathrm{Sch}}
\def\sep{\mathrm{sep}}
\def\smooth{\mathrm{smooth}}
\def\Spec{\operatorname{Spec}}
\def\Stk{\mathrm{Stk}}
\def\supp{\operatorname{supp}}
\def\Sym{\mathrm{Sym}}
\def\Tot{\mathrm{Tot}}
\def\Zar{\mathrm{Zar}}
\def\Zero{\mathrm{Z}}

\def\dAff{\mathrm{dAff}}
\def\dStk{\mathrm{dStk}}

\def\dim{\operatorname{dim}}

\def\cA{\mathscr{A}}
\def\can{\mathrm{can}}
\def\Crit{\mathrm{Crit}}

\def\CritCharts{\mathrm{CritCharts}}
\def\DCrit{\mathrm{DCrit}}
\def\der{\mathrm{der}}

\def\DR{\mathrm{DR}}
\def\ex{\mathrm{ex}}
\def\Hess{\mathrm{Hess}}
\def\rN{\mathrm{N}}
\def\ori{\mathrm{or}}
\def\cS{\mathscr{S}}
\def\Sing{\mathrm{Sing}}
\def\stab{\mathsf{stab}}

\def\T{\mathrm{T}}

\def\und{\mathrm{und}}
\def\vir{\mathrm{vir}}
\def\fX{\mathfrak{X}}
\def\fY{\mathfrak{Y}}
\def\tU{\widetilde{U}}
\def\tV{\widetilde{V}}
\def\tg{\widetilde{g}}

\makeatletter
 
  \@addtoreset{equation}{section}
 \makeatother

\tikzcdset{scale cd/.style={every label/.append style={scale=#1}, cells={nodes={scale=#1}}}}

\begin{document}

\title{Perverse pullbacks}

\address{Institute of Mathematics, Academia Sinica, Taipei, Taiwan}
\address{National Center for Theoretical Sciences, National Taiwan University, Taipei, Taiwan}
\email{adeelkhan@as.edu.tw}
\author{Adeel A. Khan}
\address{Research Institute for Mathematical Sciences, Kyoto University, Kyoto 606-8502, Japan}
\email{tkinjo@kurims.kyoto-u.ac.jp}
\author{Tasuki Kinjo}
\address{June E Huh Center for Mathematical Challenges, Korea Institute for Advanced Study, Seoul, Republic of Korea}
\email{hyeonjunpark@kias.re.kr}
\author{Hyeonjun Park}
\address{School of Mathematics and Maxwell Institute for Mathematical Sciences, University of Edinburgh, Edinburgh, UK}
\email{p.safronov@ed.ac.uk}
\author{Pavel Safronov}

\begin{abstract}
We define a new perverse $t$-exact pullback operation on derived categories of constructible sheaves which generalizes most perverse $t$-exact functors in sheaf theory, such as microlocalization, the Fourier--Sato transform and vanishing cycles. This operation is defined for morphisms of algebraic stacks equipped with a relative exact $(-1)$-shifted symplectic structure, and can be used to define cohomological Donaldson--Thomas invariants in a relative setting. We prove natural functoriality properties for perverse pullbacks, such as smooth and finite base change, compatibility with products and Verdier duality.
\end{abstract}


\maketitle

\tableofcontents

\section*{Introduction}

For a locally compact topological space $X$ the derived category $\Shv(X)$ of sheaves of complexes of $\Q$-vector spaces has two $t$-structures: one $t$-structure is determined by the condition that all stalk functors (i.e. $i^*$ for inclusions of points) are $t$-exact and another $t$-structure is determined by the condition that all costalk functors (i.e. $i^!$) are $t$-exact. By definition, arbitrary $*$-pullbacks are $t$-exact with respect to the first $t$-structure and arbitrary $!$-pullbacks are $t$-exact with respect to the second $t$-structure. Moreover, these two $t$-structures as well as $*$-pullbacks and $!$-pullbacks are exchanged under the Verdier duality functor $\bbD$. 

In the case when $X$ is the complex analytic space underlying a complex algebraic variety, there is a subcategory $\Dbc(X)\subset \Shv(X)$ of constructible sheaves which inherits the two $t$-structures. But it also has a third, perverse, $t$-structure, which is Verdier self-dual. Given an extra structure on a morphism $\pi\colon X\rightarrow B$ of complex algebraic varieties (a relative d-critical structure or a derived enhancement $\fX\rightarrow B$ equipped with a relative exact $(-1)$-shifted symplectic structure) the goal of this paper is to define a \emph{perverse pullback} functor $\pi^\varphi\colon\Dbc(B)\rightarrow \Dbc(X)$ which is perverse $t$-exact and Verdier self-dual.

\subsection*{Cohomological DT theory}

Cohomological Donaldson--Thomas theory associates a cohomology theory to certain algebro-geometric moduli spaces. Namely, for a complex algebraic stack $X$ consider the following data:
\begin{itemize}
    \item A derived enhancement $X\hookrightarrow \fX$ equipped with a $(-1)$-shifted symplectic structure in the sense of \cite{PTVV}.
    \item Orientation data, i.e. the choice of a square root line bundle of the canonical bundle $K_\fX=\det(\bL_{\fX})$.
\end{itemize}

Given such data, the works \cite{BBDJS,BBBBJ,KiemLi} define a perverse sheaf $\varphi_X$ on $X$, locally modeled on the sheaf of vanishing cycles for a function $f\colon U\rightarrow \C$ on a smooth complex scheme $U$, so that the cohomological Donaldson--Thomas (DT) invariant of $X$ is given by the cohomology of $\varphi_X$. Some of the examples of stacks which admit $(-1)$-shifted symplectic derived enhancements are moduli stacks of compactly supported coherent sheaves on smooth 3-dimensional Calabi--Yau varieties (in which case a natural orientation data was constructed in \cite{JoyceUpmeier}) and moduli stacks of local systems on a compact oriented 3-manifold (in which case a natural orientation data was constructed in \cite{NaefSafronov}).

Cohomological DT invariants relate to the usual cohomology via a dimensional reduction isomorphism constructed in \cite{KinjoDimred}. Namely, consider an algebraic stack $Y$ with a quasi-smooth derived enhancement $\fY$. Consider the stack of singularities $\Sing(Y)$ \cite{ArinkinGaitsgory} obtained as the classical truncation of the $(-1)$-shifted cotangent bundle $\T^*[-1] \fY$. The dimensional reduction theorem identifies the cohomological DT invariant of $\Sing(Y)$ with the Borel--Moore homology of $Y$:
\[\rH^\bullet(\Sing(Y), \varphi_{\Sing(Y)})\cong \HBM_{\dim(\fY)-\bullet}(Y).\]

\subsection*{Perverse pullbacks}

Motivated by relative cohomological Donaldson--Thomas theory (where we have a family of 3-dimensional Calabi--Yau varieties or an anticanonical divisor in a 4-dimensional Fano variety), in this paper we introduce a relative version of the perverse sheaf $\varphi_X$ recalled above. Namely, instead of considering a fixed complex algebraic stack $X$ and equipping it with a perverse sheaf $\varphi_X\in\Perv(X)$ we consider a family $\pi\colon X\rightarrow B$ and equip it with a perverse pullback functor $\Perv(B)\rightarrow \Perv(X)$. The following is a condensed version of \cref{thm:micropullstk}; we refer the reader to the main body of the paper to the precise definition of natural isomorphisms, their compatibility as well as a generalization to constructible sheaves with coefficients in a general commutative ring $R$.

\begin{maintheorem}\label{mainthm:perversepullback}
Let $\pi\colon X\rightarrow B$ be a morphism of higher Artin stacks over $\C$ together with a derived enhancement $X\hookrightarrow \fX$, a relative exact $(-1)$-shifted symplectic structure on $\fX\rightarrow B$, and the choice of a square root line bundle of $K_{\fX/B}=\det(\bL_{\fX/B})$. Then there is a perverse $t$-exact functor
\[\pi^\varphi\colon \Dbc(B)\longrightarrow \Dbc(X)\]
which satisfies the following properties:
\begin{enumerate}
    \item It is compatible with pullbacks along smooth morphisms in $B$.
    \item It is compatible with pullbacks along smooth morphisms in $X$.
    \item It is compatible with $(-1)$-shifted symplectic pushforwards (in the sense of \cite{ParkSymplectic}) along smooth morphisms in $B$.
    \item It is compatible with pushforwards along finite morphisms in $B$.
    \item It commutes with Verdier duality.
    \item It is compatible with products.
\end{enumerate}
\end{maintheorem}

In the opposite extremes the perverse pullback is given as follows:
\begin{itemize}
    \item For $B=\pt$ we have $\pi^\varphi \Q_B = \varphi_X$ is the perverse sheaf defined in \cite{BBDJS,BBBBJ}. In fact, our construction of perverse pullbacks extends its definition to higher Artin stacks $X$.
    \item For $X=\fX=B$ with the relative exact $(-1)$-shifted symplectic structure on $\id\colon X\rightarrow X$ determined by a function $f\colon X\rightarrow \C$ we have
    \[\pi^\varphi = \bigoplus_{c\in\C} (f^{-1}(c)\rightarrow X)_*\varphi_{f-c},\]
    where $\varphi_f\colon \Dbc(X)\rightarrow \Dbc(f^{-1}(0))$ is the vanishing cycle functor. The above sum is necessary for perverse pullbacks to be compatible with products in the naive way. In fact, this natural isomorphism along with the properties (1)-(3) of perverse pullbacks from \cref{mainthm:perversepullback} determine them uniquely.
\end{itemize}

Following \cite{JoyceDcrit}, instead of considering a $(-1)$-shifted symplectic enhancement $X\hookrightarrow \fX$ it turns out to be useful to consider relative d-critical structures on $\pi\colon X\rightarrow B$ (see \cref{def:dcriticalscheme,def:dcriticalstack}). Namely, a relative d-critical structure $s$ on $\pi$ is given by a pair of a function $\und(s)\in\cO_X$ together with a nullhomotopy of its de Rham differential $d\und(s)\in\bL_{X/B}$ in the relative cotangent complex; we moreover require that smooth-locally $X\rightarrow B$ is given by the relative critical locus of a function of a smooth $B$-scheme, compatibly with $s$. Given a $(-1)$-shifted symplectic enhancement $X\hookrightarrow \fX$, the restriction of the relative exact $(-1)$-shifted symplectic structure on $\fX$ to $X$ defines such an element $s$; the local structure is provided by the shifted Darboux theorems of \cite{BBJ,BouazizGrojnowski,BBBBJ,ParkSymplectic}. The advantage of (relative) d-critical structures over (relative) exact $(-1)$-shifted symplectic structures is that the former are obviously functorial with respect to smooth maps, which is useful in extending our constructions to higher Artin stacks.

Besides the vanishing cycle functor, the perverse pullback functor recovers most of the perverse t-exact functors that appear in sheaf theory:
\begin{itemize}
    \item Let $X = E$ be a vector bundle over a scheme $B$. Then the composite $u_E \colon E^{\vee} \to B \xhookrightarrow{0} E$ naturally carries a relative exact d-critical structure, together with a canonical choice of square root $K_{E^{\vee} / E} \cong \det(E |_E)^{\otimes 2}$. 
    In this situation, the perverse pullback functor 
    \[
    u_E^{\varphi} \colon \Dbc(E) \to \Dbc(E^{\vee})
    \]
    recovers the \emph{Fourier--Sato transform} \cite[Section 3.7]{KashiwaraSchapira} up to shift. 

    \item Let $Y \hookrightarrow B$ be a closed immersion between smooth varieties. In this situation, the composite
    $\pi \colon \mathrm{N}^*(Y/B) \to Y \to B$ naturally carries a relative exact d-critical structure with a canonical choice of square root $K_{\mathrm{N}^*(Y/B)/ B} \cong K_{Y/B} |_{\mathrm{N}^*(Y/B)}^{\otimes 2}$. 
    Then the perverse pullback functor 
    \[
    \pi^{\varphi} \colon \Dbc(B) \to \Dbc(\mathrm{N}^*(Y/B))
    \]
    recovers the \emph{microlocalization functor}  in \cite[Section 4.3]{KashiwaraSchapira}.
    In particular, the composite
    \[
     u_{\mathrm{N}^*(Y/B)}^{\varphi}\circ \pi^{\varphi} \colon \Dbc(B) \to \Dbc(\mathrm{N}(Y/B))
    \]
    recovers the \emph{specialization functor}. 
    See the next paragraph for more details on the microlocal 
    nature of the perverse pullback functor.

\end{itemize}

\subsection*{Lagrangian microlocalization}

For a complex manifold $B$ and a complex submanifold $Y\subset B$ Kashiwara and Schapira \cite[Section 4.3]{KashiwaraSchapira} defined the microlocalization functor $\mu_{Y/B}\colon \Dbc(B)\rightarrow \Dbc(\rN^*(Y/B))$ which is perverse $t$-exact. The definition of the microlocalization functor $\mu_{Y/B}$ was extended in \cite{Schefers} to the case when $Y\rightarrow B$ is a quasi-smooth closed immersion, and independently in \cite{KhanKinjo} when $Y\rightarrow B$ is a morphism of derived Artin stacks locally of finite presentation.

One can interpret the perverse pullback functor as a Lagrangian version of the microlocalization functor as follows. The starting point for this point of view is given by the following results:
\begin{itemize}
    \item By \cite[Corollary 3.1.3]{ParkSymplectic} a morphism $X\rightarrow B$ with a relative exact $(-1)$-shifted symplectic structure is the same as an exact Lagrangian structure on a morphism $X\rightarrow \T^* B$.
    \item For a complex symplectic manifold $(S, \omega)$ equipped with a $\Gm$-action which acts on $\omega$ with weight $1$ there is a category $\bPerv(S)$ of perverse microsheaves constructed in \cite{Waschkies,CKNS}. Moreover, if $S=\T^* B$ is the cotangent bundle of a complex manifold $B$, there is an equivalence $\bPerv(\T^* B)\cong \Perv_{K_B^{1/2}}(B)$, the category of perverse sheaves on $B$ twisted by the gerbe of square roots of the canonical bundle $K_B$.
\end{itemize}

\begin{mainconjecture}\label{conj:Lagrangianmicrolocalization}
For a complex $0$-shifted symplectic derived Artin stack $(S, \omega)$ equipped with a $\Gm$-action which acts on $\omega$ with weight $1$ there is a category $\bPerv(S)$ of perverse microsheaves on $S$. For a morphism $f\colon L\rightarrow S$ equipped with an exact Lagrangian structure there is a Lagrangian microlocalization functor
\[\mu^{\mathrm{Lag}}_{L/S}\colon \bPerv(S)\longrightarrow \bPerv(\T^* L).\]
Moreover, for $S=\T^* B$ there is an equivalence $\bPerv(\T^* B)\cong \Perv_{K_B^{1/2}}(B)$ and under this equivalence $\mu^{\mathrm{Lag}}_{L/S}$ is equivalent to the perverse pullback along the composite $L\rightarrow \T^* B\rightarrow B$.
\end{mainconjecture}

The Lagrangian microlocalization functor is closely related to the notion of a sheaf quantization of a Lagrangian submanifold as in \cite{NadlerShende}: a sheaf of quantization of $L\rightarrow S$ is an object $\cL_L\in \bPerv(S)$ such that $\mu^{\mathrm{Lag}}_{L/S}(\cL_L)$ is a (twisted) rank 1 local system on $L$.

In a follow-up paper we use the formalism of perverse pullbacks to construct shifted Lagrangian classes as in \cite[Conjecture 1.1]{JoyceSafronov} and \cite[Conjecture 5.18]{ABB}: for a complex oriented $(-1)$-shifted symplectic derived Artin stack $S$ and a morphism $f\colon L\rightarrow S$ equipped with an oriented Lagrangian structure there is a \emph{$(-1)$-shifted Lagrangian class}
\[[L]^{\mathrm{Lag}}\colon f^* \varphi_S\longrightarrow \omega_L[-\dim(L)]\]
generalizing the virtual classes in Borel--Moore homology defined in \cite{BorisovJoyce} and \cite{OhThomas} for $S=\pt$. This may be viewed as a decategorified and $(-1)$-shifted version of \cref{conj:Lagrangianmicrolocalization}.

Given a morphism $Y\rightarrow B$ of derived Artin stacks locally of finite presentation we obtain a $0$-shifted Lagrangian morphism $\pi\colon L=\rN^*(Y/B)\rightarrow S=\T^* B$. We expect that the perverse pullback functor in this case coincides with the derived microlocalization functor, i.e. we expect that there is a natural isomorphism
\[\mu^{\mathrm{Lag}}_{\rN^*(Y/B)/\T^* B}\cong \mu_{Y/B}.\]

\subsection*{Conventions}

Throughout the paper we work with schemes over a field $k$ assumed to be of characteristic different from $2$ and with $i\in k$ satisfying $i^2=-1$. Starting from \cref{sec:pervpull}, $k$ will be the field $\C$ of complex numbers. We also fix a commutative ring $R$ of coefficients for our sheaves. We denote by $\Gpd_\infty$ the $\infty$-category of $\infty$-groupoids.

\subsection*{Acknowledgements}

AAK acknowledges support from the grants AS-CDA-112-M01 and NSTC 112-2628-M-001-0062030.
TK was supported by JSPS KAKENHI Grant Number 25K17229.
HP was supported by Korea Institute for Advanced Study (SG089201).

\section{Schemes and stacks}

\subsection{Stacks}\label{sect:stacks}

Recall the symmetric monoidal $\infty$-category $\PrSt$ of stable presentable $\infty$-categories with colimit-preserving functors as morphisms as in \cite[\S 4.8.1]{LurieHA}. Let $\Mod_R\in\PrSt$ be the derived $\infty$-category of chain complexes of $R$-modules. Let
\[\PrSt_R = \Mod_{\Mod_R}(\PrSt)\]
be the $\infty$-category of $R$-linear stable presentable $\infty$-categories.

Let $\Sch$ be the category of schemes over $k$ and $\Sch^{\sep\ft}\subset \Sch$ the subcategory of separated schemes of finite type. Recall the following notion of higher Artin stacks, as in \cite[Definition 1.3.3.1]{HAGII}.

\begin{definition}\label{def:Artinstacks}
A \defterm{stack} is a presheaf of $\infty$-groupoids on $\Sch$ satisfying \v{C}ech descent along \'etale surjections; we denote by $\Stk$ the $\infty$-category of such. We define $n$-geometric stacks (for $n\geq -1$) inductively:
\begin{enumerate}
    \item A stack is \defterm{$(-1)$-geometric} if it is (representable by) a scheme.
    \item A morphism of stacks $f\colon X \to Y$ is \defterm{$n$-geometric} (for $n\geq -1$) if for every scheme $V$ and every morphism $V \to Y$, the fibered product $X \times_Y V$ is $n$-geometric.
    \item An $n$-geometric morphism $f\colon X \to Y$ is \defterm{smooth}, resp. \defterm{smooth surjective}, if for every commutative diagram
    \[
    \xymatrix{
    U \ar[r] \ar[d] & X \ar[d] \\
    V \ar[r] & Y
    }
    \]
    with $U,V$ schemes and $U\rightarrow X\times_Y V$ smooth surjective as an $(n-1)$-geometric morphism, the morphism $U\rightarrow V$ is a smooth, resp. smooth surjective, morphism of schemes.
    \item A stack $X$ is \defterm{$n$-geometric} if it satisfies the following properties:
    \begin{enumerate}
        \item Its diagonal $\Delta_X\colon X \to X \times X$ is $(n-1)$-geometric.
        \item There exists a scheme $U$ and a morphism\footnote{which is automatically $(n-1)$-geometric when the diagonal $\Delta_X$ is $(n-1)$-geometric.} $p\colon U \to X$ which is a smooth surjection.
    \end{enumerate}
\end{enumerate}
We additionally introduce the following terminology:
\begin{itemize}
    \item A stack is \defterm{higher Artin} if it is $n$-geometric for some $n$; we denote by $\Art\subset\Stk$ the $\infty$-category of such.
    \item A morphism is \defterm{geometric} if it is $n$-geometric for some $n$.
    \item A higher Artin stack is \defterm{Artin} if the corresponding functor is valued in groupoids. More generally, a higher Artin stack is \defterm{$n$-Artin} if the corresponding functor is valued in $n$-groupoids.
    \item A morphism is \defterm{$n$-representable} if it is representable by $n$-Artin stacks. It is \defterm{schematic} if it is representable by schemes.
\end{itemize}
\end{definition}

For a pair of stacks $X, Y$ we denote by $\Map(X, Y)$ the mapping stack whose $S$-points are given by
\[\Map(X, Y)(S) = \Hom_{\Stk}(S\times X, Y).\]

\begin{definition}
A stack $X$ is \defterm{locally of finite type} if for every cofiltered system $\{S_i\}$ of affine schemes the natural morphism
\[\colim_i X(S_i)\longrightarrow X(\lim_i S_i)\]
is an isomorphism.
\end{definition}

\begin{remark}
For schemes over a field $k$ the above definition coincides with the usual notion of schemes locally of finite type, see \cite[01ZC]{Stacks}.
\end{remark}

Let $\Stk^{\lft}\subset \Stk$ be the full subcategory of stacks locally of finite type. We write $\Art^\lft \subset \Art$ for the full subcategory spanned by $X\in\Art$ locally of finite type.

\subsection{Extensions to stacks}\label{sect:Stkextensions}

Let $\cV$ be an $\infty$-category with limits. In this section we describe several mechanisms to extend invariants of schemes valued in $\cV$ to invariants of stacks. Equip the category of schemes $\Sch$ with the \'etale topology. We will encounter objects on schemes functorial only with respect to smooth morphisms. Namely, consider the subcategory $\Sch_{\smooth}\subset \Sch$ whose objects are schemes and whose morphisms are smooth. Let $\Art_{\smooth}\subset \Art$ be a similarly defined subcategory for higher Artin stacks.

To formalize invariants of relative schemes, let $\Fun(\Delta^1, \Sch)$ be the category whose objects are morphisms of schemes and whose morphisms $p\colon (X'\xrightarrow{\pi'} B')\rightarrow (X\xrightarrow{\pi} B)$ are commutative diagrams
\[
\xymatrix{
X' \ar^{\overline{p}}[r] \ar^{\pi'}[d] & X \ar^{\pi}[d] \\
B' \ar^{p}[r] & B.
}
\]
We extend the \'etale topology to $\Fun(\Delta^1, \Sch)$ by declaring covering families $\{(X_i\rightarrow B_i)\rightarrow (X\rightarrow B)\}$ to be families of morphisms such that both $\{B_i\rightarrow B\}$ and $\{X_i\rightarrow X\times_B B_i\}$ are \'etale covers.

Let
\[\Fun(\Delta^1, \Sch^{\sep\ft})_{0\smooth,1\smooth}\subset \Fun(\Delta^1, \Sch)_{0\smooth}\subset \Fun(\Delta^1, \Sch)\]
be the following subcategories:
\begin{itemize}
    \item $\Fun(\Delta^1, \Sch)_{0\smooth}$ has the same objects and morphisms $(X'\rightarrow B')\rightarrow (X\rightarrow B)$ such that $X'\rightarrow X\times_B B'$ is smooth.
    \item $\Fun(\Delta^1, \Sch^{\sep\ft})_{0\smooth,1\smooth}$ has the same objects and morphisms $(X'\rightarrow B')\rightarrow (X\rightarrow B)$ such that both $B'\rightarrow B$ and $X'\rightarrow X\times_B B'$ are smooth.
\end{itemize}

Let $\Fun(\Delta^1, \Stk)$ be the $\infty$-category of morphisms of stacks and let
\[\Fun(\Delta^1, \Art)_{0\smooth, 1\smooth}\subset \Fun(\Delta^1, \Stk)^{\geometric}_{0\smooth}\subset \Fun(\Delta^1, \Stk)^{\geometric}\subset \Fun(\Delta^1, \Stk)\]
be the following subcategories:
\begin{itemize}
    \item $\Fun(\Delta^1, \Stk)^{\geometric}$ is the full subcategory whose objects are geometric morphisms of stacks,
    \item $\Fun(\Delta^1, \Stk)^{\geometric}_{0\smooth}$ has the same objects and morphisms $(X'\rightarrow B')\rightarrow (X\rightarrow B)$ such that $X'\rightarrow X\times_B B'$ is smooth.
    \item $\Fun(\Delta^1, \Art)_{0\smooth, 1\smooth}$ has objects morphisms of higher Artin stacks and morphisms $(X'\rightarrow B')\rightarrow (X\rightarrow B)$ such that both $B'\rightarrow B$ and $X'\rightarrow X\times_B B'$ are smooth.
\end{itemize}

We define the extensions of invariants of schemes to invariants of stacks as follows.

\begin{itemize}
    \item For a functor
    \[F\colon \Sch^{\op}\longrightarrow \cV\]
    satisfying \'etale descent we define its value on stacks $X\in\Stk$ by a right Kan extension:
    \begin{equation}\label{eq:ShvStk}
    F^{\lis}(X) = \lim_{(S, s)} F(S),
    \end{equation}
    where the limit is taken over the $\infty$-category $\Sch_{/X}$ of pairs $(S,s)$ with $S\in\Sch$ and $s\colon S \to X$ a morphism. By \cite[Proposition 3.2.5(i), Corollary 3.2.6(i)]{KhanWeavelisse} this determines an inverse to the restriction functor from $\cV$-valued sheaves on $\Stk$ to $\cV$-valued sheaves on $\Sch$.

    \item For a functor
    \[F\colon \Sch^{\op}_{\smooth}\longrightarrow \cV\]
    satisfying \'etale descent we define its value on higher Artin stacks $X\in\Art$ by a right Kan extension:
    \begin{equation}\label{eq:ShvArt}
    F^\lis(X) = \lim_{(S, s)} F(S),
    \end{equation}
    where the limit is taken over the full subcategory $\Sch^{\smooth/X}_{\smooth}\subset (\Art_{\smooth})_{/X}$ consisting of pairs $(S,s)$ with $S\in\Sch$ and $s\colon S \to X$ a smooth morphism.

    \item For a functor
    \[F\colon \Fun(\Delta^1, \Sch)_{0\smooth}^{\op}\longrightarrow \cV\]
    satisfying \'etale descent we define its value on geometric morphisms $(X\rightarrow B)\in\Fun(\Delta^1, \Stk)^{\geometric}$ of stacks by a right Kan extension:
    \begin{equation}\label{eq:ShvStkrel}
    F^\lis(X\rightarrow B) = \lim_{(X'\rightarrow B', p)} F(X'\rightarrow B'),
    \end{equation}
    where the limit is taken over the full subcategory of $(\Fun(\Delta^1, \Stk)^{\geometric}_{0\smooth})_{/(X\rightarrow B)}$ of objects $(X'\rightarrow B', p)$ where $X'\rightarrow B'$ is a morphism of schemes and $X'\rightarrow X\times_B B'$ is smooth.

    \item For a functor
    \[F\colon \Fun(\Delta^1, \Sch)_{0\smooth,1\smooth}^{\op}\longrightarrow \cV\]
    satisfying \'etale descent we define its value on a morphism $(X\rightarrow B)\in \Fun(\Delta^1, \Art)$ of higher Artin stacks by a right Kan extension:
    \begin{equation}\label{eq:ShvArtrel}
    F^\lis(X\rightarrow B) = \lim_{(X'\rightarrow B', p)} F(X'\rightarrow B'),
    \end{equation}
    where the limit is taken over the full subcategory of $(\Fun(\Delta^1, \Art)_{0\smooth, 1\smooth})_{/(X\rightarrow B)}$ of objects $(X'\rightarrow B', p)$ where $X'\rightarrow B'$ is a morphism of schemes and both $B'\rightarrow B$ and $X'\rightarrow X\times_B B'$ are smooth.
\end{itemize}

\begin{remark}
By \cite[Corollary 3.2.7]{KhanWeavelisse}, the definitions \eqref{eq:ShvStk} and \eqref{eq:ShvArt} (and, similarly, \eqref{eq:ShvStkrel} and \eqref{eq:ShvArtrel}) of the extensions above are compatible, i.e. for a sheaf $F\colon \Sch^{\op}\rightarrow \cV$ and a higher Artin stack $X\in\Art$ the restriction morphism
\[\lim_{(S, s)\in\Sch_{/X}} F(S)\longrightarrow \lim_{(S, s)\in\Sch^{\smooth/X}_{\smooth}} F(S)\]
is an isomorphism. Thus, the notation $F^\lis(X)$ for the extension is unambiguous.
\end{remark}

\begin{remark}
Given invariants defined only on the subcategory $\Sch^{\lft} \subset \Sch$ of schemes locally of finite type, or even $\Sch^{\sep\ft}$, we may similarly extend to stacks locally of finite type.
\end{remark}

\subsection{Derived stacks}

Let us briefly introduce the theory of derived stacks. The $\infty$-category $\dAff$ of \defterm{derived affine schemes} is opposite to the $\infty$-category $\CAlg^\Delta_k$ of simplicial commutative $k$-algebras via the $\Spec$ functor. A \defterm{derived stack} is a presheaf of $\infty$-groupoids on $\dAff$ satisfying \v{C}ech descent along \'etale surjections. We denote by $\dStk$ the $\infty$-category of such. There is a fully faithful inclusion functor $\Stk\rightarrow \dStk$ whose right adjoint $(-)^{\cl}\colon \dStk\rightarrow\Stk$ is given by passing to the underlying classical stack. The notion of a geometric morphism of derived stacks is defined analogously to \cref{def:Artinstacks}.

For a derived affine scheme $S=\Spec A$ we define $\QCoh(S)$ to be the $\infty$-category of dg $A$-modules. This $\infty$-category is stable and presentable and has a natural symmetric monoidal structure. We refer to objects of $\QCoh(S)$ as quasi-coherent complexes. Denote by $\Perf(S)\subset \QCoh(S)$ the full subcategory of perfect complexes, i.e. dualizable objects in $\QCoh(S)$. For a morphism of derived affine schemes $f\colon X\rightarrow Y$ we have a symmetric monoidal pullback functor $f^*\colon \QCoh(Y)\rightarrow \QCoh(X)$. This assignment defines a functor
\[\QCoh\colon \dAff^{\op}\longrightarrow \PrSt_k\]
which satisfies \'etale descent. By right Kan extension we obtain the functor
\[\QCoh\colon \dStk^{\op}\longrightarrow \PrSt_k.\]

If $f\colon X\rightarrow Y$ is a geometric morphism of derived stacks, we have the \defterm{relative cotangent complex} $\bL_{X/Y}\in\QCoh(X)$ which is functorial in the following way. For a commutative diagram
\[
\xymatrix{
X' \ar[r] \ar[d] & X \ar[d] \\
Y' \ar[r] & Y
}
\]
of derived stacks with $X\rightarrow Y$ and $X'\rightarrow Y'\times_{Y} X$ geometric we have a pullback morphism
\begin{equation}\label{eq:cotangentpullback}
(X'\rightarrow X)^*\bL_{X/Y}\longrightarrow \bL_{X'/Y'}
\end{equation}
which is an isomorphism if the diagram is Cartesian. Moreover, for geometric morphisms $X\rightarrow Y\rightarrow Z$ of derived stacks, we have a fiber sequence
\begin{equation}\label{eq:cotangentsequence}
(X\rightarrow Y)^*\bL_{Y/Z}\longrightarrow \bL_{X/Z}\longrightarrow \bL_{X/Y}.
\end{equation}
We refer to \cite[Lemma B.10.13]{CHS} for more details on the functoriality of the relative cotangent complex.

For a morphism of classical stacks $X\rightarrow Y$ we define the relative cotangent complex by viewing them as derived stacks. For instance, if $X\rightarrow Y$ is a morphism of schemes, we have $\Omega^1_{X/Y} = h^0(\bL_{X/Y})$, the sheaf of relative K\"ahler differentials. We have the following basic fact.

\begin{lemma}
Let $X\rightarrow Y$ be a geometric morphism of derived stacks. Then the cofiber of the pullback morphism
\[(X^{\cl}\rightarrow X)^* \bL_{X/Y}\longrightarrow \bL_{X^\cl/Y^\cl},\]
i.e. $\bL_{X^{\cl}/X\times_Y Y^{\cl}}$, is 2-connective.
\end{lemma}
\begin{proof}
Consider the diagram
\[
\xymatrix{
(X^{\cl}\rightarrow Y)^* \bL_Y \ar[r] \ar[d] & (X^{\cl}\rightarrow X)^* \bL_X \ar[r] \ar[d] & (X^{\cl}\rightarrow X)^* \bL_{X/Y} \ar[d] \\
(X^{\cl}\rightarrow Y^{\cl})^* \bL_{Y^\cl} \ar[r] \ar[d] & \bL_{X^{\cl}} \ar[r] \ar[d] & \bL_{X^{\cl}/Y^{\cl}} \ar[d] \\
(X^{\cl}\rightarrow Y^{\cl})^* \bL_{Y^{\cl}/Y} \ar[r] & \bL_{X^{\cl}/X} \ar[r] & \bL_{X^{\cl}/X\times_Y Y^{\cl}}.
}
\]
In this diagram all columns and the first two rows are cofiber sequences, so the bottom row is a cofiber sequence. Therefore, the claim reduces to the case $Y=\pt$.

Let $i\colon X^{\cl}\rightarrow X$ be the natural morphism. By definition, for any derived affine scheme $S$ equipped with a morphism $f\colon S\rightarrow X^{\cl}$ and a connective quasi-coherent complex $M\in\QCoh(S)$ we have
\[\Map_{\QCoh(S)}((i\circ f)^* \bL_X, M)\cong \Map_{\dStk_{S/}}(S[M], X).\]
The right-hand side preserves colimits in $X$. Writing $X=\colim_\alpha X_\alpha$ for a diagram of affine derived schemes $\{X_\alpha\}_\alpha$, so that $X^{\cl} = \colim_\alpha X^{\cl}_\alpha$, we get a commutative diagram
\[
\xymatrix{
\colim_\alpha\Map_{\QCoh(S)}(f_\alpha^*\bL_{X_\alpha^{\cl}}, M) \ar^-{\sim}[r] \ar[d] &  \Map_{\QCoh(S)}(f^*\bL_{X^{\cl}}, M) \ar[d] \\
\colim_\alpha\Map_{\QCoh(S)}((i_\alpha\circ f_\alpha)^*\bL_{X_\alpha}, M) \ar^-{\sim}[r] &  \Map_{\QCoh(S)}(f^*\bL_X, M),
}
\]
where $f_\alpha\colon S\rightarrow X^{\cl}_\alpha$ and $i_\alpha\colon X^{\cl}_\alpha\rightarrow X^{\cl}$. As colimits in $\Gpd_\infty$ are universal, we get
\[\colim_\alpha \Map_{\QCoh(S)}(f_\alpha^* \bL_{X^{\cl}_\alpha/X_\alpha}, M)\cong \Map_{\QCoh(S)}(\bL_{X^{\cl}/X}, M).\]
The $2$-connectivity of $\bL_{X^{\cl}/X}$ is equivalent to the contractibility of the right-hand side for every $M\in\QCoh(S)$ concentrated in cohomological degrees $[-1, 0]$. Thus, by the above isomorphism the claim reduces to the case of an affine derived scheme $X=\Spec A$. Since the cofiber of $A\rightarrow \rH^0(A)$ is $2$-connective, we get that $\bL_{X^{\cl}/X}=\bL_{\rH^0(A)/A}$ is $2$-connective by \cite[Corollary 7.4.3.2]{LurieHA}.
\end{proof}

We say that a geometric morphism of derived stacks $f \colon X \to Y$ is \defterm{locally of finite presentation}, or \emph{lfp} for short, if the induced morphism of classical truncations $f^\cl \colon X^\cl \to Y^\cl$ is locally of finite presentation in the classical sense, and the relative cotangent complex $\bL_{X/Y}$ is perfect. Note that a locally of finite presentation morphism of classical stacks need not be lfp as a morphism of derived stacks.

The (relative) \defterm{dimension} of an lfp morphism $f \colon X \to Y$, denoted $\dim(X/Y)$, is the rank of the perfect complex $\bL_{X/Y}$.
Note that when $f$ is smooth, we have $\dim(X/Y) = \dim(X^\cl/Y^\cl)$.

\subsection{Determinant lines}

For a stack $X$ we denote by $\Pic^{\gr}(X)$ the $\infty$-category of pairs $(L, \alpha)$ of a line bundle $L$ on $X$ and a locally constant function $\alpha\colon X\rightarrow \Z/2\Z$. It is a Picard $\infty$-groupoid (or $E_\infty$-group) with the symmetric monoidal structure given by $(L_1, \alpha_1)\otimes (L_2, \alpha_2) = (L_1\otimes L_2, \alpha_1 + \alpha_2)$ with braiding involving the Koszul sign. We fix an identification $(L, \alpha)^\vee\cong (L^\vee, \alpha)$, under which the evaluation pairing $(L, \alpha)\otimes (L, \alpha)^\vee\rightarrow (\cO_X, 0)$ is identified with the composite $(L, \alpha)\otimes (L^\vee, \alpha) = (L\otimes L^\vee, 0)\cong (\cO_X, 0)$. If $l$ is a nonvanishing section of a line bundle $L$, we denote by $l^{-1}$ the section of $L^\vee$ which pairs to $1$ with $l$.

For a perfect complex $E\in\Perf(X)$ we denote by $\det(E)\in\Pic^{\gr}(X)$ the $\Z/2\Z$-graded determinant line bundle \cite{DeligneDeterminant,KM76}. Our conventions follow \cite[Section 2.2]{KPS}. The main isomorphisms we will use are as follows:
\begin{itemize}
    \item For a fiber sequence
    \[\Delta\colon E_1\longrightarrow E_2\longrightarrow E_3\]
    there is an isomorphism \[i(\Delta)\colon \det(E_1)\otimes \det(E_3)\xrightarrow{\sim} \det(E_2).\]
    \item For a perfect complex $E$ there is an isomorphism
    \[\iota_E\colon \det(E^\vee)\xrightarrow{\sim} \det(E)^\vee.\]
\end{itemize}
For a perfect complex $E$ we have the fiber sequence $\Delta_E\colon E\rightarrow 0\rightarrow E[1]$ given by the rotation of $E\xrightarrow{\id} E\rightarrow 0$. It gives rise to an isomorphism
\[\theta_E\colon \det(E[1])\cong \det(E)^\vee\]
so that the composite
\[\det(E)\otimes \det(E[1])\xrightarrow{\id\otimes \theta_E} \det(E)\otimes \det(E)^\vee\xrightarrow{\ev} \cO_X\]
coincides with $i(\Delta_E)$.

We will use the following explicit descriptions:
\begin{itemize}
    \item If $E$ is a trivial vector bundle with a basis of sections $\{s_1, \dots, s_n\}$, the determinant line bundle $\det(E)$ has a nonvanishing section $s_1\wedge \dots\wedge s_n$.
    \item Given a short exact sequence
    \[\Delta\colon 0\longrightarrow E_1\longrightarrow E_2\longrightarrow E_3\longrightarrow 0\]
    of trivial vector bundles with $\{s_1, \dots, s_n\}$ a basis of sections of $E_1$ and $\{s_1, \dots, s_n, s_{n+1}, \dots, s_{n+m}\}$ a basis of sections of $E_2$, then
    \[i(\Delta)(s_1\wedge \dots\wedge s_n\otimes s_{n+1}\wedge \dots\wedge s_{n+m}) = s_1\wedge \dots\wedge s_{n+m}.\]
    \item Again, if $E$ is a trivial vector bundle with a basis of sections $\{s_1, \dots, s_n\}$ and with $\{s^1, \dots, s^n\}$ the dual basis of $E^\vee$, then
    \[\iota_E(s^1\wedge \dots \wedge s^n) = (-1)^{n(n-1)/2} (s_1\wedge \dots \wedge s_n)^{-1}.\]
\end{itemize}

Given a smooth morphism of stacks $f\colon X\rightarrow Y$, the relative cotangent complex $\bL_{X/Y}$ is perfect of Tor-amplitude $\geq 0$ and the \defterm{canonical bundle} is $K_{X/Y} = \det(\bL_{X/Y})\in\Pic^{\gr}(X)$.

\subsection{Differential forms}\label{sect:differentialforms}

Let $X\rightarrow B$ be a geometric morphism of derived stacks. We have the relative cotangent complex $\bL_{X/B}\in\QCoh(X)$ equipped with the de Rham differential
\[d_B\colon \Gamma(X, \cO_X)\rightarrow \Gamma(X, \bL_{X/B})\]
defined as in \cite{CPTVV,CHS,ParkSymplectic}. Recall the following notions from \cite{PTVV}.

\begin{definition}
Let $n\in\Z$.
\begin{itemize}
    \item For $p\in \Z$ the space of \defterm{relative $p$-forms of degree $n$ on $X\rightarrow B$} is
    \[\cA^p(X/B, n) = \Map_{\QCoh(X)}(\cO_X, \wedge^p \bL_{X/B}[n]).\]
    \item The space $\cA^{2, \ex}(X/B, n)$ of \defterm{relative exact two-forms of degree $n$ on $X\rightarrow B$} is the homotopy fiber of $d_B\colon \cA^0(X/B, n+1)\rightarrow \cA^1(X/B, n+1)$ at the zero form.
\end{itemize}
\end{definition}

For a commutative diagram of derived stacks
\[
\xymatrix{
X' \ar[r] \ar[d] & X \ar[d] \\
B' \ar[r] & B
}
\]
we have a natural pullback morphism $\cA^{2, \ex}(X/B, n)\rightarrow \cA^{2, \ex}(X'/B', n)$. We have natural isomorphisms
\[\Omega\cA^p(X/B, n+1)\cong \cA^p(X/B, n),\qquad \Omega\cA^{2, \ex}(X/B, n+1)\cong \cA^{2, \ex}(X/B, n)\]
and a forgetful map $d_B\colon \cA^1(X/B, n)\rightarrow \cA^{2, \ex}(X/B, n)$. We will now give two examples of relative exact two-forms we will encounter.

\subsubsection{Case $n=0$}
Given a quasi-coherent complex $V\in \QCoh(X)$ over a stack $X$ recall the \defterm{total space} $\Tot_X(V)$ which is a stack over $X$ whose $S$-points are given by pairs $(f, \alpha)$ of a morphism $f\colon S\rightarrow X$ of stacks and a morphism $\alpha\colon \cO_S\rightarrow f^*V$ in $\QCoh(S)$.

\begin{definition}
Let $X\rightarrow B$ be a geometric morphism of stacks. The \defterm{relative cotangent bundle} is $\T^*(X/B) = \Tot_X(\bL_{X/B})$.
\end{definition}

The relative cotangent bundle is stable under base change: given a Cartesian square of stacks
\[
\xymatrix{
X' \ar[r] \ar[d] & X \ar[d] \\
B' \ar[r] & B
}
\]
which is Tor-independent (e.g. either $X\rightarrow B$ or $B'\rightarrow B$ is flat), so that it is Cartesian when regarded as a square of \emph{derived} stacks, there is an isomorphism $\T^*(X/B)\times_B B'\cong \T^*(X'/B')$ of $B'$-stacks constructed using \eqref{eq:cotangentpullback}. Moreover, for a composite $X\rightarrow B_1\rightarrow B_2$ of geometric morphisms we have a Cartesian square
\begin{equation}\label{eq:cotangentCartesian}
\xymatrix{
\T^*(B_1/B_2)\times_{B_1} X \ar[r] \ar[d] & X \ar[d] \\
\T^*(X/B_2) \ar[r] & \T^*(X/B_1)
}
\end{equation}

The relative cotangent bundle carries the so-called \defterm{Liouville one-form} $\lambda_{X/B}\in\cA^1(\T^*(X/B)/B, 0)$ defined as follows (see e.g. \cite{PTVV,CalaqueCotangent}). For a morphism $(f, \alpha)\colon S\rightarrow \T^*(X/B)$ we set
\[(f, \alpha)^\ast \lambda_{X/B}\colon \cO_S\rightarrow f^*\bL_{X/B}\rightarrow \bL_{S/B}.\]

We have the following particular cases:
\begin{enumerate}
    \item If $X\rightarrow B$ is a smooth schematic morphism, $\bL_{X/B}$ has Tor-amplitude $[0,0]$, so $\T^*(X/B)\rightarrow X$ is a vector bundle.
    \item If $X\rightarrow B$ is smooth and geometric, $\bL_{X/B}$ has Tor-amplitude $\geq 0$ in which case $\T^*(X/B)\rightarrow X$ is a cone. In this case the zero section $X\rightarrow \T^*(X/B)$ is a closed immersion.
\end{enumerate}

\begin{example}
Let $X\rightarrow B$ be a smooth morphism of schemes. Vector fields on $X\rightarrow B$ give rise to functions on $\T^*(X/B)$. So, given \'etale coordinates $\{q_1, \dots, q_n\}$ on $X\rightarrow B$ we obtain \'etale coordinates $\{q_1, \dots, q_n, p_1, \dots, p_n\}$ on $\T^*(X/B)\rightarrow B$, so that $p_i$ corresponds to $\frac{\partial}{\partial q_i}$. In these coordinates we have
\[\lambda_{X/B} = \sum_{i=1}^n p_i d_B q_i.\]
\end{example}

\subsubsection{Case $n=-1$}\label{sect:criticallocus}

\begin{definition}\label{def-relative-crit}
Let $U\rightarrow B$ be a flat geometric morphism of stacks equipped with a function $f\colon U\rightarrow \bA^1$. The \defterm{relative critical locus} is the fiber product
\begin{equation}\label{eq:criticalpullback}
\xymatrix{
\Crit_{U/B}(f) \ar[r] \ar[d] & U \ar^{\Gamma_0}[d] \\
U \ar^-{\Gamma_{d_B f}}[r] & \T^*(U/B),
}
\end{equation}
where $\Gamma_0\colon U\rightarrow \T^*(U/B)$ is the zero section and $\Gamma_{d_B f}\colon U\rightarrow \T^*(U/B)$ is the graph of the relative one-form $d_B f$.    
\end{definition}

By construction we have $\Gamma_{d_B f}^\ast \lambda_{U/B} \sim d_B f$ in $\cA^1(U/B, n)$. Thus, $f$ provides a nullhomotopy $h_f\colon \Gamma_{d_B f}^\ast \lambda_{U/B}\sim 0$ in $\cA^{2, \ex}(U/B, 0)$. Similarly, $\Gamma_0^\ast \lambda_{U/B}\sim 0$. The difference of the two nullhomotopies on $\Crit_{U/B}(f)$ provides an element
\[s_f = h_f-h_0\in\Omega \cA^{2, \ex}(\Crit_{U/B}(f)/B, 0)\cong \cA^{2, \ex}(\Crit_{U/B}(f)/B, -1).\]
When we need to specify the base scheme, we also denote it by $s_{f, B}$.

\subsection{Immersions}

In this preliminary section we collect some useful definitions and facts about immersions. Throughout we fix a scheme $B$ over $k$. We will use smoothings of schemes, which are immersions of a given scheme into a smooth scheme.

\begin{definition}
Let $X$ be a $B$-scheme and $x\in X$ a point. An immersion $\imath\colon X\rightarrow U$ into a smooth $B$-scheme is \defterm{minimal at $x$} if the pullback $\imath^\ast\colon \Omega^1_{U/B, \imath(x)}\rightarrow \Omega^1_{X/B, x}$ is an isomorphism.
\end{definition}

\begin{remark}
For any immersion $\imath\colon X\rightarrow U$ the map $\imath^\ast\colon \Omega^1_{U/B, \imath(x)}\rightarrow \Omega^1_{X/B, x}$ is surjective, so that $\dim \Omega^1_{U/B, \imath(x)}\geq \dim \Omega^1_{X/B, x}$. The minimality assumption is that this is an equality.
\end{remark}

The local existence of minimal immersions is shown in \cite[Tag 0CBL]{Stacks}. Moreover, locally, every immersion can be replaced by a minimal one.

\begin{proposition}\label{prop:minimalimmersion}
Let $X$ be a $B$-scheme and $x\in X$ a point. Let $\imath\colon X\rightarrow U$ be an immersion into a smooth $B$-scheme. Then there is an open neighborhood $X^\circ\subset X$ of $x$ and a factorization of $X^\circ\rightarrow X\xrightarrow{\imath} U$ as $X^\circ\xrightarrow{\imath^\circ} V\rightarrow U$, where $V$ is a smooth $B$-scheme and $\imath^\circ$ is a closed immersion minimal at $x$.
\end{proposition}
\begin{proof}
Let $r=\dim\Omega^1_{X/B, x}$. By \cite[Tag 0CBK]{Stacks} we may find an open neighborhood $X'\subset X$ of $x$ and a diagram
\[
\xymatrix{
X' \ar^{\pi}[r] \ar[d] & \bA^r_B \ar[d] \\
X \ar^{\imath}[r] & U
}
\]
with $\pi$ unramified. By \cite[Corollaire 18.4.7]{EGA44} we may find an open neighborhood $X^\circ\subset X'$ of $x$ and a diagram
\[
\xymatrix{
X^\circ \ar^{\imath^\circ}[r] \ar[d] & V \ar[d] \\
X' \ar[r] & \bA^r_B
}
\]
with $\imath^\circ$ a closed immersion and $V\rightarrow \bA^r_B$ \'etale. In particular, $V\rightarrow B$ is smooth of relative dimension $r$. In particular, $\imath^\circ$ is minimal.
\end{proof}

One can also show the existence of minimal immersions compatibly with smooth morphisms.

\begin{proposition}\label{prop:relativesmoothing}
Let $f\colon X\rightarrow Y$ be a smooth morphism of $B$-schemes with $Y\rightarrow B$ locally of finite type, $x\in X$ a point and $y=f(x)$. Then there are open neighborhoods $X^\circ\subset X$ of $x$ and $Y^\circ\subset Y$ of $y$, a smooth morphism $f^\circ\colon X^\circ\rightarrow Y^\circ$, a smooth morphism $\tilde{f}\colon U\rightarrow V$ of smooth $B$-schemes and closed immersions $\imath\colon X^\circ\rightarrow U$ and $\jmath\colon Y^\circ\rightarrow V$, such that the diagram
\[
\xymatrix{
X \ar^{f}[d] & X^\circ \ar^{\imath}[r] \ar^{f^\circ}[d] \ar[l] & U \ar^{\tilde{f}}[d] \\
Y & Y^\circ \ar^{\jmath}[r] \ar[l] & V
}
\]
commutes and $\imath$ and $\jmath$ are minimal at $x$ and $y$.
\end{proposition}
\begin{proof}
Let $d$ be the relative dimension of $f\colon X\rightarrow Y$ at $x$. By \cite[Tag 054L]{Stacks}, we may find open neighborhoods $X^\circ\subset X$ of $x$ and $Y^\circ\subset Y$ of $y$ together with \'etale morphism $g\colon X^\circ\rightarrow \bA^d\times Y^\circ$ and a smooth morphism $f^\circ\colon X^\circ\rightarrow Y^\circ$ fitting into a commutative diagram
\[
\xymatrix{
X \ar^{f}[d] & X^\circ \ar^-{g}[r] \ar^{f^\circ}[d] \ar[l] & \bA^d\times Y^\circ \ar^{\pi_2}[dl] \\
Y & Y^\circ. \ar[l]
}
\]
By \cite[Tag 0CBL]{Stacks}, by shrinking $Y^\circ$ we may find a closed immersion $\jmath\colon Y^\circ\rightarrow V$ minimal at $y\in Y^\circ$ with $V$ a smooth $B$-scheme. Therefore, we obtain a commutative diagram
\[
\xymatrix{
X \ar^{f}[d] & X^\circ \ar^-{g}[r] \ar^{f^\circ}[d] \ar[l] & \bA^d\times Y^\circ \ar^{\pi_2}[dl] \ar^{\id\times\jmath}[r] & \bA^d\times V \ar^{\pi_2}[dl] \\
Y & Y^\circ \ar[l] \ar^{\jmath}[r] & V.
}
\]
The composite $X^\circ\xrightarrow{g} \bA^d\times Y^\circ\xrightarrow{\id\times \jmath} \bA^d\times V$ is unramified. Therefore, by \cite[Corollaire 18.4.7]{EGA44}, shrinking $X^\circ$ we may factor $X^\circ\rightarrow \bA^d\times V$ as $X^\circ\xrightarrow{\imath} U\xrightarrow{h} \bA^d\times V$, where $\imath$ is a closed immersion and $h$ is \'etale. Thus, we obtain a commutative diagram
\[
\xymatrix{
X \ar^{f}[d] & X^\circ \ar^{\imath}[r] \ar^{f^\circ}[d] \ar[l] & U \ar^-{h}[r] & \bA^d\times V \ar^{\pi_2}[dl] \\
Y & Y^\circ \ar[l] \ar^{\jmath}[r] & V.
}
\]
Let $\tilde{f} \coloneqq \pi_2\circ h\colon U\rightarrow V$ which is smooth, as it is a composite of an \'etale morphism and a projection. By construction $\jmath\colon Y^\circ\rightarrow V$ is minimal at $y$, so that $\dim(\Omega^1_{Y^\circ/B, y})=\dim(\Omega^1_{V/B, \jmath(y)})$. Since $f^\circ$ is smooth of relative dimension $d$ at $x$, we have $\dim(\Omega^1_{X^\circ/B, x}) = \dim(\Omega^1_{Y^\circ/B, y}) + d$. Since $h$ is \'etale, $\pi_2$ is smooth of relative dimension $d$ at $\imath(x)$, i.e. $\dim(\Omega^1_{U/B, \imath(x)}) = \dim(\Omega^1_{V/B, \jmath(y)}) + d$. Therefore, $\dim(\Omega^1_{X^\circ/B, x}) = \dim(\Omega^1_{U/B, \imath(x)})$.
\end{proof}

\section{Constructible sheaves}\label{sect:constructible}

In this section we summarize the properties of a theory of constructible sheaves on schemes which we will use in the paper. Such theories will involve choosing the ground field $k$ the schemes will be defined over and a commutative ring $R$ of coefficients.

\subsection{Six-functor formalisms}

For an $\infty$-category $\cC$ which admits finite limits and which has a class $c$ of morphisms stable under compositions and pullbacks, recall the symmetric monoidal $\infty$-category $\Corr(\cC)_{c;\all}$ which has the following informal description:
\begin{itemize}
    \item Its objects are objects of $\cC$.
    \item Morphisms from $X_1$ to $X_2$ are given by correspondences $X_1\leftarrow X_{12}\rightarrow X_2$ with $X_{12}\rightarrow X_2$ in $c$.
    \item The composite of $X_1\leftarrow X_{12}\rightarrow X_2$ and $X_2\leftarrow X_{23}\rightarrow X_3$ is given by the pullback $X_1\leftarrow X_{12}\times_{X_2} X_{23}\rightarrow X_3$.
    \item The symmetric monoidal structure is given by the Cartesian symmetric monoidal structure in $\cC$: $X_1, X_2\mapsto X_1\times X_2$.
\end{itemize}

We have the following notion of 6-functor formalisms from \cite[Definition A.5.7]{Mann}.

\begin{definition}
A \defterm{weak 6-functor formalism on $\cC$} is a lax symmetric monoidal functor
\[\bD^*_!\colon \Corr(\Sch^{\sep\ft})_{\all;\all}\longrightarrow \PrSt_R.\]
\end{definition}

We denote the values of a weak 6-functor formalism as follows:
\begin{itemize}
    \item The image of $X\in\Sch^{\sep\ft}$ is denoted by $\bD(X)\in\PrSt_R$.
    \item The image of $X\xleftarrow{\id} X\xrightarrow{f} Y$ is given by $f_!\colon \bD(X)\rightarrow \bD(Y)$. We denote its right adjoint by $f^!\colon \bD(Y)\rightarrow \bD(X)$.
    \item The image of $Y\xleftarrow{f}X\xrightarrow{\id} X$ is given by $f^*\colon \bD(Y)\rightarrow \bD(X)$. We denote its right adjoint by $f_*\colon \bD(X)\rightarrow \bD(Y)$.
    \item The lax symmetric monoidal structure is given by $\boxtimes\colon \bD(X)\otimes \bD(Y)\rightarrow \bD(X\times Y)$ and $\Mod_R\rightarrow \bD(\pt)$ denoted by $M\in\Mod_R\mapsto \underline{M}\in\bD(\pt)$.
\end{itemize}

The $\infty$-category $\bD(X)$ carries a symmetric monoidal structure defined by
\begin{equation}\label{eq:Dbcmonoidal}
\cF\otimes\cG = \Delta^*(\cF\boxtimes \cG),\qquad \bu_X = p^* \underline{R},
\end{equation}
where $\Delta\colon X\rightarrow X\times X$ is the diagonal and $p\colon X\rightarrow \pt$. We denote by $\cHom(-, -)\colon \bD(X)^{\op}\times \bD(X)\rightarrow \bD(X)$ the internal Hom. We denote by $\R\Gamma_c(X, -)$ the $!$-pushforward along $X\rightarrow \pt$.

\begin{definition}
An object $\cF\in\bD(X)$ is \defterm{lisse} if it is dualizable with respect to \eqref{eq:Dbcmonoidal}.
\end{definition}

Given a Cartesian square
\begin{equation}\label{eq:Cartesian}
\xymatrix{
X' \ar^{g}[r] \ar^{f'}[d] & X \ar^{f}[d] \\
Y' \ar^{h}[r] & Y
}
\end{equation}
in $\Sch^{\sep\ft}$ we obtain the following natural transformations:
\begin{enumerate}
    \item There is a natural isomorphism
    \[\Ex^*_!\colon f'_!g^*\xrightarrow{\sim} h^* f_!\]
    witnessed by the 2-cell
    \[
    \xymatrix{
    && X' \ar^{g}[dl] \ar^{f'}[dr] && \\
    & X \ar^{\id}[dl] \ar^{f}[dr] && Y' \ar^{h}[dl] \ar^{\id}[dr] & \\
    X && Y && Y'
    }
    \]
    in the $\infty$-category of $\Corr(\Sch^{\sep\ft})_{\all;\all}$. By construction the exchange natural isomorphism $\Ex^*_!$ is compatible with compositions.
    \item The natural isomorphism
    \[\Ex^!_*\colon f^! h_*\xrightarrow{\sim} g_* (f')^!\]
    is defined to be the mate of $\Ex^*_!$.
    \item There is a natural transformation
    \[\Ex^{*, !}\colon g^*f^!\longrightarrow (f')^!h^*\]
    obtained as the mate of the composite
    \[f'_! g^* f^!\xrightarrow{\Ex^*_!}h^*f_!f^!\xrightarrow{\counit} h^*.\]
\end{enumerate}

For a morphism $f\colon X\rightarrow Y$ we have the following natural transformations:
\begin{enumerate}
    \item $\Pr^*_!\colon f_!((-)\otimes f^*(-))\xrightarrow{\sim} f_!(-)\otimes (-)$ witnessing the projection formula.
    \item $\Pr^{!, *}\colon f^!(-)\otimes f^*(-)\longrightarrow f^!(-\otimes -)$ is the mate of
    \[f_!(f^!(-)\otimes f^*(-))\xrightarrow{\Pr^*_!} f_!f^!(-)\otimes(-)\xrightarrow{\counit} (-)\otimes (-).\]
\end{enumerate}

\begin{definition}
Let $\bD$ be a weak 6-functor formalism. A morphism $f\colon X\rightarrow Y$ in $\Sch^{\sep\ft}$ is \defterm{weakly cohomologically smooth} if the following conditions are satisfied:
\begin{enumerate}
    \item $\Pr^{!, *}\colon f^!(-)\otimes f^*(-)\rightarrow f^!(-\otimes -)$ is an isomorphism.
    \item The object $f^!\bu_Y\in\bD(X)$ is invertible.
    \item For any Cartesian diagram \eqref{eq:Cartesian} the morphism $\Ex^{*, !}\colon g^* f^!\bu_Y\rightarrow (f')^! \bu_{Y'}$ is an isomorphism.
\end{enumerate}
A morphism $f\colon X\rightarrow Y$ in $\Sch^{\sep\ft}$ is \defterm{cohomologically smooth} if for any Cartesian diagram \eqref{eq:Cartesian} the morphism $f'\colon X'\rightarrow Y'$ is weakly cohomologically smooth.
\end{definition}

Next we introduce the notion of a cohomologically proper morphism. For a monomorphism $f\colon X\rightarrow Y$ we have a Cartesian diagram
\[
\xymatrix{
X \ar@{=}[r] \ar@{=}[d] & X \ar^{f}[d] \\
X \ar^{f}[r] & Y
}
\]
In this case the base change isomorphism is $\Ex^!_*\colon f^!f_*\xrightarrow{\sim} \id$. We denote its mates by
\[\fgsp_f\colon f_!\longrightarrow f_*\qquad \epsilon_f\colon f^!\longrightarrow f^*.\]
Equivalently, they are mates of the base change isomorphism $\Ex^*_!\colon \id\xrightarrow{\sim} f^*f_!$.

\begin{definition}
Let $\bD$ be a weak 6-functor formalism. A monomorphism $f\colon X\rightarrow Y$ in $\Sch^{\sep\ft}$ is \defterm{cohomologically proper} if $\fgsp_f$ is an isomorphism.
\end{definition}

For a general morphism $f\colon X\rightarrow Y$ consider the diagram
\[
\xymatrix{
X \ar@/_1.0pc/@{=}[ddr] \ar^{\Delta}[dr] \ar@/^1.0pc/@{=}[drr] && \\
& X\times_Y X \ar^{p_1}[r] \ar^{p_2}[d] & X \ar^{f}[d] \\
& X \ar^{f}[r] & Y
}
\]

If $\Delta\colon X\rightarrow X\times_Y X$ is cohomologically proper, we define $\fgsp_f\colon f_!\rightarrow f_*$ as the mate of the composite
\[\id\cong (p_1)_*\Delta_*\Delta^! p_2^!\xleftarrow[\sim]{\fgsp_\Delta} (p_1)_*\Delta_!\Delta^! p_2^!\xrightarrow{\counit} (p_1)_*p_2^!\xleftarrow[\sim]{\Ex^!_*} f^!f_*.\]

\begin{definition}
Let $\bD$ be a weak 6-functor formalism. A morphism $f\colon X\rightarrow Y$ in $\Sch^{\sep\ft}$ is \defterm{cohomologically proper} if $\fgsp_\Delta$ and $\fgsp_f$ are isomorphisms.
\end{definition}

\begin{remark}
If $f\colon X\rightarrow Y$ is a monomorphism, its diagonal $\Delta\colon X\rightarrow X\times_Y X$ is an isomorphism in which case the map $\fgsp_\Delta$ is an isomorphism. Therefore, the above definitions of cohomological properness are consistent.
\end{remark}

We are ready to state the definition of a 6-functor formalism we will use in this paper.

\begin{definition}
A \defterm{6-functor formalism} is a weak 6-functor formalism $\bD^*_!\colon \Corr(\Sch^{\sep\ft})_{\all;\all}\longrightarrow \PrSt_R$ satisfying the following conditions:
\begin{enumerate}
    \item Every proper morphism is cohomologically proper.
    \item Every smooth morphism is cohomologically smooth.
\end{enumerate}
\end{definition}

Next we introduce the Verdier duality functor for a 6-functor formalism. For $X\in\Sch^{\sep\ft}$ we denote by $\omega_X = p^! \bu_{\pt}$, where $p\colon X\rightarrow \pt$. Similarly, for $f\colon X\rightarrow Y$ we denote $\omega_{X/Y} = f^! \bu_Y$. For $\cF\in\bD(X)$ we define
\[\bbD(\cF) = \cHom(\cF, \omega_X)\]
which comes with natural transformations
\[\psi_\cF\colon \cF\longrightarrow \bbD(\bbD(\cF))\]
and
\[\Ex^{\bbD, \boxtimes}\colon \bbD(-)\boxtimes \bbD(-)\longrightarrow \bbD(-\boxtimes -).\]
For a morphism $f\colon X\rightarrow Y$ we have natural isomorphisms
\[\Ex^{!,\bbD}\colon f^!\bbD\xrightarrow{\sim} \bbD f^*\]
and
\[\Ex_{*, \bbD}\colon f_*\bbD\xrightarrow{\sim}\bbD f_!\]
which fit into a commutative diagram
\begin{equation}\label{eq:Ex!*selfdual}
\xymatrix{
f^!h_*\bbD \ar^{\Ex_{*, \bbD}}[r] \ar^{\Ex^!_*}[d] & f^!\bbD h_! \ar^{\Ex^{!, \bbD}}[r] & \bbD f^*h_! \ar^{\Ex^*_!}[d] \\
g_*(f')^! \bbD \ar^{\Ex^{!, \bbD}}[r] & g_* \bbD (f')^* \ar^{\Ex_{*, \bbD}}[r] & \bbD g_! (f')^*.
}
\end{equation}

\begin{proposition}\label{prop:epsilonfgspD}
Let $f\colon X\rightarrow Y$ be a monomorphism. Then the diagram
\[
\xymatrix{
f_* f^!\bbD\ar^{\Ex^{!, \bbD}}_{\sim}[r] \ar^{\epsilon_f}[d] & f_* \bbD f^* \ar^{\Ex_{*, \bbD}}_{\sim}[r] & \bbD f_!f^* \\
f_* f^*\bbD & \bbD \ar_-{\unit}[l] & \bbD f_*f^* \ar_{\unit}[l] \ar_{\fgsp_f}[u]
}
\]
commutes.
\end{proposition}
\begin{proof}
Applying \eqref{eq:Ex!*selfdual} we get a commutative diagram
\[
\xymatrix{
\bbD \ar^{\Ex^*_!}[r] \ar^{\Ex^!_*}[d] & \bbD f_! f^* \\
f_*f^!\bbD \ar^{\Ex^{!, \bbD}}[r] & f_* \bbD f^*\ar_{\Ex_{*, \bbD}}[u]
}
\]
Using the definitions of $\fgsp_f$ and $\epsilon_f$ as mates of $\Ex^*_!$ and $\Ex^!_*$, respectively, we obtain the claim.
\end{proof}

For a lisse object $V\in\bD(X)$ with dual $V^\vee\in\bD(X)$ and another object $\cF\in\bD(X)$ we have a natural isomorphism
\[\Ex^{V, \bbD}\colon \bbD(\cF)\otimes V^\vee\xrightarrow{\sim} \bbD(V\otimes \cF).\]

The natural transformation $\psi$ commutes with the operation $f^*$ in the following sense: given a morphism $f \colon X \to Y$, there is a commutative square
\[\begin{tikzcd}
  f^* \ar{r}{f^*\psi}\ar{d}{\psi f^*}
  & f^* \bbD \bbD \ar{d}{\psi_{f^*\bbD\bbD}}
  \\
  \bbD\bbD f^* \ar{r}{\bbD\bbD f^* \psi}
  & \bbD\bbD f^* \bbD\bbD.
\end{tikzcd}\]
We also have the following compatibility between $\psi$ and $\Ex^{!,\bbD}$: there is a commutative square
\[\begin{tikzcd}
  f^! \bbD \bbD \bbD \ar{r}{\Ex^{!,\bbD}}\ar[swap]{d}{f^! \bbD (\psi)}
  & \bbD f^* \bbD \bbD \ar{d}{\bbD f^*(\psi)}
  \\
  f^! \bbD \ar{r}{\Ex^{!,\bbD}}
  & \bbD f^*.
\end{tikzcd}\]

\subsection{Complex analytic constructible sheaves on schemes and stacks}

In this section $k=\C$ and $R$ a commutative Noetherian ring of finite global dimension. Let $\Top^{\lch}$ be the category of locally compact Hausdorff topological spaces.

Using results of \cite{KashiwaraSchapira,Volpe}, an $\infty$-categorical six-functor formalism for locally compact Hausdorff topological spaces was constructed in \cite[Theorem 8.1.8]{KhanWeavelisse}. In particular, we get a lax symmetric monoidal functor
\[\Shv^*_!\colon \Corr(\Top^{\lch})_{\all;\all}\longrightarrow\PrSt_R\]
which sends a locally compact Hausdorff space $X$ to $\Shv(X; R)$, the derived $\infty$-category of complexes of sheaves of $R$-modules on $X$ with $*$-pullbacks and $!$-pushforwards along arbitrary morphisms which admit right adjoints $f_*$ and $f^!$.

Given a scheme $X\in\Sch^{\sep\ft}$ which is separated and of finite type over $\C$ we denote by $X^{\an}\in\Top^{\lch}$ its analytification. For a scheme $X\in\Sch^{\sep\ft}$ an object $\cF\in\Shv(X^{\an}; R)$ is \defterm{constructible} if there is a finite stratification $X=\sqcup_i X_i$ of $X$ into locally closed subschemes $X_i$ such that $\cF|_{X_i}$ is locally constant with perfect stalks. Let $\Dbc(X)\subset \Shv(X^{\an}; R)$ be the full subcategory of constructible sheaves and $\bD(X) = \Ind(\Dbc(X))$ the $\infty$-category of ind-constructible sheaves. By \cite{Verdier} (see also the overview in \cite{MaximSchurmann}) the six-functor formalism restricts to constructible sheaves and, hence, ind-constructible sheaves, i.e. we have a lax symmetric monoidal functor
\[\bD^*_!\colon \Corr(\Sch^{\sep\ft})_{\all;\all}\longrightarrow \PrSt_R\]
which sends $X\in\Sch^{\sep\ft}$ to $\bD(X)$. Let us summarize some properties of constructible sheaves that can be found in \cite{KashiwaraSchapira,Achar,MaximSchurmann}.

\begin{proposition}\label{prop:constructible6functor}
$\bD^*_!\colon \Corr(\Sch^{\sep\ft})_{\all;\all}\rightarrow \PrSt_R$ is a 6-functor formalism which satisfies the following properties:
\begin{enumerate}
    \item For every $X\in\Sch^{\sep\ft}$ the unit $\id\rightarrow \pi_*\pi^*$ is an isomorphism for $\pi\colon \bA^1\times X\rightarrow X$ the projection.
    \item The functor $\bD^*\colon \Sch^{\sep\ft, \op}\rightarrow \PrSt_R$ given by sending $f\colon X\rightarrow Y$ to $f^*\colon \bD(Y)\rightarrow \bD(X)$ satisfies \'etale descent.
    \item For a closed immersion $i\colon Z\rightarrow X$ and its complementary open immersion $j\colon U\rightarrow X$ the base change isomorphism gives $j^!i_*\cong 0$. The corresponding commutative diagram
    \begin{equation}\label{eq:localization}
    \xymatrix@C=2cm{
    j_!j^! \ar^{\counit}[r] \ar^{\unit}[d] & \id \ar^{\unit}[d] \\
    j_!j^!i_*i^*\cong 0 \ar^{\counit}[r] & i_*i^*
    }
    \end{equation}
    is Cartesian.
    \item For every $X\in\Sch^{\sep\ft}$ and $\cF\in\Dbc(X)$ the morphism $\psi_\cF\colon \cF\rightarrow \bbD(\bbD(\cF))$ is an isomorphism.
    \item For every $X, Y\in\Sch^{\sep\ft}$ and $\cF\in\Dbc(X),\cG\in\Dbc(Y)$ the morphism $\Ex^{\bbD, \boxtimes}\colon \bbD(\cF)\boxtimes \bbD(\cG)\rightarrow \bbD(\cF\boxtimes \cG)$ is an isomorphism.
    \item There is a perverse $t$-structure on $\Dbc(X)$ for $X\in\Sch^{\sep\ft}$ with heart $\Perv(X)$ with the following properties:
    \begin{enumerate}
        \item If $f\colon X\rightarrow Y$ is smooth, $f^\dag\coloneqq f^![-\dim(X/Y)]\colon \Dbc(Y)\rightarrow \Dbc(X)$ is $t$-exact.
        \item If $f\colon X\rightarrow Y$ is smooth with connected fibers, $f^\dag\colon \Perv(Y)\rightarrow\Perv(X)$ is fully faithful.
        \item If $f\colon X\rightarrow Y$ is finite, $f_*\colon \Dbc(X)\rightarrow \Dbc(Y)$ is $t$-exact.
    \end{enumerate}
    \item If $R$ is a field and $X\in\Sch^{\sep\ft}$, the natural functor $\Ind(\Db(\Perv(X))) \rightarrow \bD(X)$ is an equivalence.
    \item Lisse objects in $\bD(X)$ are precisely locally constant sheaves with perfect stalks.
\end{enumerate}
\end{proposition}

Let us now present several corollaries of the assumptions. First, we may extend a 6-functor formalism to stacks as follows.

\begin{proposition}\label{prop:6functorextension}
There is a functor 
\[\bD^*_!\colon \Corr(\Art^{\lft})_{\all;\all}\longrightarrow \PrSt_R\]
with the following properties:
\begin{enumerate}
    \item Its restriction to $\Corr(\Sch^{\sep\ft})_{\all;\all}$ coincides with the original 6-functor formalism.
    \item The functor
    \[\bD(X)\longrightarrow \lim_{(S, s)}\bD(S)\]
    induced by $*$-pullbacks, where the limit is taken over the $\infty$-category $\Sch^{\sep\ft}_{/X}$ of schemes $S\in\Sch^{\sep\ft}$ together with a morphism $s\colon S\rightarrow X$, is an equivalence.
    \item Every smooth morphism $f\colon X\rightarrow Y$ in $\Art^{\lft}$ is cohomologically smooth with respect to this extension of the 6-functor formalism.
\end{enumerate}
\end{proposition}
\begin{proof}
We will construct an extension of the 6-functor formalism in three steps using the subcategories
\[\Corr(\Sch^{\sep\ft})_{\all;\all}\subset \Corr(\Sch^{\lft})_{\sep\ft;\all}\subset \Corr(\Sch^{\lft})_{\all;\all}\subset \Corr(\Art^{\lft})_{\all;\all},\]
where morphisms in $\Corr(\Sch^{\lft})_{\sep\ft;\all}$ are given by correspondences $X_1\leftarrow X_{12}\rightarrow X_2$ of schemes locally of finite type, where $X_{12}\rightarrow X_2$ is separated and of finite type.

\begin{enumerate}
    \item The extension of $\bD^*_!\colon \Corr(\Sch^{\sep\ft})_{\all;\all}\rightarrow \PrSt_R$ to
    \[\bD^*_!\colon \Corr(\Sch^{\lft})_{\sep\ft;\all}\longrightarrow \PrSt_R\]
    is provided by \cite[Proposition A.5.16]{Mann}.
    \item Since \'etale morphisms are cohomologically smooth and $\bD$ satisfies \'etale descent with respect to $*$-pullbacks, it also satisfies \'etale descent with respect to $!$-pullbacks. Therefore, the functor
    \[\bD(X)\longrightarrow \lim_{(S, s)}\bD(S)\]
    induced by $!$-pullbacks, where $X\in\Sch^{\lft}$ and the limit is taken over the category $\Sch^{\sep\ft}_{/X}$ of schemes $S\in\Sch^{\sep\ft}$ together with a morphism $s\colon S\rightarrow X$, is an equivalence. Therefore, by \cite[Lemma A.5.11]{Mann} we obtain an extension
    \[\bD^*_!\colon \Corr(\Sch^{\lft})_{\all;\all}\longrightarrow \PrSt_R.\]
    \item By assumption every morphism $f\colon X\rightarrow Y$ in $\Sch^{\lft}$ which is smooth, separated and of finite type is cohomologically smooth. Using \'etale descent we get that any smooth morphism in $\Sch^{\lft}$ is cohomologically smooth. In particular, we obtain a $\sharp$-direct image (in the sense of \cite[Definition 2.30]{KhanWeaves}) for every smooth morphism $f\colon X\rightarrow Y$, where $f_\sharp(-) = f_!(f^!(\bu_Y)\otimes(-))$. Therefore, by \cite[Theorem 3.4.3(ii)]{KhanWeavelisse} we obtain an extension
    \[\bD^*_!\colon \Corr(\Art^{\lft})_{\all;\all}\longrightarrow \PrSt_R.\]
\end{enumerate}
\end{proof}

We extend the notion of constructible and perverse sheaves to stacks as follows. For $Y\in\Art^{\lft}$ we have:
\begin{enumerate}
    \item $\cF\in\bD(Y)$ is constructible if for every smooth morphism $f\colon X\rightarrow Y$ with $X\in\Sch^{\sep\ft}$ the object $f^*\cF\in\bD(X)$ is constructible.
    \item There is a $t$-structure on $\Dbc(Y)$ such that for every smooth morphism $f\colon X\rightarrow Y$ with $X\in\Sch^{\sep\ft}$ the functor $f^\dag$ is $t$-exact.
\end{enumerate}

\begin{proposition}\label{prop:lissetensorconstructible}
Let $X\in\Art^{\lft}$, $V\in\bD(X)$ a lisse object and $\cF\in\Dbc(X)$. Then $V\otimes\cF$ is constructible.
\end{proposition}
\begin{proof}
By definition of constructibility the claim for any $X\in\Art^{\lft}$ reduces to the same claim for $X\in\Sch^{\sep\ft}$. In this case we have
\[\Hom_{\bD(X)}(V\otimes \cF, -)\cong \Hom_{\bD(X)}(\cF, V^\vee\otimes(-)).\]
Since $\cF$ is constructible, this functor preserves colimits.
\end{proof}

The definition of the Verdier duality functor $\bbD$ extends verbatim to stacks. We have the following functoriality of the isomorphism $\psi$.

\begin{lemma}\label{lm:psifunctorial}
Let $f\colon X\rightarrow Y$ be a smooth morphism in $\Art^{\lft}$ and $\cF\in\Dbc(Y)$. Then the composite
\begin{align*}
f^!\cF&\xrightarrow{\psi_\cF} f^!\bbD^2\cF \\
&\xrightarrow{\Ex^{!, \bbD}} \bbD f^*\bbD\cF \\
&\xrightarrow{\Pr^{!, *}} \bbD(\omega_{X/Y}^{-1}\otimes f^!\bbD\cF) \\
&\xrightarrow{(\Ex^{!, \bbD})^{-1}} \bbD(\omega_{X/Y}^{-1}\otimes \bbD(f^*\cF)) \\
&\xrightarrow{(\Ex^{\omega_{X/Y}^{-1}, \bbD})^{-1}}\bbD^2(\omega_{X/Y}\otimes f^*\cF) \\
&\xrightarrow{\Pr^{!, *}} \bbD^2 f^!\cF,
\end{align*}
is equivalent to $\psi_{f^!\cF}$.
\end{lemma}

\begin{proposition}
For $Y\in\Art^{\lft}$ and $\cF\in\bD(Y)$ constructible, the morphism $\psi_\cF\colon \cF\rightarrow \bbD(\bbD(\cF))$ is an isomorphism.
\end{proposition}
\begin{proof}
Let $X\in\Sch^{\sep\ft}$ and $f\colon X\rightarrow Y$ a smooth morphism. In the composite in \cref{lm:psifunctorial} all morphisms, except possibly the first, are isomorphisms. Since $\omega_{X/Y}$ is invertible and $f^*\cF$ is constructible, $f^!\cF$ is also constructible by \cref{prop:lissetensorconstructible}. Thus, the composite is an isomorphism by \cref{prop:constructible6functor}(5). Therefore, $f^!\cF\rightarrow f^!\bbD^2\cF$ is an isomorphism. Since the functors $\{f^!\colon \bD(Y)\rightarrow \bD(X)\}$, where $f$ ranges over smooth morphisms from a separated scheme of finite type, are jointly conservative, we get that $\psi_\cF\colon \cF\rightarrow \bbD^2\cF$ is an isomorphism.
\end{proof}

Given a morphism $f \colon X \to Y$ in $\Art^{\lft}$, we obtain a unique natural isomorphism
\begin{equation}
  \Ex^{*,\bbD}\colon f^*\bbD\xrightarrow{\sim} \bbD f^!
\end{equation}

fitting into the commutative squares
\begin{equation}\label{eq:psi vs Ex* Ex!}
  \begin{tikzcd}
    f^* \ar{r}{\psi_{f^*}}\ar[swap]{d}{f^*\psi}
    & \bbD \bbD f^* \ar{d}{\bbD \Ex^{!,\bbD}}
    \\
    f^* \bbD \bbD \ar{r}{\Ex^{*,\bbD}\bbD}
    & \bbD f^! \bbD,
  \end{tikzcd}
  \qquad
  \begin{tikzcd}
    f^! \ar{r}{\psi_{f^*}}\ar[swap]{d}{f^!\psi}
    & \bbD \bbD f^! \ar{d}{\bbD \Ex^{*,\bbD}}
    \\
    f^! \bbD \bbD \ar{r}{\Ex^{!,\bbD}\bbD}
    & \bbD f^* \bbD.
  \end{tikzcd}
\end{equation}

For a pair of morphisms $f_1\colon X_1\rightarrow Y_1$ and $f_2\colon X_2\rightarrow Y_2$ in $\Art^{\lft}$ we have a natural transformation
\[
\Ex^{!, \boxtimes}\colon f_1^!\cF_1\boxtimes f_2^!\cF_2\longrightarrow (f_1\times f_2)^!(\cF_1\boxtimes \cF_2)
\]
which is an isomorphism for $\cF_1\in\Dbc(Y_1)$ and $\cF_2\in\Dbc(Y_2)$ as can be seen by applying the Verdier duality functor $\bbD$ and using \cref{prop:constructible6functor}(5). See \cite[Proposition 1.3]{MasseyTS}.

\begin{proposition}\label{prop:fundamentalclass}
For a smooth morphism $f\colon X\rightarrow Y$ in $\Art^{\lft}$ there is a natural isomorphism
\[\pur_f\colon f^*[2\dim(X/Y)]\xrightarrow{\sim} f^!\]
of functors $\bD(Y)\rightarrow \bD(X)$ which is functorial for compositions up to coherent homotopy and compatible with products.
\end{proposition}
\begin{proof}
By functoriality the natural isomorphism $\pur_f$ extends uniquely from schemes to stacks, so it is sufficient to construct it for smooth morphisms of schemes.

Since a smooth morphism is cohomologically smooth, the natural morphism $\Pr^{!, *}\colon \omega_{X/Y}\otimes f^*(-)\rightarrow f^!(-)$ is an isomorphism. Thus, it is enough to construct an isomorphism $\omega_{X/Y}\cong \bu_X[2\dim(X/Y)]$, functorial in $f$. The construction of this isomorphism and its 1-categorical functoriality (i.e. the isomorphisms $\pur_{\id}\cong\id$ and $\pur_{g\circ f}\cong \pur_f\circ\pur_g$) are well-known. But since $\omega_{X/Y}[-2\dim(X/Y)]$ is a local system of vector spaces, the 1-categorical functoriality implies the $\infty$-categorical functoriality.
\end{proof}

Given a smooth morphism $f \colon X \to Y$ in $\Art^\lft$ of relative dimension $d$, we obtain a natural isomorphism
\begin{equation}\label{eq:ExdagD}
\Ex^{\dag, \bbD}\colon f^\dag \bbD = f^![-d]\bbD\xrightarrow{\Ex^{!,\bbD}} \bbD f^*[d]\xleftarrow[\sim]{\pur_f} \bbD f^\dag.
\end{equation}

\begin{proposition}\label{prop:fdagDD}
    Let $f \colon X \to Y$ in $\Art^\lft$ be a smooth morphism.
    Then the diagram
    \[\begin{tikzcd}[column sep=large]
        f^\dag \ar{d}{\psi}\ar{r}{\psi}
        & \bbD \bbD f^\dag \ar{d}{\Ex^{\dag,\bbD}}
        \\
        f^\dag \bbD \bbD \ar{r}{\Ex^{\dag,\bbD}}
        & \bbD f^\dag \bbD
    \end{tikzcd}\]
    commutes.
\end{proposition}
\begin{proof}
The claim follows from \cref{lm:psifunctorial}.
\end{proof}

Similarly, for a schematic and proper morphism $f\colon X\rightarrow Y$ we obtain a natural isomorphism
\[\Ex_{*, \bbD}\colon f_*\bbD\xrightarrow{\Ex_{*, \bbD}} \bbD f_!\xleftarrow[\sim]{\fgsp_f} \bbD f_*.\]

\begin{proposition}\label{prop:Ex!*selfdual}
Consider a Cartesian square \eqref{eq:Cartesian}, where $h$ is proper and $f$ is smooth. Then the diagram
\[
\xymatrix{
f^\dag h_*\bbD \ar^{\Ex_{*, \bbD}}[r] \ar^{\Ex^!_*}[d] & f^\dag\bbD h_* \ar^{\Ex^{\dag, \bbD}}[r] & \bbD f^\dag h_* \\
g_* (f')^\dag \bbD\ar^{\Ex^{\dag, \bbD}}[r] & g_* \bbD (f')^\dag \ar^{\Ex_{*, \bbD}}[r] & \bbD g_* (f')^\dag \ar^{\Ex^!_*}[u]
}
\]
\end{proposition}
\begin{proof}
Let $d$ be the relative dimension of $f\colon X\rightarrow Y$. Let $\BC\colon f^*h_*\rightarrow g_* (f')^*$ be the Beck--Chevalley transformation obtained as the mate of the isomorphism $g^*f^*\cong (f')^*h^*$. By \cite[Lemma 7.4]{KPS} the diagram
\[
\xymatrix{
f^*h_* \ar^{\BC}[r] \ar^{\pur_f}[d] & g_* (f')^* \ar^{\pur_{f'}}[d] \\
f^!h_*[-2d] \ar^{\Ex^!_*}[r] & g_*(f')^![-2d].
}
\]
commutes. By \cite[Example 3.2.1]{FengYunZhang} the diagram
\[
\xymatrix{
g_! (f')^*\ar^{\Ex^*_!}[r] \ar^{\fgsp_g}[d] & f^*h_! \ar^{\fgsp_h}[d] \\
g_* (f')^* & f^*h_* \ar_-{\BC}[l]
}
\]
commutes. Combining the above two diagrams with \eqref{eq:Ex!*selfdual} we get the result.
\end{proof}

We may use \cref{prop:constructible6functor}(7) to (uniquely) extend exact functors on $\Perv(-)$, and natural transformations between them, to $\bD(-)$ as follows.
For this it will be convenient to record the following consequence of the universal property of bounded derived $\infty$-categories.
Consider the following $\infty$-categories:
\begin{enumerate}
    \item The $\infty$-category $\Cat^{\mathrm{ab}}$ of $R$-linear abelian categories and $R$-linear exact functors.  (This is a $(2,1)$-category.)
    \item The $\infty$-category $\mathrm{Pr}^{\mathrm{St},t}_R$ of stable presentable $R$-linear $\infty$-categories equipped with a t-structure, and $t$-exact $R$-linear colimit-preserving functors.
    \item The subcategory $\mathrm{Pr}^{\mathrm{St},t,\mathrm{cg}}_R \subset \mathrm{Pr}^{\mathrm{St},t}_R$ consisting of those objects $\cE$ for which the $t$-structure restricts to $\cE^\omega$, and the colimit-preserving functor $\Ind(\cE^\omega) \to \cE$ is a $t$-exact equivalence; and those morphisms given by functors that preserve compact objects.
    \item The full subcategory $\mathrm{Pr}^{\mathrm{St},t,\mathrm{cg},\mathrm{comp}}_R \subset \mathrm{Pr}^{\mathrm{St},t,\mathrm{cg}}_R$ spanned by those $\cE$ satisfying the following condition.
    By \cite[Corollary 7.4.12]{BunkeCisinski} and \cite[Proposition 5.3.6.2]{LurieHTT}, there is a unique $t$-exact $R$-linear colimit-preserving functor $\Ind(\Db(\cE^{\heartsuit})) \to \cE$ extending the inclusion $\cE^\heartsuit \hookrightarrow \cE$.
    Then $\cE$ belongs to $\mathrm{Pr}^{\mathrm{St},t,\mathrm{cg},\mathrm{comp}}_R$ if this functor is an equivalence.
\end{enumerate}

We then have:

\begin{proposition}\label{prop:Catab}
    Consider the canonical functor $\mathrm{Pr}^{\mathrm{St},t,\mathrm{cg}}_R \to \Cat^{\mathrm{ab}}$ sending $\cE$ to $\cE^{\omega,\heartsuit}$, the heart of the full subcategory of compact objects (with respect to its induced t-structure).
    This induces an equivalence of $\infty$-categories
    \[ \mathrm{Pr}^{\mathrm{St},t,\mathrm{cg},\mathrm{comp}}_R \xlongrightarrow{\sim} \Cat^{\mathrm{ab}}. \]
\end{proposition}
\begin{proof}
    By \cite[Proposition 5.3.6.2]{LurieHTT} and \cite[Corollary 7.4.12]{BunkeCisinski}, the restriction functor
    \[ \Fun^{\mathrm{L},\omega,t-\ex}(\Ind(\Db(\cA)), \Ind(\Db(\mathscr{B}))) \cong \Fun^{t-\ex}(\Db(\cA), \Db(\mathscr{B})) \to \Fun^{\ex}(\cA, \mathscr{B}) \]
    is an equivalence for any $\cA, \mathscr{B} \in \Cat^{\mathrm{ab}}$, where the source denotes $t$-exact functors that preserve colimits and compact objects.
    Passing to underlying $\infty$-groupoids, we find that the functor $\mathrm{Pr}^{\mathrm{St},t,\mathrm{cg},\mathrm{comp}}_R \to \Cat^{\mathrm{ab}}$ is fully faithful.
    Essential surjectivity is obvious.
\end{proof}

\begin{corollary}\label{cor:extendfromPerv}
Let $X\in\Sch^{\sep\ft}$ and $\cA\in\Cat^{\mathrm{ab}}$ an $R$-linear abelian category. Suppose $R$ is a field. Then the restriction
\[\Fun^{\mathrm{L}, t-\ex}(\bD(X), \Ind(\Db(\cA)))^\sim \longrightarrow \Fun^{\ex}(\Perv(X), \cA)^\sim \]
is an equivalence of $\infty$-groupoids, where $\Fun^{\ex}(\Perv(X), \cA)^\sim$ denotes the groupoid of exact $R$-linear functors and $\Fun^{\mathrm{L}}(\bD(X), \Ind(\Db(\cA)))^\sim$ the $\infty$-groupoid of colimit-preserving $R$-linear functors sending $\Perv(X)$ to $\cA$.
\end{corollary}
\begin{proof}
By \cref{prop:constructible6functor}(7) we have the natural equivalence $\Ind(\Dbc(X))\rightarrow \bD(X)$. The functor in question is thus the functor induced on mapping $\infty$-groupoids by the equivalence of \cref{prop:Catab}.
\end{proof}

\begin{remark}
    The proof shows that \cref{cor:extendfromPerv} holds at the level of functor $\infty$-categories, and \cref{prop:Catab} admits a corresponding $(\infty,2)$-categorical upgrade.
\end{remark}

\begin{remark}
\cref{cor:extendfromPerv} does not extend to Artin stacks.
Even though the natural functors $\Db(\Perv(X))\rightarrow \Dbc(X)$ and $\Ind(\Dbc(X))\rightarrow \bD(X)$ are equivalences for $X\in\Sch^{\sep\ft}$, neither functor is an equivalence for $X=\B\Gm$.
\end{remark}


\subsection{Vanishing cycles}
\label{sect:vanishingcycles}

For a complex analytic space $X$ and a holomorphic function $t\colon X\rightarrow\C$, let
\[X_{\Re\leq 0} = \{x\in X\mid \Re(t(x))\leq 0\}.\]
The \defterm{vanishing cycles functor} $\varphi_t \colon \Shv(X; R)\rightarrow \Shv(t^{-1}(0); R)$ is defined by
\[\varphi_t = (t^{-1}(0)\rightarrow X_{\Re\leq 0})^*(X_{\Re\leq 0}\rightarrow X)^!.\]

\begin{proposition}\label{prop:constructiblevanishingcycles}
For $X\in\Sch^{\sep\ft}$ and $t\colon X\rightarrow \bA^1$ the functor
\[\phi_t = \bigoplus_{c\in\C} (t^{-1}(c)\rightarrow X)_*\varphi_{t-c}\]
preserves constructibility, is perverse $t$-exact, and satisfies the following property:
\begin{enumerate}
    \item For a smooth morphism $f\colon X\rightarrow B$ in $\Sch^{\sep\ft}$, a function $t\colon X\rightarrow \bA^1$ and a perverse sheaf $\cF\in\Perv(B)$, the object $\phi_t f^\dag \cF$ is supported on the $B$-relative critical locus of $t$.
\end{enumerate}

For a morphism $f\colon X'\rightarrow X$ with $t\colon X\rightarrow \bA^1$ and $t'\coloneqq f^\ast t$ there are natural transformations
\[\Ex^!_\phi\colon \phi_{t'} f^! \to f^! \phi_t, \qquad \Ex^\phi_*\colon \phi_t f_* \to f_* \phi_{t'}\]
which satisfy the following properties:
\begin{enumerate}[resume]
    \item $\Ex_\phi^!$ and $\Ex^\phi_*$ are functorial for compositions.
    \item Let
    \[
    \xymatrix{
    X_{11} \ar^{f_1}[r] \ar^{g_1}[d] & X_{12} \ar^{g_2}[d] \\
    X_{21} \ar^{f_2}[r] & X_{22}
    }
    \]
    be a Cartesian diagram in $\Sch^{\sep\ft}$. Let $t_{22}\colon X_{22}\to\bA^1$ be a morphism and denote by $t_{11}, t_{12}, t_{21}$ its restrictions to $X_{11}, X_{12}, X_{21}$. Then the diagram
    \[
    \xymatrix{
    \phi_{t_{21}} f_2^! g_{2, *} \ar^{\Ex^!_\phi}[r] \ar^{\Ex^!_*}[d] & f^!_2 \phi_{t_{22}} g_{2,*} \ar^{\Ex^\phi_*}[r] & f^!_2 g_{2,*} \phi_{t_{12}} \ar^{\Ex^!_*}[d] \\
    \phi_{t_{21}} g_{1,*} f^!_1 \ar^{\Ex^\phi_*}[r] & g_{1,*} \phi_{t_{11}} f^!_1 \ar^{\Ex^!_\phi}[r] & g_{1,*} f^!_1 \phi_{t_{12}}
    }
    \]
    is commutative.
    \item If $f$ is smooth, $\Ex_\phi^!$ is invertible.
    \item\label{item:constructiblevanishingcycles/proper} If $f$ is proper, $\Ex^\phi_*$ is invertible.
\end{enumerate}
\end{proposition}
\begin{proof}
We first show that $\varphi_t$ preserves constructibility and the perverse t-structure. Indeed, the claim is local, so by embedding $X$ into a smooth scheme and using proper base change we are reduced to the corresponding claim for $X$ smooth. The fact that $\varphi_f$ preserves constructibility is shown in \cite[Corollary 4.12]{MaximSchurmann}. Perverse t-exactness is shown in \cite[Corollary 10.3.13]{KashiwaraSchapira}.

Next we show that $\phi_t$ preserves constructibility and the perverse t-structure. Fix $\cF\in\Dbc(X)$ and let $\cS$ be a Whitney stratification of $X$ such that for all strata $S\in\cS$ the restriction $\cF|_S$ is locally constant. By \cite[Remark 1.10]{MasseyCritical} for every $c\in\C$ we have
\begin{equation}\label{eq:vanishingsupport}
\supp(\varphi_{t-c}\cF)\subset \bigcup_{S\in\cS} \Crit_S(t|_S)\cap t^{-1}(c).
\end{equation}
Since $t|_S$ is locally constant on $\Crit_S(t|_S)^{\red}$ and $\Crit_S(t|_S)$ has finitely many irreducible components, $t|_S$ takes finitely many values on $\Crit_S(t|_S)$. Since $\cS$ is finite, this implies that only finitely many summands in the definition of $\phi_f\cF$ are nonzero. Thus, $\phi_t\cF$ is constructible and, if $\cF$ is perverse, so is $\phi_t\cF$.

The exchange natural transformations $\Ex^!_\phi$ and $\Ex^\phi_*$ are constructed from the usual exchange natural transformations between the six-functor operations which satisfy analogs of (2)-(3). The smooth and proper base change theorems verify (4)-(5).

Let us now show (1). Let $\cS$ be a Whitney stratification of $B$ such that for all strata $S\in\cS$ the restriction $\cF|_S$ is locally constant. Then $f^{-1}(\cS)$ is a Whitney stratification \cite[Chapter I, Proposition 1.4]{GWPL} and $f^\dag \cF$ is constructible with respect to $f^{-1}(\cS)$. Using \eqref{eq:vanishingsupport} we get
\[\supp(\phi_t f^\dag \cF)\subset \bigcup_{S\in\cS} \Crit_{f^{-1}(S)}(t|_{f^{-1}(S)})\subset \bigcup_{S\in\cS} \Crit_{f^{-1}(S)/S}(t|_{f^{-1}(S)}).\]
Using the Cartesian diagram
\[
\xymatrix{
f^{-1}(S) \ar[d] \ar[r] & X \ar[d] \\
S \ar[r] & B
}
\]
we get that $\Crit_{f^{-1}(S)/S}(t|_{f^{-1}(S)})\subset \Crit_{X/B}(t)$.
\end{proof}

Taking mates of $\Ex^!_\phi$ and $\Ex^\phi_*$ we obtain
\[\Ex^\phi_!\colon f_! \phi_{t'}\longrightarrow \phi_t f_! ,\qquad \Ex^*_\phi\colon f^*\phi_t\longrightarrow \phi_{t'} f^*.\]
Moreover, by the proper and smooth base change theorems $\Ex^\phi_!$ is an isomorphism if $f$ is proper and $\Ex^*_\phi$ is an isomorphism if $f$ is smooth.

\begin{proposition}\label{prop:phiD}
For a scheme $X\in\Sch^{\sep\ft}$ equipped with a function $t\colon X\rightarrow \bA^1$ there is a natural isomorphism $\Ex^{\phi,\bbD} \colon \phi_t \bbD \xrightarrow{\sim} \bbD \phi_{-t}$ on constructible objects, satisfying the following properties:
\begin{enumerate}
    \item The diagram
    \[
    \xymatrix@C=1.3cm{
    \phi_{t} \ar^{\psi}[d]\ar^-{\psi}[r] & \bbD \bbD \phi_{t} \ar^{\Ex^{\phi,\bbD}}[d] \\
    \phi_{t} \bbD \bbD \ar^{\Ex^{\phi,\bbD}}[r] & \bbD \phi_{-t} \bbD
    }
    \]
    commutes.
    \item For a morphism $f\colon X'\rightarrow X$ in $\Sch^{\sep\ft}$, the diagram
    \[
    \xymatrix{
    \phi_{t'} f^!\bbD\ar^{\Ex^{!, \bbD}}[r] \ar^{\Ex^!_\phi}[d] & \phi_{t'} \bbD f^* \ar^{\Ex^{\phi, \bbD}}[r] & \bbD \phi_{-t'} f^* \ar^{\Ex^*_\phi}[d] \\
    f^! \phi_t \bbD \ar^{\Ex^{\phi, \bbD}}[r] & f^!\bbD\phi_{-t} \ar^{\Ex^{!, \bbD}}[r] & \bbD f^*\phi_{-t}
    }
    \]
    commutes, where $t'\coloneqq f^\ast t$.
    \item For a morphism $f\colon X'\rightarrow X$ in $\Sch^{\sep\ft}$, the diagram
    \[
    \xymatrix{
    \phi_t f_*\bbD\ar^{\Ex_{*, \bbD}}[r] \ar^{\Ex^\phi_*}[d] & \phi_t \bbD f_! \ar^{\Ex^{\phi, \bbD}}[r] & \bbD \phi_{-t} f_! \ar^{\Ex^\phi_!}[d] \\
    f_* \phi_{t'} \bbD \ar^{\Ex^{\phi, \bbD}}[r] & f_*\bbD\phi_{-t'} \ar^{\Ex_{*, \bbD}}[r] & \bbD f_!\phi_{-t'}
    }
    \]
    commutes, where $t'\coloneqq f^\ast t$.
\end{enumerate}
\end{proposition}
\begin{proof}
We begin by recalling a natural isomorphism
\[\Ex^{\varphi, \bbD}\colon \varphi_t \bbD\xrightarrow{\sim} \bbD \varphi_{-t}\]
constructed in \cite{MasseyVerdier} (the construction is reviewed in \cite[\S2.3]{KinjoVirtual}).

Consider the diagram of inclusions
    \[\begin{tikzcd}[/tikz/column 1/.style={column sep=0.8em}, /tikz/row 1/.style={row sep=0.8em}]
        X_0 \ar{rd}{g_0} & &
        \\
        & X_{\Re=0} \ar{r}{h_+}\ar{d}{h_-}
        & X_{\Re\ge 0} \ar{d}{i_+}
        \\
        & X_{\Re\le 0} \ar{r}{i_-}
        & X,
    \end{tikzcd}\]
    where the square is Cartesian. The natural isomorphism $\Ex^{\varphi, \bbD}$ is defined by the composite
    \begin{align*}
    \varphi_t\bbD &= g_0^*h_-^*i_-^!\bbD \\
    &\xleftarrow{\epsilon_{g_0}} g_0^!h_-^*i_-^!\bbD \\
    &\xrightarrow{\Ex^{*, !}} g_0^! h_+^!i_+^* \bbD \\
    &\xrightarrow{\Ex^{*, \bbD}} g_0^! h_+^! \bbD i_+^! \\
    &\xrightarrow{\Ex^{!, \bbD}} \bbD g_0^* h_+^* i_+^! \\
    &= \bbD \varphi_{-t},
    \end{align*}
    where the natural transformations on lines 2 and 3 are isomorphisms by \cite{MasseyVerdier}. $\Ex^{\varphi, \bbD}$ induces a natural isomorphism $\Ex^{\phi, \bbD}\colon \phi_t\bbD\xrightarrow{\sim} \bbD\phi_{-t}$.

    Property (1) follows from the fact that the natural transformations $\epsilon_f$ and $\Ex^{*, !}$ are Verdier self-dual while the natural isomorphisms $\Ex^{*, \bbD}$ and $\Ex^{!, \bbD}$ are exchanged under Verdier duality. Properties (2) and (3) follows from the standard commutativity diagrams between the six functors.
\end{proof}

\begin{remark}\label{rmk:phireverse}
There is a natural isomorphism $T_\pi\colon \phi_t \xrightarrow{\sim} \phi_{-t}$ such that the composite
\[\phi_t\xrightarrow{T_\pi} \phi_{-t}\xrightarrow{T_\pi} \phi_t\]
is the monodromy operator $T\colon \phi_t\rightarrow \phi_t$ for the sheaf of vanishing cycles. Consider the composite isomorphism
\[\overline{\Ex}^{\phi, \bbD}\colon \phi_t \bbD\xrightarrow{\Ex^{\phi, \bbD}} \bbD \phi_{-t}\xrightarrow{T_\pi} \bbD \phi_t.\]
Then \cref{prop:phiD}(1) implies that the diagram
\[
\xymatrix@C=1.3cm{
\phi_{t} \ar^{\psi}[d] \ar^{T}[r] & \phi_t \ar^-{\psi}[r]& \bbD \bbD \phi_{t} \ar^{\overline{\Ex}^{\phi,\bbD}}[d] \\
\phi_{t} \bbD \bbD \ar^{\overline{\Ex}^{\phi,\bbD}}[rr] && \bbD \phi_{-t} \bbD
}
\]
commutes.
\end{remark}

\begin{proposition}\label{prop:phiTS}
For a pair of schemes $X_1,X_2\in\Sch^{\sep\ft}$ equipped with functions $t_i\colon X_i\rightarrow \bA^1$ there is a natural Thom--Sebastiani isomorphism $\TS\colon\phi_{t_1}(-)\boxtimes \phi_{t_2}(-)\xrightarrow{\sim}\phi_{t_1\boxplus t_2}(-\boxtimes -)$ which satisfies the following properties:
\begin{enumerate}
    \item $\TS$ is unital, associative and commutative.
    \item For smooth morphisms $f_i\colon X_i\rightarrow Y_i$ in $\Sch^{\sep\ft}$ and functions $t_i\colon Y_i\rightarrow \bA^1$ with $t'_i\coloneqq f_i^\ast t_i$ the diagram
    \[
    \xymatrix{
    \phi_{t'_1} f_1^\dag\cF_1\boxtimes \phi_{t'_2} f_2^\dag \cF_2 \ar^-{\TS}[r] \ar^{\Ex^!_\phi\boxtimes \Ex^!_\phi}[d] & \phi_{t'_1\boxplus t'_2}(f_1\times f_2)^\dag(\cF_1\boxtimes \cF_2) \ar^{\Ex^!_\phi}[d] \\
    (f_1\times f_2)^\dag(\phi_{t_1}\cF_1\boxtimes \phi_{t_2} \cF_2) \ar^{\TS}[r] & (f_1\times f_2)^\dag \phi_{t_1\boxplus t_2}(\cF_1\boxtimes \cF_2)
    }
    \]
    commutes.
    \item For proper morphisms $f_i\colon X_i\rightarrow Y_i$ in $\Sch^{\sep\ft}$ and functions $t_i\colon Y_i\rightarrow \bA^1$ with $t'_i\coloneqq f_i^\ast t_i$ the diagram
    \[
    \xymatrix{
    \phi_{t_1} f_{1,*}\cF_1\boxtimes \phi_{t_2} f_{2,*} \cF_2 \ar^-{\TS}[r] \ar^{\Ex^\phi_*\boxtimes \Ex^\phi_*}[d] & \phi_{t_1\boxplus t_2}(f_1\times f_2)_*(\cF_1\boxtimes \cF_2) \ar^{\Ex^\phi_*}[d] \\
    (f_1\times f_2)_*(\phi_{t'_1}\cF_1\boxtimes \phi_{t'_2} \cF_2) \ar^{\TS}[r] & (f_1\times f_2)_* \phi_{t'_1\boxplus t'_2}(\cF_1\boxtimes \cF_2)
    }
    \]
    commutes.
    \item The diagram
    \[
    \xymatrix{
    \phi_{t_1}(\bbD\cF_1)\boxtimes \phi_{t_2}(\bbD\cF_2) \ar_{\Ex^{\phi, \bbD}\boxtimes \Ex^{\phi, \bbD}}[d] \ar^{\TS}[r] & \phi_{t_1\boxplus t_2}(\bbD\cF_1\boxtimes \bbD\cF_2) \ar^{\Ex^{\bbD, \boxtimes}}[r] & \phi_{t_1\boxplus t_2}(\bbD(\cF_1\boxtimes \cF_2)) \ar_{\Ex^{\phi, \bbD}}[d] \\
    \bbD\phi_{-t_1}(\cF_1)\boxtimes \bbD\phi_{-t_2}(\cF_2)\ar^{\Ex^{\bbD, \boxtimes}}[r] & \bbD(\phi_{-t_1}(\cF_1)\boxtimes \phi_{-t_2}(\cF_2)) & \bbD(\phi_{-(t_1\boxplus t_2)}(\cF_1\boxtimes \cF_2)) \ar_-{\TS}[l]
    }
    \]
    commutes.
\end{enumerate}
\end{proposition}
\begin{proof}
We begin by recalling the natural isomorphism
\[\TS^\varphi\colon \varphi_{t_1}(-)\boxtimes \varphi_{t_2}(-)\xrightarrow{\sim} \varphi_{t_1\boxplus t_2}(-\boxtimes -)|_{t_1^{-1}(0)\times t_2^{-1}(0)}\]
constructed in \cite{MasseyTS} and \cite[Corollary 1.3.4]{Schuermann}. Let
\[i_1\colon (X_1)_{\Re\leq 0}\longrightarrow X_1,\qquad z_1\colon t_1^{-1}(0)\longrightarrow (X_1)_{\Re\leq 0}\]
and similarly for $i_2,z_2$. Let
\[i\colon (X_1\times X_2)_{\Re\leq 0}\longrightarrow X_1\times X_2,\qquad z\colon (t_1\boxplus t_2)^{-1}(0)\longrightarrow (X_1\times X_2)_{\Re\leq 0}.\]
Finally, let
\[l\colon (X_1)_{\Re\leq 0}\times (X_2)_{\Re\leq 0}\longrightarrow (X_1\times X_2)_{\Re\leq 0},\]
which is an inclusion of a closed subset. The natural isomorphism $\TS^\varphi$ is defined by the composite
\begin{align*}
\varphi_{t_1}(-)\boxtimes \varphi_{t_2}(-) &= z_1^* i_1^!(-)\boxtimes z_2^*i_2^!(-) \\
&\xrightarrow{\sim} (z_1\times z_2)^*(i_1^!(-)\boxtimes i_2^!(-)) \\
&\xrightarrow{\Ex^{!, \boxtimes}} (z_1\times z_2)^*(i_1\times i_2)^!(-\boxtimes -) \\
&\xrightarrow{\sim} (z_1\times z_2)^*l^! i^!(-\boxtimes -) \\
&\xrightarrow{\epsilon_l} (z_1\times z_2)^* l^* i^!(-\boxtimes -) \\
&= \varphi_{t_1\boxplus t_2}(-\boxtimes -)|_{t_1^{-1}(0)\times t_2^{-1}(0)},
\end{align*}
where the natural transformations on lines 3 and 5 are isomorphisms by \cite{MasseyTS}.

Fix $\cF_1\in\Dbc(X_1)$ and $\cF_2\in\Dbc(X_2)$. In \cite{MasseyTS} it is shown that $(t_1^{-1}(0)\times t_2^{-1}(0))\cap \supp \varphi_{t_1\boxplus t_2}(\cF_1\boxtimes \cF_2)$ is an open subset of $\supp \varphi_{t_1\boxplus t_2}(\cF_1\boxtimes \cF_2)$. Therefore, we have a direct sum decomposition
\[\varphi_{t_1\boxplus t_2}(\cF_1\boxtimes \cF_2)\cong \bigoplus_{c\in\C} (t_1^{-1}(c)\times t_2^{-1}(-c)\rightarrow (t_1\boxplus t_2)^{-1}(0))_*\left(\varphi_{t_1\boxplus t_2}(\cF_1\boxtimes \cF_2)|_{t_1^{-1}(c)\times t_2^{-1}(-c)}\right).\]
Applying the Thom--Sebastiani isomorphism $\TS^\varphi$ we thus get an isomorphism
\[\bigoplus_{c\in\C} (t_1^{-1}(c)\times t_2^{-1}(-c)\rightarrow (t_1\boxplus t_2)^{-1}(0))_* \left(\varphi_{t_1-c}(\cF_1)\boxtimes \varphi_{t_2+c}(\cF_2)\right)\xrightarrow{\sim}\varphi_{t_1\boxplus t_2}(\cF_1\boxtimes \cF_2)\]
natural in $\cF_1,\cF_2$. This gives rise to a natural isomorphism
\[
\begin{tikzcd}
\left(\bigoplus_{c_1\in\C} (t_1^{-1}(c_1)\rightarrow X_1)_* \varphi_{t_1-c_1}(\cF_1)\right)\boxtimes\left(\bigoplus_{c_2\in\C} (t_2^{-1}(c_2)\rightarrow X_2)_* \varphi_{t_2-c_2}(\cF_2)\right) \ar{d}{\TS}[swap]{\sim} \\
\bigoplus_{c\in\C} ((t_1\boxplus t_2)^{-1}(c)\rightarrow X_1\times X_2)_*\varphi_{(t_1\boxplus t_2)-c}(\cF_1\boxplus \cF_2).
\end{tikzcd}
\]

Unitality, associativity and commutativity of the isomorphism $\TS$ follow from the corresponding properties of $\TS^\varphi$ which are obvious from the construction. Properties (2) and (3) follow from the natural compatibilities between the exchange natural transformations. Property (4) follows from the compatibility of exchange natural transformations with external tensor products.
\end{proof}

\section{D-critical structures}

\subsection{Oriented orthogonal bundles}\label{sect:orthogonal}

Let $U$ be a scheme. An \defterm{orthogonal bundle} is a vector bundle $E\rightarrow U$ equipped with a nondegenerate quadratic form $q$. For an orthogonal vector bundle $(E, q)$ over $U$ the quadratic form induces an isomorphism $q^\sharp\colon E\xrightarrow{\sim} E^\vee$. Taking its determinant we obtain an isomorphism  
\[\det(E)\xrightarrow{\det(q^\sharp)} \det(E^\vee)\xrightarrow{\iota_E} \det(E)^\vee.\]
In particular, we obtain a \defterm{squared volume form}, i.e. a trivialization
\[\vol_q^2\colon \cO_U\cong\det(E)^{\otimes 2}\]
so that
\[\cO_U\xrightarrow{\vol_q^2} \det(E)^{\otimes 2}\xrightarrow{(\iota_E\circ\det(q^\sharp))\otimes \id} \det(E)^\vee\otimes \det(E)\]
is the coevaluation of the duality between $\det(E)$ and $\det(E)^\vee$.

\begin{example}
Suppose $E$ has an orthonormal basis $\{s_1, \dots, s_n\}$ of sections. Then the squared volume form $\vol_q^2$ is given by
\[1\mapsto (-1)^{n(n-1)/2} (s_1\wedge \dots \wedge s_n)^{\otimes 2}.\]
\end{example}

\begin{lemma}\label{lm:determinantorthogonal}
Let $(E_1, q_1)$ and $(E_2, q_2)$ be two orthogonal vector bundles over a scheme $U$. Then the diagram
\[
\begin{tikzcd}
\det(E_1 \oplus E_2)^{\otimes 2} \arrow[r, "\sim"] & (\det(E_1)\otimes \det(E_2))^{\otimes 2} \arrow[r, "\sim"] & \det(E_1)^{\otimes 2}\otimes \det(E_2)^{\otimes 2} \\
& \cO_U \arrow[ul, "\vol^2_{q_1+ q_2}"] \arrow[ur, "\vol^2_{q_1}\otimes \vol^2_{q_2}" below] &
\end{tikzcd}
\]
is commutative.
\end{lemma}
\begin{proof}
The statement is local on $U$, so we may assume that $E_1$ has an orthonormal basis of sections $\{e_1, \dots, e_n\}$ and $E_2$ has an orthonormal basis of sections $\{f_1, \dots, f_m\}$. Then the image of $\vol^2_{q_1+q_2}$ under the top isomorphism is
\begin{align*}
1&\mapsto (-1)^{(n+m)(n+m-1)/2} (e_1\wedge \dots \wedge e_n\wedge f_1\wedge \dots f_m)^{\otimes 2} \\
&\mapsto (-1)^{(n+m)(n+m-1)/2}(-1)^{nm}(e_1\wedge \dots \wedge e_n)^{\otimes 2}\otimes (f_1\wedge \dots \wedge f_m)^{\otimes 2}.
\end{align*}
But $(-1)^{(n+m)(n+m-1)/2}(-1)^{nm} = (-1)^{n(n-1)/2}(-1)^{m(m-1)/2}$, so this expression coincides with $\vol^2_{q_1}\otimes \vol^2_{q_2}$.
\end{proof}

We will now define orientations of orthogonal bundles.

\begin{definition}
Let $(E, q)$ be an orthogonal bundle over $U$. An \defterm{orientation} of $E$ is an isomorphism $\vol\colon \cO_U\xrightarrow{\sim} \det(E)$ whose square is $\vol_q^2$. We denote by $\ori_E\rightarrow U$ the $\Z/2\Z$-graded \defterm{orientation $\mu_2$-torsor} of $E$: its parity coincides with the parity of $\rank(E)$ and the underlying $\mu_2$-torsor parametrizes orientations.
\end{definition}

For a pair of orthogonal bundles $(E_1, q_1)$ and $(E_2, q_2)$ we have an isomorphism
\begin{equation}\label{eq:orientationsumisomorphism}
\ori_{E_1}\otimes_{\mu_2} \ori_{E_2}\longrightarrow \ori_{E_1\oplus E_2}
\end{equation}
which sends
\[(\vol_1, \vol_2)\mapsto \vol_1\otimes \vol_2,\]
which squares to $\vol_{q_1+q_2}^2$ by \cref{lm:determinantorthogonal}.

For an orthogonal bundle $(E, q)$ we denote by $\overline{E}$ the orthogonal bundle $(E, -q)$. Then we have an isomorphism
\begin{equation}\label{eq:orientationreverse}
\ori_E\longrightarrow \ori_{\overline{E}}
\end{equation}
which sends $\vol\mapsto i^{\rank(E)}\vol$, using the fact that $\vol^2_{-q} = (-1)^{\rank(E)} \vol^2_q$.

Besides the direct sum of orthogonal bundles we will also consider reductions of orthogonal bundles by isotropic subbundles.

\begin{definition}
Let $(E, q)$ be an orthogonal bundle over a scheme $U$ and $K\subseteq E$ an isotropic subbundle. The \defterm{reduction} of $E$ by $K$ is the orthogonal bundle $K^{\perp}/K$.
\end{definition}

\begin{lemma}\label{lm:orientationreduction}
Let $(E, q)$ be an orthogonal bundle over a scheme $U$, $K\subseteq E$ an isotropic subbundle and $F=K^\perp/K$ the reduction of $E$ by $K$. Consider the composite isomorphism
\begin{align*}
\red_K\colon \det(E)&\xleftarrow{i(\Delta_1)} \det(K^\perp)\otimes \det(K^\vee)\\
&\xrightarrow{\id\otimes \iota_K}\det(K^\perp)\otimes \det(K)^\vee\\
&\xleftarrow{i(\Delta_2)} \det(K)\otimes \det(K)^\vee\otimes \det(F)\\
&\cong \det(F)
\end{align*}
induced by the exact sequences
\[\Delta_1\colon 0\longrightarrow K^{\perp}\longrightarrow E\longrightarrow K^\vee\longrightarrow 0\]
and
\[\Delta_2\colon 0\longrightarrow K\longrightarrow K^\perp\longrightarrow F\longrightarrow 0.\]
Then
\[\red_K(\vol^2_{E, q}) = (-1)^{\rank K} \vol^2_{F, q}.\]
\end{lemma}
\begin{proof}
The claim is local, so we may assume that $E=F\oplus K\oplus K^\vee$ with $\{s_1, \dots, s_n\}$ an orthonormal basis of sections of $F$, $\{e_1, \dots, e_m\}$ a basis of sections of $K$ and $\{e^1, \dots, e^m\}$ the dual basis of sections of $K^\vee$. Then
\[\red_K(s_1\wedge \dots\wedge s_n\wedge e_1\wedge \dots\wedge e_m\wedge e^1\wedge \dots \wedge e^m) = (-1)^{m(m-1)/2} s_1\wedge \dots \wedge s_n.\]
Thus,
\begin{align*}
\red_K(\vol^2_{E, q}) &= (-1)^{(n+2m)(n+2m-1)/2} \red_K(s_1\wedge \dots\wedge s_n\wedge e_1\wedge \dots\wedge e_m\wedge e^1\wedge \dots \wedge e^m)^2 \\
&= (-1)^{(n+2m)(n+2m-1)/2} (s_1\wedge \dots\wedge s_n)^2 \\
&= (-1)^{n(n-1)/2} (-1)^m (s_1\wedge \dots\wedge s_n)^2 \\
&= (-1)^m\vol^2_{F, q}.
\end{align*}
\end{proof}

In particular, using the notation of \cref{lm:orientationreduction} we obtain a canonical isomorphism
\begin{equation}\label{eq:orientationreduction}
\ori_E\cong \ori_{K^\perp/K}
\end{equation}
which sends a volume form $\vol$ on $E$ to the volume form $i^{\rank K} \red_K(\vol)$ on $K^\perp/K$.

\begin{definition}
Let $(E, q)$ be an orthogonal bundle of even rank over a scheme $U$.
\begin{itemize}
    \item An isotropic subbundle $K\subseteq E$ is \defterm{Lagrangian} if $\rank(E) = 2\rank(K)$.
    \item Suppose $E$ carries an orientation. A Lagrangian subbundle $K\subseteq E$ is \defterm{positive} if the image of the orientation of $E$ under the isomorphism \eqref{eq:orientationreduction} is the standard orientation of the zero bundle $K^\perp/K=0$.
\end{itemize}
\end{definition}

Given an orthogonal bundle $(E, q)$ over a scheme $U$ we consider the following maps:
\begin{itemize}
\item $\pi_E\colon E \to U$ is the projection; 
\item $0_E\colon U \to E$ is the zero section;  
\item $\q_E\colon E\to \bA^1$ the function quadratic along the fibers of $\pi_E$ corresponding to the quadratic form $q$.
\end{itemize}
Then we can form the diagram
\[
\xymatrix{
E \ar[r]^{\q_E} \ar[d]^{\pi_E} & \bA^1 \\
U. \ar@/^1pc/[u]^{0_E}
}
\]

\subsection{D-critical structures on schemes}

For a scheme $X$ over $B$ we introduce a presheaf on the \'etale site of $X$ by the formula
\[(U\rightarrow X)\mapsto \Gamma(U, \cS_{U/B})\coloneqq \pi_0(\cA^{2, \ex}(U/B, -1))\]

\begin{proposition}\label{prop:Sproperties}
Let $X\rightarrow B$ be a morphism of schemes.
\begin{enumerate}
    \item $\cS_{X/B}$ is a sheaf on $X$ in the \'etale topology.
    \item There is a long exact sequence
    \begin{equation}\label{eq:Slongexactsequence}
    0\longrightarrow h^{-1}(\bL_{X/B})\longrightarrow \cS_{X/B}\xrightarrow{\und} \cO_X\xrightarrow{d_B} h^0(\bL_{X/B})=\Omega^1_{X/B}
    \end{equation}
    of sheaves on $X$.
    \item Let $X\hookrightarrow U$ be a closed immersion into a smooth $B$-scheme $U$ with $\cI_{X,U}$ the ideal sheaf. Then there is an exact sequence
    \[0\longrightarrow \cS_{X/B}\xrightarrow{\iota_{X, U}} \cO_U/\cI_{X,U}^2\xrightarrow{d_B} \Omega^1_{U/B}/\cI_{X,U}\Omega^1_{U/B},\]
    so that the composite $\cS_{X/B}\xrightarrow{\iota_{X, U}} \cO_U/\cI_{X, U}^2\rightarrow \cO_U/\cI_{X, U}\cong \cO_X$ is equal to $\und\colon \cS_{X/B}\rightarrow \cO_X$.
\end{enumerate}
\end{proposition}
\begin{proof}
By the definition of $\cA^{2, \ex}(X/B, -1)$ we obtain a long exact sequence \eqref{eq:Slongexactsequence} of presheaves. Since $h^{-1}(\bL_{X/B})$ and $h^0(\bL_{X/B})$ are quasi-coherent presheaves of $\cO_X$-modules, they are sheaves. Therefore, using the long exact sequence \eqref{eq:Slongexactsequence} we get that $\cS_{X/B}$ is a sheaf since $h^{-1}(\bL_{X/B})$, $\cO_X$ and $h^0(\bL_{X/B})$ are sheaves. This proves parts (1) and (2).

The definition of $\cS_{X/B}$ is insensitive to replacing $\bL_{X/B}$ by its truncation $\tau_{\geq -1} \bL_{X/B}$. By \cite[Tags 0FV4 and 08UW]{Stacks} we may model $\cO_X\xrightarrow{d_B} \tau_{\geq -1}\bL_{X/B}$ by
\[
\xymatrix{
\cI_{X, U}/\cI_{X,U}^2 \ar[r] & \Omega^1_{U/B} / \cI_{X,U}\Omega^1_{U/B} \\
\cI_{X, U} \ar[r] \ar[u] & \cO_U \ar^{d_B}[u]
}
\]
This gives the description of $\cS_{X/B}$ as in part (3).
\end{proof}

\begin{example}
If $X\rightarrow B$ is smooth, using the exact sequence \eqref{eq:Slongexactsequence} we identify sections of $\cS_{X/B}$ with functions $f\colon X\rightarrow \bA^1$ such that $d_B f=0$.
\end{example}

It is useful to use the following terminology.

\begin{definition}
Let $B$ be a scheme.
\begin{enumerate}
    \item An \defterm{LG pair over $B$} is a smooth $B$-scheme $U\rightarrow B$ together with a function $f\colon U\rightarrow \bA^1$.
    \item A \defterm{morphism $\Phi\colon (U, f)\rightarrow (V, g)$} of LG pairs over $B$ is a morphism of $B$-schemes $\Phi\colon U\rightarrow V$ such that $\Phi^\ast g = f$.
\end{enumerate}
\end{definition}

Recall that in \cref{sect:criticallocus} for an LG pair $(U, f)$ over $B$ we have considered the relative critical locus $\Crit_{U/B}(f)$ which is equipped with a section
\[s_f\in\Gamma(\Crit_{U/B}(f), \cS_{\Crit_{U/B}(f)/B}).\]
It has the following explicit description.

\begin{proposition}\label{prop:Critsdescription}
Consider an LG pair $(U, f)$ over $B$. Under the embedding $\Crit_{U/B}(f)\rightarrow U$ we have $\iota_{\Crit_{U/B}(f), U}(s_f) = f\pmod {\cI_{\Crit_{U/B}(f), U}^2}$.
\end{proposition}
\begin{proof}
Let $\cI_{\Gamma_{d_B f}}$ be the ideal defining the closed immersion $\Gamma_{d_B f}$ and $R=\Crit_{U/B}(f)$. The pullback
\[\Gamma_{d_B f}^\ast\colon (\cO_{\T^* (U/B)}\xrightarrow{d_B} \Omega^1_{\T^*(U/B)/B})\rightarrow (\cO_U\xrightarrow{d_B} \Omega^1_{U/B})\]
factors as the composite
\[
\xymatrix{
\cO_{\T^*(U/B)} \ar^{d_B}[r] \ar[d] & \Omega^1_{\T^*(U/B)/B} \ar[d] \\
\cO_{\T^*(U/B)}/\cI^2_{\Gamma_{d_B f}} \ar^-{d_B}[r] \ar^{\Gamma_{d_B f}^\ast}[d] & \Omega^1_{\T^*(U/B)/B} / \cI_{\Gamma_{d_B f}} \Omega^1_{\T^* (U/B)} \ar^{\Gamma_{d_B f}^\ast}[d] \\
\cO_U \ar^{d_B}[r] & \Omega^1_{U/B},
}
\]
where the top vertical morphisms are given by modding out by powers of $\cI_{\Gamma_{d_B f}}$. As in the proof of \cref{prop:Sproperties}, the bottom vertical morphisms assemble into a quasi-isomorphism of two-term complexes in degrees $[0, 1]$.

Consider the Liouville one-form $\lambda_{U/B}\in \Omega^1_{\T^*(U/B)/B}$. The nullhomotopy $h_f$ of $\Gamma_{d_B f}^\ast \lambda_{U/B}$ represented by $f\in\cO_U$ in the bottom complex lifts to a nullhomotopy represented by $\pi^\ast f\in\cO_{\T^*(U/B)}/\cI^2_{\Gamma_{d_B f}}$ in the middle complex, where $\pi\colon \T^*(U/B)\rightarrow U$ is the projection.

Next, the pullback
\[\Gamma_0^\ast\colon (\cO_U\xrightarrow{d_B} \Omega^1_{U/B})\rightarrow (\cO_R\xrightarrow{d_B} \tau_{\geq -1} \bL_{R/B})\]
using the description of the truncated cotangent complex as in the proof of \cref{prop:Sproperties}(3) is given by
\[
\xymatrix{
\cO_{\T^*(U/B)}/\cI^2_{\Gamma_{d_B f}} \ar^-{d_B}[r] \ar^{\Gamma_0^\ast}[d] & \Omega^1_{\T^*(U/B)/B} / \cI_{\Gamma_{d_B f}} \Omega^1_{\T^* (U/B)} \ar^{\Gamma_0^\ast}[d] \\
\cO_U/\cI_{R, U}^2 \ar^-{d_B}[r] & \Omega^1_{U/B}/\cI_{R, U}\Omega^1_{U/B}.
}
\]
Under the morphism
\[\Gamma_0^\ast\colon \Omega^1_{\T^*(U/B)/B} / \cI_{\Gamma_{d_B f}} \Omega^1_{\T^* (U/B)}\longrightarrow \Omega^1_{U/B}/\cI_{R, U}\Omega^1_{U/B}\]
we have $\Gamma_0^\ast \lambda_{U/B} = 0$ which represents the nullhomotopy $h_0$. Therefore, the difference $s_f = h_f-h_0\in\cS_{R/B}$ has image $\Gamma_0^\ast \pi^\ast f = f$ under $\iota_{R, U}\colon \cS_{R/B}\rightarrow \cO_U/\cI_{R, U}^2$.
\end{proof}

\begin{corollary}
Consider an LG pair $(U, f)$ over $B$. Then $\und(s_f) = f|_{\Crit_{U/B}(f)}\in\cO_{\Crit_{U/B}(f)}$.
\end{corollary}
\begin{proof}
The claim follows from \cref{prop:Critsdescription} as well as the compatibility of $\iota_{\Crit_{U/B}(f), U}$ and $\und$ established in \cref{prop:Sproperties}(3).
\end{proof}

We have the following functoriality of relative critical loci. For a morphism $\Phi\colon (U, f)\rightarrow (V, g)$ of LG pairs we have a correspondence
\begin{equation}\label{eq:cotangentcorrespondence}
\xymatrix{
& \T^*(V/B)\times_V U \ar_{\pi_V}[dl] \ar^{\pi_U}[dr] & \\
\T^*(V/B) && \T^*(U/B)
}
\end{equation}
of relative cotangent bundles. Now consider the diagram
\begin{equation}\label{eq:cotangentcriticalcorrespondence}
\xymatrix{
& U \ar_-{\Gamma_{d_B g}\times \id}[dr] \ar_{\Phi}[dl] \ar@/^2pc/^{\Gamma_{d_B f}}[ddrr] & \\
V \ar_{\Gamma_{d_B g}}[dr] && \T^*(V/B)\times_V U \ar^{\pi_U}[dr] \ar_{\pi_V}[dl] \\
& \T^*(V/B) && \T^*(U/B),
}
\end{equation}
where the square is Cartesian. The triangle commutes because $\Phi^\ast d_B g = d_B f$. Taking the zero loci of the three sections $V\rightarrow \T^*(V/B)$ (with zero locus $\Crit_{V/B}(g)$), $U\rightarrow \T^*(V/B)\times_V U$ (with zero locus $\Crit_{V/B}(g)\times_V U$) and $U\rightarrow \T^*(U/B)$ (with zero locus $\Crit_{U/B}(f)$) we obtain a correspondence
\begin{equation}\label{eq:criticalcorrespondence}
\xymatrix{
& \Crit_{V/B}(g)\times_V U \ar_{\pi_V}[dl] \ar^{\pi_U}[dr] & \\
\Crit_{V/B}(g) && \Crit_{U/B}(f)
}
\end{equation}
of relative critical loci.

\begin{proposition}\label{prop:Critfunctoriality}
Let $\Phi\colon (U, f)\rightarrow (V, g)$ be a morphism of LG pairs over $B$. Then
\[\pi_V^\ast s_g = \pi_U^\ast s_f\]
in the correspondence \eqref{eq:criticalcorrespondence}.
\end{proposition}
\begin{proof}
In the correspondence \eqref{eq:cotangentcorrespondence} we have $\pi_U^\ast \lambda_{U/B} = \pi_V^\ast \lambda_{V/B}$ as can be easily checked in local coordinates. Now consider the correspondence \eqref{eq:cotangentcriticalcorrespondence}. In the following we consider homotopies in the spaces of exact relative two-forms of degree $0$. The nullhomotopy $h_g$ of $\Gamma_{d_B g}^\ast \lambda_{V/B}$ is represented by $g\in\cO_V$. Its pullback along $\Phi\colon U\rightarrow V$ therefore provides a nullhomotopy $h'_g\colon (\Gamma_{d_B g}\times \id)^\ast \pi_V^\ast \lambda_{V/B}\sim 0$ represented by $\Phi^\ast g=f\in\cO_U$. Equating the one-forms $\pi_V^\ast \lambda_{V/B}=\pi_U^\ast \lambda_{U/B}$ and $(\Gamma_{d_B g}\times \id)^\ast \pi_V^\ast \lambda_{V/B}=\Gamma_{d_B f}^\ast \pi_U^\ast \lambda_{U/B}$ we see that the resulting nullhomotopy of $\Gamma_{d_B f}^\ast \pi_U^\ast \lambda_{U/B}$ coincides with $h_f$. Passing to the zero loci we get $\pi_U^\ast s_f = \pi_V^\ast s_g$.
\end{proof}

\begin{remark}
In the following sections we will encounter morphisms of LG pairs $\Phi\colon (U, f)\rightarrow (V, g)$ over $B$ such that $\Phi\colon U\rightarrow V$ restricts to $\Phi\colon \Crit_{U/B}(f)\rightarrow \Crit_{V/B}(g)$ (as $\Crit_{V/B}(g)\rightarrow V$ is a closed immersion, such a restriction is unique, if it exists). In this case $\Phi$ defines a splitting of $\pi_U\colon \Crit_{V/B}(g)\times_V U\rightarrow \Crit_{U/B}(f)$ and thus \cref{prop:Critfunctoriality} implies that $\Phi^\ast s_g = s_f$.
\end{remark}

We will use the Hessian quadratic form associated to an LG pair.

\begin{proposition}\label{prop:criticalHessian}
Let $(U, f)$ be an LG pair over $B$. There is a (degenerate) quadratic form $\Hess(f)$, the \defterm{Hessian} of $f$, on the restriction of the tangent bundle $\T_{U/B}|_{\Crit_{U/B}(f)}$. It satisfies the following properties:
\begin{enumerate}
    \item Let $x\in\Crit_{U/B}(f)$. Then $\Ker(\Hess(f)_x) = \T_{\Crit_{U/B}(f)/B, x}$, i.e. the Hessian $\Hess(f)_x$ restricts to a nondegenerate quadratic form on the normal bundle $\rN_{\Crit_{U/B}(f)/U, x}$.
    \item Let $\Phi\colon (U, f)\rightarrow (V, g)$ be a morphism of LG pairs. Then in the correspondence \eqref{eq:criticalcorrespondence} we have
    \[\pi_V^\ast \Hess(g) = \pi_U^\ast \Hess(f)\]
    as quadratic forms on $\T_{U/B}|_{\Crit_{V/B}(g)\times_V U}$.
\end{enumerate}
\end{proposition}
\begin{proof}
For a smooth morphism $U\rightarrow B$ we may define its $n$-th jet bundle $J^n_{U/B}$, which is an $\cO_U$-bimodule on $U$, as in \cite[\S16.7]{EGA44}. It has the following properties:
\begin{enumerate}
    \item $J^0_{U/B} = \cO_U$.
    \item For $n\geq m$ there is a morphism $J^n_{U/B}\rightarrow J^m_{U/B}$.
    \item There is a splitting $i_n\colon \cO_U\rightarrow J^n_{U/B}$ of $J^n_{U/B}\rightarrow \cO_U$ as left $\cO_U$-modules and a splitting $d^n_B\colon \cO_U\rightarrow J^n_{U/B}$ as right $\cO_U$-modules.
    \item There is a short exact sequence
    \[0\longrightarrow \Sym^n \Omega^1_{U/B}\longrightarrow J^n_{U/B}\longrightarrow J^{n-1}_{U/B}\longrightarrow 0.\]
\end{enumerate}

Let $\overline{J}^n_{U/B}$ be the quotient of $J^n_{U/B}$ by $i_n\colon \cO_U\rightarrow J^n_{U/B}$. Then we obtain a short exact sequence
\[0\longrightarrow \Sym^2 \Omega^1_{U/B}\longrightarrow \overline{J}^2_{U/B}\longrightarrow \overline{J}^1_{U/B}\cong \Omega^1_{U/B}\longrightarrow 0.\]
For $f\in\cO_U$ consider the element $d^2_B f\in\overline{J}^2_{U/B}$ whose image in $\overline{J}^1_{U/B}\cong \Omega^1_{U/B}$ is $d_B f$. Restricting to $\Crit_{U/B}(f)$ we get that $d^2_B f|_{\Crit_{U/B}(f)}$ defines a section $\Hess(f)$ of $\Sym^2 \Omega^1_{U/B}|_{\Crit_{U/B}(f)}$ which is the relevant quadratic form. The functoriality of the Hessian (i.e. property (2)) follows from the functoriality of jet bundles and the map $d^2_B$.

Let us now show property (1). Choose \'etale coordinates $\{z_1, \dots, z_n\}$ on $U\rightarrow B$ in a neighborhood of $x$. Let $\cI$ be the ideal defining the closed immersion $\Crit_{U/B}(f)\rightarrow U$. In a neighborhood of $x$ we have $\cI=\left(\frac{\partial f}{\partial z_1}, \dots, \frac{\partial f}{\partial z_n}\right)$. A vector field $v=\sum_i v_i\frac{\partial}{\partial z_i}$ on $U$ tangent to $\Crit_{U/B}(f)$ if, and only if,
\[\sum_i v_i \frac{\partial^2 f}{\partial z_i\partial z_j}\in\cI\]
for every $j$. Restricting to $x\in\Crit_{U/B}(f)$ we get that the subspace $\T_{\Crit_{U/B}(f)/B, x}\subset \T_{U/B, x}$ is given by tangent vectors $v=\sum_i v_i\frac{\partial}{\partial z_i}$ such that
\[\sum_i v_i\frac{\partial^2 f}{\partial z_i\partial z_j} = 0.\]
But this is precisely the definition of the subspace $\Ker(\Hess(f)_x)\subset \T_{U/B, x}$.
\end{proof}

Global versions of relative critical loci are given as follows.

\begin{definition}\label{def:dcriticalscheme}
Let $X\rightarrow B$ be a morphism of schemes. A \defterm{relative d-critical structure on $X\rightarrow B$} is a section $s\in \Gamma(X, \cS_{X/B})$ which satisfies the following property:
\begin{itemize}
    \item There exists a collection of LG pairs $\{U_a, f_a\}_{a\in A}$ together with open immersions $u_a\colon \Crit_{U_a/B}(f_a)\rightarrow X$ such that $\{u_a\colon \Crit_{U_a/B}(f_a)\rightarrow X\}$ is an open cover of $X$ and $u_a^\ast s = s_{f_a}$. We call such a triple $(U_a, f_a, u_a)$ a \defterm{critical chart}.
\end{itemize}
\end{definition}

\begin{remark}
Take $B=\Spec \C$ and suppose $X\rightarrow \Spec \C$ carries a relative d-critical structure $s$ with $\und(s)=f\colon X\rightarrow \bA^1$. Then $f$ is locally constant on $X^{\red}$. If we further assume that $f|_{X^{\red}}=0$, then the relative d-critical structure on $X\rightarrow \Spec \C$ is the same as a d-critical structure on $X$ in the sense of \cite{JoyceDcrit}.
\end{remark}

\begin{example}\label{ex:trivialdcritical}
Let $\pi\colon X\rightarrow B$ be a smooth morphism of schemes and $f\colon B\rightarrow \bA^1$ be a function. Then $\Crit_{X/B}(\pi^\ast f)=X$ and hence $\pi^\ast f\in\cS_{X/B}$ defines a relative d-critical structure on $X\rightarrow B.$
\end{example}

We can define products of relative d-critical structures as follows. Given two morphisms of schemes $X_1\rightarrow B_1, X_2\rightarrow B_2$ equipped with sections $s_1\in\Gamma(X_1, \cS_{X_1/B_1})$ and $s_2\in\Gamma(X_2, \cS_{X_2/B_2})$ let
\[s_1\boxplus s_2 = \pi_1^\ast s_1 + \pi_2^\ast s_2\in\Gamma(X_1\times X_2, \cS_{X_1\times X_2/B_1\times B_2}),\]
where $\pi_i\colon X_1\times X_2\longrightarrow X_i$ is the projection.

\begin{proposition}$ $\label{prop:criticalproduct}
\begin{enumerate}
    \item Let $(U_1, f_1)$ be an LG pair over $B_1$ and $(U_2, f_2)$ be an LG pair over $B_2$. Then there is an isomorphism
    \[\Crit_{U_1/B_1}(f_1)\times \Crit_{U_2/B_2}(f_2)\cong \Crit_{U_1\times U_2/B_1\times B_2}(f_1\boxplus f_2)\]
    under which $s_{f_1}\boxplus s_{f_2}\mapsto s_{f_1\boxplus f_2}$.
    \item Let $X_1\rightarrow B_1$ and $X_2\rightarrow B_2$ be morphisms of schemes equipped with relative d-critical structures $s_i\in\Gamma(X_i, \cS_{X_i/B_i})$. Then $s_1\boxplus s_2$ is a relative d-critical structure on $X_1\times X_2$.
\end{enumerate}
\end{proposition}
\begin{proof}$ $
\begin{enumerate}
    \item We have a natural isomorphism $\bL_{U_1/B_1}\boxplus \bL_{U_2/B_2}\cong \bL_{U_1\times U_2/B_1\times B_2}$. Using this isomorphism we obtain an isomorphism $\T^*(U_1/B_1)\times \T^*(U_2/B_2)\cong \T^*(U_1\times U_2/B_1\times B_2)$ under which $\Gamma_{d_B f_1}\times \Gamma_{d_B f_2}\mapsto \Gamma_{d_B (f_1\boxplus f_2)}$. This shows that the two closed subschemes $\Crit_{U_1/B_1}(f_1)\times \Crit_{U_2/B_2}(f_2)$ and $\Crit_{U_1\times U_2/B_1\times B_2}(f_1\boxplus f_2)$ of $U_1\times U_2$ are equal. The fact that under this isomorphism $s_{f_1}\boxplus s_{f_2}\mapsto s_{f_1\boxplus f_2}$ follows from \cref{prop:Critsdescription}.
    \item Given a collection of critical charts $(U^1_a, f^1_a, u^1_a)$ of $X_1$ and $(U^2_a, f^2_a, u^2_a)$ of $X_2$, by part (1) we get that $(U^1_a\times U^2_a, f^1_a\boxplus f^2_a, u^1_a\times u^2_a)$ is a critical chart for $(X_1\times X_2, s_1\boxplus s_2)$. Since $\{u^i_a\colon \Crit_{U_a^i/B_i}(f^i_a)\rightarrow X_i\}$ is an open cover, $\{u^1_a\times u^2_b\colon \Crit_{U_a^1/B_1}(f^1_a)\times \Crit_{U_a^2/B_2}(f^2_a)\rightarrow X_1\times X_2\}$ is an open cover.
\end{enumerate}
\end{proof}

For a morphism of schemes $X\rightarrow B$ equipped with a d-critical structure $s$ the opposite $-s$ is also a d-critical structure, so that if $\{U_a, f_a, u_a\}$ is a collection of critical charts of $(X, s)$, then $\{U_a, -f_a, u_a\}$ is a collection of critical charts of $(X, -s)$.




\subsection{Critical morphisms}\label{sect:chartcomparison}

In this section we introduce a particularly nice class of morphisms of LG pairs and prove their local structure. Throughout the section we fix a base scheme $B$.

\begin{definition}
A \defterm{critical morphism} of LG pairs $(U, f)\rightarrow (V, g)$ over $B$ is an unramified morphism $\Phi\colon (U, f)\rightarrow (V, g)$ of LG pairs over $B$ such that $\Phi$ restricts to a morphism $\Phi\colon \Crit_{U/B}(f)\rightarrow \Crit_{V/B}(g)$.
\end{definition}

Clearly, any \'etale morphism $(U^\circ, f^\circ)\rightarrow (U, f)$ is a critical morphism.

\begin{remark}
In \cite{JoyceDcrit} Joyce introduces the notion of an \emph{embedding} of critical charts, which is equivalent to critical morphisms with $\Phi$ assumed to be an immersion instead of just an unramified morphism.
\end{remark}

\begin{example}\label{ex:criticalstabilization}
Let $(U, f)$ be an LG pair over $B$ and $(E, q)$ an orthogonal bundle over $U$. Let $\pi_E\colon E\rightarrow U$ and $\q_E\colon E\rightarrow \bA^1$ be as in \cref{sect:orthogonal}. The inclusion $0_E\colon U\rightarrow E$ of the zero section identifies $\Crit_{U/B}(f)\cong\Crit_{E/B}(f\circ \pi_E+\q_E)$ compatibly with the natural exact two-forms of degree $-1$. Thus, $(E, f\circ \pi_E + \q_E)$ is an LG pair over $B$ and
\[0_E\colon (U, f)\longrightarrow (E, f\circ \pi_E + \q_E)\]
is a critical morphism.
\end{example}

The rest of the section is devoted to local structure results for relative LG pairs and critical morphisms. First, an LG pair can always be replaced by one which admits an open immersion into $\bA^n_B$; the following is a family version of \cite[Proposition 2.19]{JoyceDcrit}.

\begin{proposition}\label{prop:chartopenaffine}
Let $(U, f)$ be an LG pair over $B$. For every $x\in \Crit_{U/B}(f)$ there is an open neighborhood $U^\circ\subset U$ of $x$, a smooth $B$-scheme $V$ which admits an open immersion $V\hookrightarrow \bA^k_B$ together with a function $g\colon V\rightarrow \bA^1$ and a closed immersion $\Phi\colon U^\circ\rightarrow V$ such that $\Phi^\ast g = f|_{U^\circ}$ and $\Phi$ restricts to an isomorphism $\Crit_{U^\circ/B}(f|_{U^\circ})\cong \Crit_{V/B}(g)$.
\end{proposition}
\begin{proof}
Choose an open neighborhood $U^\circ\subset U$ of $x$ which is affine over $B$ and which admits an \'etale morphism $c^\circ\colon U^\circ\rightarrow \bA^n_B$. Then we have may find a closed immersion $\Phi\colon U^\circ\hookrightarrow V=\bA^k_B$. Possibly shrinking $V$ to a neighborhood of $\Phi(x)$ we may extend $c^\circ\colon U^\circ\rightarrow \bA^n_B$ to $p\colon V\rightarrow \bA^n_B$ such that $p\circ\Phi = c^\circ$ and $f^\circ\coloneqq f|_{U^\circ}$ to $g'\colon V\rightarrow \bA^1$ such that $g'\circ\Phi = f^\circ$.

Since $\Phi\colon U^\circ\rightarrow V$ is a regular immersion, by shrinking $U^\circ$ and $V$ we may find a function $r\colon V\rightarrow \bA^m$ such that $U^\circ$ is the zero locus of $r$ and in the commutative diagram
\[
\xymatrix{
U^\circ \ar^{\Phi}[r] \ar^{c^\circ}[d] & V^\circ \ar^{(p, r)}[d] \\
\bA^n_B \ar^{(1, 0)}[r] & \bA^{n+m}_B
}
\]
the vertical arrows are \'etale.

Consider
\[g \coloneqq g' - \sum_i\frac{\partial g'}{\partial r_i}r_i + \frac12\sum_{i,j}\frac{\partial^2 g'}{\partial r_i\partial r_j}r_i r_j + \frac12\sum_i r_i^2.\]
It satisfies
\[\left.\frac{\partial g}{\partial r_i}\right|_{U^\circ} = 0,\qquad \left.\frac{\partial^2 g}{\partial r_i\partial r_j}\right|_{U^\circ} = \begin{cases}
	1 & \text{if $i=j$}, \\ 0 & \text{if $i\neq j$}.
\end{cases}\]
The first equation implies that $\Phi$ restricts to $U^\circ\rightarrow \Crit_{V/\bA^n_B}(g)$. The second equation implies that we may shrink $V$ so that $\Phi\colon U^\circ\rightarrow \Crit_{V/\bA^n_B}(g)$ is an isomorphism, which we assume. Therefore, taking the full critical locus relative to $B$ we get that $\Phi$ restricts to an isomorphism $\Crit_{U^\circ/B}(f^\circ)\rightarrow \Crit_{V/B}(g)$.
\end{proof}

We next show that any critical morphism is locally of the form as in \cref{ex:criticalstabilization}; the following is a family version of \cite[Proposition 2.23]{JoyceDcrit}.

\begin{proposition}\label{prop:criticalembeddinglocal}
Let $\Phi\colon (U, f)\rightarrow (V, g)$ be a critical morphism of LG pairs over $B$ and $x\in \Crit_{U/B}(f)$. Then there is an open immersion $\imath\colon (U^\circ, f^\circ)\rightarrow (U, f)$, an \'etale morphism $\jmath\colon (V^\circ, g^\circ)\rightarrow (V, g)$, a point $x^\circ\in\Crit_{U^\circ/B}(f^\circ)$ mapping to $x$, a critical morphism $\Phi^\circ\colon (U^\circ, f^\circ)\rightarrow (V^\circ, g^\circ)$, a trivial orthogonal bundle $(E, q)\rightarrow U$ and an \'etale morphism $\alpha\colon V^\circ\rightarrow E$ such that $\alpha^\ast(f\circ \pi_E+\q_E) = g^\circ$ and the diagram
\[
\begin{tikzcd}
U \arrow[r, "0_E"] & E \\
U^\circ \arrow[r, "\Phi^\circ"] \arrow[d, "\imath" left] \arrow[u, "\imath"] & V^\circ \arrow[d, "\jmath" left] \arrow[u, "\alpha"] \\
U \arrow[r, "\Phi"] & V
\end{tikzcd}
\]
is commutative.
\end{proposition}
\begin{proof}
As a first step, we construct a local splitting of $\Phi\colon U\rightarrow V$. For this, consider an open neighborhood $U^\circ\subset U$ of $x$ affine over $B$ together with an \'etale morphism $c\colon U^\circ\rightarrow \bA^n_B$. Consider the unramified morphism $U^\circ\hookrightarrow U\xrightarrow{\Phi} V$. By \cite[Corollaire 18.4.7]{EGA44} we may find a commutative diagram
\[
\xymatrix{
U^\circ \ar^{\Phi'}[r] \ar[d] & V' \ar[d] \\
U \ar^{\Phi}[r] & V,
}
\]
where $V'\rightarrow V$ is an \'etale morphism and $\Phi'$ a closed immersion. Possibly shrinking $V'$ (and, correspondingly, $U^\circ$) we may assume that $V'\rightarrow V\rightarrow B$ is affine. Then $c$ lifts to $p'\colon V'\rightarrow \bA^n_B$ such that $p'\circ \Phi' = c$. Let $f^\circ\colon U^\circ\rightarrow \bA^1$ be the restriction of $f$ to $U^\circ$ and $g'\colon V'\rightarrow \bA^1$ the restriction of $g$ to $V'$.

Since $\Phi'\colon U^\circ\rightarrow V'$ is a regular immersion, by shrinking $V'$ we can find a function $r\colon V'\rightarrow \bA^m$ such that $U^\circ$ is the zero locus of $r$ and in the fiber square
\[\xymatrix{
U^\circ \ar@{^{(}->}[r]^{\Phi'} \ar[d]^c \cart & V' \ar[d]^{(p',r)} \\
\bA_B^n \ar@{^{(}->}[r]^-{(1,0)} & \bA_B^{n+m},
}\]
the vertical arrows are \'etale.

By assumption we have a commutative diagram
\[
\xymatrix{
\Crit_{U^\circ/B}(f^\circ) \ar[r] \ar[d] & \Crit_{V'/B}(g') \ar[d] \\
U^\circ \ar^{\Phi'}[r] & V'
}
\]
Thus, $\frac{\partial g'}{\partial r_j}|_{U^\circ}$ lies in the ideal generated by $\frac{\partial f^\circ}{\partial c_i}$. Since $p'\circ \Phi^\circ = c$ and $g'\circ \Phi' = f^\circ$, we have $\left.\frac{\partial g^\circ}{\partial p'_i}\right|_{U^\circ} = \frac{\partial f^\circ}{\partial c_i}$. So, we may write in the $\{p'_i, r'_i\}$ coordinates
\[\left.\frac{\partial g'}{\partial r_j}\right|_{U^\circ} = \sum_i a_{ij}' \left.\frac{\partial g'}{\partial {p'_i}}\right|_{U^\circ}\]
for some functions $a_{ij}'\colon U^\circ \to \bA^1$. Choose any lifts $a_{ij}\colon V' \to \bA^1$ of $a'_{ij}$ and consider
\[ p_i \coloneqq p_i' + \sum_j {a_{ij}} r_j\colon V' \longrightarrow \bA^1.\]
Possibly shrinking $V'$ we may assume that $V'\xrightarrow{(p, r)} \bA^{n+m}_B$ is \'etale. Moreover, in the $\{p_i, r_i\}$ coordinates we have
\[\left.\frac{\partial g'}{\partial r_j}\right|_{U^\circ} = 0.\]
Let $V^\circ = U^\circ\times_{\bA^n_B} V'$ with the projection $\tilde{p}\colon V^\circ\rightarrow U$. Then we obtain a commutative diagram
\[
\xymatrix{
U^\circ \ar^{\Phi^\circ}[r] \ar[d] & V^\circ \ar[d] \\
U \ar^{\Phi}[r] & V
}
\]
with vertical maps \'etale and with $\Phi^\circ\colon U^\circ\rightarrow V^\circ$ an immersion as it is the composite $U^\circ\rightarrow U^\circ\times_{\bA^n_B} U^\circ\rightarrow U^\circ\times_{\bA^n_B} V'$ of an open immersion and a closed immersion. Since $(p, r)\colon V'\rightarrow \bA^{n+m}_B$ is \'etale, so is its base change $(\tilde{p}, r)\colon V^\circ\rightarrow U\times \bA^m$. Let $g^\circ$ be the restriction of $g$ to $V^\circ$.

Since $g^\circ\circ \Phi^\circ = f^\circ$ and $\frac{\partial g^\circ}{\partial r_j}|_{U^\circ} = 0$, we see that $g^\circ - f^\circ\circ\tilde{p}$ vanishes to the second order along $U^\circ\hookrightarrow V^\circ$. Thus, shrinking $V^\circ$ we may find $q_{ij}\in\cO_{V^\circ}$ such that
\[g^\circ = f^\circ\circ\tilde{p} + \sum_{i, j} q_{ij} r_ir_j,\]
where $q_{ij}$ is a symmetric invertible matrix. Using the Gram--Schmidt process, we may assume that $q_{ij} = 0$ for $i\neq j$ and $q_{ii}$ are nowhere vanishing functions on $V^\circ$. Possibly replacing $V^\circ$ by an \'etale cover so that $q_{ii}$ admit square roots, we obtain an \'etale morphism $\alpha = (\tilde{p}, \sqrt{q} r)\colon V^\circ\rightarrow E=U\times \bA^m$, so that if we equip $E$ with the trivial quadratic form $\q_{\bA^m}$ we get $\alpha^\ast (f + \q_{\bA^m}) = g'$.
\end{proof}

\begin{example}
Consider an LG pair $(V, g)$ over $B$ and a $B$-point $\sigma\colon B\rightarrow \Crit_{V/B}(g)$ of the relative critical locus. Let $x\in B$. Assume that $g$ is \emph{relatively Morse} at $\sigma$, i.e. $\sigma$ is an inclusion of a smooth connected component. Then $\sigma\colon (B, g|_\sigma)\rightarrow (V, g)$ is a critical morphism. In this case \cref{prop:criticalembeddinglocal} reduces to the relative Morse lemma: we may find an open neighborhood $B^\circ\subset B$ of $x$, an \'etale morphism $V^\circ\rightarrow V$, a lift $\sigma^\circ\colon B^\circ\rightarrow \Crit_{V^\circ/B}(g|_{V^\circ})$ of $\sigma$ and \'etale coordinates $y\colon V^\circ\rightarrow \bA^m_B$ such that
\[g|_{V^\circ} = \sum_{i=1}^m y_i^2.\]
\end{example}

The following proposition shows that any two critical charts are locally related by a zigzag of critical morphisms; this is a family version of \cite[Theorem 2.20]{JoyceDcrit}.

\begin{proposition}\label{prop:stabilizationzigzag}
Let $X\rightarrow B$ be a morphism of schemes equipped with a relative d-critical structure $s$. Let $(U, f, u)$ and $(V, g, v)$ be two critical charts and $x\in\Crit_{U/B}(f), y\in\Crit_{V/B}(g)$ points such that $u(x)=v(y)$. Then there are open immersions $(U^\circ, f^\circ, u^\circ)\rightarrow (U, f, u)$ and $(V^\circ, g^\circ, v^\circ)\rightarrow (V, g, v)$, a critical chart $(W, h, w)$, critical morphisms
\[(U^\circ, f^\circ, u^\circ)\xrightarrow{\Phi} (W, h, w)\xleftarrow{\Psi} (V^\circ, g^\circ, v^\circ)\]
and points $x^\circ\in\Crit_{U^\circ/B}(f^\circ),y^\circ\in\Crit_{V^\circ/B}(g^\circ)$ which map to $x$ and $y$.
\end{proposition}
\begin{proof}
Using \cref{prop:chartopenaffine} we may find an open neighborhood $V^\circ\subset V$ of $y$, a critical chart $(\tilde{V}, \tilde{g}, \tilde{v})$ which admits an open immersion into $\bA^n_B$ and a critical morphism $\Xi\colon (V^\circ, g^\circ, v^\circ)\rightarrow (\tilde{V}, \tilde{g}, \tilde{v})$. Moreover, by construction $\Xi$ restricts to an isomorphism $\Crit_{V^\circ/B}(g^\circ)\cong \Crit_{\tilde{V}/B}(\tilde{g})$.

Consider the diagram
\[
\xymatrix{
& \Crit_{U/B}(f)\times_X \Crit_{V^\circ/B}(g^\circ) \ar[dl] \ar[dr] & \\
U && \tilde{V} \ar[r] & \bA^n_B.
}
\]
Since $\Crit_{U/B}(f)\times_X \Crit_{V^\circ/B}(g^\circ)\rightarrow U$ is an immersion, we may find a Cartesian diagram
\begin{equation}\label{eq:stabilizationzigzagdiagram}
\xymatrix{
R \ar[r] \ar[d] & \Crit_{U/B}(f)\times_X \Crit_{V^\circ/B}(g^\circ) \ar[d] \\
U^\circ \ar[r] & U
}
\end{equation}
with $U^\circ\rightarrow U$ open with $x\in R$ and a commutative diagram
\[
\xymatrix{
R \ar[r] \ar[d] & \Crit_{U/B}(f)\times_X \Crit_{V^\circ/B}(g^\circ) \ar[d] & \\
U^\circ \ar^{\Theta}[r] & \bA^n_B
}
\]
Shrinking $U^\circ$ further, we may assume that $\Theta\colon U^\circ\rightarrow \bA^n_B$ factors through $\tilde{V}\subset \bA^n_B$. The Cartesian diagram \eqref{eq:stabilizationzigzagdiagram}, the fact that $\Crit_{V^\circ/B}(g^\circ)\rightarrow X$ is a monomorphism and the Cartesian diagram from \cref{prop:smoothcriticallocus} imply that $R\cong \Crit_{U^\circ/B}(f^\circ)$, where $f^\circ=f|_{U^\circ}$. Let $u^\circ = u|_{\Crit_{U^\circ/B}(f^\circ)}$. Thus, we get a commutative diagram
\[
\xymatrix{
R\cong\Crit_{U^\circ/B}(f^\circ) \ar[r] \ar[d] & \Crit_{V^\circ/B}(g^\circ) \ar[d] \\
U^\circ \ar^{\Theta}[r] & \tilde{V}
}
\]
with $\Crit_{U^\circ/B}(f^\circ)\rightarrow \Crit_{V^\circ/B}(g^\circ)$ an open immersion.

Using the functoriality of the description of $\cS_{X/B}$ from \cref{prop:Sproperties}(3) with respect to $\Theta$ we obtain a commutative diagram
\[
\xymatrix{
0 \ar[r] & \cS_{X/B}|_R \ar[r] \ar@{=}[d] & \cO_{\tilde{V}} / \cI_{R,\tilde{V}}^2 \ar[r] \ar[d] & \Omega^1_{\tilde{V}/B} / \cI_{R,\tilde{V}}\Omega^1_{\tilde{V}/B} \ar[d] \\
0 \ar[r] & \cS_{X/B}|_R \ar[r] & \cO_{U^\circ} / \cI_{R,U^\circ}^2 \ar[r] & \Omega^1_{U^\circ/B} / \cI_{R,U^\circ}\Omega^1_{U^\circ/B}
}
\]
Using that $s_{f^\circ} = s_{g^\circ}$ on $R$ we thus obtain that $f^\circ - g\circ \Theta\in \cI_{R/U^\circ}^2$. Shrinking $U^\circ$, we may find a (trivial) orthogonal bundle $(E', q')\rightarrow \tilde{V}$ together with a section $s'$ of $\Theta^* E'$, vanishing on $R$, such that
\[f^\circ - \tilde{g}\circ \Theta = q'(s').\]
Now let $W' = E'$ equipped with the function $\tilde{g}\circ \pi_{E'} + \q_{E'}\colon W'\rightarrow \bA^1$.

Shrinking $U^\circ$ we may choose an \'etale morphism $y\colon U^\circ\rightarrow \bA^m_B$, i.e. \'etale coordinates on $U^\circ\rightarrow B$ near $x$. Consider the hyperbolic quadratic form $q_m$ on $\bA^{2m}$ given by $\sum_{i=1}^m y_iz_i$ and let $E=E'\times \bA^{2m}$ with the sum quadratic form $q=q' \boxplus q_m$. Consider the section $s = (s', y_1, \dots, y_m, 0, \dots, 0)$ of $\Theta^* E$. By construction $\Omega^1_{E/B}\rightarrow \Omega^1_{U^\circ/B}$ is surjective, so $\Phi\colon (U^\circ, f^\circ)\rightarrow (W, \tilde{g}\circ \pi_E + \q_E)$ given by the section $s$ is unramified. Since $s'$ vanishes on $R$, $\Phi$ is a critical morphism.

Define $\tilde{\Psi}\colon (\tilde{V}, \tilde{g})\rightarrow (W, \tilde{g}\circ \pi_E + \q_E)$ to be the zero section which is a critical morphism by \cref{ex:criticalstabilization}. We set
\[\Psi\colon (V^\circ, g^\circ)\xrightarrow{\Xi} (\tilde{V}, \tilde{g})\xrightarrow{\tilde{\Psi}} (W, \tilde{g}\circ \pi_E + \q_E),\]
which is a composite of critical morphisms.
\end{proof}

Finally, any critical chart can be replaced by a minimal chart at a point.

\begin{proposition}\label{prop:minimalchart}
Let $(V, g)$ be an LG pair over $B$ and $y\in \Crit_{V/B}(g)$ a point. Then there is an open subscheme $V^\circ\subset V$, a smooth $B$-scheme $U$, a point $x\in\Crit_{U/B}(f)$, a function $f\colon U\rightarrow \bA^1$ and a critical morphism $\Phi\colon (U, f)\rightarrow (V^\circ, g|_{V^\circ})$, such that $\Phi(x) = y$, $\Crit_{U/B}(f)\rightarrow U$ is minimal at $x$ and $\Phi\colon\Crit_{U/B}(f)\rightarrow \Crit_{V^\circ/B}(g|_{V^\circ})$ is an isomorphism.
\end{proposition}
\begin{proof}
Consider the short exact sequence
\[0\longrightarrow \T_{\Crit_{V/B}(g)/B, y}\rightarrow \T_{V/B, y}\longrightarrow \rN_y\longrightarrow 0\]
and consider an arbitrary splitting $\T_{V/B, y}\cong \T_{\Crit_{V/B}(g)/B, y}\oplus \rN_y$. In a neighborhood $V^\circ\subset V$ of $y$ we may extend it to an isomorphism
\[\T_{V^\circ/B}\cong \T_{\Crit_{V/B}(g)/B, y}\otimes \cO_{V^\circ} \oplus \rN_y\otimes \cO_{V^\circ}.\]
With respect to this decomposition let $dg|_{V^\circ} = \alpha + \beta$ for
\[\alpha\in \Omega^1_{\Crit_{V/B}(g)/B, y}\otimes \cO_{V^\circ},\qquad \beta\in \rN^\vee_y\otimes \cO_{V^\circ}.\]
Consider $U = \beta^{-1}(0)$ with $f=g|_U$ and the morphism $\nabla \beta\colon \T_{V^\circ/B}\rightarrow \rN^\vee_y\otimes \cO_{V^\circ}$.

We have $\Hess(f)_y=(\nabla df)_y\colon \T_{V^\circ/B, y}\rightarrow \Omega^1_{V^\circ/B, y}$ which is nondegenerate on $\rN_y$ by \cref{prop:criticalHessian}(1). Thus, possibly shrinking $V^\circ$ to a smaller neighborhood of $y$ we may assume that $\nabla \beta$ is surjective, i.e. $U\rightarrow B$ is smooth. Moreover, $\Crit_{U/B}(f)=\Crit_{V^\circ/B}(g|_{V^\circ})$ as both are defined by the equations $\alpha=\beta=0$ in $V^\circ$.
\end{proof}

\subsection{Gluing objects on critical charts}

We will use the results comparing different critical charts from \cref{sect:chartcomparison} to glue objects over schemes equipped with relative d-critical structures. We begin with the following general paradigm for gluing. Consider a site $\cC$ and a functor $\pi\colon \cF\rightarrow \cC$. We consider a collection of morphisms of $\cF$ called \defterm{localization morphisms}, which satisfies the following properties:
\begin{enumerate}
    \item The collection of localization morphisms is closed under composition and it contains identities.
    \item For every localization morphism $x\rightarrow z$ and any morphism $y\rightarrow z$ the fiber product $x\times_z y$ exists, it is preserved by $\pi$ and $x\times_z y\rightarrow y$ is a localization morphism.
\end{enumerate}

We say $\cF$ is \defterm{locally connected} if the following conditions are satisfied:
\begin{enumerate}
    \item For every $X\in\cC$ there is a collection of objects $\{x_a\}$ of $\cF$ together with a covering $\{\pi(x_a)\rightarrow X\}$.
    \item For every $x,y\in\cF$ together with morphisms $\pi(x)\rightarrow z\leftarrow \pi(y)$ in $\cC$, there are collections of localization morphisms $\{x_a\rightarrow x\}$ and $\{y_a\rightarrow y\}$ such that for each $a$ there exists a diagram $x_a\rightarrow z_a\leftarrow y_a$ in $\cF$, and $\{\pi(x_a)\times_z \pi(y_a)\rightarrow \pi(x)\times_z \pi(y)\}$ is a cover.
\end{enumerate}

\begin{example}\label{ex:Cartesianlocallyconnected}
Suppose $\pi\colon \cF\rightarrow \cC$ is a Cartesian fibration, so that it is classified by a presheaf $\tilde{\cF}\colon \cC^{\op}\rightarrow \Cat$. Consider the class of $\pi$-Cartesian morphisms as the class of localization morphisms on $\cF$. Then $\cF$ is locally connected precisely if the sheafification of $\tilde{\cF}$ is connected.
\end{example}

\begin{proposition}\label{prop:connecteddescent}
Let $\cC$ be a site, $\pi_i\colon \cF_i\rightarrow \cC$ for $i=1, 2$ two functors, where $\cF_1\rightarrow \cC$ is locally connected and $\cF_2\rightarrow \cC$ is a Cartesian fibration satisfying descent. Consider a functor $F\colon \cF_1\rightarrow \cF_2$ over $\cC$ satisfying the following conditions:
\begin{enumerate}
\item For every $f\colon x\rightarrow y$ in $\cF_1$ the morphism $F(f)\colon F(x)\rightarrow F(y)$ in $\cF_2$ is $\pi_2$-Cartesian.
\item Given two morphisms $f_1,f_2\colon x\rightarrow y$ such that $\pi_1(f_1)=\pi_1(f_2)$ we have $F(f_1)=F(f_2)$.    
\end{enumerate}
Then there is a Cartesian section $s$ of $\pi_2\colon \cF_2\rightarrow \cC$ determined by the following conditions:
\begin{itemize}
    \item For an object $x\in\cF_1$ we have an isomorphism $i_x\colon s(\pi_1(x))\xrightarrow{\sim} F(x)$.
    \item For a morphism $f\colon x\rightarrow y$ in $\cF_1$ there is a commutative diagram
    \[
    \begin{tikzcd}
    s(\pi_1(x)) \arrow[r, "s(\pi_1(f))"] \arrow[d, "i_x"] & s(\pi_1(y)) \arrow[d, "i_y"] \\
    F(x) \arrow[r, "F(f)"] & F(y)
    \end{tikzcd}
    \]
\end{itemize}
\end{proposition}
\begin{proof}
The proof is similar to the proofs of \cite[Theorem 2.28]{JoyceDcrit} and \cite[Theorem 6.9]{BBDJS}. For an object $X\in\cC$ consider a collection $\{x_a\mid a\in A\}$ of objects of $\cF_1$ together with a cover
\begin{equation}\label{eq:criticaldescentcover1}
\{\pi_1(x_a)\rightarrow X\mid a\in A\}.
\end{equation}
Since $\cF_1$ is locally connected, for every $a,b\in A$ there is a set $D_{ab}$ and for each $d\in D_{ab}$ there are morphisms $\{x'^d_a\rightarrow x_a\}$ and $\{x'^d_b\rightarrow x_b\}$ and objects $\{y^d\}$ of $\cF_1$ together with morphisms $x'^d_a\xrightarrow{\Phi^d} y^d\xleftarrow{\Psi^d} x'^d_b$ such that for every $a,b\in A$
\begin{equation}\label{eq:criticaldescentcover2}
\{\pi_1(x'^d_a)\times_X \pi_1(x'^d_b)\rightarrow \pi_1(x_a)\times_X \pi_1(x_b)\mid d\in D_{ab}\}
\end{equation}
is a cover.

Since $\cF_2$ satisfies descent, using the cover \eqref{eq:criticaldescentcover1} we have to specify an isomorphism
\[\alpha_{ab}\colon F(x_a)|_{\pi_1(x_a)\times_X \pi_1(x_b)}\xrightarrow{\sim} F(x_b)|_{\pi_1(x_a)\times_X \pi_1(x_b)}\]
for every $a,b\in A$ which satisfies the cocycle conditions
\begin{equation}\label{eq:criticaldescentcocycle}
\alpha_{bc}\circ \alpha_{ab} = \alpha_{ac},\qquad \alpha_{aa} = \id.
\end{equation}

We proceed to the definition of $\alpha_{ab}$. By descent, using \eqref{eq:criticaldescentcover2} to construct $\alpha_{ab}$ it is enough to construct, for every $d\in D_{ab}$,
\[\alpha^d_{ab}\colon F(x_a)|_{\pi_1(x'^d_a)\times_X \pi_1(x'^d_b)}\xrightarrow{\sim}F(x_b)|_{\pi_1(x'^d_a)\times_X \pi_1(x'^d_b)}\]
such that for every $d, e\in D_{ab}$ we have
\begin{equation}\label{eq:criticaldescentisomorphism}
\alpha^d_{ab}|_{\pi_1(x'^d_a)\times_X \pi_1(x'^d_b)\times_X \pi_1(x'^e_a)\times_X \pi_1(x'^e_b)} = \alpha^e_{ab}|_{\pi_1(x'^d_a)\times_X \pi_1(x'^d_b)\times_X \pi_1(x'^e_a)\times_X \pi_1(x'^e_b)}.
\end{equation}

We have isomorphisms $\Phi^d\colon x'^d_a\xrightarrow{\sim} y^d|_{\pi_1(x'^d_a)}$ and $\Psi^d\colon x'^d_b\xrightarrow{\sim} y^d|_{\pi_1(x'^d_b)}$ and 
we define
\[\alpha_{ab}^d = F(\Psi^d)^{-1}\circ F(\Phi^d).\]

To show \eqref{eq:criticaldescentisomorphism}, using the condition of local connectedness, we get a collection of localization morphisms $\{y'^d_c\rightarrow y^d\}$ and $\{y'^e_c\rightarrow y^e\}$ together with zigzags $y'^d_c\xrightarrow{\Phi_c} z_c\xleftarrow{\Psi_c} y'^e_c$ such that
\[\{\pi_1(y'^d_c)\times_X \pi_1(y'^e_c)\rightarrow \pi_1(y^d)\times_X \pi_1(y^e)\}\]
is a cover.

Since $x'^d_a\rightarrow x_a$ and $y'^d_c\rightarrow y^d$ are localization morphisms, we can form the fiber product
\[x'^{de}_{ac} = y'^d_c\times_{y^d} x'^d_a\times_{x_a} x'^e_a\times_{y^e} y'^e_c.\]
Then we have a (non-commutative) diagram
\[
\begin{tikzcd}
x'^{de}_{ac} \arrow[rr, "\Phi^d"] \arrow[dd, "\Phi^e"] && y'^d_c \arrow[dl, "\Phi_c"] \\
& z_c & \\
y'^e_c \arrow[ur, "\Psi_c"] && x'^{de}_{bc} \arrow[ll, "\Psi^e"] \arrow[uu, "\Psi^d"]
\end{tikzcd}
\]
in $\cF_1$. Using the functoriality of $F$ as well as the second assumption in the statement, we get, on the cover
\[\{\pi_1(y'^{de}_{ac})\times_X \pi_1(y'^{de}_{bc})\rightarrow \pi_1(x'^d_a)\times_X \pi_1(x'^d_b)\times_X \pi_1(x'^e_a)\times_X \pi_1(x'^d_b)\},\]
that
\begin{align*}
\alpha^d_{ab} &= F(\Psi^d)^{-1}\circ F(\Phi^d) \\
&= F(\Psi^d)\circ F(\Phi_c)^{-1}\circ F(\Phi_c)\circ F(\Phi^d) \\
&= F(\Psi^e)^{-1}\circ F(\Psi_c)^{-1}\circ F(\Psi_c)\circ F(\Phi^e) \\
&= F(\Psi^e)^{-1}\circ F(\Phi^e) = \alpha^e_{ab}.
\end{align*}

This finishes the construction of $\alpha_{ab}$ and shows that it is independent of the choices of the intermediate zigzags. Let us next check the cocycle conditions \eqref{eq:criticaldescentcocycle}. The condition $\alpha_{aa} = \id$ follows from $F(\id) = \id$ since we can choose the zigzag for $x_a$ and itself to be given by the identity maps. Let $a,b,c\in A$. Using local connectedness we can choose localizations $\{x'^f_i\rightarrow x_i\}$ for $i=a,b,c$ such that
\[\{\pi_1(x'^f_a)\times_X \pi_1(x'^f_b)\times_X \pi_1(x'^f_c)\rightarrow \pi_1(x_a)\times_X \pi_1(x_b)\times_X \pi_1(x_c)\}\]
and an object $y$ together with morphisms $\Phi_i\colon x'^f_i\rightarrow y$ for $i=a,b,c$. Then we have
\begin{align*}
\alpha_{ab}|_{\pi_1(x'^f_a)\times_X \pi_1(x'^f_b)\times_X \pi_1(x'^f_c)} &= \alpha_{\Phi_b}^{-1}\circ \alpha_{\Phi_a} \\
\alpha_{bc}|_{\pi_1(x'^f_a)\times_X \pi_1(x'^f_b)\times_X \pi_1(x'^f_c)} &= \alpha_{\Phi_c}^{-1}\circ \alpha_{\Phi_b} \\
\alpha_{ac}|_{\pi_1(x'^f_a)\times_X \pi_1(x'^f_b)\times_X \pi_1(x'^f_c)} &= \alpha_{\Phi_c}^{-1}\circ \alpha_{\Phi_a}.
\end{align*}

This implies that $\alpha_{bc}\circ \alpha_{ab} = \alpha_{ac}$ on $\pi_1(x'^f_a)\times_X \pi_1(x'^f_b)\times_X \pi_1(x'^f_c)$. As we can cover $\pi_1(x_a)\times_X \pi_1(x_b)\times_X \pi_1(x_c)$ by such morphisms, this proves the claim.
\end{proof}

\begin{remark}
Consider the setting of \cref{prop:connecteddescent} when $\pi_1\colon \cF_1\rightarrow \cC$ is a Cartesian fibration (see \cref{ex:Cartesianlocallyconnected}). Then the sheafification of $\cF_1$ has trivial $\pi_0$. So, a morphism to $\cF_2$ defines a global object precisely if all the obstructions coming from the nontrivial $\pi_1$ vanish; this is guaranteed by the second condition in \cref{prop:connecteddescent}.
\end{remark}

\begin{corollary}\label{cor:criticaldescent}
Let $X\rightarrow B$ be a morphism of schemes equipped with a relative d-critical structure $s$ and $\cF$ a sheaf of categories over $X$ in the Zariski topology. Fix a locally constant function $d\colon X\rightarrow \Z/2\Z$. Consider the following data:
\begin{enumerate}
    \item For every critical chart $(U, f, u)$, such that $\dim(U/B)\equiv d\pmod{2}$, we have an object \[x_u\in \Gamma(\Crit_{U/B}(f), \cF).\]
    \item For every critical morphism $\Phi\colon (U, f, u)\rightarrow (V, g, v)$ of even relative dimension we have an isomorphism \[J_\Phi\colon x_u\xrightarrow{\sim} (x_v)|_{\Crit_{U/B}(f)}.\]
\end{enumerate}
which satisfy the following conditions:
\begin{enumerate}
    \item For a composite $(U, f, u)\xrightarrow{\Phi} (V, g, v)\xrightarrow{\Psi} (W, h, w)$ of critical morphisms of even relative dimension we have
    \[J_{\Psi\circ\Phi} = (J_{\Psi})|_{\Crit_{U/B}(f)}\circ J_{\Phi}.\]
    \item For the identity critical morphism $(U, f, u)\xrightarrow{\id} (U, f, u)$ we have
    \[J_{\id} = \id.\]
    \item Given two critical morphisms $\Phi_1, \Phi_2\colon (U, f, u)\rightarrow (V, g, v)$ of even relative dimension such that $\Phi_1=\Phi_2\colon \Crit_{U/B}(f)\rightarrow \Crit_{V/B}(g)$, we have an equality
    \[J_{\Phi_1}=J_{\Phi_2}.\]
\end{enumerate}
Then there is an object $x\in\Gamma(X, \cF)$ restricting to $x_u$ on each critical chart and with $J_\Phi$ as the isomorphisms of these local objects for critical morphisms.
\end{corollary}
\begin{proof}
Consider the following category $\CritCharts_d(X/B)$:
\begin{itemize}
    \item Its objects are critical charts $(U, f, u)$ for $(X \rightarrow B, s)$ with $\dim(U/B)\equiv d\pmod{2}$.
    \item Its morphisms are critical morphisms of critical charts of even relative dimension.
\end{itemize}
We have a natural functor $\pi_1\colon \CritCharts_d(X/B)\rightarrow X_{\Zar}$ to the Zariski site of $X$ given by sending $(U, f, u)\mapsto \Crit_{U/B}(f)$. As the class of localization morphisms we take open immersions of critical charts. Given a critical chart $(U, f, u)$ we may construct a new critical chart $(U\times \bA^1, f\boxplus x^2, u)$ of dimension one higher. Thus, the fact that $\pi_1\colon\CritCharts_d(X/B)\rightarrow X_{\Zar}$ is locally connected follows from \cref{prop:stabilizationzigzag}. Let $\pi_2\colon \int \cF\rightarrow X_{\Zar}$ be the Grothendieck construction applied to the sheaf of categories $\cF$. Then the data given in the statement determines a functor $F \colon \CritCharts_d(X/B) \to \int \cF$ over $X_\Zar$ satisfying the conditions of \cref{prop:connecteddescent}, whence the claim.
\end{proof}

Restricting \cref{cor:criticaldescent} to sheaves of sets we obtain the following statement.

\begin{corollary}\label{cor:criticaldescentsets}
Let $X\rightarrow B$ be a morphism of schemes equipped with a relative d-critical structure $s$ and $\cF$ a sheaf of sets over $X$ in the Zariski topology. Fix a locally constant function $d\colon X\rightarrow \Z/2\Z$. Consider the following data:
\begin{enumerate}
    \item For every critical chart $(U, f, u)$, such that $\dim(U/B)\equiv d\pmod{2}$, we have an element $x_u\in \Gamma(\Crit_{U/B}(f), \cF)$.
\end{enumerate}
which satisfy the following condition:
\begin{enumerate}
\item For every critical morphism $\Phi\colon (U, f, u)\rightarrow (V, g, v)$ of even relative dimension we have an equality $x_u= (x_v)|_{\Crit_{U/B}(f)}$.
\end{enumerate}
Then there is a section $x\in\Gamma(X, \cF)$ restricting to $x_u$ on each critical chart.
\end{corollary}

\subsection{Virtual canonical bundle}

We are now going to define the virtual canonical bundle of a relative d-critical structure on a morphism of schemes $X\rightarrow B$, which is a line bundle on the reduced scheme $X^{\red}$. For this, we begin by defining the canonical quadratic form associated to a critical morphism. Observe that in the local model of a critical morphism $\Phi\colon (U, f)\rightarrow (V, g)$ as in \cref{prop:criticalembeddinglocal} there is a canonical isomorphism
\begin{equation}\label{eq:criticalembeddingnormallocal}
\rN_{U/V}|_{U^\circ}\cong \rN_{U/E}|_{U^\circ}\cong E|_{U^\circ}.
\end{equation}
In particular, the normal bundle of $U\rightarrow V$ admits a quadratic form \'etale locally. We will now show that it extends to a global quadratic form.

\begin{proposition}\label{prop:criticalembeddingnormalquadratic}
Let $\Phi\colon (U, f)\rightarrow (V, g)$ be a critical morphism of LG pairs over $B$. Then there is a nondegenerate quadratic form $q_\Phi$ on the normal bundle $\rN_{U/V}|_{\Crit_{U/B}(f)}$ which is uniquely determined by the following property:
\begin{enumerate}
    \item Consider a local form of the critical morphism $\Phi$ as in \cref{prop:criticalembeddinglocal}. Then under the natural isomorphism $\rN_{U/V}|_{U^\circ}\cong E|_{U^\circ}$ \eqref{eq:criticalembeddingnormallocal} the quadratic form $q_\Phi$ identifies with the quadratic form $q$ on $E$.
\end{enumerate}
Moreover, it satisfies the following properties:
\begin{enumerate}[resume]
    \item For another critical morphism $\Psi\colon(V, g)\rightarrow (W, h)$ consider the unique splitting of the short exact sequence
    \[0\longrightarrow \rN_{U/V}|_{\Crit_{U/B}(f)}\longrightarrow \rN_{U/W}|_{\Crit_{U/B}(f)}\longrightarrow \rN_{V/W}|_{\Crit_{U/B}(f)}\longrightarrow 0\]
    obtained by taking the orthogonal complement of $\rN_{U/V}|_{\Crit_{U/B}(f)}\subset \rN_{U/W}|_{\Crit_{U/B}(f)}$ with respect to $q_{\Psi\circ\Phi}$. Then the resulting isomorphism \[\rN_{U/W}|_{\Crit_{U/B}(f)}\cong \rN_{U/V}|_{\Crit_{U/B}(f)}\oplus \rN_{V/W}|_{\Crit_{U/B}(f)}\] sends the quadratic form $q_{\Psi\circ \Phi}$ to $q_\Phi+q_\Psi$.
    \item Let $\Phi'\colon (U, f)\rightarrow (V, g)$ be another critical morphism between the same critical charts. Then
    \[\vol^2_{q_\Phi}|_{\Crit_{U/B}(f)^{\red}} = \vol^2_{q_{\Phi'}}|_{\Crit_{U/B}(f)^{\red}}.\]
    \item Consider a commutative diagram
    \[
    \xymatrix{
    (U_1, f_1) \ar^{\Phi_1}[r] \ar^{\pi_U}[d] & (V_1, g_1) \ar^{\pi_V}[d] \\
    (U_2, f_2) \ar^{\Phi_2}[r] & (V_2, g_2)
    }
    \]
    with $\pi_U$ and $\pi_V$ smooth morphisms and $\Phi_1$ and $\Phi_2$ critical morphisms such that $U_1\rightarrow V_1\times_{V_2} U_2$ is \'etale. Then under the isomorphism $\pi_U^* \rN_{U_2/V_2}\cong \rN_{U_1/V_1}$ we have $q_{\Phi_2}\mapsto q_{\Phi_1}$.
    \item For a pair of critical morphisms $\Phi_i\colon (U_i, f_i)\rightarrow (V_i, g_i)$ of LG pairs over $B_i$ under the isomorphism
    \[\rN_{U_1\times U_2/V_1\times V_2}|_{\Crit_{U_1\times U_2/B_1\times B_2}(f_1\boxplus f_2)}\cong \rN_{U_1/V_1}|_{\Crit_{U_1/B_1}(f_1)}\boxplus \rN_{U_2/V_2}|_{\Crit_{U_2/B_2}(f_2)}\]
    we have $q_{\Phi_1\times \Phi_2}\mapsto q_{\Phi_1}+q_{\Phi_2}$.
    \item If $\overline{\Phi}\colon (U, -f)\rightarrow (V, -g)$ is a morphism of LG pairs over $B$ equal to $\Phi$ on the level of underlying schemes, then $q_{\Phi}=-q_{\overline{\Phi}}$.
\end{enumerate}
\end{proposition}
\begin{proof}
Since we can cover the critical locus $\Crit_{U/B}(f)$ by the local models as in \cref{prop:criticalembeddinglocal}, uniqueness is clear. To show existence, we have to show that the quadratic form is independent of the choice of the local form of the critical morphism. For this, suppose $(U^\circ_1, V^\circ_1, E_1, q_1)$ and $(U^\circ_2, V^\circ_2, E_2, q_2)$ are two choices of the data as in \cref{prop:criticalembeddinglocal}. Let
\[\overline{U} = U^\circ_1\times_U U^\circ_2,\qquad \overline{V} = V^\circ_1\times_V V^\circ_2.\]
Let $R = \Crit_{\overline{U}/B}(f)$ and let $p\colon R\rightarrow B$ be the projection. Then we have a commutative diagram of short exact sequences
\[
\begin{tikzcd}
0 \rar & \T_{U/B}|_R\rar & \T_{E_1/B}|_R \rar & E_1|_R \rar & 0 \\
0 \rar & \T_{U/B}|_R \rar \arrow[u, "\id"] \arrow[d, "\id" left] & \T_{V/B}|_R \rar \arrow[u, "\sim"] \arrow[d, "\sim" left] & \rN_{U/V}|_R \rar \arrow[u, "\sim"] \arrow[d, "\sim" left] & 0 \\
0 \rar & \T_{U/B}|_R \rar & \T_{E_2/B}|_R \rar & E_2|_R \rar & 0
\end{tikzcd}
\]

Choose local splittings $\T_{E_1/B}|_R\cong \T_{U/B}|_R\oplus E_1|_R$ and $\T_{E_2/B}|_R\cong \T_{U/B}|_R\oplus E_2|_R$. Denote the composite isomorphisms by
\[t\colon E_1|_R\xrightarrow{\sim} E_2|_R,\qquad \begin{pmatrix} \id & s \\ 0 & t\end{pmatrix}\colon \T_{U/B}|_R\oplus E_1|_R\xrightarrow{\sim} \T_{U/B}|_R\oplus E_2|_R.\]
Under the vertical isomorphisms the quadratic form $\Hess(g)$ on $\T_{V/B}|_R$ restricts to $\Hess(f) + q_1$ and $\Hess(f)+q_2$, respectively. Therefore, for $v\in\T_{U/B}|_R$ and $w\in E_1|_R$ we have
\[\Hess(f)(v) + q_1(w) = \Hess(f)(v + s(w)) + q_2(t(w)).\]
By considering the associated symmetric bilinear form we get that $s(w)\in\Ker(\Hess(f))$. Thus, $q_1(w) = q_2(t(w))$, i.e. $t\colon (E_1|_R, q_1)\rightarrow (E_2|_R, q_2)$ preserves quadratic forms. So, $q_\Phi$ is well-defined.

Let us now check the remaining properties:
\begin{enumerate}
    \setcounter{enumi}{1}
    \item It is enough to establish this equality locally. For this, apply \cref{prop:criticalembeddinglocal} to get $(U^\circ, V^\circ_1, E_1, q_1)$ fitting into a commutative diagram
    \[
    \begin{tikzcd}
    U \rar & E_1 \\
    U^\circ \rar \uar \dar & V^\circ_1 \uar \dar\\
    U \arrow[r, "\Phi"] & V
    \end{tikzcd}
    \]
    Similarly, apply \cref{prop:criticalembeddinglocal} to get $(V^\circ_2, W^\circ, E_2, q_2)$ fitting into a commutative diagram
    \[
    \begin{tikzcd}
    V \rar & E_2 \\
    V^\circ_2 \rar \uar \dar & W^\circ \uar \dar\\
    V \arrow[r, "\Psi"] & W
    \end{tikzcd}
    \]
    Let
    \[\overline{U} = U^\circ\times_V V^\circ_2.\]
    Then we get a diagram
    \[
    \begin{tikzcd}
    U \rar & E_1\times_V E_2 \\
    \overline{U} \rar \uar \dar & W^\circ \uar \dar\\
    U \arrow[r, "\Psi\circ\Phi"] & W
    \end{tikzcd}
    \]
    Thus, under the local isomorphism of $\rN_{U/W}$ and $E_1\oplus E_2|_U$ we see that $q_{\Psi\circ\Phi}$ is sent to $q_1+q_2$.
    \item We have to show that the function $\vol^2_{q_\Phi}/\vol^2_{q_{\tilde{\Phi}}}$ on $\Crit_{U/B}(f)^{\red}$ is identically $1$. This can be checked pointwise on $\Crit_{U/B}(f)^{\red}$. For a $k$-point $b\in B$ let $R_b = \Crit_{U/B}(f)\times_B \Spec k=\Crit_{U\times_B \Spec k/\Spec k}(f)$ be the corresponding fiber product which carries a d-critical structure relative to $\Spec k$. Then it is sufficient to show that $\vol^2_{q_\Phi} = \vol^2_{q_{\tilde{\Phi}}}$ when restricted to $R^{\red}_b$ for every $b$. But this is shown in \cite[Theorem 2.27]{JoyceDcrit}.
    \item It is enough to establish this claim locally. Apply \cref{prop:criticalembeddinglocal} to get $(U_2^\circ, V_2^\circ, E, q)$ fitting into a commutative diagram
    \[
    \xymatrix{
    U_2 \ar[r] & E \\
    U_2^\circ \ar^{\Phi_2^\circ}[r] \ar[u] \ar[d] & V_2^\circ \ar[u] \ar[d] \\
    U_2 \ar^{\Phi_2}[r] & V_2
    }
    \]
    By property (1) under the isomorphism $\rN_{U_2^\circ/V_2^\circ}\cong E|_{U_2^\circ}$ we have $q_{\Phi_2}\mapsto q$.

    Then $\Phi_2^\circ\times \Phi_1\colon U_2^\circ\times_{U_2} U_1\rightarrow V_2^\circ\times_{V_2} V_1$ has a local model given by the pullback of $U_2\rightarrow E$. Thus, again by property (1) we get that under the isomorphism $\rN_{U_2^\circ\times_{U_2} U_1/V_2^\circ\times_{V_2} V_1}\cong E|_{U_2^\circ\times_{U_2} U_1}$ we have $q_{\Phi_1}\mapsto q$.
    \item The claim follows from property (2) as we may write $\Phi_1\times \Phi_2$ as the composite $(\id\times \Phi_2)\circ (\Phi_1\times \id)$ of critical embeddings.
    \item The claim is local and follows from the fact that given a local model of $\Phi\colon (U, f)\rightarrow (V, g)$ as in \cref{prop:criticalembeddinglocal} specified by a trivial orthogonal bundle $(E, q)\rightarrow U$ the local model of $\overline{\Phi}\colon (U, -f)\rightarrow (V, -g)$ is specified by the same data with the trivial orthogonal bundle $(E, -q)\rightarrow U$.
\end{enumerate}
\end{proof}

Given a critical morphism $\Phi\colon (U, f)\rightarrow (V, g)$ we denote by $P_\Phi\rightarrow \Crit_{U/B}(f)$ the $\Z/2\Z$-graded orientation $\mu_2$-torsor for the orthogonal bundle $(\rN_{U/V}|_{\Crit_{U/B}(f)}, q_\Phi)$. Then \cref{prop:criticalembeddingnormalquadratic} implies the following properties of $P_\Phi$:
\begin{enumerate}
    \item Given another critical morphism $\Psi\colon (V, g)\rightarrow (W, h)$ we obtain an isomorphism
    \[\Xi_{\Phi, \Psi}\colon P_{\Psi\circ \Phi}\xrightarrow{\sim} P_{\Psi}|_{\Crit_{U/B}(f)}\otimes_{\mu_2}P_{\Phi}\]
    by combining \eqref{eq:orientationsumisomorphism} and \cref{prop:criticalembeddingnormalquadratic}(2).
    \item If $\Phi'\colon (U, f)\rightarrow (V, g)$ is another critical morphism equal to $\Phi$ on the critical loci, there is a canonical isomorphism
    \[P_\Phi\cong P_{\Phi'}\]
    given by the identity on volume forms, using \cref{prop:criticalembeddingnormalquadratic}(3).
    \item Given a diagram
    \[
    \xymatrix{
    (U_1, f_1) \ar^{\Phi_1}[r] \ar^{\pi_U}[d] & (V_1, g_1) \ar^{\pi_V}[d] \\
    (U_2, f_2) \ar^{\Phi_2}[r] & (V_2, g_2)
    }
    \]
    with horizontal morphisms critical morphisms, vertical morphisms smooth and $U_1\rightarrow V_1\times_{V_2} U_2$ \'etale, there is a canonical isomorphism
    \begin{equation}\label{eq:Ppullback}
    P_{\Phi_1}\cong (\Crit_{U_1/B}(f_1)\rightarrow \Crit_{U_2/B}(f_2))^* P_{\Phi_2}
    \end{equation}
    given by the identity on volume forms, using \cref{prop:criticalembeddingnormalquadratic}(4).
    \item Given a pair of critical morphisms $\Phi_i\colon (U_i, f_i)\rightarrow (V_i, g_i)$ of LG pairs over $B_i$, so that $\Phi_1\times \Phi_2\colon (U_1\times U_2, f_1\boxplus f_2)\rightarrow (V_1\times V_2, g_1\boxplus g_2)$ is a critical morphism of LG pairs over $B_1\times B_2$, there is a canonical isomorphism
    \begin{equation}\label{eq:Pproduct}
    P_{\Phi_1}\boxtimes P_{\Phi_2}\cong P_{\Phi_1\times \Phi_2}
    \end{equation}
    defined via the isomorphism \eqref{eq:orientationsumisomorphism}, using \cref{prop:criticalembeddingnormalquadratic}(5).
    \item If $\overline{\Phi}\colon (U, -f)\rightarrow (V, -g)$ is the same morphism between the opposite LG pairs, there is a canonical isomorphism
    \begin{equation}\label{eq:Popposite}
    P_{\Phi}\cong P_{\overline{\Phi}}
    \end{equation}
    defined via the isomorphism \eqref{eq:orientationreverse}, using \cref{prop:criticalembeddingnormalquadratic}(6).
\end{enumerate}

The quadratic form $q_\Phi$ for a critical morphism allows us to define the virtual canonical bundle as follows.

\begin{theorem}\label{thm:virtualcanonicalschemes}
Let $X\rightarrow B$ be a morphism of schemes equipped with a relative d-critical structure $s$. Then there is a line bundle $K^{\vir}_{X/B}$ on $X^{\red}$, the \defterm{virtual canonical bundle}, uniquely determined by the following conditions:
\begin{enumerate}
    \item For every critical chart $(U, f, u)$ its restriction to $\Crit_{U/B}(f)^{\red}$ is canonically isomorphic to $K_{U/B}^{\otimes 2}|_{\Crit_{U/B}(f)^{\red}}$.
    \item For every critical morphism $\Phi\colon (U, f, u)\rightarrow (V, g, v)$ the corresponding isomorphism 
    \[J_{\Phi}\colon K_{U/B}^{\otimes 2}|_{\Crit_{U/B}(f)^{\red}}\xrightarrow{\sim} K_{V/B}^{\otimes 2}|_{\Crit_{U/B}(f)^{\red}}\]
    fits into a commutative diagram
    \[
    \xymatrix{
    K_{U/B}^{\otimes 2}|_{\Crit_{U/B}(f)^{\red}} \ar^{\id\otimes \vol^2_{q_\Phi}}[dr] \ar^{J_\Phi}[rr] && K_{V/B}^{\otimes 2}|_{\Crit_{U/B}(f)^{\red}} \\
    &K_{U/B}^{\otimes 2}|_{\Crit_{U/B}(f)^{\red}}\otimes (\det \rN_{U/V}^\vee)^{\otimes 2}|_{\Crit_{U/B}(f)^{\red}}, \ar_{i(\Delta)^2}[ur]
    }
    \]
    where the diagonal morphism on the right is determined by the short exact sequence
    \[\Delta\colon 0\longrightarrow \rN_{U/V}^\vee\longrightarrow \Omega^1_{V/B}|_U\longrightarrow \Omega^1_{U/B}\longrightarrow 0.\]
\end{enumerate}
For every point $x\in X$ there is an isomorphism
\[\kappa_x\colon K^{\vir}_{X/B, x}\xrightarrow{\sim} \det(\Omega^1_{X/B, x})^{\otimes 2}\]
uniquely determined by the following condition:
\begin{enumerate}[resume]
    \item For every critical chart $(U, f, u)$ with a point $y\in \Crit_{U/B}(f)$ such that $u(y) = x$ the isomorphism
    \[K_{U/B, y}^{\otimes 2}\cong K^{\vir}_{X/B, x}\xrightarrow{\kappa_x} \det(\Omega^1_{X/B, x})^{\otimes 2}\]
    coincides with the composite of the natural isomorphism of determinant lines induced by the exact sequence
    \[0\longrightarrow \rN^\vee_{X/U, x}\longrightarrow \Omega^1_{U/B, y}\longrightarrow \Omega^1_{X/B, x}\longrightarrow 0\]
    as well as the squared volume form on $\rN_{X/U, x}$ induced by the quadratic form $\Hess(f)_x$.
\end{enumerate}
In addition, we have the following isomorphisms:
\begin{enumerate}[resume]
    \item For a pair $(X_1\rightarrow B_1, s_1)$, $(X_2\rightarrow B_2, s_2)$ of morphisms of schemes equipped with relative d-critical structures consider the relative d-critical structure $s_1\boxplus s_2$ on $X_1\times X_2\rightarrow B_1\times B_2$. Then there is an isomorphism
    \begin{equation}\label{eq:Kproduct}
    K^{\vir}_{X_1/B_1}\boxtimes K^{\vir}_{X_2/B_2}\cong K^{\vir}_{X_1\times X_2/B_1\times B_2}
    \end{equation}
    uniquely determined by the condition that for every point $(x_1, x_2)\in X_1\times X_2$ there is a commutative diagram
    \[
    \xymatrix{
    K^{\vir}_{X_1/B_1, x_1}\otimes K^{\vir}_{X_2/B_2, x_2} \ar^{\eqref{eq:Kproduct}}[r] \ar^{\kappa_{x_1}\otimes \kappa_{x_2}}[d] & K^{\vir}_{X_1\times X_2/B_1\times B_2, (x_1, x_2)} \ar^{\kappa_{(x_1, x_2)}}[d] \\
    \det(\Omega^1_{X_1/B_1, x_1})^{\otimes 2}\otimes \det(\Omega^1_{X_2/B_2, x_2})^{\otimes 2} \ar^-{\sim}[r] & \det(\Omega^1_{X_1\times X_2/B_1\times B_2, (x_1, x_2)})^{\otimes 2}.
    }
    \]
    \item For a morphism of schemes $X\rightarrow B$ equipped with a relative d-critical structure $s$ let $K^{\vir}_{X/B, s}$ be the virtual canonical bundle. For $d\colon X\rightarrow \Z/2\Z$ there is an isomorphism
    \begin{equation}\label{eq:Kreverse}
    R_d\colon K^{\vir}_{X/B, s}\cong K^{\vir}_{X/B, -s}
    \end{equation}
    squaring to the identity.
\end{enumerate}
\end{theorem}
\begin{proof}
Consider the Zariski stack on $X$ whose value on an open immersion $R\rightarrow X$ is given by the groupoid $\Pic_{R^{\red}}$ of line bundles on $R^{\red}$. We may then apply \cref{cor:criticaldescent} to glue the objects $K_{U/B}^{\otimes 2}|_{\Crit_{U/B}(f)^{\red}}$ defined for a critical chart $(U, f, u)$ using the isomorphisms $J_\Phi$ for a critical morphism. The relevant conditions of \cref{cor:criticaldescent} are verified as follows:
\begin{enumerate}
    \item The first condition follows from \cref{prop:criticalembeddingnormalquadratic}(2).
    \item The second condition follows since $\rN_{U/U}=0$ with the zero quadratic form.
    \item The third condition follows from \cref{prop:criticalembeddingnormalquadratic}(3).
\end{enumerate}
The uniqueness of the isomorphism $\kappa_x$ is clear. The existence will follow once we show that the condition is compatible with critical morphisms $(U, f, u)\rightarrow (V, g, v)$. Since this is a pointwise statement, it reduces to the statement about d-critical loci over a field which is shown in \cite[Theorem 2.28(iv)]{JoyceDcrit}.

In a critical chart $(U_1, f_1, u_1)$ for $X_1$ and $(U_2, f_2, u_2)$ for $X_2$ we define the isomorphism \eqref{eq:Kproduct} to be the obvious isomorphism
\[K^{\otimes 2}_{U_1/B_1}|_{\Crit_{U_1/B_1}(f_1)^{\red}}\boxtimes K^{\otimes 2}_{U_2/B_2}|_{\Crit_{U_2/B_2}(f_2)^{\red}}\cong K^{\otimes 2}_{U_1\times U_2/B_1\times B_2}|_{\Crit_{U_1\times U_2/B_1\times B_2}(f_1\boxplus f_2)^{\red}},\]
under the identification of \cref{prop:criticalproduct}.
Since the quadratic form $q_\Phi$ is compatible with products by \cref{prop:criticalembeddingnormalquadratic}(5), $J_\Phi$ is also compatible with products, so the above isomorphism is compatible with critical morphisms.

In a critical chart $(U, f, u)$ for $X$ we define the isomorphism $R_d$ to be given by the multiplication by $(-1)^{\dim(U/B)+d}$ on $K^{\otimes 2}_{U/B}|_{\Crit_{U/B}(f)^{\red}}$. By \cref{prop:criticalembeddingnormalquadratic}(6) we have $q_{\overline{\Phi}} = -q_\Phi$. Since
\[\vol^2_{-q_\Phi} = (-1)^{\dim(V/B)-\dim(U/B)} \vol^2_{q_\Phi},\]
we have \[J_{\overline{\Phi}} = (-1)^{\dim(V/B)-\dim(U/B)} J_\Phi,\]
so thus defined isomorphism is compatible with critical morphisms.
\end{proof}

\begin{remark}\label{rem:kappa-discrepancy-KPS}
    In the setting of \cref{thm:virtualcanonicalschemes} assume that $B$ is a point. 
    In this case, the map $\kappa_x$ has also appeared in \cite[(3.13)]{KPS}.
    We note that our choice of $\kappa_x$ differs from the one in loc. cit. by $(-1)^{\rank \Omega_{X/B, x}}$. See \cite[Remark 3.36]{KPS} for the origin of the sign.
\end{remark}

\subsection{Deformation of morphisms of LG pairs}

In this section we fix a scheme $B\in\Sch^{\sep\ft}$. Given a pair $\Phi_0,\Phi_1$ of \'etale morphisms $(U, f)\rightarrow (V, g)$ of LG pairs over $B$ which induce equal morphisms on relative critical loci and which satisfy an extra condition we show that, \'etale locally, they can be extended to an $\bA^1$-family $\Phi_t\colon (U, f)\rightarrow (V, g)$. The following is a family version of \cite[Proposition 3.4]{BBDJS}.

\begin{proposition}\label{prop:BBDJS34}
Let $\Phi_0, \Phi_1\colon (U, f)\rightarrow (V, g)$ be \'etale morphisms of LG pairs over $B$ and $u \in \Crit_{U/B}(f)$ be a point.
Assume that
\begin{equation}\label{Eq:def17}
\Phi_0|_{\Crit_{U/B}(f)} = \Phi_1|_{\Crit_{U/B}(f)}\colon \Crit_{U/B}(f)\longrightarrow \Crit_{V/B}(g)
\end{equation}
and
\begin{equation}\label{Eq:def4}
(d\Phi_1|_u^{-1} \circ d\Phi_0|_u -\id)^2=0 \colon \T_{U/B,u} \to \T_{U/B,u}.
\end{equation}
Then we can find {\'e}tale morphisms of LG pairs over $B \times \bA^1$,
\[\xymatrix{
& (W,h) \ar[ld]_{\Psi_U} \ar[rd]^{\Psi_V} & \\
(U \times \bA^1, f\boxplus 0)   & & (V \times \bA^1, g \boxplus 0)
}\]
and a map $w\colon \bA^1 \to W$ such that $\Psi_U \circ w = (u, \id) \colon \bA^1 \to U\times \bA^1$, $\Psi_V \circ w = (v, \id) \colon \bA^1 \to V\times \bA^1$,
and
\begin{equation}\label{Eq:def13}
\xymatrix@C-3pc{
& \Crit_{W/B\times \bA^1}(h) \ar[ld]_{\Psi_U} \ar[rd]^{\Psi_V} & \\
\Crit_{U/B}(f)\times \bA^1 \ar[rr]^{\Phi_0=\Phi_1} && \Crit_{V/B}(g)\times \bA^1,
}\quad
\xymatrix{
& W_0 \ar[ld]_{\Psi_{U,0}} \ar[rd]^{\Psi_{V,0}} & \\
U \ar[rr]^{\Phi_0} & &  V , 
}
\quad
\xymatrix{
& W_1 \ar[ld]_{\Psi_{U,1}} \ar[rd]^{\Psi_{V,1}} & \\
U \ar[rr]^{\Phi_1} & &  V 
}
\end{equation}
commute, where $W_t$, $\Psi_{U, t}\colon W_t\rightarrow U$, and $\Psi_{V, t}\colon W_t\rightarrow V$ are the fibers of $W$, $\Psi_U$, and $\Psi_V$ at $t\in\bA^1$.
\end{proposition}

For a scheme $X\in\Sch^{\sep\ft}$, a vector bundle $E\rightarrow X$ and a section $s\in\Gamma(X, E)$ we denote by $\Zero_X(s)$ the zero section of $s$ as a subscheme of $X$. The key to prove \cref{prop:BBDJS34} is the following perturbation lemma. 

\begin{lemma}[Perturbation]\label{lm:perturbationlemma}
Let $(U,f) \hookrightarrow (V,g)$ be a closed embedding of codimension $r$ of LG pairs over $B$.
Assume that there exist a vector bundle $E$ over $V$ of rank $r$, a section $s\in \Gamma(V,E)$, a cosection $\sigma \in \Gamma(V,E\dual)$, and a $2$-tensor $a\in \Gamma(V,E\otimes E)$ such that
$U=\Zero_V(s)$ as subschemes of $V$ and
\[g  = (\sigma \otimes \sigma) (a) + \sigma(s)  \in \Gamma(V,\O_V/\I_{U/V}^2).\]
Then we can find a closed embedding $(\tU,0) \hookrightarrow (\tV,\tg)$ of LG pairs over $B$, an {\'e}tale morphism $ (\tV,\tg) \to (V,g)$ of LG pairs over $B$,
and commutative diagrams
\begin{equation}\label{Eq:def10}
\xymatrix{
\Zero_{\tU}(\sigma) \ar@{^{(}->}[r] \ar@{.>}[d] & \tU \ar@{^{(}->}[r] & \tV \ar[d]^{}  \\
\Zero_{U}(\sigma) \ar@{^{(}->}[r] & U \ar@{^{(}->}[r] & V,
}\qquad
\xymatrix{
\Zero_{\tU}(a) \ar@{^{(}->}[r] \ar@{.>}[d] & \tU \ar@{^{(}->}[r] & \tV \ar[d]^{}  \\
\Zero_{U}(a) \ar@{^{(}->}[r] & U \ar@{^{(}->}[r] & V,}
\qquad
\xymatrix{
& \tU \ar@{^{(}->}[r] & \tV \ar[d]^{} \\
U_0:=\Zero_{U}(\sigma,a) \ar@{^{(}->}[r] \ar@{.>}[ru]^{} & U \ar@{^{(}->}[r] & V}
\end{equation}
for some dotted arrows,
such that
\begin{equation}\label{Eq:def11}
\T_{\tU/B}|_{U_0} = \T_{U/B}|_{U_0}
\end{equation}
as subspaces of $\T_{\tV/B}|_{U_0} \cong  \T_{V/B}|_{U_0}$.
\end{lemma}

\begin{proof}
Choose a $(0,2)$-tensor $c\in \Gamma(V,E\dual \otimes E\dual)$ such that
\[g = (\sigma \otimes \sigma) (a)+ \sigma(s) + c(s,s) \in \Gamma(V,\O_V).\]
We first define $\tV\coloneqq\Zero(\widetilde{\tau})\subseteq E\otimes E$ as the zero locus of the section
\[ \widetilde{\tau}\coloneqq a|_{E\otimes E} + \tau + c|_{E\otimes E}(\tau,\tau) \in \Gamma(E\otimes E, E\otimes E|_{E\otimes E}) , \]
where $\tau$ is the tautological section,
and $c|_{E\otimes E}(\tau,\tau)$ is the image of $c|_{E\otimes E} \otimes \tau \otimes \tau$ under the contraction map 
\[c_{13,25} \colon (E\dual\otimes E\dual) \otimes (E\otimes E) \otimes (E\otimes E) \colon (\alpha_1,\alpha_2,\alpha_3,\alpha_4,\alpha_5,\alpha_6)\mapsto \alpha_1(\alpha_3)\cdot \alpha_2(\alpha_5) \cdot \alpha_4\otimes \alpha_6.\]
We then define $\tU:=\Zero(\widetilde{s}) \subseteq \tV$ as the zero locus of the section
\[\widetilde{s}\coloneqq s|_{\tV} - \sigma|_{\tV} (\tau|_{\tV}) \in \Gamma(\tV, E|_{\tV}),\]
where $\sigma|_{\tV} (\tau|_{\tV})$ is the image of $\sigma|_{\tV} \otimes \tau|_{\tV}$ under the contraction map \[c_{13} \colon E\dual \otimes (E\otimes E) \to E \colon (\alpha_1,\alpha_2,\alpha_3) \mapsto \alpha_1(\alpha_3)\cdot \alpha_2.\]
Then the function $\tg\coloneqq g|_{\tV} \colon \tV \to \bA^1$ vanishes on $\tU\coloneqq\Zero_{\tV}(\widetilde{s})=\Zero_{E\otimes E}(\widetilde{\tau},s-\sigma(\tau))$,
\begin{align*}
\tg|_{\tU} 
= \left( (\sigma \otimes \sigma) (a)+ \sigma(s) + c(s,s)\right)|_{\tU}     
&=\left( (\sigma \otimes \sigma) (a)+ \sigma(\sigma(\tau)) + c(\sigma(\tau),\sigma(\tau))\right)|_{\tU} \\
& = (\sigma \otimes \sigma) (a + \tau + c(\tau,\tau))|_{\tU} = (\sigma \otimes \sigma)(\widetilde{\tau})|_{\tU} = 0.
\end{align*}

The existence of the first dotted arrow $\Zero_{\tU}(\sigma)  \to \Zero_U(\sigma)=\Zero_V(s,\sigma)$ over $V$ in \eqref{Eq:def10} follows from:
\[\Zero_{\tU}(\sigma) =\Zero_{\tV}(\widetilde{s},\sigma) \subseteq \Zero_{\tV}(s).\]
The existence of the second dotted arrow $\Zero_{\tU}(a) \to \Zero_U(a)=\Zero_V(s,a)$ over $V$ in \eqref{Eq:def10} is equivalent to:
\[\Zero_{\tU}(a) =\Zero_{\tV}(\widetilde{s},a) =\Zero_{E\otimes E}(\widetilde{\tau},\widetilde{s},a) \subseteq \Zero_{E\otimes E}(s).\]
Since $\Zero_{E\otimes E}(a,\tau) \hookrightarrow \Zero_{E\otimes E}(a,\widetilde{\tau})$ is an {\'e}tale closed embedding, the complement is closed, and thus we have such dotted arrow for the open subschemes
\[\tV^{\circ} \coloneqq \tV - ( \Zero_{\tV}(a) - \Zero_{V}(a) ) \subseteq \tV, \qquad \tU^{\circ} \coloneqq\tU \cap \tV^{\circ} \subseteq \tU .\] 
We define the third dotted arrow $U_0 \to \tU$ in \eqref{Eq:def10} as the canonical closed embedding
\[U_0\coloneqq\Zero_U(\sigma,a) = \Zero_V(s,\sigma,a) = \Zero_{E\otimes E}(\tau,s,\sigma,a) \subseteq \Zero_{E\otimes E}(\widetilde{\tau},\widetilde{s}) = \Zero_{\tV}(\widetilde{s})=\tU,\]
under the identification of $V$ with the zero section $0\colon V \to E\otimes E$.
By direct computations of de Rham differentials, we can show that
\begin{itemize}
\item $\tV  \hookrightarrow E\otimes E \to V$ is {\'e}tale near $\Zero_{\tV}(a)\supseteq U_0$;
\item $\tU$ is smooth over $B$ near $\Zero_{\tU}(\tau,\sigma)=U_0$ since $U=\Zero_V(s)$ is smooth over $B$;
\item $\T_{\tU/B}|_{U_0}=\T_{U/B}|_{U_0}$ under the identification $\T_{\tV/B}|_{U_0} = \T_{V/B}|_{U_0}$.
\end{itemize}
Consequently, we have \cref{lm:perturbationlemma} after replacing $\tV$ and $\tU$ with suitable open subschemes containing $U_0$.
\end{proof}

We now prove \cref{prop:BBDJS34} using \cref{lm:perturbationlemma} and \cref{Lem:computation} below.

\begin{proof}[Proof of \cref{prop:BBDJS34}]
Choose an {\'e}tale coordinate $y\colon V \to \bA^n_B$ near $v\in V$ after shrinking $V$ if necessary.
Consider the zero locus $W' \coloneqq \Zero (z) \subseteq U \times_B V \times \bA^1$ of the section
\begin{equation}\label{Eq:def1}
z\coloneqq y \circ \pr_2 - ((1-t)x_0 \circ \pr_1 + tx_1 \circ \pr_1) \in \Gamma(U\times_B V\times \bA^1,\O^{n}_{U\times_B V\times \bA^1}),
\end{equation}
where $x_0\coloneqq y \circ \Phi_0\colon U \to \bA^n_B$, $x_1\coloneqq y \circ \Phi_1\colon U \to \bA^n_B$, and $t\coloneqq\pr_3\colon U \times_B V \times \bA^1 \to \bA^1$.
Then $W'$ has most of the desired properties: the projection maps
\[\xymatrix{
& W' \ar[ld]_{\pr_{13}} \ar[rd]^{\pr_{23}} &\\
U\times \bA^1 && V\times \bA^1
}\]
are {\'e}tale since $\frac{\partial z}{\partial y}=\id$ and $\frac{\partial z}{\partial x_0} = 1 - t (1 - \frac{\partial x_1}{\partial x_0})$ is invertible by \eqref{Eq:def4}; 
the induced triangles
\begin{equation}\label{Eq:def14}
\xymatrix@C-4pc{
&  \Crit_{W'/B\times \bA^1}(f|_{W'}) \ar[ld]^{\pr_{13}} \ar[rd]^{\pr_{23}} &\\
\Crit_{U/B}(f)\times \bA^1 \ar[rr]_{\Phi_0=\Phi_1} \ar@/^2pc/@{.>}[ru]^-{\exists} &&  \Crit_{V/B}(g)\times \bA^1,
}\quad
\xymatrix@C-2pc{
& W'_0\coloneqq W'\times_{\bA^1}\{0\} \ar[ld]^{\pr_{1}} \ar[rd]^{\pr_{2}} &\\
U \ar[rr]_{\Phi_0} \ar@/^1pc/@{.>}[ru]^-{\exists} && V ,
}\quad \xymatrix@C-2pc{
& W'_1\coloneqq W'\times_{\bA^1}\{1\} \ar[ld]^{\pr_{1}} \ar[rd]^{\pr_{2}} &\\
U\ar[rr]_{\Phi_1} \ar@/^1pc/@{.>}[ru]^-{\exists} && V
}
\end{equation}
commute after shrinking $W'$,
since the graphs of the three lower horizontal arrows in \eqref{Eq:def14} are sections of the three {\'e}tale maps induced by $\pr_{13}$;
we have a map
$w'\coloneqq (u,v,\id) \colon \bA^1 \to W' \subseteq U \times_B V.$
However, $W'$ is not sufficient to have \cref{prop:BBDJS34} since
\[f \circ \pr_1 \neq g \circ \pr_2 \textin W' \subseteq U\times_B V\times \bA^1.\]

The computations in \cref{Lem:computation} below ensure that we can apply \cref{lm:perturbationlemma} above to the closed embedding $W' \hookrightarrow U\times_B V \times \bA^1$ and its function $f \circ \pr_1 - g \circ \pr_2$.
Then we can find:
\begin{itemize}
\item an {\'e}tale morphism $e\colon R \to U \times_B V \times \bA^1$, 
\item a smooth closed subscheme $W \subseteq R$ such that $h\coloneqq f \circ \pr_1 \circ e |_W = g \circ \pr_2 \circ e|_W$, and
\item a map $w'\colon\bA^1 \to W' \subseteq U\times_B V \times \bA^1$ lifts to a map $w:\bA^1 \to W \subseteq R$ via the third dotted arrow in \eqref{Eq:def10} since the $2$-tensor $a$ in \cref{Lem:computation} vanishes on $\{(u,v)\}\times \bA^1$.
\end{itemize}
Moreover, the induced maps 
\begin{equation}\label{Eq:def12}
\Psi_U \colon W \hookrightarrow R \to U\times_B V \times \bA^1 \xrightarrow{\pr_{13}} U \times \bA^1, \quad \Psi_V \colon W \hookrightarrow R \to U\times_B V \times \bA^1 \xrightarrow{\pr_{23}} V \times \bA^1
\end{equation}
are {\'e}tale near the image of $w$, since the tangent space remains the same by \eqref{Eq:def11};
the commutativity of the first triangle in \eqref{Eq:def13} follows from the commutativity of the first triangle in \eqref{Eq:def14} and the first dotted arrow in \eqref{Eq:def10}; the commutativity of the remaining two triangles in \eqref{Eq:def13} follows from the commutativity of the remaining two triangles in \eqref{Eq:def14} and the second dotted arrow in \eqref{Eq:def10} since the $2$-tensor $a$ in \cref{Lem:computation} vanishes on $W'_0 \cup W'_1$.
\end{proof}

We need the following lemma to complete the proof of \cref{prop:BBDJS34}.

\begin{lemma}\label{Lem:computation}
In the situation of \cref{prop:BBDJS34}, there exist
\begin{itemize}
\item an open subscheme $V^\circ \subseteq V$ containing $v$,
\item an open subscheme $S \subseteq (U\times_B V \times \bA^1)$ containing $\{(u,v)\} \times \bA^1$,
\item an {\'e}tale coordinate $y \colon V^{\circ} \to \bA^n_B$,
\item a section 
$a \in \Gamma(S, \pr_1^* (\T_{U/B} \otimes \T_{U/B})|_S)$ that vanishes in
$(\{(u,v)\}\times \bA^1) \cup (S \times_{\bA^1} \{0,1\}),$
\item an isomorphism of vector bundles $b \colon\O^{\oplus n}_{S} \cong  \pr_1^* \T_{U/B}|_S $,
\end{itemize}
such that we have
\[ ( f \boxplus (-g) \boxplus 0 ) = (df \otimes df)(a) +df \circ b(z) \in \Gamma(S, \O_{S}/\I_{\Zero(z)/S}^2)
,\]
where $z$ is defined as in \eqref{Eq:def1} 
and $df \in \Gamma(S, \pr_1^*\Omega^1_{U/B}|_S)$ is the pullback of $df \in \Gamma(U, \Omega^1_{U/B})$.
\end{lemma}

\begin{proof}

Choose an {\'e}tale coordinate $y\coloneqq(y^m,y^q)\colon V \to \bA^m_B \times_B \bA^q_B$ such that the Hessian at $v\in V$ is:
\begin{equation}\label{Eq:def8}
\frac{\partial^2 g}{\partial y^2}|_v
=
\left[\begin{array}{c|c} 
	\frac{\partial^2 g}{\partial y^m \partial y^m}|_v & \frac{\partial^2 g}{\partial y^m \partial y^q}|_v \\ 
	\hline 
	\frac{\partial^2 g}{\partial y^q \partial y^m}|_v & \frac{\partial^2 g}{\partial y^q \partial y^q} |_v
\end{array}\right] =
\left[\begin{array}{c|c} 
	0 & 0\\ 
	\hline 
	0 & \id 
\end{array}\right] .
\end{equation}
Indeed, this is always possible after shrinking $V$ and a coordinate change of $y$ by $GL_{m+q}$.

Let $x_0\coloneqq y \circ \Phi_0$ and $x_1\coloneqq y \circ \Phi_1$ be the induced coordinates.
Then there is a $q\times q$-matrix $N$ such that
\begin{equation}\label{Eq:def9}
\left(\id -\frac{\partial x_1}{\partial x_0}\right)|_u =
\left[\begin{array}{c|c} 
	0 & *\\ 
	\hline 
	0 & N
\end{array}\right] , \where N^t=-N, \quad N^2=0
\end{equation}
Indeed, \eqref{Eq:def17} gives $\frac{\partial x_1^m}{\partial x_0^m}|_u=0$; the formula $\frac{\partial^2 g}{\partial x_1^2}|_u = \frac{\partial^2 g}{\partial y^2}|_v = \frac{\partial^2 g}{\partial x^2_0}|_u$ gives
\[
\left(\frac{\partial x_1}{\partial x_0}\right)^t|_u
\left[\begin{array}{c|c} 
	0 & 0\\ 
	\hline 
	0 & \id 
\end{array}\right]  
\left(\frac{\partial x_1}{\partial x_0}\right)|_u
= 
\left[\begin{array}{c|c} 
	0 & 0\\ 
	\hline 
	0 & \id 
\end{array}\right]   .
\]
and hence $\frac{\partial x_1^q}{\partial x_0^m}|_u=0$ and $\frac{\partial x_1^q}{\partial x_0^q}|_u^t=-\frac{\partial x_1^q}{\partial x_0^q}|_u$; \eqref{Eq:def4} gives $\frac{\partial x_1^q}{\partial x_0^q}|_u^2=0$.

We also note that there exists a matrix $L \in \Hom_U(\O_U^{\oplus m+q},\O_U^{\oplus m+q})$ of functions on $U$ such that
\begin{equation}\label{Eq:def6}
x_0 - x_1 = L \cdot \frac{\partial f}{\partial x_0}, \where L|_u =
\left[\begin{array}{c|c} 
	* & *\\ 
	\hline 
	0 & N
\end{array}\right]. 
\end{equation}
Indeed, the existence of a matrix satisfying the first formula follows from \eqref{Eq:def17}.
Moreover, \eqref{Eq:def8}  gives $\Crit_{V/B}(g) \subseteq \Zero_V(y^q)$ and hence we can find a $q\times(m+q)$-matrix $M$ of functions on $V$ such that
\[y^q = M \cdot \frac{\partial g} {\partial y} , \where M|_v =  \left[\begin{array}{c|c} 
	0 & 1
\end{array}\right].\]
Then we have
\[x_0^q - x_1^q = \Phi_0^*(M) \frac{\partial f} {\partial x_0} -\Phi_1^*(M) \frac{\partial f} {\partial x_1}  =  \left( \Phi_0^*(M)  - \Phi_1^*(M)  \frac{\partial x_0}{\partial x_1}\right)  \frac{\partial f} {\partial x_0},\]
where $\left( \Phi_0^*(M)  - \Phi_1^*(M)  \frac{\partial x_0}{\partial x_1}\right)|_u =  \left[\begin{array}{c|c} 
	0 & N
\end{array}\right]$ by \eqref{Eq:def9}. 
Hence we have \eqref{Eq:def6} as claimed.

We claim that there exists a  matrix $A$ of functions on $S\subseteq U\times_B V\times \bA^1$ such that
\begin{equation}\label{Eq:def5}
( f \boxplus (-g) \boxplus 0 ) = \left(\frac{\partial f}{\partial x_0}\right)^t \cdot A \cdot \frac{\partial f}{\partial x_0} \textin \Zero(z), \where A|_{(\{(u,v)\} \times \bA^1) \cup (S\times_{\bA^1} \{0,1\})}=0 .
\end{equation}
Since we have \eqref{Eq:def6} and $( f \boxplus (-g) \boxplus 0 )|_{\Zero(z,t(1-t))}=0$ after shrinking $S$,
it suffices to find $A$ such that
\begin{equation}\label{Eq:def15}
( f \boxplus (-g) \boxplus 0 ) = \left(\frac{\partial f}{\partial x_0}\right)^t \cdot A \cdot \frac{\partial f}{\partial x_0} \textin \Zero(z,(x_0-x_1)^{\otimes3}), \where A|_{\{(u,v)\} \times \bA^1 }=0 ,
\end{equation}
where $(x_0-x_1)^{\otimes3} \in \Gamma(U, (\O_U^{q+m})^{\otimes 3})$.
Indeed, we have
\begin{align}
f \boxplus (-g) \boxplus 0 
 & = \frac{\partial g}{\partial y} \cdot (x_0 - y) + (x_0-y)^t \cdot \frac{\partial^2 g}{\partial y^2} \cdot (x_0-y) \textin \Zero((x_0-y)^{\otimes3}) \label{Eq:def2} \\
 & =  \frac{\partial g}{\partial y} \cdot (x_1 - y) + (x_1-y)^t \cdot \frac{\partial^2 g}{\partial y^2} \cdot (x_1-y) \textin \Zero((x_1-y)^{\otimes3}) . \label{Eq:def3}
\end{align}
Then the restriction of $(1-t)\times$\eqref{Eq:def2}$+t\times$\eqref{Eq:def3} to $\Zero(z)$ is:
\begin{equation}\label{Eq:def7}
     ( f \boxplus (-g) \boxplus 0 ) = \frac12 t(1-t) (x_0-x_1)^t \cdot\frac{\partial^2 f}{\partial x_0^2} \cdot  (x_0-x_1)  \textin \Zero(z,(x_0-x_1)^{\otimes3}) .
\end{equation}
Consider the matrix
\[A\coloneqq  \frac12 t(1-t) \cdot L^t \cdot \frac{\partial^2 f}{\partial x_0^2} \cdot L .\]
Then the left equality in \eqref{Eq:def15} follows from \eqref{Eq:def7} and the left equality in  \eqref{Eq:def6}.
The right equality in \eqref{Eq:def15} follows from the right equality in  \eqref{Eq:def6} and \eqref{Eq:def8}.

By the left equality in \eqref{Eq:def5}, there exists a matrix $B'$ of functions on $S\subseteq U \times V \times \bA^1$ such that
\begin{equation}\label{Eq:4}
f \boxplus -g \boxplus 0 =   \left(\frac{\partial f}{\partial x_0}\right)^t \cdot A \cdot \frac{\partial f}{\partial x_0}  +   \left(\frac{\partial f}{\partial x_0}\right)^t \cdot B' \cdot z  \textin \Zero(z^{\otimes2})
.\end{equation}
By pulling back \eqref{Eq:4} to $\Zero(z)$ and taking the differential, we get
\[ \frac{\partial f}{\partial x_0} - \frac{\partial g}{\partial y} \left((1-t) + t \frac{\partial x_1}{\partial x_0}\right) = \left(\frac{\partial f}{\partial x_0}\right)^t \cdot \frac{\partial A}{\partial x_0} \cdot \frac{\partial f}{\partial x_0} + 
\left(\left(\frac{\partial f}{\partial x_0}\right)^t \cdot A \cdot \frac{\partial^2 f}{\partial x_0^2}\right)^t  \textin \Zero(z),\]
where $\frac{\partial A}{\partial x_0} \in \Hom_S (\O_S^{q+m},\O_S^{q+m} \otimes \O_S^{q+m})$ and $\left(\frac{\partial f}{\partial x_0}\right)^t \cdot \frac{\partial A}{\partial x_0}$ is its contraction in the first factor.
Hence
\[\frac{\partial g}{\partial y} = B\cdot  \frac{\partial f}{\partial x_0} \textin \Zero(z), \where B:=\left(1- t \left(1- \frac{\partial x_1}{\partial x_0}\right)\right)^{-1}\left(1-\left(\frac{\partial f}{\partial x_0}\right)^t \cdot \frac{\partial A}{\partial x_0}  - \left(\frac{\partial^2 f}{\partial x_0^2}\right)^t\cdot A^t\right).\]
Since $A|_{\{(u,v)\}\times \bA^1}=0$, $\frac{\partial f}{\partial x_0}|_u=0$, and $\left(1- \frac{\partial x_1}{\partial x_0}\right)^2=0$, $B$ is invertible.

Applying the partial differential of \eqref{Eq:4} by the $y$-coordinate and pulling back to $\Zero(z)$, we get
\[\frac{\partial g}{\partial y} = B' \cdot \frac{\partial f}{\partial x_0} \textin \Zero(z).\]
Then $B' \cdot \frac{\partial f}{\partial x_0} - B \cdot \frac{\partial f}{\partial x_0}$ vanishes on $\Zero(z)$ and hence we get
\[f \boxplus -g \boxplus 0 =   \left(\frac{\partial f}{\partial x_0}\right)^t \cdot A \cdot \frac{\partial f}{\partial x_0}  +   \left(\frac{\partial f}{\partial x_0}\right)^t \cdot B \cdot z  \textin \Zero(z^{\otimes2})
,\]
as desired.
\end{proof}

\section{Functoriality of d-critical structures}

\subsection{Pullbacks of d-critical structures}

In the definition of relative d-critical structures (see \cref{def:dcriticalscheme}) we have considered critical charts which are Zariski open. We may replace this condition by requiring \'etale or smooth critical charts; the goal of this section is to show that the resulting notion of a relative d-critical structure does not change (see \cref{thm:criticallocussmoothdescent}). We begin with the following observation regarding smooth functoriality of relative critical loci.

\begin{proposition}\label{prop:smoothcriticallocus}
Let $(V, g)$ be an LG pair over $B$. Let $\pi\colon U\rightarrow V$ be a smooth morphism and denote $f=\pi^\ast g$. Then we have a Cartesian diagram
\[
\xymatrix{
\Crit_{U/B}(f) \ar[r] \ar[d] & U \ar^{\pi}[d] \\
\Crit_{V/B}(g) \ar[r] & V.
}
\]
\end{proposition}
\begin{proof}
Consider the diagram
\[
\xymatrix{
& U \ar[dr] \ar[dl] & \\
V \ar^{\Gamma_{d_B g}}[dr] && \T^*(V/B)\times_V U \ar[dr] \ar^{\pi_1}[dl] \\
& \T^*(V/B) && \T^*(U/B)
}
\]
where the square is Cartesian and the composite $\Gamma_{d_B f}\colon U\rightarrow \T^*(U/B)$ is given by the graph of $d_B f$. Since $U\rightarrow V$ is smooth, $\T^*(V/B)\times_V U\rightarrow \T^*(U/B)$ is injective. Therefore, the zero locus of $U\rightarrow \T^*(V/B)\times_V U$ (i.e. $\Crit_{V/B}(g)\times_V U$) coincides with the zero locus of $\Gamma_{d_B f}\colon U\rightarrow \T^*(U/B)$ (i.e. $\Crit_{U/B}(f)$).
\end{proof}

We will now prove the following technical statement, which will be used to recognize minimal critical charts; it is a family version of \cite[Proposition 2.7]{JoyceDcrit}.

\begin{proposition}\label{prop:Joyce27}
Let $X\rightarrow B$ be a morphism of schemes equipped with a relative d-critical structure $s$, a point $x\in X$, a smooth $B$-scheme $U$, a closed immersion $\imath\colon X\hookrightarrow U$ and a function $f\colon U\rightarrow \bA^1$ which satisfy the following properties:
\begin{enumerate}
    \item $\iota_{X, U}(s) = f\in\cO_U/\cI_{X,U}^2$ (see \cref{prop:Sproperties}(3) for the notation).
    \item $\imath\colon X\rightarrow U$ is minimal at $x$.
\end{enumerate}
Then there is an open neighborhood $U^\circ\subset U$ of $\imath(x)$ such that $\Crit_{U^\circ/B}(f|_{U^\circ})\times_U X\rightarrow X$ is an open immersion.
\end{proposition}
\begin{proof}
Since $s$ is a relative d-critical structure on $X\rightarrow B$, we may find a critical chart $(V, g, v)$ at $x$. Let $X^\circ\subset X$ be the image of $v\colon \Crit_{V/B}(g)\rightarrow X$ and denote by $\jmath\colon X^\circ\rightarrow V$ the corresponding closed immersion. By \cref{prop:minimalchart} we may further assume that $\jmath\colon X^\circ\rightarrow V$ is minimal at $x$. Let $s^\circ = s|_{X^\circ}$.

The morphism $X^\circ\xrightarrow{\imath\times \jmath} U\times_B V$ is an immersion into a smooth $B$-scheme, so by \cref{prop:minimalimmersion}(1), possibly shrinking $X^\circ$, we may find a factorization of this morphism as $X^\circ\xrightarrow{\tilde{\jmath}} \tilde{V}\xrightarrow{\pi_U\times \pi_V} U\times_B V$, where $\tilde{V}$ is a smooth $B$-scheme and $\tilde{\jmath}$ is a closed immersion minimal at $x$. The morphisms
\[\tilde{\jmath}^\ast\colon \Omega^1_{\tilde{V}/B, \tilde{\jmath}(x)}\rightarrow\Omega^1_{X^\circ/B, x},\qquad \imath^\ast\colon \Omega^1_{U/B, \imath(x)}\rightarrow \Omega^1_{X^\circ/B, x},\qquad \jmath^\ast\colon \Omega^1_{V/B, \jmath(x)}\rightarrow\Omega^1_{X^\circ/B, x}\]
are all isomorphisms as the corresponding immersions are minimal at $x$. Therefore,
\[\pi_U^\ast\colon \Omega^1_{U/B, \imath(x)}\rightarrow \Omega^1_{\tilde{V}/B, \tilde{\jmath}(x)},\qquad \pi_V^\ast\colon \Omega^1_{V/B, \jmath(x)}\rightarrow \Omega^1_{\tilde{V}/B, \tilde{\jmath}(x)}\]
are isomorphisms. Therefore, possibly shrinking $\tilde{V}$ (and, correspondingly, $X^\circ$), we may assume that $\pi_U\colon \tilde{V}\rightarrow U$ and $\pi_V\colon \tilde{V}\rightarrow V$ are \'etale. We will now compare $\Crit_{U/B}(f)$ and $\Crit_{V/B}(g)$ using the correspondence $V\xleftarrow{\pi_V} \tilde{V}\xrightarrow{\pi_U} U$.
\begin{enumerate}
    \item Consider the diagram
    \[
    \xymatrix{
    X^\circ \ar[dr] \ar@/^1.5pc/[drr] \ar@/_1.5pc/_{\tilde{\jmath}}[ddr] \\
    & \Crit_{V/B}(g)\times_V \tilde{V} \ar[r] \ar[d] & \Crit_{V/B}(g) \ar[d] \\
    & \tilde{V} \ar^{\pi_V}[r] & V
    }
    \]
    where the square is Cartesian. By \cref{prop:smoothcriticallocus} we have $\Crit_{V/B}(g)\times_V \tilde{V}\cong \Crit_{\tilde{V}/B}(\tilde{g})$, where $\tilde{g} = \pi_V^\ast g$. Since $X^\circ\rightarrow \Crit_{V/B}(g)$ is an open immersion and $\Crit_{\tilde{V}/B}(\tilde{g})\rightarrow \Crit_{V/B}(g)$ is \'etale, the morphism $X^\circ\rightarrow \Crit_{\tilde{V}/B}(\tilde{g})$ is \'etale and, therefore, its image is open. Similarly, $\tilde{\jmath}$ and $\Crit_{\tilde{V}/B}(\tilde{g})\rightarrow \tilde{V}$ are closed immersions, so $X^\circ\rightarrow \Crit_{\tilde{V}/B}(\tilde{g})$ is a closed immersion. Therefore, its image is also closed. Thus, shrinking $\tilde{V}$ we may assume that $X^\circ\rightarrow \Crit_{\tilde{V}/B}(\tilde{g})$ is an isomorphism. Thus, $(\tilde{V}, \tilde{g})$ provides a critical chart for $(X\rightarrow B, s)$ near $x$.
    \item Let $\cI$ be the ideal defining the closed immersion $X^\circ\cong\Crit_{\tilde{V}/B}(\tilde{G})\rightarrow \tilde{V}$. Let $\tilde{f} = \pi_U^\ast f$ and $\tilde{x} = \tilde{\jmath}(x)$. By assumption we have $\iota_{X^\circ, \tilde{V}}(s^\circ) = \tilde{f}\in\cO_{\tilde{V}}/\cI^2$. Since $(V, g, v)$ is a critical chart we also get $\iota_{X^\circ, \tilde{V}}(s^\circ) = \tilde{g}\in\cO_{\tilde{V}}/\cI^2$. Thus, $\tilde{f} - \tilde{g}\in\cI^2$. Shrinking $\tilde{V}$ (and $X^\circ$ so that $X^\circ\cong \Crit_{\tilde{V}/\tilde{B}}(\tilde{g})$ remains true) we may choose \'etale coordinates $\{z_1, \dots, z_n\}$ on $\tilde{V}$, which is a smooth $B$-scheme. Then we may find a symmetric matrix $a_{ij}\in\cO_{\tilde{V}}$ such that
    \[\tilde{f} - \tilde{g} = \sum_{i,j} a_{ij}\frac{\partial \tilde{g}}{\partial z_i}\frac{\partial \tilde{g}}{\partial z_j}.\]
    Taking the derivative, we obtain
    \[\frac{\partial \tilde{f}}{\partial z_k} = \frac{\partial\tilde{g}}{\partial z_k} + \sum_{i,j} \frac{\partial a_{ij}}{\partial z_k}\frac{\partial \tilde{g}}{\partial z_i}\frac{\partial \tilde{g}}{\partial z_j} + \sum_{i, j}2a_{ij} \frac{\partial^2 \tilde{g}}{\partial z_i\partial z_k} \frac{\partial \tilde{g}}{\partial z_j}.\]
    We can write it as
    \[d_B \tilde{f} = (\id + \alpha) d_B \tilde{g},\]
    where we introduce the matrix
    \[\alpha_{ij} = \sum_k \frac{\partial a_{kj}}{\partial z_i}\frac{\partial \tilde{g}}{\partial z_k} + \sum_k 2a_{kj} \frac{\partial^2 \tilde{g}}{\partial z_i\partial z_k}.\]
    Since $\tilde{x}\in\Crit_{\tilde{V}/B}(\tilde{g})$, we have $\frac{\partial \tilde{g}}{\partial z_k}(\tilde{x}) = 0$. Since the closed immersion $\tilde{\jmath}$ is minimal at $x$, by \cref{prop:criticalHessian} we have $\frac{\partial^2 \tilde{g}}{\partial z_i\partial z_k}(\tilde{x}) = 0$. Thus, $\alpha(\tilde{x}) = 0$ and hence, shrinking $\tilde{V}$ and $X^\circ$, we may arrange $(\id+\alpha)$ to be invertible. Thus, we get $\Crit_{\tilde{V}/B}(\tilde{f}) = \Crit_{\tilde{V}/B}(\tilde{g})$.
    \item Consider the diagram
    \[
    \xymatrix{
    \Crit_{\tilde{V}/B}(\tilde{f}) \ar[r] \ar[d] & \tilde{V} \ar[d] \\
    \Crit_{U/B}(f) \ar[r] & U
    }
    \]
    which is Cartesian by \cref{prop:smoothcriticallocus}. The composite $\Crit_{\tilde{V}/B}(\tilde{f})\cong X^\circ\rightarrow X\xrightarrow{\imath} U$ is an immersion, so it is a monomorphism. Therefore, $\Crit_{\tilde{V}/B}(\tilde{f})\rightarrow \Crit_{U/B}(f)\rightarrow U$ is a monomorphism and hence $\Crit_{\tilde{V}/B}(\tilde{f})\rightarrow \Crit_{U/B}(f)$ is a monomorphism. Since $\pi_U\colon \tilde{V}\rightarrow U$ is \'etale, we have that $\Crit_{\tilde{V}/B}(\tilde{f})\rightarrow \Crit_{U/B}(f)$ is \'etale. Therefore, by \cite[Tag 025G]{Stacks} we get that $\Crit_{\tilde{V}/B}(\tilde{f})\rightarrow \Crit_{U/B}(f)$ is an open immersion. As $\pi_U\colon \tilde{V}\rightarrow U$ is \'etale, $U^\circ = \pi_U(\tilde{V})\subset U$ is open. Thus, we get that $\Crit_{\tilde{V}/B}(\tilde{f})\rightarrow \Crit_{U^\circ/B}(f|_{U^\circ})$ is a surjective open immersion, hence an isomorphism.
\end{enumerate}
\end{proof}

The following is a family version of \cite[Proposition 2.8]{JoyceDcrit}.

\begin{theorem}\label{thm:criticallocussmoothdescent}
Let $Y\rightarrow B$ be a morphism of schemes equipped with a section $t\in\Gamma(Y, \cS_{Y/B})$, $\pi\colon X\rightarrow Y$ a smooth morphism and let $s=\pi^\ast t\in\Gamma(X, \cS_{X/B})$.
\begin{enumerate}
    \item If $t$ is a relative d-critical structure, then $s$ is a relative d-critical structure. Explicitly, for every point $x\in X$ we may find a critical chart $(U, f, u)$ for $X\rightarrow B$ around $x$, a critical chart $(V, g, v)$ for $Y\rightarrow B$ around $\pi(x)$ together with a smooth morphism $\tilde{\pi}\colon (U, f)\rightarrow (V, g)$ which fits into a commutative diagram
    \[
    \xymatrix{
    X \ar^{\pi}[d] & \Crit_{U/B}(f) \ar_-{u}[l] \ar[d] \ar[r] & U \ar^{\tilde{\pi}}[d] \\
    Y & \Crit_{V/B}(g) \ar_-{v}[l] \ar[r] & V
    }
    \]
    \item If $s$ is a relative d-critical structure and $\pi$ is surjective, then $t$ is a relative d-critical structure.
\end{enumerate}
\end{theorem}
\begin{proof}
In the first point $Y\rightarrow B$ has a relative d-critical structure, so it is locally of finite type. In the second point $X\rightarrow B$ is locally of finite type, so by \cite[Tag 01T8]{Stacks} we also get that $Y\rightarrow B$ is locally of finite type.

Let $x\in X$ be a point and $y=\pi(x)\in Y$. By \cref{prop:relativesmoothing} we may find a diagram
\[
\xymatrix{
X \ar^{\pi}[d] & X^\circ \ar^{\imath}[r] \ar^{\pi^\circ}[d] \ar[l] & U \ar^{\tilde{\pi}}[d] \\
Y & Y^\circ \ar^{\jmath}[r] \ar[l] & V
}
\]
with $X^\circ\subset X$ and $Y^\circ \subset Y$ open neighborhoods of $x$ and $y$, $\tilde{\pi}\colon U\rightarrow V$ smooth and $\imath\colon X^\circ\rightarrow U$ and $\jmath\colon Y^\circ\rightarrow V$ closed immersions into smooth $B$-schemes minimal at $x$ and $y$. Possibly shrinking $X^\circ,Y^\circ,U,V$ we may find a function $g\colon V\rightarrow \bA^1$ such that $\iota_{Y^\circ, V}(t) = g\in\cO_V/\cI_{Y^\circ, V}^2$. Let $f = \tilde{\pi}^\ast g$, so that $\iota_{X^\circ, U}(s) = f\in\cO_U/\cI_{X^\circ, U}^2$.

We now claim that $s$ is a relative d-critical structure in a neighborhood of $x$ if, and only if, we can shrink $U$ to a neighborhood of $x$ so that $\Crit_{U/B}(f)\times_U X^\circ\rightarrow X^\circ$ becomes an open immersion. Indeed, if this morphism is an open immersion, $(U, f)$ provides a critical chart at $x$. The converse is provided by \cref{prop:Joyce27}. The same claim applies to $t$.

Since $\tilde{\pi}\colon U\rightarrow V$ is smooth, by \cref{prop:smoothcriticallocus} we have $\Crit_{U/B}(f)\cong\Crit_{V/B}(g)\times_V U$. Thus, $\Crit_{U/B}(f)\times_U X^\circ\rightarrow X^\circ$ is an open immersion if, and only if, $\Crit_{V/B}(g)\times_V X^\circ\rightarrow X^\circ$ is an open immersion.

\begin{enumerate}
    \item Assume that $s$ is a relative d-critical structure. Then shrinking $V$ we get that $\Crit_{V/B}(g)\times_V Y^\circ\rightarrow Y^\circ$ is an open immersion. By base change $\Crit_{V/B}(g)\times_V X^\circ\rightarrow X^\circ$ is also an open immersion. By the above argument we get that $(X\rightarrow B, s)$ is a relative d-critical structure in a neighborhood of $x$. Varying $x$ we get that $(X\rightarrow B, s)$ is a relative d-critical structure.
    \item Assume that $s$ is a relative d-critical structure and $\pi$ is surjective. By the above argument we may shrink $V$ so that $\Crit_{V/B}(g)\times_V X^\circ\rightarrow X^\circ$ is an open immersion. By faithfully flat descent for open immersions \cite[Tag 02L3]{Stacks} this implies that $\Crit_{V/B}(g)\times_V Y^\circ\rightarrow Y^\circ$ is an open immersion. Again, the above argument implies that $t$ is a relative d-critical structure near $y$. As $\pi$ is surjective, we may vary $x\in X$ to cover $Y$, so we get that $t$ is a relative d-critical structure.
\end{enumerate}
\end{proof}

In \cref{thm:criticallocussmoothdescent} we have described a smooth functoriality of critical charts. We will also need a description of a smooth functoriality for critical morphisms of critical charts. The following statement is an analog of \cref{prop:stabilizationzigzag} for smooth morphisms of schemes equipped with relative d-critical structures. While we do not know how to show that the zigzag factorization in \cref{prop:stabilizationzigzag} can be made smooth functorial (this is claimed in the proof of \cite[Proposition 4.5]{BBBBJ}), the following alternative local model (with $U^\circ\rightarrow U$ being \'etale rather than an open immersion) is sufficient.

\begin{proposition}\label{prop:stabilizationzigzagsmoothmorphism}
Let $(X_2\rightarrow B, s_2)$ be a morphism of schemes equipped with a d-critical structure, $\pi\colon X_1\rightarrow X_2$ a smooth morphism and let $s_1=\pi^\ast s_2\in\Gamma(X_1, \cS_{X_1/B})$. Consider critical charts $(U_1, f_1, u_1)$ and $(V_1, f_1, v_1)$ of $X_1$ and $(U_2, f_2, u_2)$ and $(V_2, f_2, v_2)$ of $X_2$ together with smooth morphisms $\pi_U\colon (U_1, f_1)\rightarrow (U_2, f_2)$ and $\pi_V\colon (V_1, f_1)\rightarrow (V_2, f_2)$ compatible with $\pi$. Let $x_1\in\Crit_{U_1/B}(f_1)$ and $y_1\in\Crit_{V_1/B}(g_1)$ be points with $u_1(x_1)=v_1(y_1)$ and let $x_2=\pi_U(x_1)$ and $y_2=\pi_V(y_1)$. Then there is a commutative diagram
\[
\xymatrix{
(U_1, f_1, u_1) \ar^{\pi_U}[d] & (U_1^\circ, f_1^\circ, u_1^\circ) \ar[l] \ar^{\Phi_1}[r] \ar^{\pi^\circ_U}[d] & (W_1, h_1, w_1) \ar^{\pi_W}[d] & (V_1^\circ, g_1^\circ, v_1^\circ) \ar[r] \ar_-{\Psi_1}[l] \ar^{\pi^\circ_V}[d] & (V_1, g_1, v_1) \ar^{\pi_V}[d] \\
(U_2, f_2, u_2) & (U_2^\circ, f_2^\circ, u_2^\circ) \ar[l] \ar^{\Phi_2}[r] & (W_2, h_2, w_2) & (V_2^\circ, g_2^\circ, v_2^\circ) \ar[r] \ar_-{\Psi_2}[l] & (V_2, g_2, v_2)
}
\]
with $(U_i^\circ, f_i^\circ, u_i^\circ)$ \'etale critical charts and the rest Zariski critical charts, $(U_i^\circ, f_i^\circ, u_i^\circ)\rightarrow (U_i, f_i, u_i)$ an \'etale morphism surjective at $x_i$, $(V_i^\circ, g_i^\circ)\rightarrow (V_i, g_i)$ an open immersion surjective at $y_i$, $\Phi_i,\Psi_i$ critical morphisms, vertical morphisms smooth and $U_1^\circ\rightarrow W_1\times_{W_2} U_2^\circ$ and $V_1^\circ\rightarrow W_1\times_{W_2} V_2^\circ$ \'etale.
\end{proposition}
\begin{proof}
Let $V_2^\circ$ be an open neighborhood of $y_2$ which admits an \'etale morphism $V_2^\circ\rightarrow \bA^n_B$ over $B$ and $V_1^\circ = \pi_V^{-1}(V_2^\circ)$. Shrinking $V_1^\circ$ further, we may assume that there is an \'etale morphism $V_1^\circ\rightarrow \bA^d_{V_2^\circ}$ over $V_2^\circ$.

Replacing $\tilde{V}$ (equipped with an open immersion to $\bA^n_B$) by $V_2^\circ$ (equipped with an \'etale morphism to $\bA^n_B$) in the proof of \cref{prop:stabilizationzigzag}, we obtain an \'etale critical chart $(U_2^\circ, f_2^\circ, u_2^\circ)$ of $X_2$ with an \'etale morphism $(U_2^\circ, f_2^\circ, u_2^\circ)\rightarrow (U_2, f_2, u_2)$ and a commutative diagram
\[
\xymatrix{
\Crit_{U_2^\circ/B}(f_2^\circ) \ar[r] \ar[d] & \Crit_{V_2^\circ/B}(g_2^\circ) \ar[d] \\
U_2^\circ \ar^{\Theta_2}[r] & V_2^\circ
}
\]
of $B$-schemes. Repeating the same argument with $U_1$ and $V_1^\circ$ (as $V_2^\circ$-schemes), we obtain an \'etale critical chart $(U_1^\circ, f_1^\circ, u_2^\circ)$ of $X_1$ which fits into a commutative diagram
\[
\xymatrix{
U_1 \ar^{\pi_U}[d] & U_1^\circ \ar^{\pi_U^\circ}[d] \ar^{\Theta_1}[r] \ar[l] & V_1^\circ \ar[r] \ar^{\pi_V^\circ}[d] & V_1 \ar^{\pi_V}[d] \\
U_2 & U_2^\circ \ar^{\Theta_2}[r] \ar[l] & V_2^\circ \ar[r] & V_2.
}
\]

Passing to the critical loci of the middle square, we get a commutative diagram
\[
\xymatrix{
\Crit_{U_1^\circ/B}(f_1^\circ) \ar^{\Theta_1}[r] \ar^{\pi_U^\circ}[d] & \Crit_{V_1^\circ/B}(g_1^\circ) \ar^{\pi_V^\circ}[d] \\
\Crit_{U_2^\circ/B}(f_2^\circ) \ar^{\Theta_2}[r] & \Crit_{V_2^\circ/B}(g_2^\circ),
}
\]
where both horizontal morphisms are \'etale and vertical morphisms smooth. Thus,
\[\Theta_1^*\Omega^1_{\Crit_{V_1^\circ/B}(g_1^\circ)/ \Crit_{V_2^\circ/B}(g_2^\circ)}\longrightarrow \Omega^1_{\Crit_{U_1^\circ/B}(f_1^\circ)/\Crit_{U_2^\circ/B}(f_2^\circ)}\]
is an isomorphism. In particular,
\begin{equation}\label{eq:stabilizationzigzagsmoothmorphismpullback}
\Theta_1^*\Omega^1_{V_1^\circ/V_2^\circ}\longrightarrow \Omega^1_{U_1^\circ/U_2^\circ}
\end{equation}
is an isomorphism in a neighborhood of $x_1$. Thus, shrinking $U_1^\circ$ we may assume this morphism is an isomorphism.

Continuing with the rest of the proof of \cref{prop:stabilizationzigzag}, shrinking $U_1^\circ$ and $U_2^\circ$ we obtain a trivial orthogonal bundle $(E_2, q_2)$ over $V_2^\circ$ with a section $s_2\in\Theta_2^* E_2$ such that
\[f_2^\circ - g_2^\circ\circ\Theta_2 = q(s)\]
and such that $\Phi_2\colon (U_2^\circ, f^\circ_2)\rightarrow (W_2=E_2, g_2^\circ\circ\pi_{E_2} + \q_{E_2})$ given by $s_2$ is a critical morphism. Let $\Psi_2\colon (V_2^\circ, g^\circ_2)\rightarrow (W_2, g_2^\circ\circ\pi_{E_2} + \q_{E_2})$ be the zero section which is again a critical morphism.

Let $(E_1, q_1) = (\pi^\circ_V)^* (E_2, q_2)$ and $s_1 = (\pi^\circ_U)^* s_2$. Define $\Phi_1\colon (U_1^\circ, f_1^\circ)\rightarrow (W_1=E_1, g_1^\circ\circ\pi_{E_1} + \q_{E_1})$ using $s_1$ and $\Psi_1 \colon (V_1^\circ, g^\circ_1)\rightarrow (W_1, g_1^\circ\circ\pi_{E_1} + \q_{E_1})$ to be the zero section. Let $\pi_W\colon W_1\rightarrow W_2$ be the obvious projection. By \cref{ex:criticalstabilization} $\Psi_1$ is a critical morphism. Moreover, by construction $\Phi_1$ restricts to $\Theta_1\colon\Crit_{U_1^\circ/B}(f_1^\circ)\rightarrow \Crit_{V_2^\circ/B}(g_2^\circ)$. Thus, to show that $\Phi_1$ is a critical morphism, we have to show that it is unramified. For this, consider the commutative diagrams
\[
\xymatrix{
U_1^\circ \ar^{\Phi_1}[r] \ar^{\pi_U^\circ}[d] & W_1 \ar^{\pi_{E_1}}[r] \ar^{\pi_W}[d] & V^\circ_1 \ar^{\pi_V^\circ}[d] \\
U_2^\circ \ar^{\Phi_2}[r] & W_2 \ar^{\pi_{E_2}}[r] & V^\circ_2.
}
\]
By construction the right square is Cartesian, so
\[\pi_{E_1}^* \Omega^1_{V_1^\circ/V_2^\circ}\longrightarrow \Omega^1_{W_1/W_2}\]
is an isomorphism. Using \eqref{eq:stabilizationzigzagsmoothmorphismpullback} we get that
\[\Phi_1^* \Omega^1_{W_1/W_2}\longrightarrow \Omega^1_{U_1^\circ/U_2^\circ}\]
is an isomorphism. Therefore, $\Phi_1$ factors into a composite of an \'etale morphism $U_1^\circ\rightarrow W_1\times_{W_2} U_2^\circ$ and an unramified morphism $W_1\times_{W_2} U_2^\circ\rightarrow W_1$, hence it is unramified.
\end{proof}

Using \cref{thm:criticallocussmoothdescent} we establish smooth functoriality of the virtual canonical bundle.

\begin{proposition}\label{prop:virtualcanonicalpullback}
Let $(Y\rightarrow B, t)$ be a morphism of schemes equipped with a relative d-critical structure and $\pi\colon X\rightarrow Y$ a smooth morphism. Consider the pullback relative d-critical structure on $X\rightarrow B$. Then there is a natural isomorphism
\[\Upsilon_\pi\colon K^{\vir}_{Y/B}|_{X^{\red}}\otimes K_{X/Y}^{\otimes 2}|_{X^{\red}}\cong K^{\vir}_{X/B}.\]
It satisfies the following properties:
\begin{enumerate}
    \item $\Upsilon_{\id} = \id$.
    \item For a composite $X\xrightarrow{\pi} Y\xrightarrow{\sigma} Z$ of smooth morphisms with a relative d-critical structure on $Z\rightarrow B$ and pullback relative d-critical structures on $X\rightarrow B$ and $Y\rightarrow B$ the diagram
    \begin{equation}\label{eq:virtualcanonicalpullbackassociativity}
    \xymatrix{
    K^{\vir}_{Z/B}|_{X^{\red}}\otimes K_{X/Z}^{\otimes 2}|_{X^{\red}} \ar^-{\Upsilon_{\sigma\circ \pi}}[r] & K^{\vir}_{X/B} \\
    K^{\vir}_{Z/B}|_{X^{\red}}\otimes K_{Y/Z}^{\otimes 2}|_{X^{\red}} \otimes K_{X/Y}^{\otimes 2}|_{X^{\red}} \ar^-{\Upsilon_\sigma\otimes \id}[r] \ar^{\id\otimes i(\Delta)^2}[u] & K^{\vir}_{Y/B}|_{X^{\red}}\otimes K_{X/Y}^{\otimes 2}|_{X^{\red}} \ar^{\Upsilon_\pi}[u]
    & 
    }
    \end{equation}
    commutes, where the left vertical morphism is induced by the short exact sequence
    \[\Delta\colon 0\longrightarrow \pi^*\Omega^1_{Y/Z}\longrightarrow \Omega^1_{X/Z}\longrightarrow \Omega^1_{X/Y}\longrightarrow 0.\]
    \item For a point $x\in X$ the diagram
    \begin{equation}\label{eq:virtualcanonicalpullbackdiagram}
    \xymatrix{
    K^{\vir}_{Y/B, \pi(x)} \otimes K_{X/Y, x}^{\otimes 2} \ar^-{\Upsilon_\pi|_x}[r] \ar^{\kappa_{\pi(x)}\otimes \id}[d] & K^{\vir}_{X/B, x} \ar^{\kappa_x}[d] \\
    \det(\Omega^1_{Y/B, \pi(x)})^{\otimes 2} \otimes K_{X/Y, x}^{\otimes 2} \ar^-{i(\Delta)^2}[r] & \det(\Omega^1_{X/B, x})^{\otimes 2}
    }
    \end{equation}
    commutes, where the bottom isomorphism is induced by the short exact sequence
    \[\Delta\colon 0\longrightarrow \Omega^1_{Y/B, \pi(x)}\longrightarrow \Omega^1_{X/B, x}\longrightarrow \Omega^1_{X/Y, x}\longrightarrow 0.\]
\end{enumerate}
\end{proposition}
\begin{proof}
Let $x\in X$ be a point, $y$ its image in $Y$ and $b$ its image in $B$. Let $X_b$ and $Y_b$ be the fibers of $X\rightarrow B$ and $Y\rightarrow B$ at $b\in B$. We will first construct the isomorphism $\Upsilon_\pi$ in a neighborhood of $x$, possibly depending on additional choices.

By \cref{thm:criticallocussmoothdescent}(1) we may find a smooth morphism $\tilde{\pi}\colon U\rightarrow V$ of smooth $B$-schemes, a function $g\colon V\rightarrow \bA^1$ with $f=\tilde{\pi}^\ast g$ and a commutative diagram
\begin{equation}\label{eq:virtualcanonicalpullbackdata}
\xymatrix{
X \ar^{\pi}[d] & \Crit_{U/B}(f) \ar_-{u}[l] \ar[d] \ar[r] & U \ar^{\tilde{\pi}}[d] \\
Y & \Crit_{V/B}(g) \ar_-{v}[l] \ar[r] & V
}
\end{equation}
with the square on the right Cartesian by \cref{prop:smoothcriticallocus}, such that $(U, f, u)$ defines a critical chart for $X$ near $x$ (so that there is a point $x'\in \Crit_{U/B}(f)$ with $u(x')=x$) and $(V, g, v)$ defines a critical chart for $Y$ near $y$ (so that there is a point $y'\in\Crit_{V/B}(g)$ such that $v(y')=y$). By definition we have canonical isomorphisms
\[v^* K^{\vir}_{Y/B}\cong K_{V/B}^{\otimes 2}|_{\Crit_{V/B}(g)^{\red}},\qquad u^* K^{\vir}_{X/B}\cong K_{U/B}^{\otimes 2}|_{\Crit_{U/B}(f)^{\red}}.\]
Under these isomorphisms we define $\Upsilon_\pi|_{\Crit_{U/B}(f)^{\red}}$ to be 
the composite isomorphism
\[
\xymatrix{
K_{V/B}^{\otimes 2}|_{\Crit_{U/B}(f)^{\red}}\otimes K_{\Crit_{U/B}(f)/\Crit_{V/B}(g)}^{\otimes 2}|_{\Crit_{U/B}(f)^{\red}} \ar^{\sim}[d] \\
K_{V/B}^{\otimes 2}|_{\Crit_{U/B}(f)^{\red}} \otimes K_{U/V}^{\otimes 2}|_{\Crit_{U/B}(f)^{\red}} \ar^{i(\Delta)^2}[d] \\
K_{U/B}^{\otimes 2}|_{\Crit_{U/B}(f)^{\red}}
}
\]
where the first morphism is obtained from the isomorphism $K_{\Crit_{U/B}(f)/\Crit_{V/B}(g)}\cong K_{U/V}|_{\Crit_{U/B}(f)}$ coming from the Cartesian square
\[
\xymatrix{
\Crit_{U/B}(f) \ar[r] \ar[d] & U \ar[d] \\
\Crit_{V/B}(g) \ar[r] & V
}
\]
and the second morphism comes from the short exact sequence
\[\Delta\colon 0\longrightarrow \tilde{\pi}^*\Omega^1_{V/B}\longrightarrow \Omega^1_{U/B}\longrightarrow \Omega^1_{U/V}\longrightarrow 0.\]

The claim that the diagram \eqref{eq:virtualcanonicalpullbackdiagram} commutes reduces to the same claim about d-critical loci $X_b$ and $Y_b$ which is shown in \cite[Proposition 2.30]{JoyceDcrit}.

Now consider two choices $\{(U_1, f_1, u_1), (V_1, g_1, v_1), \tilde{\pi}_1\}$ and $\{(U_2, f_2, u_2), (V_2, g_2, v_2), \tilde{\pi}_2\}$ fitting into a diagram \eqref{eq:virtualcanonicalpullbackdata} and let $\Upsilon^1_\pi$ and $\Upsilon^2_\pi$ be the two local models of the isomorphism $\Upsilon_\pi$ defined using these local data. Then for every point $\tilde{x}\in\Crit_{U_1/B}(f_1)\times_X \Crit_{U_2/B}(f_2)$ with image $x$ in $X$ and $y$ in $Y$ both $\Upsilon^1_\pi|_x$ and $\Upsilon^2_\pi|_x$ fit into the same commutative diagram \eqref{eq:virtualcanonicalpullbackdiagram}. Therefore, they are equal. This proves that
\[\Upsilon^1_\pi|_{\Crit_{U_1/B}(f_1)\times_X \Crit_{U_2/B}(f_2)} = \Upsilon^2_\pi|_{\Crit_{U_1/B}(f_1)\times_X \Crit_{U_2/B}(f_2)}\]
and hence we obtain a global isomorphism $\Upsilon_\pi$ independent of choices.

It is enough to establish both properties of the isomorphism $\Upsilon_\pi$ pointwise, as they involve comparing isomorphisms of line bundles on reduced schemes. Property (1) is immediate from \eqref{eq:virtualcanonicalpullbackdiagram}.

Now consider the setting of property (2). Applying the commutative diagram \eqref{eq:virtualcanonicalpullbackdiagram}, diagram \eqref{eq:virtualcanonicalpullbackassociativity} restricted to $x\in X$ (with $y=\pi(x)\in Y$ and $z=\sigma(y)\in Z$) becomes
\begin{equation}\label{eq:virtualcanonicalpullbackassociativity2}
\xymatrix@C=2.5cm{
\det(\Omega^1_{Z/B, z})^{\otimes 2}\otimes K_{X/Z, x}^{\otimes 2} \ar^{i(\Delta_{X\rightarrow Z\rightarrow B})^2}[r] & \det(\Omega^1_{X/B, x})^{\otimes 2} \\
\det(\Omega^1_{Z/B, z})^{\otimes 2}\otimes K_{Y/Z, y}^{\otimes 2} \otimes K_{X/Y, x}^{\otimes 2} \ar^-{i(\Delta_{Y\rightarrow Z\rightarrow B})^2\otimes \id}[r] \ar^{\id\otimes i(\Delta_{X\rightarrow Y\rightarrow Z})^2}[u] & \det(\Omega^1_{Y/B, y})^{\otimes 2}\otimes K_{X/Y, x}^{\otimes 2}, \ar^{i(\Delta_{X\rightarrow Y\rightarrow B})^2}[u]
}
\end{equation}
where the individual isomorphisms are induced by the following double short exact sequence:
\[
        \begin{aligned}
        \begin{tikzcd}
        &
        \begin{array}{c} \substack{\displaystyle{\phantom{\Delta} } \\ \displaystyle{\phantom{\rotatebox{90}{\mbox{:}}}} \\ \displaystyle{0}} \end{array} &
        \begin{array}{c} \substack{\displaystyle{\ \Delta_{X\rightarrow Y\rightarrow B} } \\ \displaystyle{  \rotatebox{90}{\mbox{:}}} \\  \displaystyle{0}} \end{array} &
        \begin{array}{c} \substack{\displaystyle{\Delta_{X\rightarrow Y\rightarrow Z}} \\ \displaystyle{\rotatebox{90}{\mbox{:}}} \\  \displaystyle{0}}\end{array} \\
	\Delta_{Y\rightarrow Z\rightarrow B} \colon 0 & \Omega^1_{Z/B, z} & \Omega^1_{Y/B, y} & \Omega^1_{Y/Z, y} & 0 \\
	\Delta_{X\rightarrow Z\rightarrow B} \colon 0 & \Omega^1_{Z/B, z} & \Omega^1_{X/B, x} & \Omega^1_{X/Z, x} & 0 \\
	\phantom{\Delta_{X\rightarrow X\rightarrow X}:} 0 & 0 & \Omega^1_{X/Y, x} & \Omega^1_{X/Y, x} & 0 \\
    & 0 & 0 & 0 & 
        \arrow[from=2-1, to=2-2]
        \arrow[from=2-2, to=2-3]
        \arrow[from=2-3, to=2-4]
        \arrow[from=2-4, to=2-5]
        \arrow[from=3-1, to=3-2]
        \arrow[from=3-2, to=3-3]
        \arrow[from=3-3, to=3-4]
        \arrow[from=3-4, to=3-5]
        \arrow[from=4-1, to=4-2]
        \arrow[from=4-2, to=4-3]
        \arrow[equal, from=4-3, to=4-4]
        \arrow[from=4-4, to=4-5]
	\arrow[from=1-2, to=2-2]
        \arrow[equal, from=2-2, to=3-2]
        \arrow[from=3-2, to=4-2]
        \arrow[from=4-2, to=5-2]
        \arrow[from=1-3, to=2-3]
        \arrow[from=2-3, to=3-3]
        \arrow[from=3-3, to=4-3]
        \arrow[from=4-3, to=5-3]
        \arrow[from=1-4, to=2-4]
        \arrow[from=2-4, to=3-4]
        \arrow[from=3-4, to=4-4]
        \arrow[from=4-4, to=5-4]
\end{tikzcd}
\end{aligned}
\]
The commutativity of the diagram \eqref{eq:virtualcanonicalpullbackassociativity2} follows from the corresponding property of determinant lines, see \cite[(2.4)]{KPS}.
\end{proof}

\subsection{Base change of d-critical structures}

Recall that for a commutative diagram of schemes
\[
\begin{tikzcd}
X' \rar \dar & X \dar \\
B' \rar & B
\end{tikzcd}
\]
there is a natural pullback map
\begin{equation}\label{eq:formpullback}
\cA^{2, \ex}(X/B, k)\longrightarrow \cA^{2, \ex}(X'/B', k).
\end{equation}

We have the following base change property of the relative critical locus.

\begin{proposition}\label{prop:criticalbasechange}
Let $(U, f)$ be an LG pair over $B$. Let $B'\rightarrow B$ be a morphism and consider the fiber product
\[
\xymatrix@C=0.3cm@R=0.3cm{
U' \ar[rr] \ar[dd] && U \ar[dd] \\
& \Box & \\
B' \ar[rr] && B
}
\]
Let $f'\colon U'\rightarrow \bA^1$ be the pullback of $f$ to $U'$, so that $(U', f')$ is an LG pair over $B'$. Then there is a natural isomorphism
\[\Crit_{U/B}(f)\times_B B'\cong \Crit_{U'/B'}(f')\]
under which $s_f$ maps to $s_{f'}$.
\end{proposition}
\begin{proof}
Let $R=\Crit_{U/B}(f)$ and $R'=\Crit_{U'/B}(f')$. By definition we have fiber products
\[
\xymatrix{
R \ar[r] \ar[d] & U \ar[d] \\
B \ar[r] & \T^*(U/B),
}\qquad
\xymatrix{
R' \ar[r] \ar[d] & U' \ar[d] \\
B' \ar[r] & \T^*(U'/B').
}
\]
Applying the base change functor $(-)\times_B B'$ to the first square we get a fiber product
\[
\xymatrix{
R\times_B B' \ar[r] \ar[d] & U' \ar[d] \\
B' \ar[r] & \T^*(U/B)\times_B B'.
}
\]
Using $\T^*(U/B)\times_B B'\cong \T^*(U'/B')$, by the universal property of fiber products we obtain an isomorphism $R\times_B B'\cong R'$.

Let us now show that under this isomorphism $s_f$ is sent to $s_{f'}$. The description of the sheaf $\cS_{R/B}$ using the embedding $R\rightarrow U$ given in \cref{prop:Sproperties} is compatible with base change. Thus, the pullback $\cS_{R/B}|_{R'}\rightarrow \cS_{R'/B'}$ fits into a commutative diagram
\[
\xymatrix{
0 \ar[r] & \cS_{R/B}|_{R'} \ar^-{\iota_{R, U}}[r] \ar[d] & \left.\cO_U/\cI_{R, U}^2\right|_{R'} \ar^-{d_B}[r] \ar[d] & \left.\Omega^1_{U/B}/\cI_{R, U}\Omega^1_{U/B}\right|_{R'} \ar[d] \\
0 \ar[r] & \cS_{R'/B'} \ar^-{\iota_{R', U'}}[r] & \cO_{U'}/\cI_{R', U'}^2 \ar^-{d_{B'}}[r] & \Omega^1_{U'/B'}/\cI_{R', U'}\Omega^1_{U'/B'}
}
\]
Thus, it is enough to show that $\iota_{R, U}(s_f)$ is sent to $\iota_{R', U'}(s_{f'})$. But this immediately follows from \cref{prop:Critsdescription}.
\end{proof}

The above proposition allows us to define base change of d-critical structures.

\begin{theorem}\label{thm:dcritbasechange}
Let $X\rightarrow B$ be a morphism of schemes equipped with a section $s\in\Gamma(X, \cS_{X/B})$. Let $B'\rightarrow B$ be a morphism and consider the fiber product
\[
\xymatrix{
X' \ar[r] \ar[d] & X \ar[d] \\
B' \ar[r] & B
}
\]
Denote by $s'\in\Gamma(X', \cS_{X'/B'})$ the pullback of $s\in\Gamma(X, \cS_{X/B})$.
\begin{enumerate}
    \item If $s$ is a relative d-critical structure, then $s'$ is a relative d-critical structure.
    \item If $s'$ is a relative d-critical structure and $B'\rightarrow B$ is an \'etale cover, then $s$ is a relative d-critical structure.
\end{enumerate}
\end{theorem}
\begin{proof}
In the first point $X\rightarrow B$ has a relative d-critical structure, so it is locally of finite type. In the second point $X'\rightarrow B'$ is locally of finite type and $B'\rightarrow B$ is an \'etale cover, so by faithfully flat descent \cite[Tag 02KX]{Stacks} we get that $X\rightarrow B$ is locally of finite type.

Let $x'\in X'$ be a point and $x\in X$ its image in $X$. By \cite[Tag 0CBL]{Stacks} we may find morphisms $X\leftarrow X^\circ\rightarrow U$, where $X^\circ\rightarrow X$ is an open immersion surjective at $x$, $U$ is a smooth $B$-scheme and $X^\circ\rightarrow U$ is a closed immersion minimal at $x$. Let $X'\leftarrow X^{\circ\prime}\rightarrow U'$ be the base change of this correspondence along $B'\rightarrow B$. We obtain a diagram
\[
\xymatrix{
U'\ar[r] & U \\
X^{\circ\prime} \ar[r] \ar[d] \ar[u] & X^\circ \ar[u] \ar[d] \\
X' \ar[r] \ar[d] & X \ar[d] \\
B' \ar[r] & B
}
\]
where all squares are Cartesian. Therefore, $X^{\circ\prime}\rightarrow X'$ is an open immersion surjective at $x'$, $U'$ is a smooth $B'$-scheme and $X^{\circ'}\rightarrow U'$ is minimal at $x'$. Possibly shrinking $X^\circ, U$ we may find a function $f\colon U\rightarrow \bA^1$ such that $\iota_{X^\circ, U} = f\in\cO_U/\cI^2_{X^\circ, U}$. Let $f'$ be the composite $X'\rightarrow X\xrightarrow{f} \bA^1$.

As in the proof of \cref{thm:criticallocussmoothdescent}, $s$ is a relative d-critical structure near $x$ if, and only if, we can shrink $U$ to a neighborhood of $x$ so that $\Crit_{U/B}(f)\times_U X^\circ\rightarrow X^\circ$ becomes an open immersion. The same claim applies to $s'$.

\begin{enumerate}
    \item Assume that $s$ is a relative d-critical structure. Then by the above argument shrinking $U$ we get that $\Crit_{U/B}(f)\times_U X^\circ\rightarrow X^\circ$ is an open immersion. By \cref{prop:criticalbasechange} its base change along $B'\rightarrow B$ is $\Crit_{U'/B'}(f')\times_{U'} X^{\circ\prime}\rightarrow X^{\circ\prime}$ which is therefore also an open immersion. By the above argument we get that $s'$ is a relative d-critical structure in a neighborhood of $x'$. Varying $x'$ we get that $s'$ is a relative d-critical structure.
    \item Assume that $s$ is a relative d-critical structure and $B'\rightarrow B$ is an \'etale cover. By the above argument shrinking $U'$ we get that $\Crit_{U'/B}(f)\times_{U'} X^{\circ\prime}\rightarrow X^{\circ\prime}$ is an open immersion. Since $B'\rightarrow B$ is universally open \cite[Tag 01UA]{Stacks}, $U'\rightarrow U$ is open. Therefore, we may shrink $U$ so that $U'$ is still the base change of $U$ along $B'\rightarrow B$. By the faithfully flat descent for open immersions \cite[Tag 02L3]{Stacks} we get that $\Crit_{U/B}(f)\times_U X^\circ\rightarrow X^\circ$ is an open immersion and hence, by the above argument, that $s$ is a relative d-critical structure in a neighborhood of $x$. Since $X'\rightarrow X$ is surjective, this implies that we may vary $x'\in X'$ to cover $X$, so we get that $s$ is a relative d-critical structure.
\end{enumerate}
\end{proof}

Combining \cref{thm:criticallocussmoothdescent} and \cref{thm:dcritbasechange} we obtain the following.

\begin{corollary}\label{cor:dcritbasechange}
Consider a diagram of schemes
\[
\xymatrix{
X' \ar[r] \ar[d] & X \ar[d] \\
B' \ar[r] & B
}
\]
where $X'\rightarrow X\times_B B'$ is smooth, $s\in\Gamma(X, \cS_{X/B})$ and let $s'\in\Gamma(X', \cS_{X'/B'})$ be the pullback.
\begin{enumerate}
    \item If $s$ is a relative d-critical structure, then $s'$ is a relative d-critical structure.
    \item If $s'$ is a relative d-critical structure, $B'\rightarrow B$ is an \'etale cover and $X'\rightarrow X\times_B B'$ is surjective, then $s$ is a relative d-critical structure.
\end{enumerate}
\end{corollary}

Using \cref{thm:dcritbasechange} we can establish the following compatibility of the virtual canonical bundle with respect to base change.

\begin{proposition}\label{prop:virtualcanonicalbasechange}
Let $X\rightarrow B$ be a morphism of schemes equipped with a relative d-critical structure $s\in\Gamma(X, \cS_{X/B})$. Let $B'\rightarrow B$ be a morphism of schemes and consider the fiber product
\[
\xymatrix{
X' \ar[r] \ar[d] & X \ar[d] \\
B' \ar[r] & B
}
\]
Denote by $s'\in\Gamma(X', \cS_{X'/B'})$ the pullback of $s\in\Gamma(X, \cS_{X/B})$. Then there is an isomorphism
\[\Upsilon\colon K^{\vir}_{X/B}|_{(X')^{\red}}\xrightarrow{\sim} K^{\vir}_{X'/B'}\]
compatible with compositions of Cartesian squares. Moreover, for a point $x'\in X'$ and its image $x\in X$ the diagram
\begin{equation}\label{eq:Upsilonkappapullback}
\xymatrix{
K^{\vir}_{X/B, x} \ar^{\Upsilon|_x}[r] \ar^{\kappa_x}[d] & K^{\vir}_{X'/B', x'} \ar^{\kappa_{x'}}[d] \\
\det(\Omega^1_{X/B, x})^{\otimes 2} \ar^{\sim}[r] & \det(\Omega^1_{X'/B', x'})^{\otimes 2}
}
\end{equation}
commutes.
\end{proposition}
\begin{proof}
A cover of $X$ by critical charts $(U, f, u)$ defines a cover of $X'$ by critical charts $(U'=U\times_B B', f'=f|_{U'}, u'=u\times \id)$. Thus, it is enough to construct the isomorphism $\Upsilon$ for each critical chart $(U, f, u)$ and show that the diagram \eqref{eq:Upsilonkappapullback} commutes. The latter condition ensures that the locally defined isomorphisms glue into a global isomorphism $\Upsilon$ .

Let $(U, f, u)$ be a critical chart for $X$ and $(U', f', u')$ the corresponding critical chart for $X'$. By definition we have canonical isomorphisms
\[u^* K^{\vir}_{X/B}\cong K_{U/B}^{\otimes 2}|_{\Crit_{U/B}(f)^{\red}},\qquad (u')^* K^{\vir}_{X'/B'}\cong K_{U'/B'}^{\otimes 2}|_{\Crit_{U'/B'}(f')^{\red}}.\]
We define $\Upsilon|_{\Crit_{U'/B'}(f')^{\red}}$ to be the isomorphism
\[K_{U/B}^{\otimes 2}|_{\Crit_{U'/B'}(f')^{\red}}\cong K_{U'/B'}^{\otimes 2}|_{\Crit_{U'/B'}(f')^{\red}}\]
coming from the bottom Cartesian square in
\[
\xymatrix{
X' \ar[r] \ar[d] & X \ar[d] \\
U' \ar[r] \ar[d] & U \ar[d] \\
B' \ar[r] & B.
}
\]

The above diagram of Cartesian squares defines an isomorphism of short exact sequences
\[
\xymatrix{
0 \ar[r] & \rN^\vee_{X/U, x} \ar[r] \ar^{\sim}[d] & \Omega^1_{U/B, x} \ar[r] \ar^{\sim}[d] & \Omega^1_{X/B, x} \ar[r] \ar^{\sim}[d] & 0 \\
0 \ar[r] & \rN^\vee_{X'/U', x'} \ar[r] & \Omega^1_{U'/B', x'} \ar[r] & \Omega^1_{X'/B', x'} \ar[r] & 0
}
\]
The Hessian quadratic form on $\Omega^1_{U/B, x}$ restricts to the Hessian quadratic form on $\Omega^1_{U'/B', x'}$. Thus, using the description of $\kappa_x$ from \cref{thm:virtualcanonicalschemes}(3) we obtain a commutative diagram
\[
\xymatrix{
K_{U/B, x}^{\otimes 2} \ar^{\Upsilon|_x}[r] \ar^{\kappa_x}[d] & K_{U'/B', x'} \ar^{\kappa_{x'}}[d] \\
\det(\Omega^1_{X/B, x})^{\otimes 2} \ar^{\sim}[r] & \det(\Omega^1_{X'/B', x'})^{\otimes 2}.
}
\]

For a critical morphism $\Phi\colon (U, f, u)\rightarrow (V, g, v)$ its base change along $B'\rightarrow B$ defines a crtical morphism $\Phi'\colon (U', f', u')\rightarrow (V', g', v')$. Moreover, under the natural isomorphism $\rN_{U/V}|_{U'}\cong \rN_{U'/V'}$ the quadratic form $q_\Phi$ restricts to the quadratic form $q_{\Phi'}$. Thus, using the description of the gluing isomorphisms $J_\Phi$ for the virtual canonical bundle from \cref{thm:virtualcanonicalschemes}(2) we get that $\Upsilon^V|_{\Crit_{U'/B'}(f')^{\red}} = \Upsilon^U$. In particular, there is a global isomorphism $K^{\vir}_{X/B}|_{(X')^{\red}}\cong K^{\vir}_{X'/B'}$ which restricts to $\Upsilon^U$ in each critical chart.
\end{proof}

Combining \cref{prop:virtualcanonicalpullback} and \cref{prop:virtualcanonicalbasechange} we obtain the following.

\begin{corollary}\label{cor:virtualcanonicalbasechange}
Consider a diagram of schemes
\[
\xymatrix{
X' \ar^{\overline{p}}[r] \ar[d] & X \ar[d] \\
B' \ar^{p}[r] & B,
}
\]
where $X'\rightarrow X'\times_B X$ is smooth, and a relative d-critical structure $s\in\Gamma(X, \cS_{X/B})$. Let $s'\in\Gamma(X', \cS_{X'/B'})$ be the pullback. Then there is an isomorphism
\[\Upsilon_{X'\rightarrow X}\colon K^{\vir}_{X/B}|_{(X')^{\red}}\otimes K_{X'/X\times_B B'}^{\otimes 2}|_{(X')^{\red}}\xrightarrow{\sim} K^{\vir}_{X'/B'}.\]
It satisfies the following properties:
\begin{enumerate}
    \item For the diagram
    \[
    \xymatrix{
    X \ar@{=}[r] \ar[d] & X \ar[d] \\
    B \ar@{=}[r] & B
    }
    \]
    we have $\Upsilon_{X\rightarrow X}=\id$.
    \item For a commutative diagram
    \[
    \xymatrix{
    X'' \ar[r] \ar[d] & X' \ar[r] \ar[d] & X \ar[d] \\
    B'' \ar[r] & B' \ar[r] & B,
    }
    \]
    with $X''\rightarrow X'\times_{B'} B''$ and $X'\rightarrow X\times_B B'$ smooth, a relative d-critical structure $s\in\Gamma(X, \cS_{X/B})$ and $s'$, $s''$ its pullbacks to $X'\rightarrow B'$ and $X''\rightarrow B''$, the diagram
    \[
    \xymatrix{
    K^{\vir}_{X/B}|_{(X'')^{\red}}\otimes K^{\otimes 2}_{X'/X\times_B B'}|_{(X'')^{\red}}\otimes K^{\otimes 2}_{X''/X'\times_{B'} B''}|_{(X'')^{\red}} \ar^{\Upsilon_{X'\rightarrow X}\otimes \id}[d] \ar^-{\id\otimes i(\Delta)^2}[r] & K^{\vir}_{X/B}|_{(X'')^{\red}}\otimes K^{\otimes 2}_{X''/X\times_B B''}|_{(X'')^{\red}} \ar^{\Upsilon_{X''\rightarrow X}}[d] \\
    K^{\vir}_{X'/B'}|_{(X'')^{\red}}\otimes K^{\otimes 2}_{X''/X'\times_{B'} B''}|_{(X'')^{\red}} \ar^-{\Upsilon_{X''\rightarrow X'}}[r] & K^{\vir}_{X''/B''}
    }
    \]
    commutes, where the top horizontal isomorphism is induced by the short exact sequence
    \[
    \Delta\colon 0\longrightarrow \Omega^1_{X'/X\times_B B'}|_{X''} \longrightarrow \Omega^1_{X''/X\times_B B''}\longrightarrow \Omega^1_{X''/X'\times_{B'} B''}\longrightarrow 0.
    \]
    \item For a point $x\in X'$ the diagram
    \[
    \xymatrix{
    K^{\vir}_{X/B, \overline{p}(x')}\otimes K^{\otimes 2}_{X'/X\times_B B', x'} \ar^-{\Upsilon_{X'\rightarrow X}|_{x'}}[r] \ar^{\kappa_{\overline{p}(x')}\otimes \id}[d] & K^{\vir}_{X'/B', x'} \ar^{\kappa_{x'}}[d] \\
    \det(\Omega^1_{X/B, \overline{p}(x')})^{\otimes 2}\otimes K^{\otimes 2}_{X'/X\times_B B', x'} \ar^-{i(\Delta)^2}[r] & \det(\Omega^1_{X'/B', x'})^{\otimes 2}
    }
    \]
    commutes, where the bottom horizontal isomorphism is induced by the short exact sequence
    \[\Delta\colon 0\longrightarrow \Omega^1_{X/B, x}\longrightarrow \Omega^1_{X'/B', x'}\longrightarrow \Omega^1_{X'/X\times_B B', x'}\longrightarrow 0.\]
\end{enumerate}
\end{corollary}

\subsection{D-critical structures on stacks}

We now extend the notion of relative d-critical structures to higher Artin stacks. Let $X\rightarrow B$ be a geometric morphism of stacks. In \cref{sect:differentialforms} we have defined the space $\cA^{2, \ex}(X/B, -1)$ of relative exact two-forms of degree $-1$. In particular, we again have a sheaf on the big \'etale site of $X$ denoted by $\cS_{X/B}$ so that
\[\Gamma(X, \cS_{X/B}) = \pi_0(\cA^{2, \ex}(X/B, -1)).\]

For a morphism of schemes $X\rightarrow B$ let $\DCrit(X/B)\subset \Gamma(X, \cS_{X/B})$ be the set of relative d-critical structures. By \cref{cor:dcritbasechange} we see that the collection of relative d-critical structures defines a functor
\[\DCrit\colon \Fun(\Delta^1, \Sch)^{\op}_{0\smooth}\longrightarrow \Set\]
satisfying \'etale descent. Using \eqref{eq:ShvStkrel} we obtain the notion of a relative d-critical structure on a geometric morphism of stacks as follows.

\begin{definition}\label{def:dcriticalstack}
Let $X\rightarrow B$ be a geometric morphism. A \defterm{relative d-critical structure on $X\rightarrow B$} is a section $s\in\Gamma(X, \cS_{X/B})$ which satisfies the following property: for every diagram
\[
\xymatrix{
X' \ar[r] \ar[d] & X \ar[d] \\
B' \ar[r] & B, \\
}
\]
where $X'\rightarrow B'$ is a morphism of schemes, $X'\rightarrow X\times_B B'$ is smooth and $s'\in\Gamma(X', \cS_{X'/B'})$ is the pullback of $s$, we have that $s'$ is a relative d-critical structure.
\end{definition}

\begin{remark}
If $X\rightarrow B$ is a morphism of schemes, \cref{cor:dcritbasechange} ensures that \cref{def:dcriticalstack} is compatible with \cref{def:dcriticalscheme}.    
\end{remark}

Using \eqref{eq:ShvStkrel} we get an extension of \cref{cor:dcritbasechange} to geometric morphisms.

\begin{proposition}\label{prop:dcritbasechangestacks}
Consider a commutative diagram of stacks
\[
\xymatrix{
X' \ar[r] \ar[d] & X \ar[d] \\
B' \ar[r] & B,
}
\]
where $X\rightarrow B$ is geometric and $X'\rightarrow X\times_B B'$ is smooth (in particular, $X'\rightarrow B'$ is geometric). Consider $s\in\Gamma(X, \cS_{X/B})$ and let $s'\in\Gamma(X', \cS_{X'/B'})$ be the pullback.
\begin{enumerate}
    \item If $s$ is a relative d-critical structure, then $s'$ is a relative d-critical structure.
    \item If $s'$ is a relative d-critical structure, $B'\rightarrow B$ is an \'etale cover and $X'\rightarrow X\times_B B'$ is surjective, then $s$ is a relative d-critical structure.
\end{enumerate}
\end{proposition}

We will now define the virtual canonical bundle associated to a relative d-critical structure on stacks.

\begin{proposition}\label{prop:virtualcanonicalstacks}
Let $X\rightarrow B$ be a geometric morphism of stacks equipped with a relative d-critical structure $s$. There is a line bundle $K^{\vir}_{X/B}$ on $X^{\red}$, the \defterm{virtual canonical bundle}, uniquely determined by the following conditions:
\begin{enumerate}
    \item If $X\rightarrow B$ is a morphism of schemes, then $K^{\vir}_{X/B}$ coincides with the virtual canonical bundle defined in \cref{thm:virtualcanonicalschemes}.
    \item For every commutative diagram of stacks
    \[
    \xymatrix{
    X' \ar^{\overline{p}}[r] \ar^{\pi'}[d] & X \ar^{\pi}[d] \\
    B' \ar^{p}[r] & B
    }
    \]
    with $\pi, \pi'$ geometric and $X'\rightarrow X\times_B B'$ a smooth morphism (so that $X'\rightarrow B'$ has a pullback relative d-critical structure by \cref{cor:dcritbasechange}), we have a canonical isomorphism
    \[\Upsilon_{X'\rightarrow X}\colon K^{\vir}_{X/B}|_{(X')^{\red}}\otimes K_{X'/X\times_B B'}^{\otimes 2}|_{(X')^{\red}}\xrightarrow{\sim} K^{\vir}_{X'/B'}.\]
    \item For a commutative diagram of stacks
    \[
    \xymatrix{
    X'' \ar^{\overline{q}}[r] \ar^{\pi''}[d] & X' \ar^{\overline{p}}[r] \ar^{\pi'}[d] & X \ar^{\pi}[d] \\
    B'' \ar^{q}[r] & B' \ar^{p}[r] & B
    }
    \]
    with $\pi'', \pi', \pi$ geometric morphisms, $X''\rightarrow X'\times_{B'} B''$ and $X'\rightarrow X\times_B B'$ smooth, a relative d-critical structure $s\in\Gamma(X, \cS_{X/B})$ and $s'$, $s''$ its pullbacks to $X'\rightarrow B'$ and $X''\rightarrow B''$, the diagram
    \[
    \xymatrix{
    K^{\vir}_{X/B}|_{(X'')^{\red}}\otimes K^{\otimes 2}_{X'/X\times_B B'}|_{(X'')^{\red}}\otimes K^{\otimes 2}_{X''/X'\times_{B'} B''}|_{(X'')^{\red}} \ar^{\Upsilon_{X'\rightarrow X}\otimes \id}[d] \ar^-{\id\otimes i(\Delta)^2}[r] & K^{\vir}_{X/B}|_{(X'')^{\red}}\otimes K^{\otimes 2}_{X''/X\times_B B''}|_{(X'')^{\red}} \ar^{\Upsilon_{X''\rightarrow X}}[d] \\
    K^{\vir}_{X'/B'}|_{(X'')^{\red}}\otimes K^{\otimes 2}_{X''/X'\times_{B'} B''}|_{(X'')^{\red}} \ar^-{\Upsilon_{X''\rightarrow X'}}[r] & K^{\vir}_{X''/B''}
    }
    \]
    commutes, where the top horizontal isomorphism is induced by the fiber sequence
    \[
    \Delta\colon \bL_{X'/X\times_B B'}|_{X''} \longrightarrow \bL_{X''/X\times_B B''}\longrightarrow \bL_{X''/X'\times_{B'} B''}.
    \]
\end{enumerate}
For every point $x\in X$ there is an isomorphism
\[\kappa_x\colon K^{\vir}_{X/B, x}\xrightarrow{\sim} \det(\tau^{\geq 0}\bL_{X/B, x})^{\otimes 2}\]
which coincides with $\kappa_x$ defined in \cref{thm:virtualcanonicalschemes} for $X\rightarrow B$ a morphism of schemes and which satisfies the following property:
\begin{enumerate}[resume]
    \item Let
    \[
    \xymatrix{
    X' \ar^{\overline{p}}[r] \ar^{\pi'}[d] & X \ar^{\pi}[d] \\
    B' \ar^{p}[r] & B
    }
    \]
    be a diagram of stacks as in (2). For a point $x'\in X'$  the diagram
    \[
    \xymatrix{
    K^{\vir}_{X/B, \overline{p}(x')}\otimes K_{X'/X\times_B B', x'}^{\otimes 2} \ar^-{\Upsilon_{X'\rightarrow X}|_{x'}}[r] \ar^{\kappa_{\overline{p}(x')}\otimes \id}[d] & K^{\vir}_{X'/B', x'} \ar^{\kappa_{x'}}[d] \\
    \det(\tau^{\geq 0}\bL_{X/B, \overline{p}(x')})^{\otimes 2}\otimes K_{X'/X\times_B B', x'}^{\otimes 2} \ar^-{i(\Delta_{x'\rightarrow x})^2}[r] & \det(\tau^{\geq 0} \bL_{X'/B', x'})^{\otimes 2}
    }
    \]
    commutes, where the bottom horizontal isomorphism is induced by the fiber sequence
    \[\Delta_{x'\rightarrow x}\colon \tau^{\geq 0}\bL_{X/B, \overline{p}(x')}\longrightarrow \tau^{\geq 0} \bL_{X'/B', x'}\longrightarrow \bL_{X'/X\times_B B', x'}.\]
\end{enumerate}
In addition, the virtual canonical bundle satisfies the following properties:
\begin{enumerate}[resume]
    \item For a pair $X_1\rightarrow B_1$, $X_2\rightarrow B_2$ of geometric morphisms of stacks equipped with relative d-critical structures there is an isomorphism
    \begin{equation}\label{eq:Kproductstacks}
    K^{\vir}_{X_1/B_1}\boxtimes K^{\vir}_{X_2/B_2}\cong K^{\vir}_{X_1\times X_2/B_1\times B_2},
    \end{equation}
    which is unital, commutative and associative and such that the isomorphism $\Upsilon_{X'\rightarrow X}$ from (2) is compatible with products. Moreover, \eqref{eq:Kproductstacks} is uniquely determined by the condition that for every point $(x_1, x_2)\in X_1\times X_2$ there is a commutative diagram
    \[
    \xymatrix{
    K^{\vir}_{X_1/B_1, x_1}\otimes K^{\vir}_{X_2/B_2, x_2} \ar^{\eqref{eq:Kproductstacks}}[r] \ar^{\kappa_{x_1}\otimes \kappa_{x_2}}[d] & K^{\vir}_{X_1\times X_2/B_1\times B_2, (x_1, x_2)} \ar^{\kappa_{(x_1, x_2)}}[d] \\
    \det(\tau^{\geq 0}\bL_{X_1/B_1, x_1})^{\otimes 2}\otimes \det(\tau^{\geq 0}\bL_{X_2/B_2, x_2})^{\otimes 2} \ar^-{\sim}[r] & \det(\tau^{\geq 0}\bL_{X_1\times X_2/B_1\times B_2, (x_1, x_2)})^{\otimes 2}.
    }
    \]
    \item For a locally constant function $d\colon X\rightarrow\Z/2\Z$ there is an isomorphism
    \begin{equation}\label{eq:Kreversestacks}
    R_d\colon K^{\vir}_{X/B, s}\cong K^{\vir}_{X/B, -s}
    \end{equation}
    squaring to the identity. For a pair $X_1\rightarrow B_1$, $X_2\rightarrow B_2$ of geometric morphisms of stacks equipped with relative d-critical structures $s_1,s_2$ the diagram
    \[
    \xymatrix{
    K^{\vir}_{X_1/B_1, s_1}\boxtimes K^{\vir}_{X_2/B_2, s_2}\ar^{\eqref{eq:Kproductstacks}}[r] \ar^{R_{d_1}\boxtimes R_{d_2}}[d] & K^{\vir}_{X_1\times X_2/B_1\times B_2, s_1\boxplus s_2} \ar^{R_{d_1\boxplus d_2}}[d] \\
    K^{\vir}_{X_1/B_1, -s_1}\boxtimes K^{\vir}_{X_2/B_2, -s_2}\ar^{\eqref{eq:Kproductstacks}}[r] & K^{\vir}_{X_1\times X_2/B_1\times B_2, -(s_1\boxplus s_2)}
    }
    \]
    commutes. For a commutative diagram of stacks as in (2) the diagram
    \[
    \xymatrix{
    K^{\vir}_{X/B, s}|_{(X')^{\red}}\otimes K^{\otimes 2}_{X'/X\times_B B'} \ar^-{\Upsilon_{X'\rightarrow X}}[r] \ar^{R_d\otimes \id}[d] & K^{\vir}_{X'/B', s'} \ar^{R_{d+\dim(X'/X\times_B B')}}[d] \\
    K^{\vir}_{X/B, -s}|_{(X')^{\red}}\otimes K^{\otimes 2}_{X'/X\times_B B'} \ar^-{\Upsilon_{X'\rightarrow X}}[r] & K^{\vir}_{X'/B', -s'}
    }
    \]
    commutes.
\end{enumerate}
\end{proposition}
\begin{proof}
Consider the functor
\[\Pic_{K^2}\colon \Fun(\Delta^1, \Stk)^{\geometric,\op}_{0\smooth}\longrightarrow \Gpd\]
defined as follows:
\begin{itemize}
    \item For a geometric morphism of stacks $X\rightarrow B$ we assign the groupoid $\Pic(X^{\red})$ of line bundles on $X^{\red}$.
    \item For a commutative diagram
    \[
    \xymatrix{
    X' \ar^{\overline{p}}[r] \ar^{\pi'}[d] & X \ar^{\pi}[d] \\
    B' \ar^{p}[r] & B
    }
    \]
    with $X\rightarrow B$ geometric and $X'\rightarrow X\times_B B'$ smooth, we assign the functor $\Pic(X^{\red})\rightarrow \Pic((X')^{\red})$ which sends a line bundle $L$ on $X^{\red}$ to $\overline{p}^* L\otimes K^{\otimes 2}_{X'/X\times_B B'}|_{(X')^{\red}}$.
    \item For a commutative diagram
    \[
    \xymatrix{
    X'' \ar^{\overline{q}}[r] \ar^{\pi''}[d] & X' \ar^{\overline{p}}[r] \ar^{\pi'}[d] & X \ar^{\pi}[d] \\
    B'' \ar^{q}[r] & B' \ar^{p}[r] & B
    }
    \]
    with $X\rightarrow B$ geometric, $X'\rightarrow X\times_B B'$ and $X''\rightarrow X'\times_{B'} B''$ smooth, to a natural isomorphism between the composite \[L\mapsto \overline{p}^* L\otimes K^{\otimes 2}_{X'/X\times_B B'}|_{(X')^{\red}}\mapsto \overline{q}^* \overline{p}^* L\otimes K^{\otimes 2}_{X'/X\times_B B'}|_{(X'')^{\red}}\otimes K^{\otimes 2}_{X''/X'\times_{B'} B''}|_{(X'')^{\red}}\]
    and $L\mapsto \overline{q}^*\overline{p}^* L\otimes K^{\otimes 2}_{X''/X\times_B B''}|_{(X'')^{\red}}$ given by $\id\otimes i(\Delta)^2$, where
    \[\Delta\colon \bL_{X'/X\times_B B'}|_{X''} \longrightarrow \bL_{X''/X\times_B B''}\longrightarrow \bL_{X''/X'\times_{B'} B''}.\]
\end{itemize}

Consider also the functor
\[\DCrit\colon \Fun(\Delta^1, \Stk)^{\geometric,\op}_{0\smooth}\longrightarrow \Set\longrightarrow \Gpd\]
given by sending a geometric morphism $X\rightarrow B$ to the set of relative d-critical structures.

The restrictions of the functors $\Pic_{K^2}$ and $\DCrit$ to morphisms of schemes satisfies \'etale descent (\'etale descent for line bundles in the case of $\Pic_{K^2}$ and \'etale descent for relative d-critical structures proven in \cref{thm:dcritbasechange}(2) for $\DCrit$); therefore, using \eqref{eq:ShvStkrel} both $\Pic_{K^2}$ and $\DCrit$ extend to functors on $\Fun(\Delta^1, \Stk)^{\geometric}_{0\smooth}$.

By \cref{cor:virtualcanonicalbasechange} we have a natural transformation
\[
\begin{tikzcd}[column sep=2cm]
\Fun(\Delta^1, \Sch)^{\op}_{0\smooth} \ar[r, bend left=40, ""{name=U, below}, "\DCrit"{above}] \ar[r, bend right=40, ""{name=D, above}, "\Pic_{K^2}"{below}] & \Gpd \ar[Rightarrow, from=U, to=D, "K^{\vir}" description]
\end{tikzcd}
\]
defined as follows:
\begin{itemize}
    \item For a morphism of schemes $X\rightarrow B$ equipped with a relative d-critical structure we assign the virtual canonical bundle $K^{\vir}_{X/B}$ on $X^{\red}$.
    \item For a commutative diagram
    \[
    \xymatrix{
    X' \ar^{\overline{p}}[r] \ar^{\pi'}[d] & X \ar^{\pi}[d] \\
    B' \ar^{p}[r] & B
    }
    \]
    of schemes with $X'\rightarrow X\times_B B'$ smooth, a relative d-critical structure on $X\rightarrow B$ and its pullback relative d-critical structure on $X'\rightarrow B'$ we assign the isomorphism
    \[\Upsilon_{X'\rightarrow X}\colon K^{\vir}_{X/B}|_{(X')^{\red}}\otimes K^{\otimes 2}_{X'/X\times_B B'}|_{(X')^{\red}}\xrightarrow{\sim} K^{\vir}_{X'/B'}\]
\end{itemize}

Using \eqref{eq:ShvStkrel} with $\cV=\Fun(\Delta^1, \Gpd)$ we see that the natural transformation $K^{\vir}\colon \DCrit\rightarrow \Pic_{K^2}$ extends from morphisms of schemes to geometric morphisms of stacks.

For a commutative diagram of stacks as in (3) together with a point $x''\in X''$ with $x'=\overline{q}(x'')$ and $x=\overline{p}(x')$ the diagram
\[
\xymatrix{
\det(\tau^{\geq 0}\bL_{X/B, x})\otimes K^{\otimes 2}_{X'/X\times_B B', x'}\otimes K^{\otimes 2}_{X''/X'\times_{B'} B'', x''} \ar^{i(\Delta_{x'\rightarrow x})^2\otimes \id}[d] \ar^-{\id\otimes i(\Delta)^2}[r] & \det(\tau^{\geq 0}\bL_{X/B, x})\otimes K^{\otimes 2}_{X''/X\times_B B'', x''} \ar^{i(\Delta_{x''\rightarrow x})^2}[d] \\
\det(\tau^{\geq 0}\bL_{X'/B', x'})\otimes K^{\otimes 2}_{X''/X'\times_{B'} B'', x''} \ar^-{i(\Delta_{x''\rightarrow x'})^2}[r] & \det(\tau^{\geq 0}\bL_{X''/B'', x''})
}
\]
commutes, as can be seen by applying \cite[(2.4)]{KPS} to the double fiber sequence
\[
        \begin{aligned}
        \begin{tikzcd}
        \begin{array}{c} \substack{\displaystyle{\phantom{\Delta} } \\ \displaystyle{\phantom{\rotatebox{90}{\mbox{:}}}} \\ \displaystyle{\Delta_{x'\rightarrow x} \colon \tau^{\geq 0}\bL_{X/B, x}}} \end{array} &
        \begin{array}{c} \substack{\displaystyle{\Delta_{x''\rightarrow x'}} \\ \displaystyle{  \rotatebox{90}{\mbox{:}}} \\  \displaystyle{\tau^{\geq 0}\bL_{X'/B', x'}}} \end{array} &
        \begin{array}{c} \substack{\displaystyle{\Delta} \\ \displaystyle{\rotatebox{90}{\mbox{:}}} \\  \displaystyle{\bL_{X'/X\times_B B', x'}}}\end{array} \\
	\Delta_{x''\rightarrow x} \colon \tau^{\geq 0}\bL_{X/B, x} & \tau^{\geq 0}\bL_{X''/B'', x''} & \bL_{X''/X\times_B B'', x''} \\
	\phantom{\Delta_{x\rightarrow x}:} 0 & \bL_{X''/X'\times_{B'} B'', x''} & \bL_{X''/X'\times_{B'} B'', x''}.
        \arrow[shift right=3, from=1-1, to=1-2, start anchor={[xshift=-1ex]}, end anchor={[xshift=1ex]}]
        \arrow[shift right=3, from=1-2, to=1-3, start anchor={[xshift=-1ex]}, end anchor={[xshift=1ex]}]
        \arrow[from=2-1, to=2-2]
        \arrow[from=2-2, to=2-3]
        \arrow[from=3-1, to=3-2]
        \arrow[equal, from=3-2, to=3-3]
        \arrow[shift left=6, equal, from=1-1, to=2-1]
        \arrow[shift left=6, from=2-1, to=3-1]
        \arrow[from=1-2, to=2-2]
        \arrow[from=2-2, to=3-2]
        \arrow[from=1-3, to=2-3]
        \arrow[from=2-3, to=3-3]
\end{tikzcd}
\end{aligned}
\]

Using this fact and property (3) of the isomorphism $\Upsilon$, the construction of the isomorphism $\kappa_x$ reduces to the construction of $\kappa_x$ for schemes which is compatible with $\Upsilon$ for smooth morphisms of schemes by \cref{thm:virtualcanonicalschemes}(3).

The isomorphism \eqref{eq:Kproductstacks} is established for morphisms of schemes in \cref{thm:virtualcanonicalschemes}(4). By construction of the virtual canonical bundle of stacks to show that the isomorphism extends to stacks and $\Upsilon_{X'\rightarrow X}$ is compatible with products, it is enough to establish this claim for a pair
\[
\xymatrix{
X_1' \ar[r] \ar[d] & X_1 \ar[d] \\
B_1' \ar[r] & B_1
}\qquad
\xymatrix{
X_2' \ar[r] \ar[d] & X_2 \ar[d] \\
B_2' \ar[r] & B_2
}
\]
of commutative diagrams of schemes as in property (2). It is also enough to establish compatibility for every point $(x'_1, x'_2)\in X'_1\times X'_2$. Using property (4) the compatibility reduces to the fact that the isomorphisms $i(\Delta_{x'\rightarrow x})^2$ are compatible with direct sums.

The isomorphism $R_d$ is constructed for morphisms of schemes in \cref{thm:virtualcanonicalschemes}(5). The compatibility of $R_d$ with products is obvious. The compatibility with the isomorphism $\Upsilon_{X'\rightarrow X}$ when $X'\rightarrow X\times_B B'$ is an isomorphism is also obvious. By construction of $\Upsilon_{X'\rightarrow X}$ for a general smooth morphism $X'\rightarrow X\times_B B'$ it is thus enough to consider the case $B'=B$ and $X'\rightarrow X$ smooth. Moreover, this compatibility can be checked on critical charts, so by \cref{thm:criticallocussmoothdescent}(1) we may assume that there is a smooth morphism $(U, f, u)\rightarrow (V, g, v)$ of critical charts for $X'\rightarrow X$ locally near a point in $X'$. In this case $\Upsilon_{X'\rightarrow X}$ is determined by the natural isomorphism $(U\rightarrow V)^*K_{V/B}^{\otimes 2}\otimes K_{U/V}^{\otimes 2}\cong K_{U/B}^{\otimes 2}$. The isomorphism $R_d$ acts by $(-1)^{\dim(V/B)+d}$ on $K_{V/B}^{\otimes 2}$. The isomorphism $R_{d + \dim(X'/X)}$ acts by $(-1)^{\dim(U/B)+d+\dim(X'/X)}$ on $K_{U/B}^{\otimes 2}$. Since $\dim(U/V)=\dim(X'/X)$, the claim follows.
\end{proof}

\subsection{Pushforwards of d-critical structures}\label{sect:dcriticalpushforward}

In addition to the pullback functoriality of differential forms given by \eqref{eq:formpullback}, we will also consider a ``pushforward'' functoriality \cite{ParkSymplectic} given as follows. Consider geometric morphisms $X\xrightarrow{\pi} B\xrightarrow{p} S$ of stacks. Using the fiber sequence
\[\pi^*\bL_{B/S}\longrightarrow \bL_{X/S}\longrightarrow \bL_{X/B}\]
we obtain a morphism
\[\mu_{X\rightarrow B\rightarrow S}\colon \cS_{X/B}\longrightarrow \pi^*\bL_{B/S}.\]
In particular, if $X$ is equipped with a section $s\in\Gamma(X, \cS_{X/B})$, there is a canonical induced morphism, the \defterm{moment map}
\[\mu_s\colon X\longrightarrow \T^*(B/S),\]
which comes with a homotopy
\begin{equation}\label{eq:mushomotopy}
h_{s, p}\colon \mu_s^\ast \lambda_{B/S} \sim d_S \und(s)
\end{equation}
in $\cA^1(X/S, 0)$.

\begin{definition}\label{def:dcritpushforwardscheme}
Let $X\rightarrow B\xrightarrow{p} S$ be geometric morphisms of stacks together with a section $s\in\Gamma(X, \cS_{X/B})$. The \defterm{d-critical pushforward} $p_*(X, s)$ is the fiber product
\[
\xymatrix{
p_* (X, s) \ar[r] \ar[d] & X \ar^{\mu_s}[d] \\
B \ar^-{0}[r] & \T^*(B/S).
}
\]
\end{definition}

The d-critical pushforward $R=p_*(X, s)$ comes equipped with a canonical section $p_*s\in\Gamma(R, \cS_{R/S})$ obtained as follows. Consider the following commutative diagram:
\[
\xymatrix{
\cA^1(\T^*(B/S)/S, 0) \ar^-{0^\ast}[r] \ar^{\mu_s^\ast}[d] & \cA^1(B/S, 0) \ar^{(R\rightarrow B)^\ast}[d] \\
\cA^1(X/S, 0) \ar^{(R\rightarrow X)^\ast}[r] & \cA^1(R/S, 0) \\
\cA^0(X, 0) \ar^{(R\rightarrow X)^\ast}[r] \ar_{d_S}[u] & \cA^0(R, 0) \ar_{d_S}[u]
}
\]
Consider the function $f=\und(s)\in\cA^0(X, 0)$. Using the bottom commutative square we get a homotopy $d_S(f|_R)\sim (R\rightarrow X)^\ast (d_S f)$ in $\cA^1(R/S, 0)$. Using the homotopy $h_{s, p}\colon d_S f\sim \mu_s^\ast \lambda_{B/S}$ given by \eqref{eq:mushomotopy} as well as the top commutative square we get a homotopy $(R\rightarrow X)^\ast(d_S f)\sim (R\rightarrow B)^\ast 0^\ast \lambda_{B/S}$. But $0^\ast\lambda_{B/S}=0$, so in total we get a nullhomotopy of $d_S(f|_R)$. Thus, we obtain a section $p_*s\in\Gamma(R, \cS_{R/S})$ with the underlying function $\und(p_* s)=f|_R\colon R\rightarrow \bA^1$.

D-critical pushforwards are functorial in the following sense: given geometric morphisms of stacks $X\rightarrow B\xrightarrow{p} S\xrightarrow{q} T$ together with a section $s\in\Gamma(X, \cS_{X/B})$ we have an isomorphism
\[q_*(p_*(X, s), p_* s)\cong (q\circ p)_*(X, s)\]
constructed using the Cartesian square \eqref{eq:cotangentCartesian}. Moreover, under this isomorphism we have $q_*(p_* s) = (q\circ p)_* s$.

\begin{example}\label{ex:dcritpushforwardcritlocus}
Let $p\colon U\rightarrow B$ be a smooth morphism of schemes and $f\colon U\rightarrow \bA^1$. We have $\cS_{U/U}\cong \cO_U$ and hence $f$ defines a section $f\in\cS_{U/U}$. Then $p_*(U, f) = \Crit_{U/B}(f)$ and $p_* f=s_f$.
\end{example}

The following statement generalizes \cref{ex:dcritpushforwardcritlocus}.

\begin{proposition}\label{prop:pushforwardcriticallocus}
Let $U\rightarrow B\xrightarrow{p} S$ be smooth morphisms of schemes and $f\colon U\rightarrow \bA^1$. The closed immersion $p_*(\Crit_{U/B}(f), s_{f, B_1})\rightarrow \Crit_{U/B}(f)$ identifies
\[p_*(\Crit_{U/B}(f), s_{f, B})\cong \Crit_{U/S}(f).\]
Moreover, under this isomorphism $p_* s_{f, B}$ on $p_*(\Crit_{U/B}(f), s_f)$ goes to $s_{f, S}$ on $\Crit_{U/S}(f)$.
\end{proposition}
\begin{proof}
The isomorphism $p_*(\Crit_{U/B}(f), s_{f, B})\cong \Crit_{U/S}(f)$ follows from the existence of the following diagram, where all squares are Cartesian:
\[
\xymatrix{
\Crit_{U/S}(f) \ar[r] \ar[d] & \Crit_{U/B}(f) \ar[r] \ar^{\mu_{s_{f, B}}}[d] & U \ar^{\Gamma_{d_{S}f}}[d] \\
U \ar^-{\Gamma_0}[r] & \T^*(B/S)\times_{B_1} U \ar[r] \ar[d] & \T^*(U/S) \ar[d] \\
& U \ar^{\Gamma_0}[r] & \T^*(U/B)
}
\]

Let $R_B = \Crit_{U/B}(f)$ and $R = p_*(\Crit_{U/B}(f), s_f)$.  Consider the commutative diagram
\[
\xymatrix{
\cS_{R_B/S}|_R \ar^{i_p}[r] \ar^{(R\rightarrow R_B)^\ast}[d] & \cS_{R_B/B}|_R \ar^{(R\rightarrow R_B)^\ast}[d] \\
\cS_{R/S} \ar^{i_p}[r] & \cS_{R/B}
}
\]
of sheaves on $R$. Since $p\colon B\rightarrow S$ is smooth, the horizontal morphisms $i_p$ are injective. Therefore, it is enough to show that
\[i_p(p_* s_{f, B}) = i_p(s_{f, S}).\]
But by definition of the pushforward $p_* s_{f, B}$ we have
\[i_p(p_* s_{f, B}) = (R\rightarrow R_1)^\ast s_{f, B}.\]
Consider the commutative diagram
\[
\xymatrix{
\cS_{R_B/B}|_R \ar^-{\iota_{R_B, U}}[r] \ar^{(R\rightarrow R_B)^*}[d] & \cO_U/\cI^2_{R_B, U}|_R \ar[d] \\
\cS_{R/B} \ar^-{\iota_{R, U}}[r] & \cO_U/\cI^2_{R, U}
}
\]
defined in \cref{prop:Sproperties}(3) with the horizontal morphisms injective. By \cref{prop:Critsdescription} we get
\[\iota_{R_B, U}(s_{f, B}) = [f]\in \cO_U/\cI^2_{R_B, U},\qquad \iota_{R, U}(i_p(s_{f, S})) = [f]\in\cO/\cI^2_{R, U}\]
which implies that $(R\rightarrow R_B)^\ast s_{f, B} = i_p(s_{f, S})$ and hence, using the above equalities, $p_* s_{f, B} = s_{f, S}$.
\end{proof}

As a corollary, we obtain that d-critical pushforwards preserve relative d-critical structures on schemes.

\begin{corollary}\label{cor:dcritpushforwardschemes}
Consider morphisms of schemes $X\rightarrow B\xrightarrow{p} S$, where $p$ is smooth, equipped with a relative d-critical structure $s\in\Gamma(X, \cS_{X/B})$. Then $p_* s$ is a relative d-critical structure on $p_*(X, s)\rightarrow S$.
\end{corollary}
\begin{proof}
Since $s$ is a relative d-critical structure, we have a collection $\{U_a, f_a\}$ of LG pairs over $B$ together with morphisms $u_a\colon \Crit_{U_a/B}(f_a)\rightarrow X$ such that $\{u_a\colon \Crit_{U_a/B}(f_a)\rightarrow X\}$ is a Zariski cover. Then
\[\{p_*(\Crit_{U_a/B}(f_a), s_{f_a, B})\rightarrow p_*(X, s)\}\]
is also a Zariski cover. But by \cref{prop:pushforwardcriticallocus} we have $p_*(\Crit_{U_a/B}(f_a), s_{f_a, B})\cong \Crit_{U_a/S}(f_a)$ compatibly with relative d-critical structures, so we get a Zariski cover $\{\Crit_{U_a/S}(f_a)\rightarrow p_*(X, s)\}$ by critical charts.
\end{proof}

The d-critical pushforward has the following compatibility with pullbacks.

\begin{proposition}\label{prop:dcritpushforwardpullback}
Consider a commutative diagram of stacks
\[
\xymatrix{
X' \ar[r] \ar[d] & B' \ar^{p'}[r] \ar[d] & S' \ar[d] \\
X \ar[r] & B \ar^{p}[r] & S
}
\]
with all morphisms geometric, equipped with a section $s\in\Gamma(X, \cS_{X/B_1})$ and let $s'\in\Gamma(X', \cS_{X'/B'})$ be the pullback. Assume that $B'\rightarrow B\times_S S'$ is smooth.
\begin{enumerate}
    \item The above diagram can be extended to a commutative diagram
    \begin{equation}\label{eq:dcritpushforwardpullbackdiagram}
    \xymatrix{
    p'_*(X', s') \ar[r] \ar[d] & X' \ar[r] \ar[d] & B' \ar^{p'}[r] \ar[d] & S' \ar[d] \\
    p_*(X, s) \ar[r] & X \ar[r] & B_1 \ar^{p}[r] & S
    }
    \end{equation}
    so that $p'_* s'$ is equal to the pullback of $p_* s$ under the leftmost vertical morphism and the leftmost square is Cartesian.
    \item If $B'\rightarrow B$ and $X'\rightarrow X\times_B B'$ are smooth (respectively, smooth surjective), then so is $p'_*(X', s')\rightarrow p_*(X, s)$. Moreover, in this case there is a fiber sequence
    \[\bL_{B'/B\times_S S'}|_{p'_*(X', s')}\longrightarrow \bL_{p'_*(X', s')/p_*(X, s)\times_S S'}\longrightarrow \bL_{X'/X\times_B B'}|_{p'_*(X', s')}.\]
\end{enumerate}
\end{proposition}
\begin{proof}
We have a commutative diagram
\[
\xymatrix@C=1.8cm{
\cS_{X/B}|_{X'} \ar^-{\mu_{X\rightarrow B\rightarrow S}}[r] \ar[d] & (X\rightarrow B)^* \bL_{B/S}|_{X'} \ar[d] \\
\cS_{X'/B} \ar^-{\mu_{X'\rightarrow B\rightarrow S}}[r] \ar[d] & (X'\rightarrow B)^*\bL_{B/S} \ar[d] \\
\cS_{X'/B'} \ar^-{\mu_{X'\rightarrow B'\rightarrow S'}}[r] & (X'\rightarrow B')^*\bL_{B'/S'}
}
\]
Therefore, we obtain a commutative diagram
\[
\xymatrix{
X \ar^-{\mu_s}[r] & \T^*(B/S) \\
X' \ar^-{\mu'_s}[r] \ar_{\mu_{s'}}[dr] \ar[u] & \T^*(B/S)\times_B B' \ar[d] \ar[u]\\
& \T^*(B'/S'),
}
\]
where the morphism $\mu'_s\colon X'\rightarrow \T^*(B/S)\times_B B'$ is given by the composite
\[\mu'_s\colon X'\rightarrow X\times_B B'\xrightarrow{\mu_s\times \id} \T^*(B/S)\times_B B'\cong \T^*(B\times_S B'/S')\times_{B\times_S S'} B'.\]
Since $B'\rightarrow B\times_S S'$ is smooth, in the Cartesian square
\[
\xymatrix{
\T^*(B\times_S S'/S')\times_{B\times_S S'} B' \ar[r] \ar[d] & \T^*(B'/S') \ar[d] \\
B' \ar^-{0}[r] & \T^*(B'/B\times_S S')
}
\]
the bottom morphism is a closed immersion and hence the top morphism is a closed immersion, hence a monomorphism. Thus, the top square
\[
\xymatrix{
p'_*(X', s') \ar[r] \ar[d] & X' \ar^{\mu'_s}[d] \\
B' \ar^-{0}[r] \ar[d] & \T^*(B/S)\times_B B' \ar[d] \\
B \ar^-{0}[r] & \T^*(B/S)
}
\]
is Cartesian. Since the bottom square is Cartesian as well, this implies that the outer square is Cartesian. But then in the diagram
\[
\xymatrix{
p'_*(X', s') \ar[r] \ar[d] & X' \ar[d] \\
p_*(X, s) \ar[r] \ar[d] & X \ar^{\mu_s}[d] \\
B \ar^-{0}[r] & \T^*(B/S)
}
\]
the bottom square is Cartesian by definition and the outer square is Cartesian by what we have just shown. Therefore, the top square is Cartesian, which establishes the first claim.

The morphism $p'_*(X', s')\rightarrow p_*(X, s)$ factors as the composite
\begin{equation}\label{eq:dcritpushforwardpullbackcomposite}
p'_*(X', s')\rightarrow p_*(X, s)\times_B B'\rightarrow p_*(X, s).
\end{equation}
The first morphism in \eqref{eq:dcritpushforwardpullbackcomposite} is a base change of $X'\rightarrow X\times_B B'$ and the second morphism is a base change of $B'\rightarrow B$. This implies the second claim.
\end{proof}

We obtain the following extension of \cref{cor:dcritpushforwardschemes} to geometric morphisms.

\begin{corollary}\label{cor:dcritpushforwardstacks}
Consider geometric morphisms of stacks $X\rightarrow B\xrightarrow{p} S$, where $p$ is smooth, equipped with a relative d-critical structure $s\in\Gamma(X, \cS_{X/B})$. Then $p_*s$ is a relative d-critical structure on $p_*(X, s)\rightarrow S$.
\end{corollary}
\begin{proof}
By \cref{prop:dcritpushforwardpullback}(1) the claim reduces to the case $S$ a scheme. Then we can find a commutative diagram
\[
\xymatrix{
X' \ar[r] \ar[d] & B' \ar^{p'}[r] \ar[d] & S \ar@{=}[d] \\
X \ar[r] & B \ar^{p}[r] & S
}
\]
with $X'\rightarrow B'$ a morphism of schemes and $B'\rightarrow B$ and $X'\rightarrow X\times_B B'$ smooth surjective. By \cref{prop:dcritbasechangestacks}(1) $s'$ is a relative d-critical structure on $X'\rightarrow B'$. By \cref{cor:dcritpushforwardschemes} $p'_* s'$ is a relative d-critical structure on $p'_*(X', s')\rightarrow S$. By \cref{prop:dcritpushforwardpullback}(2) the morphism $p'_*(X', s')\rightarrow p_*(X, s)$ is smooth and surjective. Therefore, by descent (\cref{prop:dcritbasechangestacks}) we get that $p_*s$ is a relative d-critical structure on $p_*(X, s)\rightarrow S$.
\end{proof}

Let us now describe virtual canonical bundles of d-critical pushforwards.

\begin{proposition}\label{prop:canonicalpushforward}
Consider geometric morphisms of stacks $X\rightarrow B\xrightarrow{p} S$, where $p$ is smooth, equipped with a relative d-critical structure $s\in\Gamma(X, \cS_{X/B})$. There is an isomorphism
\[\Sigma_p\colon K^{\vir}_{X/B}|_{p_*(X, s)^{\red}}\otimes K^{\otimes 2}_{B/S}|_{p_*(X, s)^{\red}}\xrightarrow{\sim} K^{\vir}_{p_* (X, s)/S}\]
which satisfies the following properties:
\begin{enumerate}
    \item It is functorial for compositions: $\Sigma_{\id} = \id$ and given another smooth morphism $q\colon S\rightarrow T$ with $R=(q\circ p)_*(X, s)^{\red}$ the diagram
    \begin{equation}\label{eq-dcrit-push-canonical-composition}
    \xymatrix@C=1.5cm{
    K^{\vir}_{X/B}|_R\otimes K^{\otimes 2}_{B/S}|_R\otimes K^{\otimes 2}_{S/T}|_R \ar^-{\id\otimes i(\Delta)^2}[r] \ar^{\Sigma_p\otimes \id}[d] & K^{\vir}_{X/B}|_R\otimes K^{\otimes 2}_{B/T}|_R \ar^{\Sigma_{q\circ p}}[d] \\
    K^{\vir}_{p_*(X, s)/S}|_R\otimes K^{\otimes 2}_{S/T}|_R \ar^-{\Sigma_q}[r] & K^{\vir}_{(q\circ p)_*(X, s)/T}
    }
    \end{equation}
    commutes, where the top horizontal morphism is induced by the fiber sequence
    \[
    \Delta\colon p^*\bL_{S/T}\longrightarrow \bL_{B/T}\longrightarrow \bL_{B/S}.
    \]
    \item Consider a commutative diagram of stacks
    \begin{equation}\label{eq:pushforwardpullbackvirtual}
    \xymatrix{
    X' \ar[r] \ar[d] & B' \ar^{p'}[r] \ar[d] & S' \ar[d] \\
    X \ar[r] & B \ar^{p}[r] & S
    }
    \end{equation}
    with all morphisms geometric, equipped with a section $s\in\Gamma(X, \cS_{X/B})$ and let $s'\in\Gamma(X', \cS_{X'/B'})$ be the pullback. Assume that $p\colon B\rightarrow S$, $B'\rightarrow B\times_S S'$ and $X'\rightarrow X\times_B B'$ are smooth. Let $R'=p'_*(X', s')^{\red}$. Then the diagram
    \[
    \xymatrix{
    K^{\vir}_{X/B}|_{R'}\otimes K^{\otimes 2}_{B/S}|_{R'}\otimes K^{\otimes 2}_{X'/X\times_B B'}|_{R'}\otimes K^{\otimes 2}_{B'/B\times_S S'}|_{R'} \ar^-{\Sigma_p\otimes \id}[r] \ar^{\Upsilon_{X'\rightarrow X}\otimes \id}[d] & K^{\vir}_{p_*(X, s)/S}|_{R'}\otimes K^{\otimes 2}_{X'/X\times_B B'}|_{R'}\otimes K^{\otimes 2}_{B'/B\times_S S'}|_{R'} \ar^{\id\otimes i(\Delta_2)^2}[d] \\
    K^{\vir}_{X'/B'}|_{R'}\otimes K^{\otimes 2}_{B/S}|_{R'}\otimes K^{\otimes 2}_{B'/B\times_S S'}|_{R'} \ar^{\id\otimes i(\Delta_1)^2}[d] & K^{\vir}_{p_*(X, s)/S}|_{R'}\otimes K^{\otimes 2}_{p'_*(X', s')/p_*(X, s)\times_{S} S'}|_{R'} \ar^{\Upsilon_{p'_*(X', s')\rightarrow p_*(X, s)}}[d] \\
    K^{\vir}_{X'/B'}|_{R'}\otimes K^{\otimes 2}_{B'/S'}|_{R'} \ar^{\Sigma_{p'}}[r] & K^{\vir}_{p'_*(X', s')/S'},
    }
    \]
    commutes, where
    \[\Delta_1\colon \bL_{B/S}|_{B'}\longrightarrow \bL_{B'/S'}\longrightarrow \bL_{B'/B\times_S S'}\]
    and
    \[\Delta_2\colon \bL_{B'/B\times_S S'}|_{p'_*(X', s')}\longrightarrow \bL_{p'_*(X', s')/p_*(X, s)\times_{S} S'}\longrightarrow \bL_{X'/X\times_{B} B'}|_{p'_*(X', s')}.\]
    \item For a function $d\colon X\rightarrow \Z/2\Z$ the diagram
    \[
    \xymatrix{
    K^{\vir}_{X/B, s}|_{p_*(X, s)^{\red}}\otimes K^{\otimes 2}_{B/S}|_{p_*(X, s)^{\red}} \ar^-{\Sigma_p}[r] \ar^{R_d\otimes \id}[d] & K^{\vir}_{p_* (X, s)/S, p_* s} \ar^{R_{d+\dim(B/S)}}[d] \\
    K^{\vir}_{X/B, -s}|_{p_*(X, s)^{\red}}\otimes K^{\otimes 2}_{B/S}|_{p_*(X, s)^{\red}} \ar^-{\Sigma_p}[r] & K^{\vir}_{p_* (X, s)/S, -p_*s}
    }
    \]
    commutes.
    \item $\Sigma_p$ is compatible with products.
\end{enumerate}
\end{proposition}
\begin{proof}
As in the proof of \cref{prop:virtualcanonicalstacks} it is enough to construct these isomorphisms for morphisms of schemes.

Let $(U, f, u)$ be a critical chart for $X$, where $(U, f)$ is viewed as an LG pair over $B$. Let $\overline{u}\colon \Crit_{U/S}(f)\cong p_*(\Crit_{U/B}(f), s_f)\xrightarrow{u} p_*X$, where the first isomorphism is provided by \cref{prop:smoothcriticallocus}. Then $(U, f, \overline{u})$, where $(U, f)$ is viewed as an LG pair over $S$, is a critical chart for $p_*(X, s)$. A critical morphism $\Phi\colon (U, f, u)\rightarrow (V, g, v)$ of critical charts on $X$ gives rise to a critical morphism $\Phi\colon (U, f, \overline{u})\rightarrow (V, g, \overline{v})$ of critical charts on $p_*(X, s)$. By \cref{cor:dcritpushforwardschemes} this collection of critical charts covers $p_*(X, s)$. By \cref{cor:criticaldescentsets} it is thus sufficient to construct the isomorphism $\Sigma_p$ restricted to each such critical chart and check an equality of these isomorphisms for every critical morphism of critical charts on $X$.

Let $(U, f, u)$ be a critical chart for $X$ and $(U, f, \overline{u})$ the corresponding critical chart for $p_*(X, s)$. By definition we have canonical isomorphisms
\[u^* K^{\vir}_{X/B}\cong K^{\otimes 2}_{U/B}|_{\Crit_{U/B}(f)^{\red}},\qquad \overline{u}^* K^{\vir}_{p_*(X, s)/S}\cong K^{\otimes 2}_{U/S}|_{\Crit_{U/S}(f)^{\red}}.\]
We define $\Sigma_p|_{\Crit_{U/S}(f)^{\red}}$ to be the isomorphism
\[i(\Delta_U)^2\colon K^{\otimes 2}_{U/B}|_{\Crit_{U/S}(f)^{\red}}\otimes K^{\otimes 2}_{B/S}|_{\Crit_{U/S}(f)^{\red}}\xrightarrow{\sim} K^{\otimes 2}_{U/S}|_{\Crit_{U/S}(f)^{\red}}\]
associated to the exact sequence
\[\Delta_U\colon 0\longrightarrow \Omega^1_{B/S}|_U\longrightarrow \Omega^1_{U/S}\longrightarrow \Omega^1_{U/B}\longrightarrow 0.\]

Now consider a critical morphism $\Phi\colon (U, f, u)\rightarrow (V, g, v)$ of critical charts on $X$. Recall that the normal bundle $\rN_{U/V}|_{\Crit_{U/B}(f)}$ carries a nondegenerate quadratic form $q_\Phi$. Let $R=\Crit_{U/S}(f)^{\red}$. Consider the diagram
\begin{equation}\label{eq:Lambdacharts}
\xymatrix@C=1.5cm{
K^{\otimes 2}_{B/S}|_R\otimes K^{\otimes 2}_{U/B}|_R \ar^-{\id\otimes \vol^2_{q_\Phi}}[r] \ar^{i(\Delta_U)^2}[d] & K^{\otimes 2}_{B/S}|_R\otimes K^{\otimes 2}_{U/B}|_R\otimes (\det \rN^\vee_{U/V})^{\otimes 2}|_R \ar^-{\id\otimes i(\Delta_1)^2}[r] \ar^{i(\Delta_U)^2\otimes \id}[d] & K^{\otimes 2}_{B/S}|_R \otimes K^{\otimes 2}_{V/B}|_R \ar^{i(\Delta_V)^2}[d] \\
K^{\otimes 2}_{U/S}|_R \ar^-{\id\otimes \vol^2_{q_\Phi}}[r] & K^{\otimes 2}_{U/S}|_R\otimes (\det \rN^\vee_{U/V})^{\otimes 2}|_R \ar^-{i(\Delta_2)^2}[r] & K^{\otimes 2}_{V/S}|_R,
}
\end{equation}
where the individual isomorphisms are induced by the following double short exact sequence:
\[
\begin{aligned}
\begin{tikzcd}
    &
    \begin{array}{c} \substack{\displaystyle{\phantom{\Delta} } \\ \displaystyle{\phantom{\rotatebox{90}{\mbox{:}}}} \\ \displaystyle{0}} \end{array} &
    \begin{array}{c} \substack{\displaystyle{\ \Delta_V } \\ \displaystyle{  \rotatebox{90}{\mbox{:}}} \\  \displaystyle{0}} \end{array} &
    \begin{array}{c} \substack{\displaystyle{\Delta_U} \\ \displaystyle{\rotatebox{90}{\mbox{:}}} \\  \displaystyle{0}}\end{array} \\
    \phantom{\Delta_0:} 0 & 0 & \Omega^1_{B/S}|_U & \Omega^1_{B/S}|_U & 0 \\
    \Delta_2 \colon 0 & \rN^\vee_{U/V} & \Omega^1_{V/S}|_U & \Omega^1_{U/S} & 0 \\
    \Delta_1 \colon 0 & \rN^\vee_{U/V} & \Omega^1_{V/B}|_U & \Omega^1_{U/B} & 0 \\
    & 0 & 0 & 0 & 
    \arrow[from=2-1, to=2-2]
    \arrow[from=2-2, to=2-3]
    \arrow[equal, from=2-3, to=2-4]
    \arrow[from=2-4, to=2-5]
    \arrow[from=3-1, to=3-2]
    \arrow[from=3-2, to=3-3]
    \arrow[from=3-3, to=3-4]
    \arrow[from=3-4, to=3-5]
    \arrow[from=4-1, to=4-2]
    \arrow[from=4-2, to=4-3]
    \arrow[from=4-3, to=4-4]
    \arrow[from=4-4, to=4-5]
    \arrow[from=1-2, to=2-2]
    \arrow[from=2-2, to=3-2]
    \arrow[equal, from=3-2, to=4-2]
    \arrow[from=4-2, to=5-2]
    \arrow[from=1-3, to=2-3]
    \arrow[from=2-3, to=3-3]
    \arrow[from=3-3, to=4-3]
    \arrow[from=4-3, to=5-3]
    \arrow[from=1-4, to=2-4]
    \arrow[from=2-4, to=3-4]
    \arrow[from=3-4, to=4-4]
    \arrow[from=4-4, to=5-4]
\end{tikzcd}
\end{aligned}
\]

The top horizontal morphism in \eqref{eq:Lambdacharts} gives an isomorphism of the two models of $K^{\otimes 2}_{B/S}|_R\otimes K^{\vir}_{X/B}|_R$ in the critical charts $U$ and $V$. The bottom horizontal morphism gives an isomorphism of the two models of $K^{\vir}_{p_*(X, s)/S}$ in the critical charts $U$ and $V$. The outer vertical morphisms are the local models of $\Sigma_p$. The square on the left commutes by naturality. The commutativity of the square on the right follows from the corresponding property of determinant lines, see \cite[Corollary 2.3]{KPS}.

Functoriality of $\Sigma$ with respect to compositions, property (1), follows from a commutativity of the diagram of determinant lines as in \cite[Lemma 2.4]{KPS} associated to the double short exact sequence
\[
\begin{aligned}
\begin{tikzcd}
    & 0 & 0 & 0 \\
    0 & \Omega^1_{S/T}|_U & \Omega^1_{S/T}|_U & 0 & 0 \\
    0 & \Omega^1_{B/T}|_U & \Omega^1_{U/T} & \Omega^1_{U/B} & 0 \\
    0 & \Omega^1_{B/S}|_U & \Omega^1_{U/S} & \Omega^1_{U/B} & 0 \\
    & 0 & 0 & 0 & 
    \arrow[from=2-1, to=2-2]
    \arrow[equal, from=2-2, to=2-3]
    \arrow[from=2-3, to=2-4]
    \arrow[from=2-4, to=2-5]
    \arrow[from=3-1, to=3-2]
    \arrow[from=3-2, to=3-3]
    \arrow[from=3-3, to=3-4]
    \arrow[from=3-4, to=3-5]
    \arrow[from=4-1, to=4-2]
    \arrow[from=4-2, to=4-3]
    \arrow[from=4-3, to=4-4]
    \arrow[from=4-4, to=4-5]
    \arrow[from=1-2, to=2-2]
    \arrow[from=2-2, to=3-2]
    \arrow[from=3-2, to=4-2]
    \arrow[from=4-2, to=5-2]
    \arrow[from=1-3, to=2-3]
    \arrow[from=2-3, to=3-3]
    \arrow[from=3-3, to=4-3]
    \arrow[from=4-3, to=5-3]
    \arrow[from=1-4, to=2-4]
    \arrow[from=2-4, to=3-4]
    \arrow[equal, from=3-4, to=4-4]
    \arrow[from=4-4, to=5-4]
\end{tikzcd}
\end{aligned}
\]

Let us now show property (2). We can factor \eqref{eq:pushforwardpullbackvirtual} as
\[
\xymatrix{
X' \ar[d] && \\
X\times_B B' \ar[r] \ar[d] & B' \ar[d] & \\
X\times_{S} S' \ar[r] \ar[d] & B\times_S S' \ar[d] \ar[r] & S' \ar[d] \\
X \ar[r] & B \ar[r] & S,
}
\]
where all squares are Cartesian. Using the functoriality with respect to compositions of $\Sigma$ (property (1)) and $\Upsilon$ (\cref{prop:virtualcanonicalstacks}(3)), property (2) for a general diagram \eqref{eq:pushforwardpullbackvirtual} follows from the following particular cases:
\begin{enumerate}
    \item Both squares in \eqref{eq:pushforwardpullbackvirtual} are Cartesian. Let $(U, f, u)$ be a critical chart on $X$ and $(U', f', u')$ its base change along $B'\rightarrow B$ which defines a critical chart of $X'$. Let $R'=\Crit_{U'/S'}(f')^{\red}$. Then we have to prove a commutativity of the diagram
    \[
    \xymatrix{
    K^{\otimes 2}_{U/B}|_{R'}\otimes K^{\otimes 2}_{B/S}|_{R'} \ar^-{i(\Delta_U)^2}[r] \ar^{\sim}[d] & K^{\otimes 2}_{U/S}|_{R'} \ar^{\sim}[d] \\
    K^{\otimes 2}_{U'/B'}|_{R'}\otimes K^{\otimes 2}_{B'/S'}|_{R'} \ar^-{i(\Delta_{U'})^2}[r] & K^{\otimes 2}_{U'/S'}|_{R'}
    }
    \]
    which follows from the naturality of the isomorphisms $i(\Delta)$.
    \item In \eqref{eq:pushforwardpullbackvirtual} the left square is Cartesian and $B=S=S'$. As before, let $(U, f, u)$ be a critical chart on $X$ and $(U', f', u')$ its base change along $B'\rightarrow B$ which defines a critical chart of $X'$. Let $(U, f, \overline{u})$ be the corresponding critical chart of $p_*(X, s)$, where $(U, f)$ is viewed as an LG pair over $B$ and $(U', f', \overline{u}')$ be the corresponding critical chart of $p'_*(X', s')$, where $(U', f')$ is viewed as an LG pair over $B$. Let $R'=\Crit_{U'/B}(f')^{\red}$. Then we have to prove a commutativity of the diagram
    \[
    \xymatrix{
    K^{\otimes 2}_{U/B}|_{R'}\otimes K^{\otimes 2}_{B'/B}|_{R'} \ar@{=}[r] \ar^{\sim}[dd] & K^{\otimes 2}_{U/B}|_{R'} \otimes K^{\otimes 2}_{B'/B}|_{R'} \ar^{\sim}[d] \\
    & K^{\otimes 2}_{U/B}|_{R'}\otimes K^{\otimes 2}_{U'/U}|_{R'} \ar^{i(\Delta'_U)^2}[d] \\
    K^{\otimes 2}_{U'/B'}|_{R'}\otimes K^{\otimes 2}_{B'/B}|_{R'} \ar^-{i(\Delta_{U'})^2}[r] & K^{\otimes 2}_{U'/B}|_{R'},
    }
    \]
    where
    \[\Delta'_U\colon 0\longrightarrow \Omega^1_{U/B}|_{U'}\longrightarrow \Omega^1_{U'/B}\longrightarrow \Omega^1_{U'/U}\longrightarrow 0.\]
    But $\Delta'_U$ and $\Delta_{U'}$ are the two short exact sequences associated to the direct sum decomposition $\Omega^1_{U'/B}\cong \Omega^1_{U/B}|_{U'}\oplus \Omega^1_{B'/B}|_{U'}$, so the corresponding isomorphisms $i(\Delta'_U)$ and $i(\Delta_{U'})$ agree.
    \item $B'=B$ and $S'=S$. By \cref{thm:criticallocussmoothdescent}(1) we may find a smooth morphism $\tilde{\pi}\colon U'\rightarrow U$ of smooth $B$-schemes, a function $f\colon U\rightarrow \bA^1$ with $f'=\tilde{\pi}^\ast f$ and a commutative diagram
    \[
    \xymatrix{
    X' \ar[d] & \Crit_{U'/B}(f') \ar_-{u'}[l] \ar[r] \ar[d] & U' \ar^{\tilde{\pi}}[d] \\
    X & \Crit_{U/B}(f) \ar_-{u}[l] \ar[r] & U
    }
    \]
    so that $(U, f, u)$ is a critical chart for $X$ and $(U', f', u')$ is a critical chart for $X'$ and, moreover, we may cover $X$ and $X'$ by critical charts of this form. Let $R'=\Crit_{U'/S}(f')^{\red}$. Then we have to prove a commutativity of the diagram
    \[
    \xymatrix@C=1.5cm{
    K^{\otimes 2}_{B/S}|_{R'}\otimes K^{\otimes 2}_{U/B}|_{R'}\otimes K^{\otimes 2}_{U'/U}|_{R'} \ar^-{i(\Delta_U)^2\otimes \id}[r] \ar^{\id\otimes i(\Delta_1)^2}[d] & K^{\otimes 2}_{U/S}|_{R'} \otimes K^{\otimes 2}_{U'/U}|_{R'} \ar^{i(\Delta_2)^2}[d] \\
    K^{\otimes 2}_{B/S}|_{R'} \otimes K^{\otimes 2}_{U'/B}|_{R'} \ar^{i(\Delta_{U'})^2}[r] & K^{\otimes 2}_{U'/S}|_{R'},
    }
    \]
    where $\Delta_1$ and $\Delta_2$ are the short exact sequences
    \[
    \begin{aligned}
    \begin{tikzcd}
    &
    \begin{array}{c} \substack{\displaystyle{\Delta_U} \\ \displaystyle{\rotatebox{90}{\mbox{:}}} \\ \displaystyle{0}} \end{array} &
    \begin{array}{c} \substack{\displaystyle{\Delta_{U'}} \\ \displaystyle{  \rotatebox{90}{\mbox{:}}} \\  \displaystyle{0}} \end{array} &
    \begin{array}{c} \substack{\displaystyle{\phantom{\Delta}} \\ \displaystyle{\phantom{\rotatebox{90}{\mbox{:}}}} \\  \displaystyle{0}}\end{array} \\
    \phantom{\Delta_0:} 0 & \Omega^1_{B/S}|_{U'} & \Omega^1_{B/S}|_{U'} & 0 & 0 \\
    \Delta_2\colon 0 & \Omega^1_{U/S}|_{U'} & \Omega^1_{U'/S} & \Omega^1_{U'/U} & 0 \\
    \Delta_1\colon 0 & \Omega^1_{U/B}|_{U'} & \Omega^1_{U'/B} & \Omega^1_{U'/U} & 0 \\
    & 0 & 0 & 0 & 
    \arrow[from=2-1, to=2-2]
    \arrow[equal, from=2-2, to=2-3]
    \arrow[from=2-3, to=2-4]
    \arrow[from=2-4, to=2-5]
    \arrow[from=3-1, to=3-2]
    \arrow[from=3-2, to=3-3]
    \arrow[from=3-3, to=3-4]
    \arrow[from=3-4, to=3-5]
    \arrow[from=4-1, to=4-2]
    \arrow[from=4-2, to=4-3]
    \arrow[from=4-3, to=4-4]
    \arrow[from=4-4, to=4-5]
    \arrow[from=1-2, to=2-2]
    \arrow[from=2-2, to=3-2]
    \arrow[from=3-2, to=4-2]
    \arrow[from=4-2, to=5-2]
    \arrow[from=1-3, to=2-3]
    \arrow[from=2-3, to=3-3]
    \arrow[from=3-3, to=4-3]
    \arrow[from=4-3, to=5-3]
    \arrow[from=1-4, to=2-4]
    \arrow[from=2-4, to=3-4]
    \arrow[equal, from=3-4, to=4-4]
    \arrow[from=4-4, to=5-4]
    \end{tikzcd}
    \end{aligned}
    \]
    The commutativity of the diagram follows from a commutativity of the diagram of determinant lines as in \cite[Lemma 2.4]{KPS}.
\end{enumerate}
Properties (3) and (4) are straightforward.
\end{proof}

\subsection{Orientations for relative d-critical structures}

Using the virtual canonical bundle we define orientations for relative d-critical structures.

\begin{definition}
Let $X\rightarrow B$ be a geometric morphism of stacks equipped with a relative d-critical structure $s$. An \defterm{orientation} of $(X\rightarrow B,s)$ is a pair $(\cL, o)$ consisting of a $\Z/2\Z$-graded line bundle $\cL$ together with an isomorphism $o\colon \cL^{\otimes 2}\cong K^{\vir}_{X/B}$.
\end{definition}

For simplicity of notation we will often denote an orientation simply by $o$ with the graded line bundle $\cL$ being implicit.

\begin{remark}
In the above definition we consider graded orientations as in \cite{KPS}. In \cite{BBDJS} the authors consider an analogous notion where $\cL$ is an ungraded line bundle.
\end{remark}

There is a natural notion of isomorphisms of orientations: an isomorphism $(\cL_1, o_1)\rightarrow (\cL_2, o_2)$ is given by an isomorphism $f\colon \cL_1\rightarrow \cL_2$ of graded line bundles such that the diagram
\[
\xymatrix{
\cL_1^{\otimes 2} \ar_{o_1}[dr] \ar^{f^{\otimes 2}}[rr] && \cL_2^{\otimes 2} \ar^{o_2}[dl] \\
& K^{\vir}_{X/B}
}
\]
commutes.

We have the following functoriality of orientations:
\begin{enumerate}
    \item Consider a commutative diagram of stacks
    \[
    \xymatrix{
    X'\ar[r] \ar^{\pi'}[d] \ar^{\overline{p}}[r] & X \ar^{\pi}[d] \\
    B' \ar^{p}[r] & B
    }
    \]
    with $X'\rightarrow X\times_B B'$ a smooth morphism. Consider a relative d-critical structure $s$ on $X\rightarrow B$ and its pullback $s'$ to $X'\rightarrow B'$. For an orientation $(\cL, o)$ there is an orientation $(\cL|_{(X')^{\red}}\otimes K_{X'/X\times_B B'}|_{(X')^{\red}}, \overline{p}^\ast o)$, of $(X'\rightarrow B', s')$, where
    \begin{align*}
    \overline{p}^\ast o\colon (\cL|_{(X')^{\red}}\otimes K_{X'/X\times_B B'}|_{(X')^{\red}})^{\otimes 2}&\xrightarrow{\sim} \cL^{\otimes 2}|_{(X')^{\red}}\otimes K^{\otimes 2}_{X'/X\times_B B'}|_{(X')^{\red}}\\
    &\xrightarrow{o\otimes \id} K^{\vir}_{X/B}|_{(X')^{\red}}\otimes K^{\otimes 2}_{X'/X\times_B B'}|_{(X')^{\red}}\\
    &\xrightarrow{\Upsilon_{X'\rightarrow X}} K^{\vir}_{X'/B'}.
    \end{align*}
    For a composable pair
    \[
    \xymatrix{
    X''\ar[d] \ar^{\overline{q}}[r] & X'\ar[r] \ar^{\pi'}[d] \ar^{\overline{p}}[r] & X \ar^{\pi}[d] \\
    B'' \ar^{q}[r] & B' \ar^{p}[r] & B
    }
    \]
    of commutative diagram of stacks as above the natural isomorphism $\overline{q}^* K_{X'/X\times_B B'}\otimes K_{X''/X'\times_{B'} B''}\cong K_{X''/X\times_B B''}$ induces an isomorphism of orientations
    \begin{equation}\label{eq:orientationpullbackassociative}
    \overline{q}^*\overline{p}^* o\cong (\overline{q}\circ\overline{p})^* o
    \end{equation}
    using \cref{prop:virtualcanonicalstacks}(3).
    \item Consider geometric morphisms of stacks $X\rightarrow B\xrightarrow{p} S$, where $p$ is smooth and $X\rightarrow B$ is equipped with a relative d-critical structure $s$ and an orientation $(\cL, o)$. Then there is an orientation $(\cL|_{p_*(X, s)^{\red}}\otimes K_{B/S}|_{p_*(X, s)^{\red}}, p_*o)$, where
    \begin{align*}
    p_* o\colon (\cL|_{p_*(X, s)^{\red}}\otimes K_{B/S}|_{p_*(X, s)^{\red}})^{\otimes 2} &\xrightarrow{\sim} \cL^{\otimes 2}|_{p_*(X, s)^{\red}}\otimes K^{\otimes 2}_{B/S}|_{p_*(X, s)^{\red}} \\
    &\xrightarrow{o\otimes \id} K^{\vir}_{X/B}|_{p_*(X, s)^{\red}}\otimes K^{\otimes 2}_{B/S}|_{p_*(X, s)^{\red}} \\
    &\xrightarrow{\Sigma_p} K^{\vir}_{p_*(X, s)/S}.
    \end{align*}
    Given another smooth morphism $q\colon S\rightarrow T$ we obtain an isomorphism
    \begin{equation}\label{eq:orientationpushforwardassociative}
    q_*p_*o\cong (q\circ p)_* o
    \end{equation}
    using \cref{prop:canonicalpushforward}(1). For a commutative diagram of stacks
    \begin{equation}
    \xymatrix{
    X' \ar[r] \ar^{q}[d] & B' \ar^{p'}[r] \ar[d] & S' \ar[d] \\
    X \ar[r] & B \ar^{p}[r] & S
    }
    \end{equation}
    with $B\rightarrow S, B'\rightarrow B\times_S S'$ and $X'\rightarrow X\times_{B} B'$ smooth and relative d-critical structure $s$ on $X\rightarrow B$ and its pullback $s'$ on $X'\rightarrow B'$ we have a diagram
    \[
    \xymatrix{
    p'_*(X', s') \ar[r] \ar^{\overline{q}}[d] & S' \ar[d] \\
    p_*(X, s) \ar[r] & S.
    }
    \]
    Then using \cref{prop:canonicalpushforward}(2) we obtain an isomorphism
    \begin{equation}\label{eq:orientationpushforwardpullback}
    \overline{q}^* p_*o \cong p'_*q^*o.
    \end{equation}
    \item For a geometric morphism of stacks $X\rightarrow B$ equipped with a relative d-critical structure $s$ and an orientation $(\cL, o)$ of parity $d\colon X\rightarrow \Z/2\Z$ there is an orientation $(\cL, \overline{o})$ of $(X\rightarrow B, -s)$, where
    \begin{equation}\label{eq:orientationreversedcritical}
    \overline{o}\colon \cL^{\otimes 2}\xrightarrow{o} K^{\vir}_{X/B, s}\xrightarrow{R_d} K^{\vir}_{X/B, -s}.
    \end{equation}
    Since $R_d$ squares to the identity, the identity morphism of graded line bundles induces an isomorphism of orientations
    \begin{equation}\label{eq:orientationdoublereverse}
    \overline{(\overline{o})}\cong o.
    \end{equation}
    For a commutative diagram of stacks
    \[
    \xymatrix{
    X'\ar[r] \ar^{\pi'}[d] \ar^{\overline{p}}[r] & X \ar^{\pi}[d] \\
    B' \ar^{p}[r] & B
    }
    \]
    with $X'\rightarrow X\times_B B'$ smooth, a relative d-critical structure on $X\rightarrow B$ and an orientation $o$ and its pullback to $X'\rightarrow B'$, using \cref{prop:virtualcanonicalstacks}(6) we obtain an isomorphism
    \begin{equation}\label{eq:orientationpullbackreverse}
    \overline{\overline{p}^* o}\cong \overline{p}^*\overline{o}.
    \end{equation}
    For morphisms $X\rightarrow B\xrightarrow{p} S$, where $p$ is smooth and $X\rightarrow B$ is equipped with a relative d-critical structure $s$ and an orientation $o$, using \cref{prop:canonicalpushforward}(3) we obtain an isomorphism
    \begin{equation}\label{eq:orientationpushforwardreverse}
    \overline{p_* o}\cong p_*\overline{o}.
    \end{equation}
    \item For a pair $X_1\rightarrow B_1, X_2\rightarrow B_2$ of geometric morphisms of stacks equipped with relative d-critical structures $s_1,s_2$ and orientations $(\cL_1, o_1), (\cL_2, o_2)$ there is an orientation $(\cL_1\boxtimes \cL_2, o_1\boxtimes o_2)$ of $X_1\times X_2\rightarrow B_1\times B_2$ equipped with the relative d-critical structure $s_1\boxplus s_2$, where
    \[o_1\boxtimes o_2\colon (\cL_1\boxtimes \cL_2)^{\otimes 2}\xrightarrow{\sim} (\cL_1)^{\otimes 2}\boxtimes (\cL_2)^{\otimes 2}\xrightarrow{o_1, o_2} K^{\vir}_{X_1/B_1}\boxtimes K^{\vir}_{X_2/B_2}\xrightarrow{\eqref{eq:Kproductstacks}} K^{\vir}_{X_1\times X_2/B_1\times B_2}.\]
    This construction is unital and associative in the obvious sense. It is also commutative in the following sense: for the swapping isomorphism $\sigma\colon X_2\times X_1\rightarrow X_1\times X_2$ using the symmetry of \eqref{eq:Kproductstacks} we obtain an isomorphism
    \begin{equation}\label{eq:orientationbraiding}
    \sigma^*(o_1\boxtimes o_2)\cong o_2\boxtimes o_1,
    \end{equation}
    which on the level of underlying graded line bundles is the usual swapping isomorphism $\sigma^*(\cL_1\boxtimes \cL_2)\cong \cL_2\boxtimes \cL_1$ of ungraded line bundles multiplied by the Koszul sign $(-1)^{\deg(\cL_1)\deg(\cL_2)}$. For $i=1, 2$ let
    \[
    \xymatrix{
    X'_i\ar[r] \ar^{\pi'_i}[d] \ar^{\overline{p}_i}[r] & X_i \ar^{\pi_i}[d] \\
    B'_i \ar^{p_i}[r] & B_i
    }
    \]
    be a commutative diagram of stacks with $X'_i\rightarrow X_i\times_{B_i} B'_i$ smooth, equipped with an oriented relative d-critical structure $(s_i, o_i)$ on $X_i\rightarrow B_i$ and its pullback on $X'_i\rightarrow B'_i$. Then using the compatibility of the isomorphism \eqref{eq:Kproductstacks} with $\Upsilon_{X'_i\rightarrow X_i}$ we obtain an isomorphism
    \begin{equation}\label{eq:orientationproductpullback}
    (\overline{p}_1\times \overline{p}_2)^*(o_1\boxtimes o_2)\cong (\overline{p}_1^*o_1)\boxtimes (\overline{p}_2^*o_2)
    \end{equation}
    given by the obvious isomorphism involving the Koszul sign. For $i=1, 2$ let $X_i\rightarrow B_i\xrightarrow{p_i} S_i$ be morphisms of stacks, where $p$ is smooth and $X_i\rightarrow B_i$ is equipped with an oriented relative d-critical structure $(s_i, o_i)$. Then using the compatibility of $\Sigma_{p_i}$ with products we obtain an isomorphism
    \begin{equation}\label{eq:orientationproductpushforward}
    (p_1\times p_2)_*(o_1\boxtimes o_2)\cong (p_{1, *}o_1)\boxtimes (p_{2, *}o_2)
    \end{equation}
    again given by the obvious isomorphism involving the Koszul sign.
\end{enumerate}

Given a geometric morphism $X \to B$, we denote by $\DCrit^{\ori}(X/B)$ the groupoid of oriented relative d-critical structures $(s, \cL, o)$ on $X \to B$.
The assignment $(X \to B) \mapsto \DCrit^{\ori}(X/B)$ determines a functor
\[\DCrit^{\ori} \colon \Fun(\Delta^1, \Sch)^{\op}_{0\smooth}\longrightarrow \Gpd.\]
Since $\DCrit(-)$ and the stack of graded line bundles satisfy \'etale descent, so does $\DCrit^{\ori}(-)$.

We will use canonical orientations of relative critical loci defined as follows:
\begin{enumerate}
    \item For the identity morphism $Y\rightarrow Y$ of stacks equipped with a relative d-critical structure specified by a function $f\colon Y\rightarrow \bA^1$ there is a canonical trivialization $K^{\vir}_{Y/Y}\cong \cO_{Y^{\red}}$ which corresponds to the identity isomorphism under $\kappa_y$ for every $y\in Y$. We define the \defterm{canonical orientation} $o^{\can}_{Y/Y}$ of $(Y\xrightarrow{\id} Y, f)$ to be the even line bundle $\cO_{Y^{\red}}$ whose square is equipped with the above isomorphism to $K^{\vir}_{Y/Y}$.
    \item More generally, suppose $p\colon Y\rightarrow B$ is a smooth geometric morphism of stacks together with a function $f\colon Y\rightarrow \bA^1$. The relative critical locus $\Crit_{Y/B}(f)\rightarrow B$ is a d-critical pushforward of $(Y\xrightarrow{\id} Y, f)$ along $p$ and we call the pushforward orientation the \defterm{canonical orientation} $o^{\can}_{\Crit_{Y/B}(f)/B}=p_* o^{\can}_{Y/Y}$ whose underlying graded line bundle is $K_{Y/B}|_{\Crit_{Y/B}(f)^{\red}}$.
\end{enumerate}

\begin{example}
Suppose $p\colon U\rightarrow B$ is a smooth morphism of schemes equipped with a function $f\colon U\rightarrow \bA^1$. The morphism $p\colon U\rightarrow B$ provides a critical chart for $X=\Crit_{U/B}(f)\rightarrow B$. In this case the isomorphism $o^{\can}_{\Crit_{U/B}(f)/B}\colon K_{U/B}^{\otimes 2}|_{\Crit_{U/B}(f)^{\red}}\cong K^{\vir}_{X/B}$ coincides with the canonical isomorphism from \cref{thm:virtualcanonicalschemes}(1).
\end{example}

For a critical morphism $\Phi\colon (U, f)\rightarrow (V, g)$ of LG pairs over $B$ of even relative dimension recall the $\mu_2$-torsor $P_\Phi$ parametrizing orientations of the orthogonal bundle $(\rN_{U/V}|_{\Crit_{U/B}(f)}, q_\Phi)$. Via the above description it can also be interpreted as a $\mu_2$-torsor parametrizing isomorphisms $o^{\can}_{\Crit_{U/B}(f)/B}\rightarrow o^{\can}_{\Crit_{V/B}(g)/B}|_{\Crit_{U/B}(f)}$ of canonical orientations.

We have the following functoriality of canonical orientations:
\begin{enumerate}
    \item Consider a commutative diagram
    \[
    \xymatrix{
    Y' \ar^{p'}[r] \ar[d] & Y \ar[d] \\
    B' \ar[r] & B
    }
    \]
    of geometric morphisms of stacks, where $Y\rightarrow B$ and $Y'\rightarrow Y\times_B B'$ are smooth, together with a function $f\colon Y\rightarrow \bA^1$. Let $f'\colon Y'\rightarrow \bA^1$ its pullback to $Y'$, so that we obtain a commutative diagram
    \[
    \xymatrix{
    \Crit_{Y'/B'}(f') \ar^{\overline{p}}[r] \ar[d] & \Crit_{Y/B}(f) \ar[d] \\
    B' \ar^{p}[r] & B.
    }
    \]
    The natural isomorphism $p'^* K_{Y/B}\otimes K_{Y'/Y\times_B B'}\cong K_{Y'/B'}$ of graded line bundles induces an isomorphism
    \begin{equation}\label{eq:pullbackcanonicalorientation}
    \overline{p}^\ast o^{\can}_{\Crit_{Y/B}(f)}\cong o^{\can}_{\Crit_{Y'/B'}(f')}
    \end{equation}
    of orientations, using \eqref{eq:orientationpushforwardpullback} applied to the diagram
    \[
    \xymatrix{
    Y' \ar@{=}[r] \ar[d] & Y' \ar[r] \ar[d] & B' \ar[d] \\
    Y \ar@{=}[r] & Y \ar[r] & B.
    }
    \]
    \item Consider smooth geometric morphisms of stacks $Y\xrightarrow{p} B_1\xrightarrow{q} B_2$ together with a function $f\colon Y\rightarrow \bA^1$. The natural isomorphism $K_{Y/B_1}\otimes p^* K_{B_1/B_2}\cong K_{Y/B_2}$ of graded line bundles induces an isomorphism
    \begin{equation}\label{eq:pushforwardcanonicalorientation}
    q_* o^{\can}_{\Crit_{Y/B_1}(f)/B_1}\cong o^{\can}_{\Crit_{Y/B_2}(f)/B_2}
    \end{equation}
    of orientations, using \eqref{eq:orientationpushforwardassociative} applied to the morphisms $Y=Y\xrightarrow{p} B_1\xrightarrow{q} B_2$.
    \item $\delta$. Consider a smooth morphism of stacks $Y\rightarrow B$ equipped with a function $f\colon Y\rightarrow \bA^1$. The identity morphism of graded line bundles induces an isomorphism
    \begin{equation}\label{eq:canonicalorientationreverse}
    \overline{o^{\can}_{\Crit_{Y/B}(f)/B}}\cong o^{\can}_{\Crit_{Y/B}(-f)/B}
    \end{equation}
    of orientations, using \eqref{eq:orientationpushforwardreverse} applied to $Y=Y\rightarrow B$.
    \item For a pair of smooth morphisms $Y_1\rightarrow B_1$ and $Y_2\rightarrow B_2$ equipped with functions $f_1\colon Y_1\rightarrow \bA^1$ and $f_2\colon Y_2\rightarrow \bA^1$ the natural isomorphism $K_{Y_1/B_1}\boxtimes K_{Y_2/B_2}\cong K_{Y_1\times Y_2/B_1\times B_2}$ induces an isomorphism
    \begin{equation}\label{eq:canonicalorientationproduct}
    o^{\can}_{\Crit_{Y_1/B_1}(f_1)/B_1}\boxtimes o^{\can}_{\Crit_{Y_2/B_2}(f_2)/B_2}\cong o^{\can}_{\Crit_{Y_1\times Y_2/B_1\times B_2}(f_1\boxplus f_2)/B_1\times B_2}
    \end{equation}
    of orientations, using \cref{prop:canonicalpushforward}(4).
\end{enumerate}

If $X\rightarrow B$ is a morphism of schemes equipped with a relative d-critical structure $s$ and an orientation $o$, we can compare in each critical chart $o$ and the canonical orientation. This gives rise to the following data:
\begin{enumerate}
    \item For a critical chart $(U, f, u)$ of $(X\rightarrow B, s)$ with $\dim(U/B)\pmod{2}$ equal to the parity of $o$ we have a $\mu_2$-torsor $Q^o_{U, f, u}$ on $\Crit_{U/B}(f)$ parametrizing isomorphisms $o|_{\Crit_{U/B}(f)}\rightarrow o^{\can}_{\Crit_{U/B}(f)/B}$.
    \item For a critical morphism $\Phi\colon (U, f, u)\rightarrow (V, g, v)$ of even relative dimension we have an isomorphism
    \[\Lambda_\Phi\colon Q^o_{V, g, v}|_{\Crit_{U/B}(f)}\xrightarrow{\sim} P_\Phi\otimes_{\mu_2} Q^o_{U, f, u}\]
    of $\mu_2$-torsors coming from the composition $o|_{\Crit_{U/B}(f)}\rightarrow o^{\can}_{\Crit_{U/B}(f)/B}\rightarrow o^{\can}_{\Crit_{V/B}(g)/B}|_{\Crit_{U/B}(f)}$. It is associative for a composite $(U, f, u)\xrightarrow{\Phi} (V, g, v)\xrightarrow{\Psi} (W, h, w)$ of critical morphisms of even relative dimension via the isomorphism $\Xi_{\Phi, \Psi}\colon P_{\Psi\circ\Phi}\xrightarrow{\sim}P_\Psi|_{\Crit_{U/B}(f)}\otimes_{\mu_2}P_\Phi$.
    \item Consider a commutative diagram
    \[
    \xymatrix{
    X'\ar^{\overline{p}}[r] \ar[d] & X \ar[d] \\
    B' \ar^{p}[r] & B
    }
    \]
    with $X'\rightarrow X\times_{B'} B$ smooth with the pullback relative d-critical structure on $X'\rightarrow B'$. Consider a commutative diagram
    \[
    \xymatrix{
    U'\ar[r] \ar[d] & U \ar[d] \\
    B'\ar^{p}[r] & B
    }
    \]
    with $U'\rightarrow U\times_{B'} B$ smooth, where $(U, f, u)$ is a critical chart for $X\rightarrow B$ and $(U', f', u')$ is a critical chart for $X'\rightarrow B'$. Then the composite $\overline{p}^* o|_{\Crit_{U'/B'}(f')}\rightarrow \overline{p}^* o^{\can}_{\Crit_{U/B}(f)/B}\xrightarrow{\eqref{eq:pullbackcanonicalorientation}} o^{\can}_{\Crit_{U'/B'}(f')/B'}$ determines an isomorphism of $\mu_2$-torsors
    \begin{equation}\label{eq:Qpullback}
    Q^o_{U, f, u}|_{\Crit_{U'/B'}(f')}\cong Q^{\overline{p}^* o}_{U', f', u'}.
    \end{equation}
    \item For a critical chart $(U, f, u)$ of $(X\rightarrow B, s)$ and another smooth morphism $q\colon B\rightarrow S$ the composite $q_* o|_{\Crit_{U/S}(f)}\rightarrow q_* o^{\can}_{\Crit_{U/B}(f)/B}\xrightarrow{\eqref{eq:pushforwardcanonicalorientation}} o^{\can}_{\Crit_{U/S}(f)/S}$ determines an isomorphism of $\mu_2$-torsors
    \begin{equation}\label{eq:Qpushforward}
    Q^o_{U, f, u}|_{\Crit_{U/B}(f)}\cong Q^{q_* o}_{U, f, u}.
    \end{equation}
    \item For a critical chart $(U, f, u)$ of $(X\rightarrow B, s)$ the composite $\overline{o}\rightarrow \overline{o^{\can}_{\Crit_{U/B}(f)/B}}\xrightarrow{\eqref{eq:canonicalorientationreverse}} o^{\can}_{\Crit_{U/B}(-f)/B}$ determines an isomorphism of $\mu_2$-torsors
    \begin{equation}\label{eq:Qreverse}
    Q^o_{U, f, u}\cong Q^{\overline{o}}_{U, -f, u}.
    \end{equation}
    \item For another morphism of schemes $X'\rightarrow B'$ equipped with a relative d-critical structure $s'$, an orientation $o'$ and a critical chart $(U', f', u')$ the composite $o\boxtimes o'\rightarrow o^{\can}_{\Crit_{U/B}(f)/B}\boxtimes o^{\can}_{\Crit_{U'/B'}(f')/B'}\xrightarrow{\eqref{eq:canonicalorientationproduct}} o^{\can}_{\Crit_{U\times U'/B\times B'}(f\boxplus f')/B\times B'}$ determines an isomorphism of $\mu_2$-torsors
    \begin{equation}\label{eq:Qproduct}
    Q^o_{U, f, u}\boxtimes Q^{o'}_{U', f', u'}\cong Q^{o\boxtimes o'}_{U\times U', f\boxplus f', u\times u'}.
    \end{equation}
\end{enumerate}

\section{Perverse pullbacks}
\label{sec:pervpull}

The goal of this section is to construct a perverse pullback functor for a morphism of higher Artin stacks equipped with a relative d-critical structure. Locally the perverse pullback will be given by the functor of vanishing cycles. To glue it into a global functor, we will construct the stabilization isomorphisms for critical morphisms of critical charts. All schemes are assumed to be separated and of finite type over $\C$.

\subsection{Vanishing cycles for orthogonal bundles}

As the first step to construct stabilization isomorphisms, we will construct them for stabilizations defined by an orthogonal bundle. Given an orthogonal bundle $(E, q)$ over a scheme $U$ we consider the diagram 
\[
\xymatrix{
E \ar[r]^{\q_E} \ar[d]^{\pi_E} & \bA^1 \\
U. \ar@/^1pc/[u]^{0_E}
}
\]
as in \cref{sect:orthogonal}. We will be interested in the following functorialities of orthogonal bundles:
\begin{enumerate}
    \item (\textit{Base change}) Let $p\colon V \to U$ be a morphism of schemes and consider the pullback orthogonal bundle $F=p^* E$ over $V$. Let $p_E\colon F\rightarrow E$ be the corresponding projection morphism. Let $g\coloneqq f\circ p\colon V \to \bA^1$. We have a pullback diagram
    \begin{equation}\label{eq:zerosectionCartesian}
    \begin{tikzcd}[column sep=0.4cm, row sep=0.4cm]
    V \arrow[rr, "0_F"] \arrow[dd, "p"] && F \arrow[dd, "p_E"] \\
    & \Box & \\
    U \arrow[rr, "0_E"] && E
    \end{tikzcd}
    \end{equation}
    The isomorphism $\det(F)\cong p^*\det(E)$ induces an isomorphism
\begin{equation}\label{eq:pullbackorientation}
\ori_{F} \cong p^*\ori_E.
\end{equation}

\item (\textit{Isotropic reduction}) Given an isotropic subbundle $K\subseteq E$, we can form a ``Lagrangian correspondence'' $D\colon F \dashrightarrow E$, that is, a correspondence
\[\xymatrix@C=2cm{
& D \ar@{->>}[ld]^{\pi_{D/F}}_{\textrm{sm.surj}} \ar@{^{(}->}[rd]^{\text{cl.emb}}_{i_{D/E}} & \\
 F && E
}\]
with $D\coloneqq K^\perp$ and $F\coloneqq K^\perp/K$ the reduction of $E$ by $K$. Then we have a fiber square
\begin{equation}
\xymatrix{
& D \ar@{->>}[ld]_{\pi_{D/F}} \ar@{^{(}->}[rd]^{i_{D/E}} & \\
 F \cong F\dual \ar@{_{(}->}[rd]_{\pi_{D/F}\dual} & \Box & E \cong E\dual  \ar@{->>}[ld]^{i_{D/E}\dual} \\
 & D\dual &
}
\end{equation}
and a commutative square
\begin{equation}\label{Eq:15}
\xymatrix{
& D \ar@{->>}[ld]_{\pi_{D/F}}\ar@{^{(}->}[rd]^{i_{D/E}} & \\
 F \ar[rd]_{\q_F} & & E \ar[ld]^{\q_E} \\
 & \bA^1 &
}
\end{equation}
Consider the hyperbolic localization functor
\[\Loc_D\coloneqq (\pi_{D/F})_* i_{D/E}^!\colon \Dbc(E) \to \Dbc(F).\]
Then we have a natural transformation
\begin{equation}\label{Eq:16}
\Ex^\Loc_\phi\colon \phi_{f\circ \pi_F+\q_F}\Loc_D\longrightarrow \Loc_D\phi_{f\circ \pi_E+\q_E}
\end{equation}
given by the composite
\[
\phi_{f\circ \pi_F+\q_F}(\pi_{D/F})_*i_{D/E}^!\xrightarrow{\Ex^\phi_*} (\pi_{D/F})_*\phi_{f\circ\pi_D + \q_F\circ \pi_{D/F}}i_{D/E}^! \xrightarrow{\Ex^!_\phi} (\pi_{D/F})_*i_{D/E}^!\phi_{f\circ \pi_E+\q_E}
\]
and natural isomorphisms
\begin{equation}\label{Eq:17}
\Loc_D (0_E)_* = (\pi_{D/F})_*i^!_{D/E}(0_E)_*\xrightarrow{\Ex^!_*} (\pi_{D/F})_*(0_D)_*\cong (0_F)_*
\end{equation}
and
\begin{equation}\label{Eq:18}
\Loc_D\pi_E^\dag = (\pi_{D/F})_*i^!_{D/E}\pi_E^\dag\cong (\pi_{D/F})_*\pi_{D/F}^\dag\pi_F^\dag[\rank(F)-\rank(D)]\cong \pi_F^\dag,
\end{equation}
where the last isomorphism is given by the composite
\[\id\xrightarrow[\sim]{\unit} (\pi_{D/F})_*\pi_{D/F}^* \xrightarrow[\sim]{\pur_{\pi_{D/F}}}(\pi_{D/F})_*\pi_{D/F}^{\dag}[\rank(F)-\rank(D)].\]
\end{enumerate}

\begin{remark}
The natural transformation $\Ex^\Loc_\phi$ is not necessarily invertible. But \cref{lm:stabilizationmetabolicisomorphism} below shows that it becomes invertible after composing with $\pi_E^\dag$.
\end{remark}

We are ready to define the \defterm{stabilization isomorphisms} for quadratic bundles.

\begin{theorem}\label{thm:stabilizationconstruction}
For a scheme $U$, a function $f\colon U\to \bA^1$ and an orthogonal bundle $E$ of even rank on $U$ there exists a natural isomorphism
\[\stab_{E}\colon (0_E)_* \phi_f(-) \xrightarrow{\sim} \phi_{f\circ \pi_E + \q_E}\pi_E^\dag((-)\otimes_{\mu_2}\ori_E)\]
of functors $\Perv(U)\rightarrow \Perv(E)$ uniquely determined by the following properties:
\begin{enumerate}
    \item (Smooth base change) For a smooth morphism of schemes $p\colon V \to U$, the diagram
    \[\xymatrix@C+2pc{
    p_E^\dag(0_E)_*\phi_f(-) \ar[r]^-{\stab_E} \ar[d]^{\Ex^!_\phi,\Ex^!_*} & p_E^\dag\phi_{f\circ \pi_E +\q_E}\pi_E^\dag((-)\otimes_{\mu_2}\ori_E) \ar[d]^{\Ex^!_\phi\otimes \eqref{eq:pullbackorientation}} \\
    (0_{F})_*\phi_{g}p^\dag(-) \ar[r]^-{\stab_{F}} & \phi_{g \circ \pi_{F} + \q_{F}}\pi_{F}^\dag (p^\dag(-)\otimes_{\mu_2}\ori_F)
    }\]
    commutes, where $g\coloneqq f\circ p$, $F\coloneqq p^*E$, and $p_E\coloneqq p|_F$.

    \item (Isotropic reduction) Given an isotropic subbundle $K \subseteq E$, the diagram
    \[\xymatrix{
    (0_F)_*\phi_f(-) \ar[r]^-{\stab_F} \ar[d]^{\eqref{Eq:17}}_{\sim} & \phi_{f\circ\pi_F+\q_F}\pi_F^\dag((-)\otimes_{\mu_2}\ori_F) \ar[r]^-{\eqref{Eq:18}\otimes\eqref{eq:orientationreduction}}_-{\sim} & \phi_{f\circ\pi_F+\q_F}\Loc_D\pi_E^\dag((-)\otimes\ori_E) \ar[d]^{\Ex^\Loc_\phi} \\
    \Loc_D(0_E)_*\phi_f(-) \ar[rr]^-{\stab_E} && \Loc_D\phi_{f\circ\pi_E +\q_E}\pi_E^\dag((-)\otimes_{\mu_2}\ori_E)
    }\]
    commutes, where $D\coloneqq K^\perp$ and $F\coloneqq K^\perp/K$.
\end{enumerate}
\end{theorem}

From now on $E$ will be an orthogonal bundle of even rank over a scheme $U$. The rest of the section is devoted to the construction of the stabilization isomorphisms and a proof of their properties. One plausible way of constructing stabilization isomorphisms is to consider the local isomorphisms given by the Thom--Sebastiani isomorphism for trivial orthogonal bundles and glue them. However, we then need to show that these local isomorphisms are compatible with the transition maps along changing the orthogonal coordinates which is quite non-trivial---we need a relative version of \cite[Proposition 3.4]{BBDJS}.
Instead, we will use special orthogonal Grassmannians which are smooth with {\em connected} fibers so that a smooth-local construction without considering the transition maps will be sufficient (as the pullback on perverse sheaves is then fully faithful). In the special orthogonal Grassmannians we have Lagrangian subbundles so that we can define the stabilization isomorphisms via reduction by the Lagrangian subbundle.

\subsubsection{Case 1: stabilization isomorphism for metabolic bundles}\label{sss:stabilizationmetabolic}

Assume that there exists a Lagrangian subbundle $M \subseteq E$. Then the reduction of $E$ by $M$ is just $U$. We define the stabilization natural transformation
\begin{equation}\label{eq:stabilizationmetabolic}
\stab_E^M\colon 0_{E,*}\phi_f \longrightarrow \phi_{f\circ \pi_E + \q_E}\pi_E^\dag
\end{equation}
of functors $\Perv(U)\rightarrow \Perv(E)$ as follows. For any $\cF\in\Perv(U)$ the sheaf $\phi_{f\circ \pi_E + \q_E}(\pi_E^\dag \cF)$ is supported on the zero section of $E$ by \cref{prop:constructiblevanishingcycles}(1). Using \eqref{Eq:17} we see that the functor $\Loc_M\colon \Dbc(E)\rightarrow \Dbc(U)$ defines a left inverse to $(0_E)_*\colon\Dbc(U)\rightarrow \Dbc(E)$. Therefore, $\stab_E^M$ is uniquely determined by the commutative square
\begin{equation}\label{eq:stabilizationmetabolicdiagram}
\xymatrix@C+9pc{
\Loc_M0_{E,*}\phi_f \ar@{.>}[r]^-{\stab_E^M} & \Loc_M\phi_{f\circ \pi_E + \q_E}\pi_E^\dag \\
\phi_f \ar[u]^{\eqref{Eq:17}}_{\sim} \ar^{\eqref{Eq:18}}_{\sim}[r] & \phi_f\Loc_M\pi_E^\dag \ar[u]^{\Ex^\Loc_\phi}
}
\end{equation}

\begin{lemma}\label{lm:stabilizationmetabolicisomorphism}
Given a Lagrangian subbundle $M$ of an orthogonal bundle $E$ over a scheme $U$, the natural transformation
\[\Ex^\Loc_\phi\colon \phi_f\Loc_M\pi_E^\dag \to \Loc_M\phi_{f\circ \pi_E+\q_E}\pi_E^\dag\]
is an isomorphism.
\end{lemma}
\begin{proof}
Since the statement is local on $U$, we may assume that $U$ admits a closed immersion $i\colon U\rightarrow U'$ into a smooth scheme $U'$ and there is a vector bundle $M'$ on $U'$ and $E=i^* E'$, where $E'=M'\oplus (M')^\vee$ equipped with the hyperbolic quadratic form. Let $M=i^* M'$. Since $i_*$ is fully faithful, it is enough to show that $i_*\Ex^\Loc_\phi$ is an isomorphism. The diagram
\[
\xymatrix@C=1.5cm{
i_*\phi_f\Loc_M\pi_E^\dag \ar^-{\Ex^\Loc_\phi}[r] \ar^-{\Ex^\phi_*,\Ex^!_*}[d] & i_* \Loc_M \phi_{f\circ \pi_E+\q_E}\pi_E^\dag \ar^{\Ex^\phi_*, \Ex^!_*}[d] \\
\phi_{f'} \Loc_{M'} \pi_{E'}^\dag i_* \ar^-{\Ex^\Loc_\phi}[r] & \Loc_{M'} \phi_{f'\circ \pi_{E'}+\q_{E'}}\pi_{E'}^\dag i_*
}
\]
of exchange natural transformations commutes: it follows from the compatibility of $\Ex^!_*$ and $\Ex^\phi_*$ with compositions as well as the commutative diagram from \cref{prop:constructiblevanishingcycles}(3). As the vertical morphisms are isomorphisms, it is enough to show that the bottom morphism is an isomorphism. Thus, we are reduced to showing the claim for the smooth scheme $U'$.

Consider the $\Gm$-action on $E'$ given by the weight $1$ action on $M'$ and the weight $(-1)$-action on $(M')^\vee$. Then $M'$ is the attractor scheme, so the attractor correspondence for this action is $U'\leftarrow M'\rightarrow E'$. The function $\q_{E'}\colon E'\rightarrow \bA^1$ is $\Gm$-invariant, so the claim follows from the fact that vanishing cycles commute with hyperbolic restriction, \cite[Proposition 5.4.1]{Nakajima}.
\end{proof}

In particular, $\stab^M_E$ is a natural isomorphism.

\begin{remark}
When $f=0$, the statement of \cref{lm:stabilizationmetabolicisomorphism} also follows from the dimensional reduction isomorphism as in \cite[Appendix A]{DavisonCritical}.
\end{remark}

\subsubsection{Case 2: stabilization isomorphism for oriented bundles}

Assume that $E$ has an orientation $o\in\Gamma(U, \ori_E)$. Consider the \defterm{special orthogonal Grassmannian}
\[\OG^+(E)\colon (T \to U) \mapsto \{\text{positive Lagrangian subbundles on $E|_T$}\}.\]
Let $M_E \subseteq E|_{\OG(E)}$ be the universal Lagrangian subbundle. Consider the projection
\[p\colon \OG^+(E) \to U.\]

\begin{lemma}
The morphism $p\colon \OG^+(E)\rightarrow U$ is smooth with connected fibers.
\end{lemma}
\begin{proof}
The question is \'etale local on $U$, so we may assume that $E$ is trivial, i.e. $E = U\times V$ for a non-degenerate quadratic space $(V, q)$ which admits a Lagrangian. In this case $\OG^+(E)\cong U\times \OG^+(V)$. Moreover, $\OG^+(V)$ is a homogeneous space for the special orthogonal group $\SO(q)$. But $\SO(q)$ is smooth and connected, which implies that $\OG^+(V)$ is smooth and connected.
\end{proof}

Therefore, the pullback
\[p^\dag\colon \Perv(U) \to \Perv(\OG^+(E))\]
is fully faithful by \cref{prop:constructible6functor}(6b), i.e., for any $\scF,\scG \in \Perv(U)$, the pullback
\begin{equation}\label{14}
p^\dag\colon \Hom_{\Perv(U)}(\scF,\scG) \xrightarrow{\sim} \Hom_{\Perv(\OG^+(E))}(p^\dag\scF,p^\dag\scG)
\end{equation}
is an isomorphism. We define the stabilization natural isomorphism
\begin{equation}\label{eq:stabilizationoriented}
\stab_E^o\colon 0_{E,*}\phi_f \xrightarrow{\sim} \phi_{f\circ \pi_E + \q_E}\pi_E^\dag
\end{equation}
of functors $\Perv(U)\rightarrow \Perv(E)$ as the unique natural transformation that fits into the commutative square
\[
\xymatrix@C+4pc{
p_E^\dag 0_{E,*}\phi_f  \ar@{.>}[r]^-{\stab_E^o}_-{\sim} \ar[d]^{\Ex^!_\phi,\Ex^!_*}_{\sim} & p_E^\dag\phi_{f\circ \pi_E +\q_E}\pi_E^\dag \ar[d]^{\Ex^!_\phi}_{\sim} \\
0_{E|_{\OG^+(E)},*}\phi_{f|_{\OG^+(E)}}p^\dag \ar[r]^-{\stab_{E|_{\OG^+(E)}}^{M_E}}_-{\sim} & \phi_{f|_{\OG^+(E)}\pi_{E|_V} + \q_{E|_{\OG^+(E)}}}\pi_{E|_{\OG^+(E)}}^\dag p^\dag,
}\]
with respect to the isomorphism \eqref{14}, where $\stab_{E|_{\OG^+(E)}}^{M_E}$ is the stabilization morphism defined in \eqref{eq:stabilizationmetabolic}.

\subsubsection{Case 3: stabilization isomorphism in general}

Assume now $E$ is arbitrary. Consider the orientation bundle $\ori_E$ geometrically: we have an {\'e}tale surjection $p\colon \ori_E \to U$ and an automorphism $\sigma\colon \ori_E \xrightarrow{\sim} \ori_E$ over $U$ exchanging the two sheets such that $\sigma^2=1$.
\[\xymatrix@C-1pc{
\ori_E \ar[rr]^-{\sigma} \ar[rd]_p && \ori_E \ar[ld]^p \\& U&
}\]

By descent, for any $\scF,\scG \in \Perv(U)$, the pullback establishes an isomorphism
\begin{equation}\label{15}
p^\dag\colon \Hom_{\Perv(U)}(\scF,\scG) \xrightarrow{\sim} \Hom_{\Perv(\ori_E)}(p^\dag \scF,p^\dag\scG)^\sigma,
\end{equation}
where $(-)^\sigma$ denotes the $\sigma$-invariant subspace. Hence to construct the stabilization isomorphism in general, it suffices to know how the stabilization isomorphism from \eqref{eq:stabilizationoriented} changes under the change of the orientation.

\begin{proposition}\label{prop:stabilizationorientation}
Let $E$ be an orthogonal bundle over $U$ equipped with an orientation $o\in\Gamma(U, \ori_E)$. Then
\[\stab_E^{-o} = - \stab_E^o\colon 0_{E,*}\phi_f \xrightarrow{\sim} \phi_{f\circ \pi_E + \q_E}\pi_E^\dag.\]
\end{proposition}
\begin{proof}
Let $V$ be the standard one-dimensional space $\C$ equipped with the quadratic form $x^2$. The statement is \'etale local, so we may assume that $E=U\times V^{2n}$ for some $n$. By definition of $\stab_E^o$ it is enough to show that $\stab_E^{M_1} = -\stab_E^{M_2}$ for any two Lagrangian subbundles $M_1, M_2\subset E$ which induce different orientations of $E$. In turn, by naturality of the stabilization isomorphism it is equivalent to showing that for some $A\in\rO(2n)$ with $\det(A) = -1$ the induced action on $\phi_{f\boxplus q}(-\boxtimes \omega_{V^{2n}})$ is by $(-1)$. Using the Thom--Sebastiani isomorphism $\TS$ this boils down to showing that $x\mapsto -x$ acts by $(-1)$ on $\phi_{x^2}\omega_V$.

The morphism $q=x^2\colon\bA^1\rightarrow \bA^1$ is finite, so we get an isomorphism $\Ex^\phi_*\colon q_*\phi_{x^2}\omega_{\bA^1}\cong \phi_x q_*\omega_{\bA^1}$. The sheaf $q_* \omega_{\bA^1}\in\Dbc(\bA^1)$ fits into a fiber sequence
\[j_! \cK[2] \longrightarrow q_* \omega_{\bA^1}\longrightarrow \omega_{\bA^1},\]
where $j\colon \Gm\rightarrow \bA^1$ is the inclusion of the complement of the origin and $\cK$ is lisse of rank 1. Therefore, $(\phi_t q_*\omega_{\bA^1})_{\{0\}}\cong \cK|_1[2]$. By proper base change, we have a short exact sequence
\[0\longrightarrow \cK|_1\longrightarrow \rH^{\mathrm{BM}}_0(q^{-1}(1))\longrightarrow R\longrightarrow 0,\]
where the second morphism is the trace map. Since $t\mapsto -t$ acts by permuting the fibers of $q^{-1}(1)$, it acts by $-1$ on $\cK|_1$, which proves the claim.
\end{proof}

The pullback $E|_{\ori_E}$ admits a tautological orientation $o_E$. Since we have $\sigma^* (o_E) = -o_E$, the map
\[\stab_{E|_{\ori_E}}^{o_E}(-) \otimes o_E\colon  0_{E|_{\ori_E},*}\phi_{p^*f} (-) \xrightarrow{\sim} \phi_{p^*f\circ \pi_{E|_{\ori_E}} + \q_{E|_{\ori_E}}}\pi_{E|_{\ori_E}}^\dag(-\otimes_{\mu_2} \ori_{E|_{\ori_E}}) \]
is $\sigma$-invariant and hence descends to an isomorphism
\begin{equation}\label{eq:stabilizationgeneral}
\stab_E\colon 0_{E, *} \phi_f(-)\xrightarrow{\sim} \phi_{f\circ \pi_E + \q_E}\pi_E^\dag(-\otimes_{\mu_2}\ori_E).
\end{equation}

\subsubsection{Proof of \cref{thm:stabilizationconstruction}}

The stabilization isomorphism is given by \eqref{eq:stabilizationgeneral}. The uniqueness is clear:
\begin{itemize}
    \item Using the base change property for the orientation torsor $\ori_E\rightarrow U$ the stabilization isomorphism for a general even rank orthogonal bundle is uniquely determined by the stabilization isomorphism for an oriented even rank orthogonal bundle.
    \item Using the base change property for the special orthogonal Grassmannian $\OG^+(E)\rightarrow U$ the stabilization isomorphism for an oriented even rank orthogonal bundle is uniquely determined by the stabilization isomorphism for a metabolic bundle.
    \item The isotropic reduction property uniquely determines the stabilization isomorphism for a metabolic bundle using the diagram \eqref{eq:stabilizationmetabolicdiagram}.
\end{itemize}

Let us now prove the relevant properties:
\begin{enumerate}
    \item (Smooth base change) The claim is local, so we may assume that $E$ admits a Lagrangian $M_E\subset E$. Let $F=p^*E$, $M_F=p^* M_E$, $p_E = p|_F$ and $p_M= p|_{M_F}$. Consider the diagram
    \begin{equation}\label{eq:Lagrangiansubbundlepullback}
    \begin{tikzcd}
    V \arrow[d, "p"] & M_F \dar \arrow[l, "\pi_{M_F}" above] \arrow[r, "i_{M_F/F}"] & F \arrow[d, "p_E"] \\
    U & M_E \arrow[l, "\pi_{M_E}"] \arrow[r, "i_{M_E/E}" below] & E
    \end{tikzcd}
    \end{equation}
    where both squares are Cartesian. This gives rise to a base change isomorphism
    \begin{equation}\label{eq:basechangestabilization1}
    p^\dag \Loc_{M_E}\cong \Loc_{M_F}p_E^\dag
    \end{equation}
    determined by $\Ex^!_*$. We have to show that the outer rectangle in the diagram
    \[
    \begin{tikzcd}[column sep=1.5cm]
    p^\dag \Loc_{M_E} 0_{E, *} \phi_f \ar{d}{\eqref{eq:basechangestabilization1},\Ex^!_\phi,\Ex^!_*} & p^\dag\phi_f \arrow[l, "\eqref{Eq:17}" above] \arrow[r, "\eqref{Eq:18}"] \arrow[d, "\Ex^!_\phi"] & p^\dag\phi_f \Loc_{M_E}\pi_E^\dag \ar{d}{\eqref{eq:basechangestabilization1}, \Ex^!_\phi} \arrow[r, "\Ex^\Loc_{\phi}"] & p^\dag\Loc_{M_E}\phi_{f\circ \pi_E + \q_E}\pi_E^\dag \ar{d}{\eqref{eq:basechangestabilization1},\Ex^!_\phi} \\
    \Loc_{M_F}0_{F, *}\phi_g p^\dag & \phi_g p^\dag \arrow[l, "\eqref{Eq:17}" above] \arrow[r, "\eqref{Eq:18}"] & \phi_g\Loc_{M_F}\pi_F^\dag p^\dag \arrow[r, "\Ex^\Loc_\phi"] & \Loc_{M_F}\phi_{g\circ \pi_F + \q_F}\pi_F^\dag p^\dag
    \end{tikzcd}
    \]
    commutes. This follows from the commutativity of the individual squares:
    \begin{itemize}
        \item The commutativity of the leftmost square follows by applying the 6-functor formalism $\bD$ to the 3-cell
        \[
        \xymatrix{
        &&& V \ar^{p}[dl] \ar^{0_{M_F}}[dr] &&& \\
        && U \ar@{=}[dl] \ar^{0_{M_E}}[dr] && M_F \ar^{p_M}[dl] \ar^{\pi_{M_F}}[dr] && \\
        & U \ar@{=}[dl] \ar^{0_E}[dr] && M_E \ar^{i_{M_E/E}}[dl] \ar^{\pi_{M_E}}[dr] && V \ar^{p}[dl] \ar@{=}[dr] \\
        U && E && U && V
        }
        \]
        in the $\infty$-category of correspondences.
        \item The commutativity of the middle square follows from the fact that the exchange natural transformation $\Ex^!_*$ intertwines the unit of the adjunction $\pi^*_{M_E}\dashv (\pi_{M_E})_*$ and the unit of the adjunction $\pi^*_{M_F}\dashv (\pi_{M_F})_*$.
        \item The commutativity of the rightmost square follows from \cref{prop:constructiblevanishingcycles}(2)-(3).
    \end{itemize}
    \item (Isotropic reduction) The claim is local, so we may find a Lagrangian subbundle $M\subset F$ together with a splitting $F=M\oplus M^\vee$ and $E=F\oplus K\oplus K^\vee$. We have an isomorphism
    \begin{equation}\label{eq:basechangestabilization2}
    \Loc_{M\oplus K}\cong \Loc_{M}\Loc_D
    \end{equation}
    induced by $\Ex^!_*$. Then the claim reduces to the commutativity of the diagram
    \[
    \begin{tikzcd}[column sep=1.7cm]
    \phi_f\Loc_{M\oplus K} \arrow[rr, "\Ex^\Loc_\phi"] \arrow[d, "\eqref{eq:basechangestabilization2}"] && \Loc_{M\oplus K}\phi_{f\circ \pi_E + \q_E} \arrow[d, "\eqref{eq:basechangestabilization2}"] \\
    \phi_f\Loc_{M}\Loc_{D} \arrow[r, "\Ex^\Loc_\phi"] & \Loc_{M}\phi_{f\circ \pi_F + \q_F}\Loc_{D} \arrow[r, "\Ex^\Loc_\phi"] & \Loc_{M}\Loc_{D}\phi_{f\circ \pi_E + \q_E}
    \end{tikzcd}
    \]
    which follows from \cref{prop:constructiblevanishingcycles}(3).
\end{enumerate}

\subsubsection{Properties of the stabilization isomorphisms}

\begin{proposition}\label{prop:Linearization-Whitneysum}
Let $U$ be a scheme, $f\colon U\rightarrow \bA^1$ a function and $E, F$ orthogonal bundles of even rank over $U$. Then the diagram
\[\xymatrix@C-6pc{
0_{E\oplus F,*} \phi_f(-)\ar[rr]^-{\stab_{E\oplus F}}  \ar[d]_{\sim} && \phi_{f \circ \pi_{E\oplus F} +\q_{E\oplus F}}\pi_{E\oplus F}^\dag(-\otimes\ori_{E\oplus F})\ar[d]^{\eqref{eq:orientationsumisomorphism}} \\
0_{F|_E,*} 0_{E,*} \phi_f(-) \ar[rd]_{\stab_E} &  & \phi_{f\circ \pi_E \circ \pi_{F|_E} +\q_E \circ \pi_{F|_E} +\q_{F|_E}} \pi_{F|_E}^\dag (\pi_E^\dag(-\otimes_{\mu_2}\ori_E)\otimes_{\mu_2}\ori_{F|_E}) \\
& 0_{F|_E,*} \phi_{f\circ \pi_E + \q_E}\pi_E^\dag(-\otimes_{\mu_2}\ori_E)\ar[ru]_-{\stab_{F|_E}} 
}\]
is commutative.
\end{proposition}
\begin{proof}
The claim is local, so we may assume that $F$ admits a Lagrangian subbundle $M\subset F$. Let $K = 0\oplus M\subset E\oplus F$ be the corresponding isotropic subbundle. Then $K^\perp = E\oplus M$ and hence $K^\perp/K = E$. Writing $\stab_{F|_E}$ locally as $\stab^M_{F|_E}$ the claim reduces to the isotropic reduction property of the stabilization isomorphism from \cref{thm:stabilizationconstruction}.
\end{proof}

Next we prove a finite base change property.

\begin{proposition}\label{prop:stabilizationfinitebasechange}
Let $U$ be a scheme, $E$ an orthogonal bundle of even rank over $U$, $f\colon U\rightarrow \bA^1$ a function and $p\colon V\rightarrow U$ a finite morphism. Let $g\coloneqq f\circ p$, $F\coloneqq p^*E$, and $p_E\coloneqq p|_E$. The diagram
\[\xymatrix@C+2pc{
(0_E)_*\phi_f p_*(-) \ar[r]^-{\stab_E}\ar[d]^{\Ex^\phi_*} & \phi_{f\circ \pi_E +\q_E}\pi_E^\dag(p_*(-)\otimes_{\mu_2}\ori_E) \ar[d]^{\Ex^\phi_*,\Ex^!_*\otimes\eqref{eq:pullbackorientation}} \\
(p_E)_*(0_F)_*\phi_{g}(-) \ar[r]^-{\stab_{F}} & (p_E)_*\phi_{g \circ \pi_{F} + \q_{F}}\pi_{F}^\dag(-\otimes_{\mu_2}\ori_F)
}\]
commutes.
\end{proposition}
\begin{proof}
The claim is local, so as in the proof of the smooth base change property we may assume that $E$ admits a Lagrangian $M_E\subset E$. Let $F=p^*E$ and $M_F=p^* M_E$. Diagram \eqref{eq:Lagrangiansubbundlepullback} induces a base change isomorphism
\begin{equation}\label{eq:finitebasechangestabilization}
\Loc_{M_E}(p_E)_*\cong p_*\Loc_{M_F}
\end{equation}
determined by $\Ex^!_*$. We have to show that the outer rectangle in the diagram
    \[
    \begin{tikzcd}[column sep=1.5cm]
    \Loc_{M_E} 0_{E, *} \phi_f p_* \ar{d}{\Ex^\phi_*, \eqref{eq:finitebasechangestabilization}} & \phi_fp_* \arrow[l, "\eqref{Eq:17}" above] \arrow[r, "\eqref{Eq:18}"] \ar{d}{\Ex^\phi_*} & \phi_f \Loc_{M_E}\pi_E^\dag p_* \ar{d}{\Ex^!_*, \eqref{eq:finitebasechangestabilization}, \Ex^\phi_*} \arrow[r, "\Ex^\Loc_{\phi}"] & \Loc_{M_E}\phi_{f\circ \pi_E + \q_E}\pi_E^\dag p_* \ar{d}{\Ex^!_*, \Ex^\phi_*, \eqref{eq:finitebasechangestabilization}} \\
    p_* \Loc_{M_F}0_{F, *}\phi_g & p_*\phi_g \arrow[l, "\eqref{Eq:17}" above] \arrow[r, "\eqref{Eq:18}"] & p_*\phi_g\Loc_{M_F}\pi_F^\dag \arrow[r, "\Ex^\Loc_\phi"] & p_*\Loc_{M_F}\phi_{g\circ \pi_F + \q_F}\pi_F^\dag
    \end{tikzcd}
    \]
    commutes. This follows from the commutativity of the individual squares:
    \begin{itemize}
        \item The commutativity of the leftmost square follows by applying sheaf theory $\bD$ to the 3-cell
        \[
        \xymatrix{
        &&& V \ar@{=}[dl] \ar^{0_{M_F}}[dr] &&& \\
        && V \ar@{=}[dl] \ar^{0_{M_F}}[dr] && M_F \ar@{=}[dl] \ar^{p_M}[dr] && \\
        & V \ar@{=}[dl] \ar^{0_F}[dr] && M_F \ar^{i_{M_F/F}}[dl] \ar^{p_M}[dr] && M_E \ar@{=}[dl] \ar^{\pi_{M_E}}[dr] & \\
        V && F && M_E && U
        }
        \]
        in the $\infty$-category of correspondences.
        \item The commutativity of the middle square follows from the fact that the exchange natural transformation $\Ex^!_*$ intertwines the unit of the adjunction $\pi^*_{M_E}\dashv (\pi_{M_E})_*$ and the unit of the adjunction $\pi^*_{M_F}\dashv (\pi_{M_F})_*$.
        \item The commutativity of the rightmost square follows from \cref{prop:constructiblevanishingcycles}(2)-(3).
    \end{itemize}
\end{proof}

Next we prove a compatibility with Verdier duality.

\begin{proposition}\label{prop:stabilizationduality}
Let $U$ be a scheme, $E$ an orthogonal bundle of even rank over $U$ and $f\colon U\rightarrow \bA^1$ a function. The diagram
\[
\xymatrix{
(0_E)_*\phi_f\bbD(-)\ar^-{\stab_E}[r] \ar^{\Ex^{\phi, \bbD},\Ex_{*, \bbD}}[d] & \phi_{f\circ \pi_E+\q_E} \pi^\dag_E(\bbD(-)\otimes_{\mu_2}\ori_E) \ar^{\Ex^{\dag, \bbD},\Ex^{\phi, \bbD}\otimes\eqref{eq:orientationreverse}}[d] \\
\bbD (0_E)_* \phi_{-f}(-) & \bbD\phi_{-f\circ \pi_E-\q_E} \pi^\dag_E(-\otimes_{\mu_2}\ori_{\overline{E}}) \ar_-{\stab_{\overline{E}}}[l]
}
\]
commutes.
\end{proposition}
\begin{proof}
The claim is local, so we may assume $E=M\oplus M^\vee$ for a pair of Lagrangian subbundles $M,M^\vee\subset E$. The orientations of $E$ induced by $M$ and $M^\vee$ differ by $(-1)^{\rank(E)/2}$. Similarly, the morphism \eqref{eq:orientationreverse} acts on volume forms by $(-1)^{\rank(E)/2}$. Thus, we have to show that the diagram
\[
\xymatrix@C=1.5cm{
(0_E)_*\phi_f\bbD(-)\ar^-{\stab^M_E}[r] \ar^{\Ex^{\phi, \bbD},\Ex_{*, \bbD}}[d] & \phi_{f\circ \pi_E+\q_E} \pi^\dag_E\bbD(-) \ar^{\Ex^{\dag, \bbD},\Ex^{\phi, \bbD}}[d] \\
\bbD (0_E)_* \phi_{-f}(-) & \bbD\phi_{-f\circ \pi_E-\q_E} \pi^\dag_E(-) \ar_-{\stab^{M^\vee}_{\overline{E}}}[l]
}
\]
commutes. For this, it is sufficient to show that it commutes after applying $\Loc_M$, as in the definition of $\stab^M_E$.

Consider the Cartesian diagram
\[
\xymatrix{
U \ar^{0_M}[r] \ar^{0_{M^\vee}}[d] & M \ar^{i_M}[d] \\
M^\vee \ar^{i_{M^\vee}}[r] & E.
}
\]
Then we have a natural transformation $H\colon \pi_{M^\vee, *} i_{M^\vee}^!\rightarrow \pi_{M, !} i_M^*$ defined as the mate of the composite isomorphism
\[\id\cong \pi_{M, !} 0_{M, !} 0_{M^\vee}^* \pi^*_{M^\vee}\xrightarrow{\Ex^*_!} \pi_{M, !} i_M^* i_{M^\vee, !} \pi^*_{M^\vee}.\]
By \cite{Braden} (considering the $\Gm$-action on $E$ with $M$ of weight $1$ and $M^\vee$ of weight $-1$) $H$ is an isomorphism on $\Gm$-equivariant complexes; in particular, it is an isomorphism on the subcategory of constructible complexes on $E$ supported on $U$.

Using the definition of the stabilization isomorphism and commuting $\Loc_M$ past $\bbD$ the claim reduces to the following ones:
\begin{enumerate}
    \item The diagram
    \[
    \xymatrix{
    \id \ar^{\sim}[r] \ar^{\sim}[d] & \pi_{M, !} 0_{M, !} \ar^-{\Ex^*_!}[r] & \pi_{M, !} i_M^* 0_{E, !} \ar^{\fgsp_{0_E}}[d] \\
    \pi_{M^\vee, *} 0_{M^\vee, *} \ar^-{\Ex^!_*}[r] & \pi_{M^\vee, *} i_{M^\vee}^! 0_{E, *} \ar^{H}[r] & \pi_{M, !} i_M^* 0_{E, *}
    }
    \]
    commutes. Plugging in the definitions of $\Ex^!_*$, $H$ and $\fgsp_{0_E}$ as mates of the respective natural transformations, the claim follows from the compatibility of $\Ex^*_!$ with compositions.
    \item The diagram
    \[
    \xymatrix{
    \pi_{M, !} i_M^* \pi_E^*[\rank(E)] \ar^{\sim}[r] \ar^{\pur_{\pi_E}}[d] & \pi_{M, !} \pi_M^*[2\rank(M)] \ar^-{\pur_{\pi_M}}[r] & \pi_{M, !} \pi_M^! \ar^{\counit}[r] & \id \ar^{\unit}[d] \\
    \pi_{M, !} i_M^* \pi_E^![-\rank(E)] & \pi_{M^\vee, *} i_{M^\vee}^! \pi_E^![-\rank(E)] \ar_{H}[l] & \pi_{M^\vee, *} \pi_{M^\vee}^![-2 \rank(M)] \ar_{\sim}[l] & \pi_{M^\vee, *} \pi_{M^\vee}^* \ar_-{\pur_{\pi_{M^\vee}}}[l] 
    }
    \]
    commutes. Unpacking the definition of $H$ as a mate, the commutativity of this diagram reduces to the commutativity of
    \[
    \scalebox{0.85}{
    \xymatrix@C=2cm{
    \id \ar^-{\sim}[r] \ar@{=}[d] & \pi_{M, !} 0_{M, !} 0_{M^\vee}^* \pi_{M^\vee}^* \ar^{\Ex^*_!}[r] & \pi_{M, !} i_M^* i_{M^\vee, !} \pi^*_{M^\vee} \ar^-{\unit}[r] & \pi_{M, !} i_M^* \pi_E^! \pi_{E, !} i_{M^\vee, !} \pi^*_{M^\vee} \\
    \id & \pi_{M, !} \pi_M^! \pi_{M^\vee, !} \pi^!_{M^\vee} \ar_-{\counit, \counit}[l] & \pi_{M, !} \pi_M^* \pi_{M^\vee, !} \pi^*_{M^\vee} [2\rank(E)] \ar_{\pur_{\pi_M}, \pur_{\pi_{M^\vee}}}[l] & \pi_{M, !} i_M^* \pi_E^* \pi_{E, !} i_{M^\vee, !} \pi^*_{M^\vee} [2\rank(E)] \ar_{\pur_{\pi_E}}[u] \ar_{\sim}[l]
    }
    }
    \]
    which follows from the compatibility of the purity isomorphism with respect to compositions.
    \item The diagram
    \[
    \xymatrix{
    \phi_f \pi_{M^\vee, *} i_{M^\vee}^! \ar^-{\Ex^\phi_*}[r] \ar^{H}[d] & \pi_{M^\vee, *} \phi_{f\circ \pi_{M^\vee}} i_{M^\vee}^! \ar^-{\Ex^!_\phi}[r] & \pi_{M^\vee, *} i_{M^\vee}^! \phi_{f\circ \pi_E + \q_E} \ar^{H}[d] \\
    \phi_f \pi_{M, !} i_M^* & \pi_{M, !} \phi_{f\circ \pi_M}i_M^* \ar_-{\Ex^\phi_!}[l] & \pi_{M, !} i_M^* \phi_{f\circ \pi_E + \q_E} \ar_-{\Ex^*_\phi}[l]
    }
    \]
    commutes. Plugging in the definitions of $H$, $\Ex^\phi_!$ and $\Ex^*_\phi$ as mates, the commutativity of this diagram follows from \cref{prop:constructiblevanishingcycles}(3).
\end{enumerate}
\end{proof}

Finally, we prove a compatibility with products.

\begin{proposition}\label{prop:stabilizationproducts}
Let $U,V$ be schemes, $E$ an orthogonal bundle of even rank over $U$ and $f\colon U\rightarrow \bA^1$ and $g\colon V\rightarrow \bA^1$ two functions. Let $E'=E\times V$ be the pullback orthogonal bundle over $U\times V$. Then the diagram
\[
\xymatrix@C=1.5cm{
(0_E)_*\phi_f(-)\boxtimes \phi_g(-) \ar^-{\stab_E\boxtimes \id}[r] \ar^{\TS}[d] & \phi_{f\circ \pi_E + \q_E} \pi_E^\dag(-\otimes_{\mu_2}\ori_E)\boxtimes \phi_g(-) \ar^{\TS\otimes\eqref{eq:pullbackorientation}}[d] \\
(0_{E'})_*\phi_{f\boxplus g}(-\boxtimes -) \ar^-{\stab_{E'}}[r] & \phi_{f\circ \pi_E + \q_E \boxplus g}\pi_{E'}^\dag((-\boxtimes -)\otimes_{\mu_2}\ori_{E'})
}
\]
commutes.
\end{proposition}
\begin{proof}
The claim is local, so we may assume that there is a Lagrangian subbundle $M\subset E$. Let $M'\subset E'$ be its pullback to $U\times V$. Then we have to show that the outside rectangle in
\[
\scalebox{0.9}{
\xymatrix{
\Loc_M(0_E)_*\phi_f(-)\boxtimes \phi_g(-) \ar^{\TS}[d] & \phi_f(-)\boxtimes \phi_g(-) \ar_-{\eqref{Eq:17}}[l] \ar^-{\eqref{Eq:18}}[r] \ar^{\TS}[d] & \phi_f\Loc_M\pi_E^\dag(-)\boxtimes \phi_g(-) \ar^-{\Ex^\Loc_\phi}[r] \ar^{\TS}[d] & \Loc_M\phi_{f\circ \pi_E + \q_E}\pi_E^\dag(-)\boxtimes \phi_g(-) \ar^{\TS}[d] \\
\Loc_{M'}(0_{E'})_*\phi_{f\boxplus g}(-\boxtimes -) & \phi_{f\boxplus g}(-\boxtimes -) \ar_-{\eqref{Eq:17}}[l] \ar^-{\eqref{Eq:18}}[r] & \phi_{f\boxplus g}\Loc_{M'}\pi_{E'}^\dag(-\boxtimes -) \ar^-{\Ex^\Loc_\phi}[r] & \Loc_{M'}\phi_{f\circ \pi_E + \q_E\boxplus g}\pi_{E'}^\dag(-\boxtimes -)
}}
\]
commutes. But in the above diagram individual squares commute: this follows from the compatibility of the isomorphisms $\Ex^!_*$, the unit $\id\rightarrow \pi_{M, *}\pi_M^*$, $\Ex^\phi_*$ and $\Ex^!_\phi$ with products.
\end{proof}

\subsection{Symmetries of vanishing cycles}

In this section we construct and analyze stabilization isomorphisms for \'etale morphisms of critical charts.

Let $\Phi\colon (U, f)\rightarrow (V, g)$ be a smooth morphism of LG pairs over $B$, so that by \cref{prop:smoothcriticallocus} $\Phi$ restricts to a smooth morphism $\Phi\colon \Crit_{U/B}(f)\rightarrow \Crit_{V/B}(g)$ of the same relative dimension $\dim(U/V)$. There is a natural comparison isomorphism
\[\phi_f(U\rightarrow B)^\dag\xrightarrow{\Ex^!_\phi}\Phi^\dag\phi_g(V\rightarrow B)^\dag\xleftarrow{\pur_\Phi} \Phi^*\phi_g(V\rightarrow B)^\dag[\dim(U/V)]\]
of functors $\Perv(B)\rightarrow \Perv(U)$. By \cref{prop:constructiblevanishingcycles}(1) the image of these functors is supported on the relative critical locus $\Crit_{U/B}(f)$ and thus we get a natural isomorphism
\begin{equation}\label{eq:vanishingcyclesexchange}
\Ex_{\Phi}\colon (\phi_f(U\rightarrow B)^\dag)(-)|_{\Crit_{U/B}(f)}\xrightarrow{\sim}(\phi_g(V\rightarrow B)^\dag)(-)|_{\Crit_{U/B}(f)}[\dim(U/V)].
\end{equation}

\begin{lemma}\label{lm:Excomposite}
Let $(U, f)\xrightarrow{\Phi} (V, g)\xrightarrow{\Psi} (W, h)$ be a composite of smooth morphisms of LG pairs over $B$. Then
\[\Ex_{\Psi\circ \Phi} = \Ex_{\Psi}|_{\Crit_{U/B}(f)}\circ \Ex_{\Phi}.\]
\end{lemma}
\begin{proof}
The claim follows from the functoriality of $\Ex^!_\phi$ with respect to compositions, \cref{prop:constructiblevanishingcycles}(2), as well as the functoriality of the purity isomorphism, \cref{prop:fundamentalclass}.
\end{proof}

\begin{example}\label{ex:exchangequadratic}
Let $(U, f)$ be an LG pair over $B$ and $(V, q)$ a vector space equipped with a nondegenerate quadratic form. Let $M$ be an orthogonal automorphism of $(V, q)$ and consider an automorphism $\id\times M$ of the LG pair $(U\times V, f\boxplus q)$ over $B$. Then
\[\Ex_{\id\times M} = \det(M) \Ex_{\id},\]
where $\det(M)=\pm 1$ since $M$ is orthogonal. Indeed, using the naturality of the stabilization isomorphism from \cref{thm:stabilizationconstruction} it reduces to the fact that $M\colon V\rightarrow V$ acts by $\det(M)$ on the orientation $\mu_2$-torsor.
\end{example}

The following statement, which is a family version of \cite[Theorem 3.1]{BBDJS}, explains that $\Ex_\Phi$ depends, up to an explicit sign, only on the \'etale morphism
\[\Phi|_{\Crit_{U/B}(f)}\colon \Crit_{U/B}(f)\longrightarrow \Crit_{V/B}(g).\]

\begin{theorem}\label{thm:BBDJS31}
Let $\Phi_0,\Phi_1\colon (U, f)\rightarrow (V, g)$ be \'etale morphisms of LG pairs over $B$. Let $X=\Crit_{U/B}(f)$ and $Y=\Crit_{V/B}(g)$ and assume that
\[\Phi_0|_X = \Phi_1|_X\colon X\longrightarrow Y.\]
\begin{enumerate}
    \item Consider the induced isomorphisms
    \[d\Phi_0|_X, d\Phi_1|_X\colon \T_{U/B}|_{X^{\red}}\longrightarrow \T_{V/B}|_{Y^{\red}}.\]
    Then $\det(d\Phi_1|_{X^{\red}}^{-1}\circ d\Phi_0|_{X^{\red}})\colon X^{\red}\rightarrow \bA^1$ is a locally constant map with values $\pm 1$.
    \item We have
    \[\Ex_{\Phi_0} = \det(d\Phi_1|_{X^{\red}}^{-1}\circ d\Phi_0|_{X^{\red}})\cdot \Ex_{\Phi_1}.\]
\end{enumerate}
\end{theorem}

\begin{proof}[Proof of \cref{thm:BBDJS31}(1)]
The function $\Delta=\det(d\Phi_1|_{X^{\red}}^{-1}\circ d\Phi_0|_{X^{\red}})\colon X^{\red}\rightarrow \bA^1$ is a function on a reduced scheme. Therefore, the equation $\Delta^2=1$ can be checked pointwise on $X^{\red}$. For a $k$-point $b\in B$ let
\[U_b = U\times_B \Spec k,\qquad V_b = V\times_B \Spec k\]
be the corresponding smooth schemes over $k$ and $X_b = \Crit_{U_b}(f)$.
Then it is sufficient to show that $\Delta^2|_{X_b}=1$. But this is the content of \cite[Theorem 3.1 (a)]{BBDJS}.
\end{proof}

Next we will prove \cref{thm:BBDJS31}(2) under an additional assumption.

\begin{proposition}\label{prop:BBDJS35}
Consider the setting of \cref{thm:BBDJS31} and suppose that for a point $u\in \Crit_{U/B}(f)$ we have
\[(d\Phi_1|_u^{-1}\circ d\Phi_0|_u - \id)^2 = 0\colon \T_{U/B, u}\rightarrow \T_{U/B, u}.\]
For each perverse sheaf $\cF\in\Perv(B)$ there is an open neighborhood $X^\circ\subset \Crit_{U/B}(f)$ of $u$ such that $\Ex_{\Phi_0}(\cF)|_{X^\circ} = \Ex_{\Phi_1}(\cF)|_{X^\circ}$.
\end{proposition}
\begin{proof}
Applying \cref{prop:BBDJS34}, we obtain an LG pair $(W, h)$ over $B\times \bA^1$ together with \'etale morphisms
\[\xymatrix{
& (W,h) \ar[ld]_{\Psi_U} \ar[rd]^{\Psi_V} & \\
(U \times \bA^1, f\boxplus 0)   & & (V \times \bA^1, g \boxplus 0)
}\]
and a map $w\colon \bA^1 \to W$ satisfying the conditions of \cref{prop:BBDJS34}. Consider the following data:
\begin{itemize}
    \item $\cP = (\phi_f(U\rightarrow B)^\dag)(\cF)|_{\Crit_{U/B}(f)}\in\Perv(\Crit_{U/B}(f))$.
    \item $\cQ = (\phi_g(V\rightarrow B)^\dag)(\cF)|_{\Crit_{U/B}(f)}\in\Perv(\Crit_{U/B}(f))$.
    \item An isomorphism $\alpha\colon \cP\rightarrow \cQ$ in $\Perv(\Crit_{U/B}(f))$ given by $\Ex_{\Phi_0}$.
    \item An isomorphism $\beta\colon \cP\rightarrow \cQ$ in $\Perv(\Crit_{U/B}(f))$ given by $\Ex_{\Phi_1}$.
    \item An isomorphism $\gamma\colon (\cP\boxtimes R_{\bA^1}[1])|_{\Crit_{W/B\times \bA^1}(h)}\rightarrow (\cQ\boxtimes R_{\bA^1}[1])|_{\Crit_{W/B\times \bA^1}(h)}$ given by $\Ex_{\Psi_V}\circ \Ex_{\Psi_U}^{-1}$.
\end{itemize}

From the commutative diagram
\[
\xymatrix{
& W_0 \ar[ld]_{\Psi_{U,0}} \ar[rd]^{\Psi_{V,0}} & \\
U \ar[rr]^{\Phi_0} & &  V , 
}
\]
we get $\gamma|_{t=0}[-1] = \alpha|_{\Crit_{W_0/B}(h_0)}$, where $h_t\colon W_t\rightarrow \bA^1$ is the restriction of $h\colon W\rightarrow \bA^1$ to the fiber at $t\in\bA^1$. Similarly, from the commutative diagram
\[
\xymatrix{
& W_1 \ar[ld]_{\Psi_{U,1}} \ar[rd]^{\Psi_{V,1}} & \\
U \ar[rr]^{\Phi_1} & &  V , 
}
\]
we get $\gamma|_{t=1}[-1] = \beta|_{\Crit_{W_1/B}(h_1)}$. Therefore, the claim follows from \cite[Proposition 2.8]{BBDJS}.
\end{proof}

We are now ready to prove \cref{thm:BBDJS31}(2) in general.

\begin{proof}[Proof of \cref{thm:BBDJS31}(2)]
Consider a point $u\in\Crit_{U/B}(f)$ and let $v=\Phi_0(u)=\Phi_1(u)$. The differentials $d\Phi_0$ and $d\Phi_1$ fit into a commutative diagram
\[
\xymatrix{
0 \ar[r] & \T_{X/B, u} \ar[r] \ar^{d(\Phi_t|_X)|_u}[d] & \T_{U/B, u} \ar[r] \ar^{(d\Phi_t)|_u}[d] & \rN_u \ar[r] \ar[d] & 0 \\
0 \ar[r] & \T_{Y/B, v} \ar[r] & \T_{V/B, v} \ar[r] & \rN_v \ar[r] & 0
}
\]
where both rows are exact. Choosing an arbitrary splitting of the top exact sequence and using that $\Phi_0|_X = \Phi_1|_X$, we get
\begin{equation}\label{eq:Mmatrix}
d\Phi_1|^{-1}_u\circ d\Phi_0|_u = 
\left[\begin{array}{c|c} 
	\id & \ast\\ 
	\hline 
	0 & M
\end{array}\right]\colon \T_{X/B, u}\oplus \rN_u\rightarrow \T_{X/B, u}\oplus \rN_u.
\end{equation}

By \cref{prop:criticalHessian}(1) the Hessian $\Hess(f)_u$ restricts to a nondegenerate quadratic form on $\rN_u$ and by \cref{prop:criticalHessian}(2) the induced morphisms $(\rN_u, \Hess(f)_u)\rightarrow (\rN_v, \Hess(g)_v)$ are orthogonal. Thus, $M$ is an orthogonal automorphism of $(\rN_u, \Hess(f)_u)$.

If $M=\id$, we are finished by \cref{prop:BBDJS35} as then $(d\Phi_1|^{-1}_u\circ d\Phi_0|_u-\id)^2=0$. We will now adjust the setting to put the \'etale morphisms in the form suitable for the application of \cref{prop:BBDJS35}. For this, using \cref{prop:minimalchart} we may find an LG pair $(W, h)$ over $B$, a point $w\in\Crit_{W/B}(h)$ and a critical morphism $\Xi\colon (W, h)\rightarrow (U, f)$ which satisfies
\begin{enumerate}
    \item $\Xi(w) = u$.
    \item $\T_{\Crit_{W/B}(h), u}=\T_{W/B, u}$.
\end{enumerate}

The last property implies that the normal space of $\Xi\colon W\rightarrow U^\circ$ at $w$ is canonically isomorphic to $\rN_u$, such that the quadratic form $q_\Xi$ from \cref{prop:criticalembeddingnormalquadratic} is identified with $\Hess(f)_u$. Thus, applying \cref{prop:criticalembeddinglocal} to $\Xi$ we obtain a diagram
\[
\xymatrix{
W \ar^-{0}[r] & W\times \rN_u & \\
W^\circ \ar^{\Xi^\circ}[r] \ar_{\imath}[d] \ar^{\imath}[u] & U^{\circ} \ar_{\jmath}[d] \ar^{\alpha}[u] & \\
W \ar^{\Xi}[r] & U \ar^{\Phi_t}[r] & V
}
\]
with $\imath\colon (W^\circ, h^\circ)\rightarrow (W, h)$ an open immersion and $\jmath\colon (U^\circ, f^\circ)\rightarrow (U, f)$ and $\alpha\colon (U^\circ, f^\circ)\rightarrow (W\times \rN_u, h\boxplus \Hess(f)_u)$ \'etale morphisms of LG pairs over $B$. By assumption $w\in W^\circ$; denote $u^\circ=\Xi^\circ(u)\in\Crit_{U^\circ/B}(f^\circ)$.

Since $M$ is orthogonal, $\id\times M$ is an automorphism of the LG pair $(W\times \rN_u, h\boxplus \Hess(f)_u)$ over $B$. Thus, we may form the fiber product
\begin{equation}\label{eq:Pfiberproduct}
\xymatrix@C=2cm{
P \ar^{\pi_0}[r] \ar^{\pi_1}[d] & U^\circ \ar^{(\id\times M)\circ\alpha}[d] \\
U^\circ \ar^{\alpha}[r] & W\times \rN_u
}
\end{equation}
which comes equipped with a function $d\colon P\rightarrow \bA^1$ given by
\[d = f^\circ\circ \pi_0 = f^\circ\circ \pi_1.\]
By the universal property there is a point $p\in\Crit_{P/B}(d)$ such that $\pi_0(p)=\pi_1(p) = u^\circ$. For $t=0, 1$ define the \'etale morphisms
\[\Theta_t=\Phi_t\circ \jmath\circ \pi_t\colon P\longrightarrow V\]
which send $p$ to $v$. Identifying $\T_{P/B, p}\cong \T_{W/B, w}\oplus \rN_u$ using $d(\alpha\circ\pi_1)_p$ we get that
\[d\Theta_1|_p^{-1}\circ d\Theta_0|_p = \left[\begin{array}{c|c} 
	\id & \ast\\ 
	\hline 
	0 & \id
\end{array}\right]\colon \T_{P/B, p}\longrightarrow \T_{P/B, p}.\]

Fixing a perverse sheaf $\cF\in\Perv(B)$, by \cref{prop:BBDJS35} we get an open neighborhood $P^\circ\subset \Crit_{P/B}(d)$ of $p$ such that
\[\Ex_{\Theta_0}(\cF)|_{P^\circ} = \Ex_{\Theta_1}(\cF)|_{P^\circ}.\]
Using \cref{lm:Excomposite} we get
\[(\Ex_{\Phi_0}|_{P^\circ}\circ \Ex_{\jmath}|_{P^\circ}\circ \Ex_{\alpha}|^{-1}_{P^\circ}\circ \Ex_{\alpha}|_{P^\circ}\circ \Ex_{\pi_0}|_{P^\circ})(\cF) = (\Ex_{\Phi_1}|_{P^\circ}\circ \Ex_{\jmath}|_{P^\circ}\circ \Ex_{\alpha}|^{-1}_{P^\circ}\circ \Ex_{\alpha}|_{P^\circ}\circ \Ex_{\pi_1}|_{P^\circ})(\cF).\]
Using the commutative diagram \eqref{eq:Pfiberproduct} we get
\[(\Ex_{\id\times M}|_{P^\circ}\circ \Ex_{\alpha}|_{P^\circ}\circ \Ex_{\pi_0}|_{P^\circ})(\cF) = (\Ex_{\alpha}|_{P^\circ}\circ \Ex_{\pi_1}|_{P^\circ})(\cF).\]
By \cref{ex:exchangequadratic} we have $\Ex_{\id\times M} = \det(M)$ and therefore
\[\Ex_{\Phi_0}(\cF)|_{P^\circ} = \det(M)\cdot \Ex_{\Phi_1}(\cF)|_{P^\circ}.\]
Using \eqref{eq:Mmatrix} we see that
\[\det(M) = \det(d\Phi_1|_u^{-1}\circ d\Phi_0|_u).\]
Here $P^\circ\rightarrow \Crit_{U/B}(f)$ is an \'etale neighborhood of $u\in\Crit_{U/B}(f)$. Varying over different points $u\in\Crit_{U/B}(f)$ we get that
\[\Ex_{\Phi_0}(\cF) = \det(d\Phi_1|_{X^{\red}}^{-1}\circ d\Phi_0|_{X^{\red}})\cdot \Ex_{\Phi_1}(\cF)\]
is true on an \'etale cover of $\Crit_{U/B}(f)$ and hence, by \'etale descent, is true on all of $\Crit_{U/B}(f)$. Since this is true for every $\cF\in\Perv(B)$, this finishes the proof.
\end{proof}

\subsection{Perverse pullbacks for schemes}

Given an \'etale morphism $\Phi\colon (U, f)\rightarrow (V, g)$ of LG pairs over $B$, there is a natural isomorphism \eqref{eq:vanishingcyclesexchange}
\[\Ex_\Phi\colon (\phi_f(U\rightarrow B)^\dag)(-)|_{\Crit_{U/B}(f)}\xrightarrow{\sim} (\phi_g(V\rightarrow B)^\dag)(-)|_{\Crit_{U/B}(f)}.\]
In the following statement we extend this isomorphism to arbitrary critical morphisms; this is a family version of the stabilization isomorphism from \cite[Theorem 5.4]{BBDJS}.

\begin{theorem}\label{thm:nonlinearstabilization}
Let $\Phi\colon (U, f)\rightarrow (V, g)$ be a critical morphism of LG pairs over $B$ of even relative dimension. Then there is a natural isomorphism
\[\stab_\Phi\colon (\phi_f(U\rightarrow B)^\dag)(-)|_{\Crit_{U/B}(f)}\xrightarrow{\sim} (\phi_g(V\rightarrow B)^\dag)(-)|_{\Crit_{U/B}(f)}\otimes_{\mu_2} P_\Phi\]
of functors $\Perv(B)\rightarrow \Perv(\Crit_{U/B}(f))$ uniquely determined by the following properties:
\begin{enumerate}
    \item Let $(U, f)$ be an LG pair over $B$ and $(E, q)$ an orthogonal bundle over $U$ of even rank. For the zero section $0_E\colon (U, f)\rightarrow (E, f\circ\pi_E+\q_E)$ we have
    \[\stab_\Phi = \stab_E\]
    defined in \cref{thm:stabilizationconstruction}.
    \item If $\Phi\colon (U, f)\rightarrow (V, g)$ is an \'etale morphism of LG pairs over $B$, then
    \[\stab_\Phi = \Ex_\Phi.\]
\end{enumerate}

It additionally satisfies the following properties:
\begin{enumerate}[resume]
\item Let $\Phi_0, \Phi_1\colon (U, f)\rightarrow (V, g)$ be critical morphisms of even relative dimension such that
    \[\Phi_0|_{\Crit_{U/B}(f)}=\Phi_1|_{\Crit_{U/B}(f)}\colon \Crit_{U/B}(f)\rightarrow \Crit_{V/B}(g).\]
    Then we have an equality
    \[\stab_{\Phi_0}=\stab_{\Phi_1}\colon (\phi_f(U\rightarrow B)^\dag)(-)|_{\Crit_{U/B}(f)}\xrightarrow{\sim} (\phi_g(V\rightarrow B)^\dag)(-)|_{\Crit_{U/B}(f)}\otimes_{\mu_2} P_\Phi.\]
    \item For the identity critical morphism $\id\colon (U, f)\rightarrow (U, f)$ we have $\stab_{\id} = \id$.
    \item For a composite $(U, f)\xrightarrow{\Phi}(V, g)\xrightarrow{\Psi} (W, h)$ of critical morphisms of even relative dimensions we have a commutative diagram
    \[
    \begin{tikzcd}
    (\phi_f(U\rightarrow B)^\dag)(-)|_{\Crit_{U/B}(f)}\arrow[r, "\stab_\Phi"] \arrow[d, "\stab_{\Psi\circ\Phi}"] & (\phi_g(V\rightarrow B)^\dag(-))|_{\Crit_{U/B}(f)}\otimes_{\mu_2} P_{\Phi} \arrow[d, "\stab_{\Psi}\otimes\id"] \\
    (\phi_h(W\rightarrow B)^\dag)(-)|_{\Crit_{U/B}(f)}\otimes_{\mu_2} P_{\Psi\circ\Phi} \arrow[r, "\id\otimes\Xi_{\Phi, \Psi}"] & (\phi_h(W\rightarrow B)^\dag)(-)|_{\Crit_{U/B}(f)}\otimes_{\mu_2} P_\Psi|_{\Crit_{U/B}(f)}\otimes P_\Phi.
    \end{tikzcd}
    \]
    \item For a commutative diagram
    \[
    \xymatrix{
    (U_1, f_1) \ar^{\Phi_1}[r] \ar^{\pi_U}[d] & (V_1, g_1) \ar^{\pi_V}[d] \\
    (U_2, f_2) \ar^{\Phi_2}[r] & (V_2, g_2)
    }
    \]
    with $\pi_U$ and $\pi_V$ smooth morphisms and $\Phi_1$ and $\Phi_2$ critical morphisms of even relative dimension such that $U_1\rightarrow V_1\times_{V_2} U_2$ is \'etale, the diagram
    \[
    \resizebox{\textwidth}{!}{$
    \xymatrix{
    (\phi_{f_1}(U_1\rightarrow B)^\dag)(-)|_{\Crit_{U_1/B}(f_1)} \ar^-{\stab_{\Phi_1}}[r] \ar^{\Ex_{\pi_U}}[d] & (\phi_{g_1}(V_1\rightarrow B)^\dag)(-)|_{\Crit_{U_1/B}(f_1)}\otimes_{\mu_2} P_{\Phi_1} \ar^{\Ex_{\pi_V}\otimes\eqref{eq:Ppullback}}[d] \\
    (\phi_{f_2}(U_2\rightarrow B)^\dag)(-)|_{\Crit_{U_1/B}(f_1)}[\dim(U_1/U_2)] \ar^-{\stab_{\Phi_2}}[r] & (\phi_{g_2}(V_2\rightarrow B)^\dag)(-)|_{\Crit_{U_1/B}(f_1)}\otimes_{\mu_2} \pi_U^* P_{\Phi_2}[\dim(U_1/U_2)]
    }
    $}
    \]
    commutes.
    \item For a smooth morphism $p\colon B'\rightarrow B$ with $\Phi'\colon (U', f')\rightarrow (V', g')$ the base change of $\Phi$ and $\pi_U\colon \Crit_{U'/B'}(f')\rightarrow \Crit_{U/B}(f)$ the corresponding projection, the diagram
    \[
    \xymatrix{
    (\phi_{f'}(U'\rightarrow B)^\dag)(-)|_{\Crit_{U'/B'}(f')} \ar^-{\stab_{\Phi'}}[r] \ar^{\Ex^!_\phi}[d] & (\phi_{g'}(V'\rightarrow B)^\dag)(-)|_{\Crit_{U'/B'}(f')}\otimes_{\mu_2} P_{\Phi'} \ar^{\Ex^!_\phi\otimes \eqref{eq:Ppullback}}[d] \\
    \pi_U^\dag((\phi_f(U\rightarrow B)^\dag)(-)|_{\Crit_{U/B}(f)}) \ar^-{\stab_{\Phi}}[r] & \pi_U^\dag((\phi_g(V\rightarrow B)^\dag)(-)|_{\Crit_{U/B}(f)}\otimes_{\mu_2} P_{\Phi})
    }
    \]
    commutes.
    \item For a finite morphism $c\colon \tilde{B}\rightarrow B$ with $\tilde{\Phi}\colon (\tilde{U}, \tilde{f})\rightarrow (\tilde{V}, \tilde{g})$ the base change of $\Phi$ and $\tilde{c}_U\colon \tilde{U}\rightarrow U$ and $\tilde{c}_V\colon \tilde{V}\rightarrow V$ projections, the diagram
    \[
    \xymatrix{
    (\phi_f(U\rightarrow B)^\dag c_*)(-)|_{\Crit_{U/B}(f)} \ar^-{\stab_\Phi}[r] \ar^{\Ex^\phi_*,\Ex^!_*}[d] & (\phi_g(V\rightarrow B)^\dag c_*)(-)|_{\Crit_{U/B}(f)}\otimes_{\mu_2} P_{\Phi} \ar^{\Ex^\phi_*,\Ex^!_*\otimes\eqref{eq:Ppullback}}[d] \\
    (\tilde{c}_U)_*((\phi_{\tilde{f}}(\tilde{U}\rightarrow \tilde{B})^\dag)(-)|_{\Crit_{\tilde{U}/\tilde{B}}(f)}) \ar^-{\stab_{\tilde{\Phi}}}[r] & (\tilde{c}_V)_*((\phi_{\tilde{g}}(\tilde{V}\rightarrow \tilde{B})^\dag)(-)|_{\Crit_{\tilde{U}/\tilde{B}}(f)}\otimes_{\mu_2} P_{\tilde{\Phi}})
    }
    \]
    commutes.
    \item Let $p\colon B\rightarrow B'$ be a smooth morphism and $\stab^B_\Phi$ the stabilization isomorphism for $\Phi$ viewed as a morphism of LG pairs over $B$ and $\stab^{B'}_\Phi$ the stabilization isomorphism for $\Phi$ viewed as a morphism of LG pairs over $B'$. Then the diagram
    \[
    \xymatrix{
    (\phi_f(U\rightarrow B)^\dag p^\dag)(-)|_{\Crit_{U/B'}(f)}\ar^-{\stab^B_\Phi}[r] \ar^{\sim}[d] & (\phi_g(V\rightarrow B)^\dag p^\dag)(-)|_{\Crit_{U/B'}(f)}\otimes_{\mu_2} P_\Phi \ar^{\sim}[d] \\
    (\phi_f(U\rightarrow B')^\dag)(-)|_{\Crit_{U/B'}(f)}\ar^-{\stab^{B'}_\Phi}[r] & (\phi_g(V\rightarrow B')^\dag)(-)|_{\Crit_{U/B'}(f)}\otimes_{\mu_2} P_\Phi
    }
    \]
    commutes.
    \item\label{item:nonlinearstabilization/verdier}
    The diagram
    \[\begin{tikzcd}
    (\phi_f(U\rightarrow B)^\dag\bbD)(-)|_{\Crit_{U/B}(f)} \ar{r}{\stab_\Phi} \ar{d}{\Ex^{\dag,\phi},\Ex^{\phi, \bbD}}
      & (\phi_g(V\rightarrow B)^\dag \bbD)(-)|_{\Crit_{U/B}(f)}\otimes_{\mu_2} P_\Phi \ar{d}{\Ex^{\dag,\phi},\Ex^{\phi, \bbD}\otimes \eqref{eq:Popposite}} \\
      (\bbD \phi_{-f}(U\rightarrow B)^\dag)(-)|_{\Crit_{U/B}(f)} 
      & (\bbD \phi_{-g}(V\rightarrow B)^\dag)(-)|_{\Crit_{U/B}(f)} \ar{l}[swap]{\stab_{\overline{\Phi}}}\otimes_{\mu_2} P_{\overline{\Phi}}
    \end{tikzcd}\]
    commutes.
    \item\label{item:nonlinearstabilization/TS}
    Given another critical morphism $\Phi' \colon (U', f')\rightarrow (V', g')$ of LG pairs over $B'$ of even relative dimension, $\Phi \times \Phi' \colon (U \times U', f\boxplus f') \to (V \times V', g\boxplus g')$ is a critical morphism over $B \times B'$ and the diagram
    \[
    \begin{tikzcd}[column sep=large]
      ((\phi_f p^\dag)(-)|_{\Crit_{U/B}(f)} \boxtimes (\phi_{f'} p'^\dag)(-)|_{\Crit_{U'/B'}(f')}) \ar{r}{\stab_\Phi \boxtimes \stab_{\Phi'}}\ar{d}{\TS}
      & (\phi_g q^\dag)(-)|_{\Crit_{U/B}(f)} \boxtimes (\phi_{g'} q'^\dag)(-)|_{\Crit_{U'/B'}(f')}\otimes_{\mu_2} (P_\Phi\boxtimes P_{\Phi'}) \ar{d}{\TS\otimes \eqref{eq:Pproduct}}
      \\
      \phi_{f \boxplus f'} (p\times p')^\dag(-)|_{\Crit_{U\times U'/B\times B'}(f\boxplus f')} \ar{r}{\stab_{\Phi\times\Phi'}}
      & \phi_{g\boxplus g'}(q \times q')^\dag(-)|_{\Crit_{U\times U'/B\times B'}(f\boxplus f')}\otimes_{\mu_2} P_{\Phi\times \Phi'}
    \end{tikzcd}
    \]
    commutes.
\end{enumerate}
\end{theorem}
\begin{proof}
Using \cref{prop:criticalembeddinglocal} we may find a collection of $\{U_a, f_a\}_{a\in A}$ of LG pairs over $B$ with open morphisms $\imath_a\colon (U_a, f_a)\rightarrow (U, f)$ such that $\{R_a=\Crit_{U_a/B}(f_a)\rightarrow R=\Crit_{U/B}(f)\}$ is an open cover, a collection $\{V_a, g_a\}_{a\in A}$ of LG pairs over $B$ with \'etale morphisms $\jmath_a\colon (V_a, g_a)\rightarrow (V, g)$ and a collection $\{E_a, q_a\}_{a\in A}$ of trivial orthogonal bundles over $U$ with \'etale morphisms $\alpha_a\colon (V_a, g_a)\rightarrow (E_a, f\circ \pi_{E_a}+\q_{E_a})$ which fit into a commutative diagram
\[
\begin{tikzcd}
U \arrow[r, "0_{E_a}"] & E_a \\
U_a \arrow[r, "\Phi_a"] \arrow[d, "\imath_a" left] \arrow[u, "\imath_a"] & V_a \arrow[d, "\jmath_a" left] \arrow[u, "\alpha_a"] \\
U \arrow[r, "\Phi"] & V
\end{tikzcd}
\]
By descent it is enough to construct the stabilization isomorphism $\stab_\Phi$ restricted to $R_a$ and show that these stabilization isomorphisms are equal on $R_{ab} = R_a\times_R R_b$.

By construction there is a canonical isomorphism
\begin{equation}\label{eq:Plocalform}
P_\Phi|_{R_a}\cong \ori_{E_a}|_{R_a}.
\end{equation}
Let $\pi_U\colon U\rightarrow B$ and $\pi_V\colon V\rightarrow B$ be the natural projections. We define $\stab_\Phi|_{R_a}$ as the composite
\begin{align*}
\phi_f(\pi_U^\dag(-))|_{R_a}\otimes_{\mu_2} \ori_{E_a}|_{R_a}&\xrightarrow{\stab_{E_a}} (\phi_{f\circ \pi_{E_a} + \q_{E_a}}\pi_{E_a}^\dag\pi_U^\dag(-))|_{R_a} \\
&\xleftarrow{\Ex_{\alpha_a}} (\phi_{g_a}\pi_{V_a}^\dag(-))|_{R_a}\\
&\xrightarrow{\Ex_{\jmath_a}}(\phi_g\pi_V^\dag(-))|_{R_a}.
\end{align*}

The above isomorphism is determined uniquely by properties (1) and (2).

We will now show the following facts:
\begin{enumerate}
    \item $\stab_\Phi$ glues into a global isomorphism over $R$, i.e. $\stab_\Phi|_{R_a} = \stab_\Phi|_{R_b}$ above.
    \item $\stab_\Phi$ is independent of the choices of the local model given in \cref{prop:criticalembeddinglocal}. For this, making two choices for the local model, we again have to prove $\stab_\Phi|_{R_a} = \stab_\Phi|_{R_b}$.
    \item Property (3) of the stabilization isomorphism holds, i.e. given two critical morphisms $\Phi_0, \Phi_1\colon (U, f)\rightarrow (V, g)$ such that $\Phi_0|_{\Crit_{U/B}(f)} = \Phi_1|_{\Crit_{U/B}(f)}$ we have $\stab_{\Phi_0} = \stab_{\Phi_1}$. For this we repeat the above construction with $\Phi_0$ on $U_a$ and $\Phi_1$ on $U_b$. We have to prove again $\stab_{\Phi_0}|_{R_a} = \stab_{\Phi_1}|_{R_b}$.
\end{enumerate}

Since $E_a$ and $E_b$ are trivial, we may identify the two; we denote the resulting orthogonal bundle $E$. Let $U_{ab} = U_a\times_U U_b$. We have a diagram
\[
\xymatrix{
& U_a \ar^{\Phi_a}[r] \ar_{\imath_a}[d] & V_a \ar_{\jmath_a}[dr] \ar_{\alpha_a}[d] \\
U_{ab} \ar[ur] \ar[dr] & U \ar^{0_E}[r] & E & V \\
& U_b \ar^{\Phi_b}[r] \ar^{\imath_b}[u] & V_b \ar^{\jmath_b}[ur] \ar^{\alpha_b}[u]
}
\]
with the two squares as well as the left triangle commutative. When we prove the first two items in the above list, the two composite morphisms $U_{ab}\rightarrow V$ are equal (given by the composite $U_{ab}\rightarrow U\xrightarrow{\Phi} V$). When we prove the last item in the above list, we only have the weaker statement that the two composite morphisms $R_{ab}\rightarrow U_{ab}\rightarrow V$ are equal. So, from now on we only use this weaker assumption.

For a point $u\in R_{ab}$ we have an isomorphism $E|_u\cong E|_u$ given by
\[E|_u\cong \rN_{U/E_a, u}\xrightarrow{(d\alpha_a)|^{-1}_{\Phi_a(u)}} \rN_{U_a/V_a, u}\xrightarrow{(d\jmath_a)|_{\Phi_a(u)}} \rN_{U/V, u}\xrightarrow{(d\jmath_b)|^{-1}_{\Phi_b(u)}}\rN_{U_b/V_b, u}\xrightarrow{(d\alpha_b)|_{\Phi_b(u)}} \rN_{U/E_b, u}\cong E|_u.\]
Let $M_u\colon E|_u\rightarrow E|_u$ be the corresponding matrix; by \cref{prop:criticalembeddingnormalquadratic} it is an orthogonal matrix. The two isomorphisms $P_\Phi|_{R_{ab}}\cong \ori_{E_a}|_{R_{ab}}$ and $P_\Phi|_{R_{ab}}\cong \ori_{E_b}|_{R_{ab}}$ given by \eqref{eq:Plocalform} differ by $\det(M)$ in a neighborhood of $u$.

By \cref{lm:Excomposite} we have to show that we have an equality
\begin{equation}\label{eq:BBDJS54independence}
\Ex_{\alpha_b}|_{R_{ab}}\circ \Ex_{\jmath_b}|_{R_{ab}}^{-1}\circ \Ex_{\jmath_a}|_{R_{ab}}\circ \Ex_{\alpha_a}|_{R_{ab}}^{-1} = \det(M)\cdot \id
\end{equation}
of natural automorphisms of $\phi_{f\circ \pi_E + \q_E}(\pi_E^\dag(\pi_U^\dag(-)))|_{R_{ab}}$ in a neighborhood of $u$.

For this, let $V_{ab} = V_a\times_E V_b$ equipped with a function $g_{ab}\colon V_{ab}\rightarrow \bA^1$ compatible with $g_a$ and $g_b$ on $V_a$ and $V_b$ and consider the corresponding diagram
\[
\xymatrix{
&& V_a \ar_{\jmath_a}[dr] \ar_{\alpha_a}[d] \\
R_{ab} \ar^{\Phi_{ab}}[r] & V_{ab} \ar^{p_a}[ur] \ar_{p_b}[dr] \ar[r] & E & V \\
&& V_b \ar^{\jmath_b}[ur] \ar^{\alpha_b}[u]
}
\]
Then we have two \'etale morphisms
\[\jmath_a\circ p_a,\jmath_b\circ p_b\colon (V_{ab}, g_{ab})\rightarrow (V, g)\]
of LG pairs over $B$ such that the two composites
\[R_{ab}\longrightarrow \Crit_{V_{ab}/B}(g_{ab})\longrightarrow \Crit_{V/B}(g)\]
are equal. Identifying $\T_{V_{ab}/B, \Phi_{ab}(u)}\cong \T_{U/B, u}\oplus E|_u$ using the \'etale morphism $V_{ab}\rightarrow E$ we have
\[(d p_b)|^{-1}_{\Phi_{ab}(u)}\circ (d\jmath_b)|^{-1}_{\Phi_b(u)}\circ (d\jmath_a)|_{\Phi_a(u)}\circ (d p_a)|_{\Phi_a(u)} =
\left[\begin{array}{c|c} 
	\id & 0\\ 
	\hline 
	0 & M
\end{array}\right]
\]
Thus, we have
\[\det\left((d p_b)|^{-1}_{\Phi_{ab}(u)}\circ (d\jmath_b)|^{-1}_{\Phi_b(u)}\circ (d\jmath_a)|_{\Phi_a(u)}\circ (d p_a)|_{\Phi_a(u)}\right) = \det(M).\]
Applying \cref{thm:BBDJS31} we get \eqref{eq:BBDJS54independence} as claimed.

Let us now show the remaining properties of the stabilization isomorphism. It is enough to establish property (1) locally when $E$ is a trivial orthogonal bundle. In this case the critical morphism $0_E\colon U\rightarrow E$ is already in a local model, so that $\stab_\Phi=\stab_E$ by definition.

In the setting of property (2) we are again in a local model: the corresponding diagram from \cref{prop:criticalembeddinglocal} is
\[
\xymatrix{
U \ar@{=}[r] & U \\
U \ar@{=}[r] \ar@{=}[u] \ar@{=}[d] & U \ar^{\Phi}[d] \ar@{=}[u] \\
U \ar^{\Phi}[r] & V
}
\]
and by definition we get $\stab_\Phi=\Ex_\Phi$.

Property (4) of the stabilization isomorphism is obvious as for the identity critical morphism we may take $E=0$ and $U_a=U=V_a=V$ in the local model given in \cref{prop:criticalembeddinglocal}.

It is enough to show the commutativity of the diagram from property (5) locally. For this, consider $U^\circ, V^\circ_1, V^\circ_2, W^\circ, (E_1, q_1), (E_2, q_2), \overline{U}$ as in the proof of \cref{prop:criticalembeddingnormalquadratic}(2), so that we have a local model of the critical morphism $\Psi\circ\Phi$ as
\[
\begin{tikzcd}
U \rar & E_1\times_V E_2 \\
\overline{U} \rar \uar \dar & W^\circ \uar \dar\\
U \arrow[r, "\Psi\circ\Phi"] & W
\end{tikzcd}
\]

Restricting the diagram to $\overline{U}$, we get local isomorphisms $P_{\Phi}\cong \ori_{E_1}$, $P_{\Psi}\cong \ori_{E_2}$ and $P_{\Psi\circ \Phi}\cong \ori_{E_1\oplus E_2}$. The commutativity of the diagram then follows from \cref{prop:Linearization-Whitneysum}.

Let us now show property (6). First assume that $\Phi_2$ is \'etale. As $\Phi_1$ factors through a composite of \'etale morphisms $U_1\rightarrow V_1\times_{V_2} U_2\rightarrow V_1$, it is also \'etale. In this case the commutativity of the diagram follows from \cref{lm:Excomposite}. For an arbitrary critical morphism $\Phi_2$, we may use \cref{prop:criticalembeddinglocal} to assume, \'etale locally, that $\Phi_2$ is given by the zero section $U_2\rightarrow E$ of an orthogonal bundle in which case $\stab_{\Phi_2} = \stab_E$. As in the proof of \cref{prop:criticalembeddingnormalquadratic}(4), \'etale locally we have that $\Phi_1$ is given by the zero section $U_1\rightarrow \pi_U^* E$ of the pullback orthogonal bundle. In this case property (6) follows from the compatibility of $\stab_E$ with smooth base change, \cref{thm:stabilizationconstruction}(1).

Property (7) is proven similarly to property (6): if $\Phi$ is \'etale, it follows from \cref{prop:constructiblevanishingcycles}(2) and if $\Phi$ is the zero section of an orthogonal bundle $E$, it follows from the compatibility of $\stab_E$ with smooth base change, \cref{thm:stabilizationconstruction}(1).

Property (8) is proven similarly to properties (6) and (7): if $\Phi$ is \'etale, it follows from \cref{prop:constructiblevanishingcycles}(2) and if $\Phi$ is the zero section of an orthogonal bundle $E$, it follows from the compatibility of $\stab_E$ with finite base change, \cref{prop:stabilizationfinitebasechange}.

Property (9) is obvious from construction.

Property (10) follows from \cref{prop:phiD}(2) if $\Phi$ is \'etale and from \cref{prop:stabilizationduality} if $\Phi$ is the zero section of an orthogonal bundle.

Let us finally prove property (11). Writing $\Phi\times \Phi' = (\Phi\times \id)\circ (\id\times \Phi')$ and using the compatibility of the stabilization isomorphism with respect to compositions, property (5), we reduce the claim to the case $\Phi'=\id$. Then if $\Phi$ is \'etale, the claim follows from \cref{prop:phiTS}(2). If $\Phi$ is the zero section of an orthogonal bundle, the claim follows from \cref{prop:stabilizationproducts}.
\end{proof}

We are now ready to define perverse pullbacks for relative d-critical structures.

\begin{theorem}\label{thm:micropullsch}
Let $\pi\colon X\rightarrow B$ be a morphism of schemes equipped with a relative d-critical structure $s$ and an orientation $o$. There is an exact functor
\[\pi^\varphi\colon \Perv(B)\longrightarrow \Perv(X),\]
called the \defterm{perverse pullback}, uniquely determined by the following properties:
\begin{enumerate}
    \item For every critical chart $(U, f, u)$ of $(X\rightarrow B, s)$, such that $\dim(U/B)\pmod{2}$ coincides with the parity of $o$, we have a canonical natural isomorphism
    \begin{equation}\label{eq:microlocalpullbacklocalmodel}
    (\pi^\varphi(\cF))|_{\Crit_{U/B}(f)}\cong \phi_f((U\rightarrow B)^\dag(\cF))|_{\Crit_{U/B}(f)}\otimes_{\mu_2} Q^o_{U, f, u}.
    \end{equation}
    \item For a critical morphism $\Phi\colon (U, f, u)\rightarrow (V, g, v)$ of critical charts of even relative dimension the diagram
    \[
    \xymatrix{
    (\pi^\varphi(\cF))|_{\Crit_{U/B}(f)} \ar^-{\eqref{eq:microlocalpullbacklocalmodel}}[r] \ar^{\eqref{eq:microlocalpullbacklocalmodel}}[d] & \phi_f((U\rightarrow B)^\dag(\cF))|_{\Crit_{U/B}(f)}\otimes_{\mu_2} Q^o_{U, f, u} \ar^{\stab_\Phi\otimes\id}[d] \\
    \phi_g((V\rightarrow B)^\dag(\cF))|_{\Crit_{U/B}(f)}\otimes_{\mu_2} Q^o_{V, g, v} \ar^-{\id\otimes\Lambda_\Phi}[r] & \phi_g((V\rightarrow B)^\dag(\cF))|_{\Crit_{U/B}(f)}\otimes_{\mu_2}P_\Phi\otimes_{\mu_2}Q^o_{U, f, u}
    }
    \]
    commutes.
\end{enumerate}
In addition, it satisfies the following properties:
\begin{enumerate}[resume]
    \item Given an isomorphism of orientations $o_1\cong o_2$ of the same relative d-critical structure $(X\rightarrow B, s)$, there is a natural isomorphism $\pi^\varphi_{o_1}\cong \pi^\varphi_{o_2}$ of the perverse pullback functors with respect to the two orientations, which is associative. Moreover, the automorphism of an orientation $(\cL, o)$ given by multiplication on $\cL$ by a sign $\sigma\in\mu_2$ acts on $\pi^\varphi_o$ by the same sign $\mu$.
    \item If $R$ is a field, then the perverse pullback functor $\pi^\varphi$ extends uniquely to a colimit-preserving $t$-exact functor
    \[\pi^\varphi\colon \bD(B)\longrightarrow \bD(X).\]
\end{enumerate}

\end{theorem}
\begin{proof}
The assignment $X\mapsto \Perv(X)$ of the category of perverse sheaves forms a sheaf of categories over the Zariski site of $X$. Therefore, the assignment $X\mapsto \Fun_\ex(\Perv(B), \Perv(X))$ forms a sheaf of categories as well. We may then glue the local descriptions of the perverse pullback functor $\pi^\varphi$ using \cref{cor:criticaldescent}, where the relevant properties of the stabilization isomorphisms were verified in \cref{thm:nonlinearstabilization}. This establishes properties (1) and (2).

Let us now show property (3). For an isomorphism of orientations $o_1\cong o_2$ using \cref{cor:criticaldescentsets} we need to specify the natural isomorphism $\pi^\varphi_{o_1}\cong \pi^\varphi_{o_2}$ locally on critical charts. On a critical chart $(U, f, u)$ we let it be the isomorphism $Q^{o_1}_{U, f, u}\cong Q^{o_2}_{U, f, u}$ given by precomposing the isomorphism $o_2\cong o^{\can}_{\Crit_{U/B}(f)}$ with the isomorphism $o_1\cong o_2$. For an automorphism of an orientation $o$ given by a sign $\sigma\in\mu_2$ the corresponding isomorphism $Q^o_{U, f, u}\cong Q^o_{U, f, u}$ is given by acting by the same sign and hence, using the local description of perverse pullbacks, it acts on $\pi^\varphi$ by the same sign.

Property (4) follows from \cref{cor:extendfromPerv} since all exact functors on $\Perv(-)$ extend uniquely to $\bD(-)$.
\end{proof}

\subsection{Compatibility with base change}

We first establish a compatibility of perverse pullbacks with smooth pullbacks.

\begin{proposition}\label{prop:perversesmoothcompatibility}
Let $\pi\colon X\rightarrow B$ be a morphism of schemes equipped with an oriented relative d-critical structure. Let $p\colon X'\rightarrow X$ be a smooth morphism of schemes. Consider the pullback oriented relative d-critical structure on $X'\rightarrow B$. Then there is a natural isomorphism
\[\alpha_p\colon (\pi\circ p)^\varphi(-) \xrightarrow{\sim} p^\dag\pi^\varphi(-)\]
of functors $\Perv(B)\rightarrow \Perv(X')$ which satisfies the following properties:
\begin{enumerate}
    \item $\alpha_{\id} = \id$.
    \item For a composite $X''\xrightarrow{q} X'\xrightarrow{p} X$ of smooth morphisms with an oriented relative d-critical structure on $\pi\colon X\rightarrow B$ and pullback oriented relative d-critical structures on $X''\rightarrow B$ and $X'\rightarrow B$ we have $\alpha_p\circ \alpha_q = \alpha_{p\circ q}$.
\end{enumerate}
If $R$ is a field, $\alpha_p$ extends to a natural isomorphism of functors $\bD(B)\rightarrow \bD(X')$.
\end{proposition}
\begin{proof}
Denote by $o$ the orientation of $X\rightarrow B$ and by $o'$ the pullback orientation of $X'\rightarrow B$. By \cref{thm:criticallocussmoothdescent} we may find a cover of $X'$ by critical charts $(U'_a, f'_a, u'_a)$ and a cover of the image of $p\colon X'\rightarrow X$ by critical charts $(U_a, f_a, u_a)$ together with a smooth morphism $\tilde{p}_a\colon (U'_a, f'_a)\rightarrow (U_a, f_a)$ fitting into a commutative diagram
\begin{equation}\label{eq:perversesmoothcompatibilitycharts}
\xymatrix{
X' \ar^{p}[d] & \Crit_{U'_a/B}(f'_a) \ar[r] \ar_-{u'_a}[l] \ar[d] & U'_a \ar^{\tilde{p}_a}[d] \\
X & \Crit_{U_a/B}(f_a) \ar_-{u_a}[l] \ar[r] & U_a
}
\end{equation}

We define $\alpha_p|_{\Crit_{U'_a/B}(f'_a)}$ as the unique natural isomorphism which fits into the commutative diagram
\[
\begin{tikzcd}
(\pi\circ p)^\varphi(-)|_{\Crit_{U'_a/B}(f'_a)} \ar{r}{\eqref{eq:microlocalpullbacklocalmodel}} \ar{d}{\alpha_p|_{\Crit_{U'_a/B}(f'_a)}} & (\phi_{f'_a} (U'_a\rightarrow B)^\dag)(-)|_{\Crit_{U'_a/B}(f'_a)}\otimes Q^{o'}_{U'_a, f'_a, u'_a} \ar{d}{\Ex_{\tilde{p}_a}\otimes\eqref{eq:Qpullback}} \\
(p^\dag\pi^\varphi)(-)|_{\Crit_{U'_a/B}(f'_a)} \ar{r}{\eqref{eq:microlocalpullbacklocalmodel},\pur_p} & (\phi_{f_a}(U_a\rightarrow B)^\dag)(-)|_{\Crit_{U'_a/B}(f'_a)}[\dim(U'_a/U_a)]\otimes Q^o_{U_a, f_a, u_a}|_{\Crit_{U'_a/B}(f'_a)}.
\end{tikzcd}
\]
We will now show that thus defined local natural isomorphisms $\alpha_p|_{\Crit_{U'_a/B}(f'_a)}$ glue into a global natural isomorphism $\alpha_p$. For this, we need to check that the restrictions of $\alpha_p|_{\Crit_{U'_a/B}(f'_a)}$ and $\alpha_p|_{\Crit_{U'_b/B}(f'_b)}$ to $\Crit_{U'_a/B}(f'_a)\times_{X'} \Crit_{U'_b/B}(f'_b)$ coincide. Consider the \'etale local model of the intersection given by \cref{prop:stabilizationzigzagsmoothmorphism}:
\begin{equation}\label{eq:perversesmoothcompatibilitydiagram}
\xymatrix{
(U'_a, f'_a, u'_a) \ar^{\tilde{p}_a}[d] & (U^{\prime\circ}_a, f^{\prime\circ}_a, u^{\prime\circ}_a) \ar^-{\Phi'_a}[r] \ar[l] \ar[d] & (U'_{ab}, f'_{ab}, u'_{ab}) \ar^{\tilde{p}_{ab}}[d] & (U^{\prime\circ}_b, f^{\prime\circ}_b, u^{\prime\circ}_b) \ar_-{\Phi'_b}[l] \ar[r] \ar[d] & (U'_b, f'_b, u'_b) \ar^{\tilde{p}_b}[d] \\
(U_a, f_a, u_a) & (U^\circ_a, f^\circ_a, u^\circ_a) \ar^-{\Phi_a}[r] \ar[l] & (U_{ab}, f_{ab}, u_{ab}) & (U^\circ_b, f^\circ_b, u^\circ_b) \ar_-{\Phi_b}[l] \ar[r] & (U_b, f_b, u_b)
}
\end{equation}
with all vertical morphisms smooth and all morphisms to the pullback squares \'etale. Let
\[R_{ab} = \Crit_{U^\circ_a/B}(f^\circ_a)\times_X \Crit_{U^\circ_b/B}(f^\circ_b),\qquad R'_{ab} = \Crit_{U^{\prime\circ}_a/B}(f^{\prime\circ}_a)\times_X \Crit_{U^{\prime\circ}_b/B}(f^{\prime\circ}_b).\]
By definition (see \cref{thm:micropullsch}(2)) the composite natural isomorphism
\[
(\phi_{f_a}(U_a\rightarrow B)^\dag)(-)|_{R_{ab}}\otimes Q^o_{U_a, f_a, u_a}|_{R_{ab}} \xleftarrow{\sim} \pi^\varphi(-)|_{R_{ab}}\xrightarrow{\sim} (\phi_{f_b}(U_b\rightarrow B)^\dag)(-)|_{R_{ab}}\otimes Q^o_{U_b, f_b, u_b}|_{R_{ab}}
\]
is given by the composite of stabilization isomorphisms with respect to the bottom critical morphisms in \eqref{eq:perversesmoothcompatibilitydiagram}. Similarly, the composite natural isomorphism
\[
(\phi_{f'_a}(U'_a\rightarrow B)^\dag)(-)|_{R'_{ab}}\otimes Q^{o'}_{U'_a, f'_a, u'_a}|_{R'_{ab}} \xleftarrow{\sim} (\pi\circ p)^\varphi(-)|_{R'_{ab}}\xrightarrow{\sim} (\phi_{f'_b}(U'_b\rightarrow B)^\dag)(-)|_{R'_{ab}}\otimes Q^{o'}_{U'_b, f'_b, u'_b}|_{R'_{ab}}
\]
is given by the composite of stabilization isomorphisms with respect to the top critical morphisms in \eqref{eq:perversesmoothcompatibilitydiagram}. The equality of the restrictions of $\alpha_p|_{\Crit_{U'_a/B}(f'_a)}$ and $\alpha_p|_{\Crit_{U'_b/B}(f'_b)}$ to $R'_{ab}$ therefore follows from the compatibility of the stabilization isomorphisms with smooth pullbacks, \cref{thm:nonlinearstabilization}(6).

The functoriality of the natural isomorphism $\alpha_p$ follows from the functoriality of the natural isomorphism $\Ex_p$, \cref{lm:Excomposite}, and the functoriality of the purity isomorphism $\pur_p$, \cref{prop:fundamentalclass}.

The extension of $\alpha_p$ to a natural transformation of functors $\bD(B)\rightarrow \bD(X')$ is given by \cref{cor:extendfromPerv}.
\end{proof}

Next we establish a compatibility of perverse pullbacks with smooth base change.

\begin{proposition}\label{prop:perversesmoothbasechangecompatibility}
Let
\[
\xymatrix{
X' \ar^{\overline{p}}[r] \ar^{\pi'}[d] & X \ar^{\pi}[d] \\
B' \ar^{p}[r] & B
}
\]
be a commutative diagram of schemes with $p\colon B'\rightarrow B$ and $X'\rightarrow B'\times_B X$ smooth. Let $s\in\Gamma(X, \cS_{X/B})$ be a relative d-critical structure and denote by $s'\in\Gamma(X', \cS_{X'/B'})$ its pullback. Let $o$ be an orientation of $X\rightarrow B$ and $o'$ its pullback to $X'\rightarrow B'$. Then there is a natural isomorphism
\[\alpha_{p, \overline{p}}\colon (\pi')^\varphi p^\dag(-)\xrightarrow{\sim}\overline{p}^\dag \pi^\varphi(-)\]
of functors $\Perv(B)\rightarrow \Perv(X')$
which satisfies the following properties:
\begin{enumerate}
    \item $\alpha_{\id, \id} = \id$.
    \item Given a diagram
    \[
    \xymatrix{
    X'' \ar^{\overline{q}}[r] \ar^{\pi''}[d] & X' \ar^{\overline{p}}[r] \ar^{\pi'}[d] & X \ar^{\pi}[d] \\
    B'' \ar^{q}[r] & B' \ar^{p}[r] & B
    }
    \]
    with $B''\xrightarrow{q} B'\xrightarrow{p} B$, $X'\rightarrow B'\times_B X$ and $X''\rightarrow B''\times_{B'} X'$ smooth, we have
    \[\alpha_{p, \overline{p}}\circ\alpha_{q, \overline{q}} = \alpha_{p\circ q, \overline{p}\circ \overline{q}}.\]
\end{enumerate}
If $R$ is a field, $\alpha_{p, \overline{p}}$ extends to a natural transformation of functors $\bD(B)\rightarrow \bD(X')$.
\end{proposition}
\begin{proof}
Let us first assume that the diagram is Cartesian. Consider a cover of $X$ by critical charts $(U_a, f_a, u_a)$ and let $(U'_a, f'_a, u'_a)$ be their base change along $p$ which define a cover of $X'$ by critical charts. We define $\alpha_{p, \overline{p}}|_{\Crit_{U'_a/B'}(f'_a)}$ as the unique natural isomorphism which fits into the commutative diagram
\[
\xymatrix{
((\pi')^\varphi p^\dag)(-)|_{\Crit_{U'_a/B'}(f'_a)} \ar^-{\eqref{eq:microlocalpullbacklocalmodel}}[r] \ar^{\alpha_{p, \overline{p}}|_{\Crit_{U'_a/B'}(f'_a)}}[d] & (\phi_{f'_a} (U'_a\rightarrow B)^\dag)(-)|_{\Crit_{U'_a/B'}(f'_a)} \otimes Q^{o'}_{U'_a, f'_a, u'_a} \ar^{\Ex^!_\phi\otimes\eqref{eq:Qpullback}}[d] \\
(\overline{p}^\dag \pi^\varphi)(-)|_{\Crit_{U'_a/B'}(f'_a)} \ar^-{\eqref{eq:microlocalpullbacklocalmodel}}[r] & (\phi_{f_a} (U_a\rightarrow B)^\dag)(-)|_{\Crit_{U'_a/B}(f'_a)} \otimes Q^{o}_{U_a, f_a, u_a}|_{\Crit_{U'_a/B'}(f'_a)}
}
\]

The fact that these locally defined natural isomorphisms $\alpha_{p, \overline{p}}|_{\Crit_{U'_a/B'}(f'_a)}$ glue into a global natural isomorphism $\alpha_{p, \overline{p}}$ is shown as in the proof of \cref{prop:perversesmoothcompatibility}:
\begin{itemize}
    \item By \cref{prop:stabilizationzigzag} we have a Zariski local model of the intersection $\Crit_{U_a/B}(f_a)\times_X \Crit_{U_b/B}(f_b)$ given by
    \[
    \xymatrix{
    (U_a, f_a, u_a) & (U^\circ_a, f^\circ_a, u^\circ_a) \ar^-{\Phi_a}[r] \ar[l] & (U_{ab}, f_{ab}, u_{ab}) & (U^\circ_b, f^\circ_b, u^\circ_b) \ar_-{\Phi_b}[l] \ar[r] & (U_b, f_b, u_b).
    }
    \]
    \item Let
    \[
    \xymatrix{
    (U'_a, f'_a, u'_a) & (U^{\prime\circ}_a, f^{\prime\circ}_a, u^{\prime\circ}_a) \ar^-{\Phi'_a}[r] \ar[l] & (U'_{ab}, f'_{ab}, u'_{ab}) & (U^{\prime\circ}_b, f^{\prime\circ}_b, u^{\prime\circ}_b) \ar_-{\Phi'_b}[l] \ar[r] & (U'_b, f'_b, u'_b)
    }
    \]
    be the base change of the above local model along $p\colon B'\rightarrow B$ which provides a Zariski local model of the intersection $\Crit_{U'_a/B'}(f'_a)\times_{X'} \Crit_{U'_b/B'}(f'_b)$.
    \item The compatibility of the two local models of the isomorphism $\alpha_p$ follows from the compatibility of the stabilization isomorphisms with smooth base change, \cref{thm:nonlinearstabilization}(7).
\end{itemize}

For an arbitrary commutative diagram we factor it as
\[
\xymatrix{
X' \ar^{\tilde{p}}[dr] \ar@/_0.5pc/_{\pi'}[ddr] \ar@/^0.5pc/^{\overline{p}}[drr] && \\
& X\times_{B} B' \ar^{p'}[r] \ar[d] & X \ar^{\pi}[d] \\
& B' \ar^{p}[r] & B
}
\]
and then define $\alpha_{p, \overline{p}}$ as the composite $\alpha_{\tilde{p}}\circ \alpha_{p, p'}$.

The functoriality of $\alpha_{p, \overline{p}}$ can be checked in a critical chart, where it follows from the functoriality of $\Ex^!_\phi$ with respect to compositions, \cref{prop:constructiblevanishingcycles}(2).
\end{proof}

Finally, we establish a compatibility of perverse pullbacks with finite push-forwards.

\begin{proposition}\label{prop:perversefinitepushforwardcompatibility}
Let
\[
\xymatrix{
\tilde{X} \ar^{\overline{c}}[r] \ar^{\tilde{\pi}}[d] & X \ar^{\pi}[d] \\
\tilde{B} \ar^{c}[r] & B
}
\]
be a Cartesian diagram of schemes with $c\colon \tilde{B}\rightarrow B$ finite. Let $s\in\Gamma(X, \cS_{X/B})$ be a relative d-critical structure and denote by $\tilde{s}\in\Gamma(\tilde{X}, \cS_{\tilde{X}/\tilde{B}})$ its pullback. Let $o$ be an orientation of $X\rightarrow B$ and denote by $\tilde{o}$ its pullback. Then there is a natural isomorphism
\[\beta_c\colon \pi^\varphi c_*\xrightarrow{\sim} \overline{c}_* \tilde{\pi}^\varphi\]
of functors $\Perv(\tilde{B})\rightarrow \Perv(X)$
which satisfies the following properties:
\begin{enumerate}
    \item $\beta_{\id} = \id$.
    \item Given a diagram
    \[
    \xymatrix{
    \dbtilde{X} \ar^{\overline{d}}[r] \ar^{\dbtilde{\pi}}[d] & \tilde{X} \ar^{\overline{c}}[r] \ar^{\tilde{\pi}}[d] & X \ar^{\pi}[d] \\
    \dbtilde{B} \ar^{d}[r] & \tilde{B} \ar^{c}[r] & B
    }
    \]
    with $\dbtilde{B}\xrightarrow{d} \tilde{B}\xrightarrow{c} B$ a composite of finite morphisms and both squares Cartesian, we have
    \[\beta_c\circ\beta_d = \beta_{c\circ d}.\]
    \item Let
    \[
    \xymatrix{
    X' \ar^{\overline{p}}[r] \ar^{\pi'}[d] & X \ar^{\pi}[d] \\
    B' \ar^{p}[r] & B
    }
    \]
    be a diagram of schemes with $p\colon B'\rightarrow B$ and $X'\rightarrow B'\times_B X$ smooth. Let $s\in\Gamma(X, \cS_{X/B})$ be a relative d-critical structure and denote by $s'\in\Gamma(X', \cS_{X'/B'})$ its pullback. Let $c\colon \tilde{B}\rightarrow B$ be a finite morphism and denote by $\tilde{X} = X\times_B \tilde{B}$, $\tilde{B}'=B'\times_B \tilde{B}$ and $\tilde{X}'=X'\times_B \tilde{B}$ their base changes which fit into commutative diagrams
    \[
    \vcenter{
    \xymatrix{
    \tilde{X}' \ar^{\tilde{\overline{c}}}[r] \ar^{\tilde{\pi}'}[d] & X' \ar^{\overline{p}}[r] \ar^{\pi'}[d] & X \ar^{\pi}[d] \\
    \tilde{B}' \ar^{\tilde{c}}[r] & B' \ar^{p}[r] & B
    }}
    \quad=\quad\vcenter{
    \xymatrix{
    \tilde{X}' \ar^{\tilde{\overline{p}}}[r] \ar^{\tilde{\pi}'}[d] & \tilde{X} \ar^{\overline{c}}[r] \ar^{\tilde{\pi}}[d] & X \ar^{\pi}[d] \\
    \tilde{B}' \ar^{\tilde{p}}[r] & \tilde{B} \ar^{c}[r] & B.
    }}
    \]
    Let $o$ be an orientation of $X\rightarrow B$ and consider its pullbacks to $X'$, $\tilde{X}$ and $\tilde{X}'$. Then the diagram
    \begin{equation}\label{eq-pushforward-compatibility}
    \xymatrix{
    (\pi')^\varphi \tilde{c}_* \tilde{p}^\dag \ar^{\beta_{\tilde{c}}}[r] \ar^{\Ex^!_*}[d] & \tilde{\overline{c}}_*(\tilde{\pi}')^\varphi\tilde{p}^\dag \ar^-{\alpha_{\tilde{p}, \tilde{\overline{p}}}}[r] & \tilde{\overline{c}}_* \tilde{\overline{p}}^\dag \tilde{\pi}^\varphi \ar^{\Ex^!_*}[d] \\
    (\pi')^\varphi p^\dag c_* \ar^-{\alpha_{p, \overline{p}}}[r] & \overline{p}^\dag \pi^\varphi c_* \ar^{\beta_c}[r] & \overline{p}^\dag \overline{c}_* \tilde{\pi}^\varphi
    }
    \end{equation}
    commutes.
\end{enumerate}
If $R$ is a field, $\beta_c$ extends to a natural transformation of functors $\bD(\tilde{B})\rightarrow \bD(X)$.
\end{proposition}
\begin{proof}
Consider a cover of $X$ by critical charts $(U_a, f_a, u_a)$ and let $(\tilde{U}_a, \tilde{f}_a, \tilde{u}_a)$ be their base change along $c$ which define a cover of $\tilde{X}$ by critical charts. Let $c_a\colon \tilde{U}_a\rightarrow U_a$ be the projection. We define $\beta_c|_{\Crit_{U_a/B}(f_a)}$ as the unique natural isomorphism which fits into the commutative diagram
\[
\xymatrix{
(\pi^\varphi c_*)(-)|_{\Crit_{U_a/B}(f_a)} \ar^-{\eqref{eq:microlocalpullbacklocalmodel}}[r] \ar^{\beta_c|_{\Crit_{U_a/B}(f_a)}}[dd] & (\phi_{f_a} (U_a\rightarrow B)^\dag c_*)(-)\otimes Q^o_{U_a, f_a, u_a}|_{\Crit_{U_a/B}(f_a)} \ar^{\Ex^!_*\otimes \id}[d] \\
& (\phi_{f_a}(c_a)_*(\tilde{U}_a\rightarrow \tilde{B})^\dag)(-)\otimes Q^o_{U_a, f_a, u_a}|_{\Crit_{U_a/B}(f_a)} \ar^{\Ex^\phi_*\otimes\eqref{eq:Qpullback}}[d] \\
(\overline{c}_*\tilde{\pi}^\varphi)(-)|_{\Crit_{U_a/B}(f_a)} \ar^-{\eqref{eq:microlocalpullbacklocalmodel}}[r] & \overline{c}_*((\phi_{\tilde{f}_a} (\tilde{U}_a\rightarrow \tilde{B})^\dag)(-)\otimes Q^{\tilde{o}}_{\tilde{U}_a, \tilde{f}_a, \tilde{u}_a})|_{\Crit_{U_a/B}(f_a)}
}
\]

The fact that these locally defined natural isomorphisms $\beta_c|_{\Crit_{U_a/B}(f_a)}$ glue into a global natural isomorphism $\beta_c$ is shown as in the proof of \cref{prop:perversesmoothcompatibility}:
\begin{itemize}
    \item By \cref{prop:stabilizationzigzag} we have a Zariski local model of the intersection $\Crit_{U_a/B}(f_a)\times_X \Crit_{U_b/B}(f_b)$ given by
    \[
    \xymatrix{
    (U_a, f_a, u_a) & (U^\circ_a, f^\circ_a, u^\circ_a) \ar^-{\Phi_a}[r] \ar[l] & (U_{ab}, f_{ab}, u_{ab}) & (U^\circ_b, f^\circ_b, u^\circ_b) \ar_-{\Phi_b}[l] \ar[r] & (U_b, f_b, u_b).
    }
    \]
    \item Let
    \[
    \xymatrix{
    (\tilde{U}_a, \tilde{f}_a, \tilde{u}_a) & (\tilde{U}^\circ_a, \tilde{f}^\circ_a, \tilde{u}^\circ_a) \ar^-{\tilde{\Phi}_a}[r] \ar[l] & (\tilde{U}_{ab}, \tilde{f}_{ab}, \tilde{u}_{ab}) & (\tilde{U}^\circ_b, \tilde{f}^\circ_b, \tilde{u}^\circ_b) \ar_-{\tilde{\Phi}_b}[l] \ar[r] & (\tilde{U}_b, \tilde{f}_b, \tilde{u}_b)
    }
    \]
    be the base change of the above local model along $c\colon \tilde{B}\rightarrow B$ which provides a Zariski local model of the intersection $\Crit_{\tilde{U}_a/\tilde{B}}(\tilde{f}_a)\times_{\tilde{X}} \Crit_{\tilde{U}_b/\tilde{B}}(\tilde{f}_b)$.
    \item The compatibility of the two local models of the isomorphism $\beta_p$ follows from the compatibility of the stabilization isomorphisms with finite base change, \cref{thm:nonlinearstabilization}(8).
\end{itemize}

The properties of $\beta_c$ can be checked in a critical chart. The functoriality of $\beta_c$ follows from the functoriality of $\Ex^\phi_*$ and $\Ex^!_*$ with respect to compositions, \cref{prop:constructiblevanishingcycles}(2). Property (3) follows from \cref{prop:constructiblevanishingcycles}(3).
\end{proof}

\subsection{Compatibility with d-critical pushforwards}

In this section we establish a compatibility of perverse pullbacks with d-critical pushforwards. First we analyze a compatibility with d-critical pushforwards along smooth morphisms.

\begin{proposition}\label{prop:perversesmoothpushforwardcompatibility}
Let $\pi\colon X\rightarrow B$ be a morphism of schemes equipped with a relative d-critical structure $s$ and an orientation $o$. Let $p\colon B\rightarrow S$ be a smooth morphism. Let $p_*(X, s)$ be the d-critical pushforward which fits into a commutative diagram
\[
\xymatrix{
p_*(X, s) \ar^-{i}[r] \ar^{\overline{\pi}}[d] & X \ar^{\pi}[d] \\
S & B. \ar_{p}[l]
}
\]
Consider the orientation $p_*o$ of $p_*(X, s)$. Then there is a natural isomorphism
\[\gamma_p\colon \pi^\varphi p^\dag\xrightarrow{\sim} i_* \overline{\pi}^\varphi\]
of functors $\Perv(S)\rightarrow \Perv(X)$ which satisfies the following properties:
\begin{enumerate}
    \item It is functorial for compositions in $p$.
    \item Consider a commutative diagram
    \begin{equation}\label{eq:pushforwardpullbackdiagram}
    \xymatrix{
    X' \ar^{\pi'}[r] \ar^{q}[d] & B' \ar^{p'}[r] \ar^{a_1}[d] & S' \ar^{a_2}[d] \\
    X \ar^{\pi}[r] & B \ar^{p}[r] & S
    }
    \end{equation}
    of schemes with $p$, $a_2$, $B'\rightarrow S'\times_S B$ and $X'\rightarrow X\times_B B'$ smooth, a relative d-critical structure $s\in\Gamma(X, \cS_{X/B})$ and $s'=q^\ast s$ its pullback. Let
    \[
    \xymatrix{
    p'_*(X', s') \ar^{\overline{q}}[d] \ar^-{\overline{\pi}'}[r] & S' \ar^{a_2}[d] \\
    p_*(X, s) \ar^-{\overline{\pi}}[r] & S
    }
    \]
    be the diagram of d-critical pushforwards with $i'\colon p'_*(X', s')\rightarrow X'$ the natural inclusion. Let $o$ be an orientation of $X$ and consider orientations $q^* o$ of $X'$, $p_*o$ of $p_*(X, s)$ and $p'_* q^*o$ of $p'_*(X', s')$. Then the diagram of natural isomorphisms
    \[
    \xymatrix{
    q^\dag \pi^\varphi p^\dag \ar^{\gamma_p}[d] & (\pi')^\varphi a_1^\dag p^\dag \ar^{\sim}[r] \ar_{\alpha_{a_1, q}}[l] & (\pi')^\varphi (p')^\dag a_2^\dag \ar^{\gamma_{p'}}[d] \\
    q^\dag i_* \overline{\pi}^\varphi \ar^{\Ex^!_*}[r] & i'_* \overline{q}^\dag \overline{\pi}^\varphi & i'_* (\overline{\pi}')^\varphi a_2^\dag \ar_-{\alpha_{a_2, \overline{q}}}[l]
    }
    \]
    commutes.
    \item Consider a commutative diagram
    \[
    \xymatrix{
    \tilde{X} \ar^{q}[d] \ar^{\tilde{\pi}}[r] & \tilde{B} \ar^{c_1}[d] \ar^{\tilde{p}}[r] & \tilde{S} \ar^{c_2}[d] \\
    X \ar^{\pi}[r] & B \ar^{p}[r] & S
    }
    \]
    of schemes with both squares Cartesian, $c_2$ finite and $p$ smooth. Let $s\in\Gamma(X, \cS_{X/B})$ be a relative d-critical structure and let $\tilde{s} = \overline{c}^\ast s$ be its pullback to $\tilde{X}$. Let
    \[
    \xymatrix{
    \tilde{p}_*(\tilde{X}, \tilde{s}) \ar^{\overline{q}}[d] \ar^-{\overline{\tilde{\pi}}}[r] & \tilde{S} \ar^{c_2}[d] \\
    p_*(X, s) \ar^-{\overline{\pi}}[r] & S
    }
    \]
    be the Cartesian diagram of d-critical pushforwards with $\tilde{i}\colon \tilde{p}_*(\tilde{X}, \tilde{s})\rightarrow \tilde{X}$ the natural inclusion. Let $o$ be an orientation of $X$ and consider orientations $q^*o$ of $\tilde{X}$, $p_*o$ of $p_*(X, s)$ and $\tilde{p}_*q^*o$ of $\tilde{p}_*(\tilde{X}, \tilde{s})$. Then the diagram of natural isomorphisms
    \[
    \xymatrix{
    \pi^\varphi p^\dag (c_2)_* \ar^{\gamma_p}[d] \ar^{\Ex^!_*}[r] & \pi^\varphi (c_1)_*\tilde{p}^\dag \ar^{\beta_{c_1}}[r] & q_* \tilde{\pi}^\varphi \tilde{p}^\dag \ar^{\gamma_{\tilde{p}}}[d] \\
    i_*\overline{\pi}^\varphi (c_2)_* \ar^{\beta_{c_2}}[r] & i_* \overline{q}_* \overline{\tilde{\pi}}^\varphi \ar^{\sim}[r] & q_*\tilde{i}_*\overline{\tilde{\pi}}^\varphi
    }
    \]
    commutes.
\end{enumerate}
If $R$ is a field, $\gamma_p$ extends to a natural isomorphism of functors $\bD(S)\rightarrow \bD(X)$.
\end{proposition}
\begin{proof}
Consider a cover of $X$ by critical charts $(U_a, f_a, u_a)$ and let $(U_a, f_a, \overline{u}_a)$ be the corresponding cover of $p_*(X, s)$ by critical charts. Let $i_a\colon \Crit_{U_a/S}(f_a)\rightarrow \Crit_{U_a/B}(f_a)$ be the obvious inclusion. In the critical chart $(U_a, f_a, u_a)$ the isomorphism $\gamma_p$ is the isomorphism
\[\phi_{f_a} (U_a\rightarrow B)^\dag p^\dag(-)|_{\Crit_{U_a/B}(f_a)}\otimes Q^o_{U_a, f_a, u_a}\xrightarrow{\sim} (i_a)_*(\phi_{f_a} (U_a\rightarrow S)^\dag(-)|_{\Crit_{U_a/S}(f_a)}\otimes Q^{p_*o}_{U_a, f_a, u_a})\]
obtained from \eqref{eq:Qpushforward} and the unit of the adjunction $i_a^*\dashv (i_a)_*$ which is an isomorphism since $\phi_{f_a} (U_a\rightarrow S)^\dag(-)$ is supported on $\Crit_{U_a/S}(f_a)$ by \cref{prop:constructiblevanishingcycles}(1).

The compatibility of the two locally defined isomorphisms $\gamma_p$ in two critical charts follows as in the proof of \cref{prop:perversesmoothcompatibility} from the local model of the intersection given by \cref{prop:stabilizationzigzag} and \cref{thm:nonlinearstabilization}(9).

By construction $\gamma_p$ is obviously functorial for compositions in $p$. The rest of the properties can be checked in critical charts in which case they follow from the naturality of the unit of the adjunction $i^*\dashv i_*$ with respect to pullbacks.
\end{proof}

\subsection{Compatibility with duality and products}

We begin by establishing a compatibility of perverse pullbacks with Verdier duality. Using the isomorphisms $\Ex^{\dag, \bbD}$ and $\Ex_{*, \bbD}$ we observe that the category $\Perv(X)\cap \bbD(\Perv(X))$ is stable under $f^\dag$ for $f$ smooth and $f_*$ for $f$ finite. Moreover, if $R$ is a field, it coincides with the category $\Perv(X)$ of perverse sheaves.

\begin{proposition}\label{prop:perverseverdiercompatibility}
Let $\pi \colon X \to B$ be a morphism of schemes equipped with a relative d-critical structure $s$ and an orientation $o$. Let $\pi^\varphi$ be the perverse pullback for $(X\rightarrow B, s)$ with respect to $o$ and $\pi^{\varphi, -}$ the perverse pullback for $(X\rightarrow B, -s)$ with respect to $\overline{o}$. Then there is a natural isomorphism
\[\delta\colon \pi^\varphi \bbD\xrightarrow{\sim} \bbD \pi^{\varphi, -}\]
of functors $\Perv(B)\cap \bbD(\Perv(B))\rightarrow \Perv(X)$ which satisfies the following properties:
\begin{enumerate}
    \item There is a commutative diagram
    \[\begin{tikzcd}
      \pi^\varphi \ar{d}{\psi}\ar{r}{\psi}
      & \bbD\bbD\pi^\varphi \ar[leftarrow]{d}{\delta}
      \\
      \pi^\varphi \bbD\bbD\ar{r}{\delta}
      & \bbD\pi^{\varphi, -}\bbD.
    \end{tikzcd}\]
    \item Consider a commutative square of schemes
    \[\begin{tikzcd}
        X' \ar{r}{\overline{p}}\ar{d}{\pi'}
        & X \ar{d}{\pi}
        \\
        B' \ar{r}{p}
        & B
    \end{tikzcd}\]
    with $p \colon B' \to B$ and $X' \to B' \times_B X$ smooth, so that $X'\rightarrow B'$ carries the pullback oriented relative d-critical structure. Then the diagram
    \[\begin{tikzcd}
        (\pi')^\varphi p^\dag \bbD \ar{r}{\alpha_{p,\overline{p}}} \ar{d}{\Ex^{\dag,\bbD}}
        & \overline{p}^\dag\pi^\varphi \bbD \ar{r}{\delta}
        & \overline{p}^\dag \bbD \pi^{\varphi, -} \ar{d}{\Ex^{\dag,\bbD}} \\
        (\pi')^\varphi\bbD p^\dag \ar{r}{\delta}
        & \bbD(\pi')^{\varphi, -} p^\dag 
        & \bbD \overline{p}^\dag \pi^{\varphi, -} \ar{l}[swap]{\alpha_{p,\overline{p}}}
    \end{tikzcd}\]
    commutes.
    \item Let
    \[
    \xymatrix{
    \tilde{X} \ar^{\overline{c}}[r] \ar^{\tilde{\pi}}[d] & X \ar^{\pi}[d] \\
    \tilde{B} \ar^{c}[r] & B
    }
    \]
    be a Cartesian diagram of schemes with $c\colon \tilde{B}\rightarrow B$ finite, so that $\tilde{X}\rightarrow \tilde{B}$ carries the pullback oriented relative d-critical structure. Then the diagram
    \[
    \xymatrix{
    \pi^\varphi c_*\bbD \ar^{\beta_c}[r] \ar^{\Ex_{*, \bbD}}[d] & \overline{c}_* \tilde{\pi}^\varphi \bbD \ar^{\delta}[r] & \overline{c}_* \bbD \tilde{\pi}^{\varphi, -} \ar^{\Ex_{*, \bbD}}[d] \\
    \pi^\varphi \bbD c_* \ar^{\delta}[r] & \bbD \pi^{\varphi, -} c_* & \bbD \overline{c}_* \tilde{\pi}^{\varphi, -} \ar_{\beta_c}[l]
    }
    \]
    commutes.
    \item Let $\pi\colon X\rightarrow B$ be a morphism of schemes equipped with a d-critical structure $s$ and $p\colon B\rightarrow S$ a smooth morphism. Let $p_*(X, s)$ be the d-critical pushforward which fits into a commutative diagram
    \[
    \xymatrix{
    p_*(X, s) \ar^-{i}[r] \ar^{\overline{\pi}}[d] & X \ar^{\pi}[d] \\
    S & B. \ar_{p}[l]
    }
    \]
    Then the diagram
    \[
    \xymatrix{
    \pi^\varphi p^\dag\bbD\ar^{\gamma_p}[r] \ar^{\Ex^{\dag, \bbD}}[d] & i_*\overline{\pi}^\varphi \bbD \ar^{\delta}[r] & i_* \bbD \overline{\pi}^{\varphi, -} \ar^{\Ex_{*, \bbD}}[d] \\
    \pi^\varphi \bbD p^\dag \ar^{\delta}[r] & \bbD \pi^{\varphi, -} p^\dag & \bbD i_*\overline{\pi}^{\varphi, -} \ar_{\gamma_p}[l]
    }
    \]
    commutes.
\end{enumerate}
If $R$ is a field, $\delta$ extends to a natural isomorphism of functors $\bD(B)\rightarrow \bD(X)$.
\end{proposition}
\begin{proof}
Given a closed immersion of schemes $i\colon Z\rightarrow Y$ and a constructible complex $\cF\in\Dbc(Y)$ supported on $Z$ the natural morphism $\epsilon_i\colon i^!\cF\rightarrow i^*\cF$ is an isomorphism. In particular, in this case we have a natural isomorphism
\[i^*\bbD\cF\xleftarrow{\epsilon_i} i^!\bbD\cF\xrightarrow{\Ex^{!, \bbD}} \bbD i^*\cF\]
which will be implicit from now on.

Consider critical charts $(U_a, f_a, u_a)$ for $\pi \colon X \to B$. We define $\delta|_{\Crit_{U_a/B}(f_a)}$ as the unique natural isomorphism which fits into commutative diagrams
  \[\begin{tikzcd}
    (\pi^\varphi\bbD(-))|_{\Crit_{U_a/B}(f_a)} \ar[swap]{dd}{\delta|_{\Crit_{U_a/B}(f_a)}}\ar{r}{\eqref{eq:microlocalpullbacklocalmodel}}
    & (\phi_{f_a}(U_a \to B)^\dagger\bbD(-))|_{\Crit_{U_a/B}(f_a)}\otimes Q^o_{U_a, f_a, u_a} \ar{d}{\Ex^{\dag, \bbD}\otimes \id}
    \\
    & (\phi_{f_a}\bbD(U_a \to B)^\dagger(-))|_{\Crit_{U_a/B}(f_a)}\otimes Q^o_{U_a, f_a, u_a} \ar{d}{\Ex^{\phi, \bbD}\otimes \eqref{eq:Qreverse}}
    \\
    (\bbD\pi^{\varphi, -}(-))|_{\Crit_{U_a/B}(f_a)} \ar{r}{\eqref{eq:microlocalpullbacklocalmodel}}
    & (\bbD\phi_{-f_a}(U_a \to B)^\dagger(-))|_{\Crit_{U_a/B}(f_a)}\otimes Q^{\overline{o}}_{U_a, -f_a, u_a}.
  \end{tikzcd}\]
  The fact that these locally defined natural isomorphisms $\delta|_{\Crit_{U_a/B}(f_a)}$ glue into a global natural isomorphism $\delta
  $ is shown as in the proof of \cref{prop:perversesmoothcompatibility}, using the compatibility of the stabilization isomorphisms with Verdier duality, see \cref{thm:nonlinearstabilization}(\ref{item:nonlinearstabilization/verdier}).

  The properties of $\delta$ can be checked on critical charts. (1) follows by combining \cref{prop:phiD}(1) and \cref{prop:fdagDD}. (2) follows from \cref{prop:phiD}(2). (3) follows from \cref{prop:Ex!*selfdual} and \cref{prop:phiD}(3). (4) follows from \cref{prop:epsilonfgspD}.
\end{proof}

\begin{remark}
As in \cref{rmk:phireverse} we may define a natural isomorphism
\[\overline{\delta}\colon \pi^\varphi \bbD\xrightarrow{\sim} \bbD \pi^{\varphi}\]
which fits into a commutative diagram
\[
\xymatrix@C=1.3cm{
\pi^\varphi \ar^{\psi}[d] \ar^{T}[r] & \pi^\varphi \ar^-{\psi}[r]& \bbD \bbD \pi^\varphi \ar^{\overline{\delta}}[d] \\
\pi^\varphi \bbD \bbD \ar^{\overline{\delta}}[rr] && \bbD \pi^\varphi \bbD,
}
\]
where $T\colon \pi^\varphi\rightarrow \pi^\varphi$ is the natural isomorphism given locally by the monodromy operator of the functor of vanishing cycles.
\end{remark}

Next we establish a compatibility of perverse pullbacks with products. If $R$ is not a field, in general the external tensor product does not preserve perversity already in the case when the base is a point. We denote by
\[(\Perv(B_1)\times\Perv(B_2))_{\boxtimes\Perv}\subset \Perv(B_1)\times\Perv(B_2)\]
the full subcategory consisting of pairs $(\cF_1, \cF_2)$ such that $\cF_1\boxtimes\cF_2\in\Dbc(B_1\times B_2)$ is perverse.

\begin{proposition}\label{prop:perverseproductscompatibility}
Let $\pi_1 \colon X_1 \rightarrow B_1$ and $\pi_2 \colon X_2 \rightarrow B_2$ be morphisms of schemes equipped with relative d-critical structures $s_1$ and $s_2$ and orientations $o_1$ and $o_2$, respectively. Consider the morphism $\pi_1\times \pi_2\colon X_1 \times X_2 \to B_1 \times B_2$ equipped with the relative d-critical structure $s_1 \boxplus s_2$ and orientation $o_1\boxtimes o_2$. Then there is a natural isomorphism
\[\TS\colon \pi_1^\varphi(-)\boxtimes \pi_2^\varphi(-)\xrightarrow{\sim} (\pi_1\times \pi_2)^\varphi(-\boxtimes -)\]
of functors $(\Perv(B_1)\times\Perv(B_2))_{\boxtimes\Perv}\rightarrow \Perv(X_1\times X_2)$. It satisfies the following properties:
\begin{enumerate}
    \item It is associative: given another morphism of schemes $\pi_3\colon X_3\rightarrow B_3$ equipped with an oriented relative d-critical structure $(s_3, o_3)$ the diagram
    \[
    \xymatrix{
    \pi_1^\varphi(-)\boxtimes \pi_2^\varphi(-)\boxtimes \pi_3^\varphi(-)\ar^-{\TS\boxtimes \id}[r] \ar^{\id\boxtimes \TS}[d] & (\pi_1\times \pi_2)^\varphi(-\boxtimes -)\boxtimes \pi_3^\varphi(-) \ar^{\TS}[d] \\
    \pi_1^\varphi(-)\boxtimes (\pi_2\times \pi_3)^\varphi(-\boxtimes -) \ar^-{\TS}[r] & (\pi_1\times \pi_2\times \pi_3)^\varphi(-\boxtimes -\boxtimes -).
    }
    \]
    commutes, where we use the natural isomorphism $o_1\boxtimes (o_2\boxtimes o_3)\cong (o_1\boxtimes o_2)\boxtimes o_3$ of orientations.
    \item It is unital: if $\pi_1\colon \pt\rightarrow \pt$ equipped with the relative d-critical structure $s=0$ and the trivial orientation, then $\TS=\id$.
    \item It is graded-commutative: for the swapping isomorphism $\sigma\colon X_2\times X_1\rightarrow X_1\times X_2$ and $\cF_1\in\Dbc(X_1),\cF_2\in\Dbc(X_2)$ the diagram
    \[
    \xymatrix{
    \sigma^*(\pi_1^\varphi(\cF_1)\boxtimes \pi_2^\varphi(\cF_2)) \ar^{\TS}[r] \ar^{\sim}[d] & \sigma^*(\pi_1\times \pi_2)^\varphi(\cF_1\boxtimes\cF_2) \ar^{\alpha}[d] \\
    \pi_2^\varphi(\cF_2)\boxtimes \pi_1^\varphi(\cF_1) \ar^{\TS}[r] & (\pi_2\times \pi_1)^\varphi(\cF_2\boxtimes \cF_1),
    }
    \]
    where the right vertical isomorphism uses the natural isomorphism $\sigma^*(o_1\boxtimes o_2)\xrightarrow{\eqref{eq:orientationbraiding}} o_2\boxtimes o_1$ of orientations, commutes up to $(-1)^{\deg(o_1)\deg(o_2)}$.
    \item For $i=1, 2$ let
    \[
    \xymatrix{
    X'_i \ar^{\overline{p}_i}[r] \ar^{\pi'_i}[d] & X_i \ar^{\pi_i}[d] \\
    B'_i \ar^{p_i}[r] & B_i
    }
    \]
    be a commutative diagram of schemes with $p_i\colon B'_i\rightarrow B_i$ and $X'_i\rightarrow B'_i\times_{B_i} X_i$ smooth. Consider oriented relative d-critical structures on $X_i\rightarrow B_i$ and their pullbacks to $X'_i\rightarrow B'_i$. Then the diagram
    \[
    \xymatrix{
    (\pi'_1)^\varphi p_1^\dag(-)\boxtimes (\pi'_2)^\varphi p_2^\dag(-) \ar^-{\TS}[r] \ar^{\alpha_{p_1, \overline{p}_1}\boxtimes \alpha_{p_2, \overline{p}_2}}[d] & (\pi'_1\times \pi'_2)^\varphi(p_1^\dag(-)\boxtimes p_2^\dag(-)) \ar^-{\sim}[r] & (\pi'_1\times \pi'_2)^\varphi(p_1\times p_2)^\dag(-) \ar^{\alpha_{p_1\times p_2, \overline{p}_1\times \overline{p}_2}}[d] \\
    \overline{p}^\dag_1\pi_1^\varphi(-)\boxtimes \overline{p}^\dag_2\pi_2^\varphi(-) \ar^-{\sim}[r] & (\overline{p}_1\times \overline{p}_2)^\dag(\pi_1^\varphi(-)\boxtimes \pi_2^\varphi(-)) \ar^{\TS}[r] & (\overline{p}_1\times \overline{p}_2)^\dag(\pi_1\times \pi_2)^\varphi(-\boxtimes -)
    }
    \]
    commutes.
    \item For $i=1, 2$ let
    \[
    \xymatrix{
    \tilde{X}_i \ar^{\overline{c}_i}[r] \ar^{\tilde{\pi}_i}[d] & X_i \ar^{\pi_i}[d] \\
    \tilde{B}_i \ar^{c_i}[r] & B_i
    }
    \]
    be a Cartesian diagram of schemes with $c_i\colon \tilde{B}_i\rightarrow B_i$ finite. Consider oriented relative d-critical structures on $X_i\rightarrow B_i$ and their pullbacks to $\tilde{X}_i\rightarrow \tilde{B}_i$. Then the diagram
    \[
    \xymatrix{
    \pi_1^\varphi c_{1, *}(-)\boxtimes \pi_2^\varphi c_{2, *}(-) \ar^-{\TS}[r] \ar^{\beta_{c_1}\boxtimes \beta_{c_2}}[d] & (\pi_1\times \pi_2)^\varphi(c_{1, *}(-)\boxtimes c_{2, *}(-)) \ar^-{\sim}[r] & (\pi_1\times \pi_2)^\varphi(c_1\times c_2)_*(-\boxtimes -) \ar^{\beta_{c_1\times c_2}}[d] \\
    \overline{c}_{1, *} \tilde{\pi}^\varphi_1(-)\boxtimes \overline{c}_{2, *} \tilde{\pi}^\varphi_2(-) \ar^-{\sim}[r] & (\overline{c}_1\times \overline{c}_2)_* (\tilde{\pi}^\varphi_1(-)\boxtimes \tilde{\pi}^\varphi_2(-)) \ar^-{\TS}[r] & (\overline{c}_1\times \overline{c}_2)_* (\tilde{\pi}_1\times \tilde{\pi}_2)^\varphi(-\boxtimes -)
    }
    \]
    commutes.
    \item For $i=1, 2$ let $X_i\xrightarrow{\pi_i} B_i\xrightarrow{p_i} S_i$ be morphisms of schemes, where $\pi_i$ is equipped with an oriented relative d-critical structure $(s_i, o_i)$. Let $p_{i, *}(X_i, s_i)$ be the d-critical pushforward which fits into a commutative diagram
    \[
    \xymatrix{
    p_{i,*}(X_i, s_i) \ar^-{i}[r] \ar^{\overline{\pi}}[d] & X_i \ar^{\pi_i}[d] \\
    S_i & B_i. \ar_{p_i}[l]
    }
    \]
    Consider the orientation $p_{i,*}o_i$ of $p_{i,*}(X_i, s_i)$. Then the diagram
    \[
    \xymatrix{
    \pi_1^\varphi p_1^\dag(-)\boxtimes \pi_2^\varphi p_2^\dag(-) \ar^-{\TS}[r] \ar^{\gamma_{p_1}\boxtimes \gamma_{p_2}}[d] & (\pi_1\times \pi_2)^\varphi(p_1^\dag(-)\boxtimes p_2^\dag(-)) \ar^{\sim}[r] & (\pi_1\times \pi_2)^\varphi (p_1\times p_2)^\dag(-\boxtimes -) \ar^{\gamma_{p_1\times p_2}}[d] \\
    i_{1,*}\overline{\pi}^\varphi_1(-)\boxtimes i_{2,*}\overline{\pi}^\varphi_2(-) \ar^{\sim}[r] & (i_1\times i_2)_*(\overline{\pi}^\varphi_1(-)\boxtimes \overline{\pi}^\varphi_2(-)) \ar^-{\TS}[r] & (i_1\times i_2)_* (\overline{\pi}_1\times \overline{\pi}_2)^\varphi(-\boxtimes -)
    }
    \]
    commutes.
    \item The diagram
    \[
      \xymatrix{
      \pi_1^\varphi(\bbD(-))\boxtimes \pi_2^\varphi(\bbD(-)) \ar^-{\TS}[r] \ar^{\delta\boxtimes\delta}[d] & (\pi_1\times \pi_2)^\varphi(\bbD(-)\boxtimes\bbD(-)) \ar^{\Ex^{\bbD, \boxtimes}}[r] & (\pi_1\times \pi_2)^\varphi(\bbD(-\boxtimes -)) \ar^{\delta}[d] \\
      \bbD(\pi_1^{\varphi, -}(-))\boxtimes\bbD(\pi_2^{\varphi, -}(-)) \ar^{\Ex^{\bbD, \boxtimes}}[r] & \bbD(\pi_1^{\varphi, -}(-)\boxtimes\pi_2^{\varphi, -}(-)) & \bbD((\pi_1\times \pi_2)^{\varphi, -}(-\boxtimes -)) \ar_-{\TS}[l]
      }
    \]
    commutes.
\end{enumerate}
If $R$ is a field, $\TS$ extends to a natural isomorphism of functors $\bD(B_1)\times \bD(B_2)\rightarrow \bD(X_1\times X_2)$.
\end{proposition}
\begin{proof}
Let $(U_{1,a}, f_{1,a}, u_{1,a})$ and $(U_{2,a}, f_{2,a}, u_{2,a})$ be critical charts for $\pi_1$ and $\pi_2$, respectively. Write $U_{12,a} \coloneqq U_{1,a} \times U_{2,a}$, $B_{12} \coloneqq B_1 \times B_2$, and $f_{12,a} \coloneqq f_{1,a} \boxplus f_{2,a}$, and denote by $p_{1,a} \colon U_{1,a} \to B_1$, $p_{2,a} \colon U_{2,a} \to B_2$, and $p_{12,a} \colon U_{12,a} \to B_{12}$ the projections. Recall that $(U_{12,a}, f_{12,a}, u_{1,a} \times u_{2,a})$ is a critical chart for $\pi$. Identifying $\Crit_{U_{12,a}/B_{12}}(f_{12,a}) \cong \Crit_{U_{1,a}/B_1}(f_{1,a}) \times \Crit_{U_{2,a}/B_2}(f_{2,a})$, we define $\TS|_{\Crit_{U_{12, a}/B_{12}}(f_{12, a})}$ as the unique isomorphism which fits into commutative diagrams
\[
\begin{tikzcd}[column sep=1.8cm,scale cd=0.8]
(\pi_1^\phi(-) \boxtimes \pi_2^\phi(-))|_{\Crit_{U_{12,a}/B_{12}}(f_{12,a})} \ar{d}{\TS|_{\Crit_{U_{12,a}/B_{12}}}} \ar{r}{\eqref{eq:microlocalpullbacklocalmodel} \boxtimes \eqref{eq:microlocalpullbacklocalmodel}} & (\phi_{f_{1,a}} p_{1, a}^\dag(-) \boxtimes \phi_{f_{2,a}} p_{2, a}^\dag(-))|_{\Crit_{U_{12,a}/B_{12}}(f_{12,a})} \otimes Q^{o_1}_{U_{1,a},f_{1,a},u_{1,a}}\boxtimes Q^{o_2}_{U_{2,a},f_{2,a},u_{2,a}} \ar{d}{\TS\otimes \eqref{eq:Qproduct}} \\
(\pi^\phi(-\boxtimes -))|_{\Crit_{U_{12,a}/B_{12}}(f_{12,a})} \ar{r}{\eqref{eq:microlocalpullbacklocalmodel}} & (\phi_{f_{12,a}} p_{12, a}^\dag(-\boxtimes -))|_{\Crit_{U_{12,a}/B_{12}}(f_{12,a})} \otimes Q^{o_1\boxtimes o_2}_{U_{1,a}\times U_{2,a},f_{1,a}\boxplus f_{2,a},u_{1,a}\times u_{2,a}}
\end{tikzcd}
\]
for each $a$. The fact that these locally defined isomorphisms glue into the global isomorphism $\TS$ is shown as in the proof of \cref{prop:perversesmoothcompatibility}, using \cref{thm:nonlinearstabilization}(\ref{item:nonlinearstabilization/TS}).

The associativity and unitality property of $\tau$ follows from the corresponding properties of the Thom--Sebastiani isomorphism $\TS$ for the sheaves of vanishing cycles. The graded-commutativity of $\tau$ follows from the following facts:
\begin{enumerate}
    \item The Thom--Sebastiani isomorphism $\TS$ for the sheaves of vanishing cycles is commutative.
    \item The diagram
    \[
    \xymatrix{
    \sigma^*(p_{1, a}^\dag(\cF_1)\boxtimes p_{2, a}^\dag(\cF_2)) \ar^-{\sim}[r] \ar^{\sim}[d] & \sigma^* p_{12, a}^\dag(\cF_1\boxtimes \cF_2) \ar^{\sim}[d] \\
    p_{2, a}^\dag(\cF_2)\boxtimes p_{1, a}^\dag(\cF_1) \ar^-{\sim}[r] & p_{21, a}^\dag(\cF_2\boxtimes \cF_1)
    }
    \]
    commutes up to the sign $(-1)^{\dim(U_{1, a}/B_1)\dim(U_{2, a}/B_2)}=(-1)^{\deg(o_1)\deg(o_2)}$ due to the presence of the shifts $[-\dim(U_{i, a}/B_i)]$ in the definition of $p_{i, a}^\dag$.
    \item The commutativity of the diagram
    \[
    \xymatrix{
    \sigma^*(Q^{o_1}_{U_{1,a},f_{1,a},u_{1,a}}\boxtimes Q^{o_2}_{U_{2,a},f_{2,a},u_{2,a}}) \ar^{\eqref{eq:Qproduct}}[r] \ar^{\sim}[dd] & \sigma^*(Q^{o_1\boxtimes o_2}_{U_{1,a}\times U_{2,a},f_{1,a}\boxplus f_{2,a},u_{1,a}\times u_{2,a}}) \ar^{\eqref{eq:Qpullback}}[d] \\
    & Q^{\sigma^*(o_1\boxtimes o_2)}_{U_{2,a}\times U_{1,a},f_{2,a}\boxplus f_{1,a},u_{2,a}\times u_{1,a}} \ar^{\eqref{eq:orientationbraiding}}[d] \\
    Q^{o_2}_{U_{2,a},f_{2,a},u_{2,a}}\boxtimes Q^{o_1}_{U_{1,a},f_{1,a},u_{1,a}} \ar^{\eqref{eq:Qproduct}}[r] & Q^{o_2\boxtimes o_1}_{U_{2,a}\times U_{1,a},f_{2,a}\boxplus f_{1,a},u_{2,a}\times u_{1,a}}
    }
    \]
    is equivalent to the commutativity of the diagram
    \[
    \begin{tikzcd}[scale cd=0.8]
    \sigma^*(\cL_1\boxtimes \cL_2) \ar{r}{o_1, o_2} \ar{d}{\eqref{eq:orientationbraiding}} & \sigma^*(K_{U_{1, a}}|_{\Crit_{U_{1, a}/B_1}(f_{1, a})}\boxtimes K_{U_{2, a}}|_{\Crit_{U_{2, a}/B_2}(f_{2, a})}) \ar{r}{\sim} & \sigma^*(K_{U_{1, a}\times U_{2, a}}|_{\Crit_{U_{1, a}\times U_{2, a}/B_1\times B_2}(f_{1, a}\boxplus f_{2, a})}) \ar{d}{\sim} \\
    \cL_2\boxtimes \cL_1 \ar{r}{o_2, o_1} & K_{U_{2, a}}|_{\Crit_{U_{2, a}/B_2}(f_{2, a})}\boxtimes K_{U_{1, a}}|_{\Crit_{U_{1, a}/B_1}(f_{1, a})} \ar{r}{\sim} & K_{U_{2, a}\times U_{1, a}}|_{\Crit_{U_{2, a}\times U_{1, a}/B_2\times B_1}(f_{2, a}\boxplus f_{1, a})},
    \end{tikzcd}
    \]
    which commutes on the nose.
\end{enumerate}
The compatibilities of $\alpha$, $\beta$ and $\gamma$ with respect to products follow from the compatibility of the isomorphisms $\Ex^!_\phi$, $\Ex^\phi_*$ (see \cref{prop:phiTS}) as well as the unit $\id\rightarrow i_*i^*$ for a closed immersion $i$ with products. The diagram in property (3) in a critical chart reduces to the commutative diagram from \cref{prop:phiTS}(4).
\end{proof}

\subsection{Perverse pullbacks for stacks}

We will now define the operation of perverse pullback along any morphism of higher Artin stacks equipped with a relative d-critical structure.

\begin{theorem}\label{thm:micropullstk}
Let $\pi \colon X\rightarrow B$ be a morphism of higher Artin stacks locally of finite type equipped with an oriented relative d-critical structure. There is an exact \defterm{perverse pullback} functor
\[\pi^\varphi\colon \Perv(B)\longrightarrow \Perv(X)\]
together with the following natural isomorphisms:
\begin{enumerate}
    \item Let
    \[
    \xymatrix{
    X' \ar^{\overline{p}}[r] \ar^{\pi'}[d] & X \ar^{\pi}[d] \\
    B' \ar^{p}[r] & B
    }
    \]
    be a commutative diagram of higher Artin stacks locally of finite type with $p\colon B'\rightarrow B$ and $X'\rightarrow B'\times_B X$ smooth and the pullback oriented relative d-critical structure on $X'\rightarrow B'$. Then there is a natural isomorphism
    \[\alpha_{p, \overline{p}}\colon (\pi')^\varphi p^\dag\xrightarrow{\sim} \overline{p}^\dag \pi^\varphi,\]
    \item Let
    \[
    \xymatrix{
    \tilde{X} \ar^{\overline{c}}[r] \ar^{\tilde{\pi}}[d] & X \ar^{\pi}[d] \\
    \tilde{B} \ar^{c}[r] & B
    }
    \]
    be a Cartesian diagram of higher Artin stacks locally of finite type with $c\colon \tilde{B}\rightarrow B$ finite and the pullback oriented relative d-critical structure on $\tilde{X}\rightarrow \tilde{B}$. Then there is a natural isomorphism
    \[\beta_c\colon \pi^\varphi c_*\xrightarrow{\sim} \overline{c}_*\tilde{\pi}^\varphi.\]
    \item Consider morphisms $X\xrightarrow{\pi}B\xrightarrow{p} S$ of higher Artin stacks locally of finite type, where $p$ is smooth and $\pi$ is equipped with a relative d-critical structure $s$ and an orientation $o$. Let $p_*(X, s)$ be the d-critical pushforward which fits into a commutative diagram
    \[
    \xymatrix{
    p_*(X, s) \ar^-{i}[r] \ar^{\overline{\pi}}[d] & X \ar^{\pi}[d] \\
    S & B. \ar_{p}[l]
    }
    \]
    and which carries the orientation $p_* o$. Then there is a natural isomorphism
    \[\gamma_p\colon \pi^\varphi p^\dag\xrightarrow{\sim} i_* \overline{\pi}^\varphi.\]
    \item For a morphism $X\xrightarrow{\pi} B$ of higher Artin stacks locally of finite type equipped with an oriented relative d-critical structure there is a natural isomorphism
    \[\delta\colon \pi^\varphi\bbD\xrightarrow{\sim}\bbD\pi^{\varphi, -},\]
    where $\pi^{\varphi, -}$ denotes the perverse pullback with respect to the opposite oriented relative d-critical structure.
    \item For a pair of morphisms $\pi_1\colon X_1\rightarrow B_1$ and $\pi_2\colon X_2\rightarrow B_2$ equipped with oriented relative d-critical structures with $\pi_1\times \pi_2\colon X_1\times X_2\rightarrow B_1\times B_2$ equipped with the product oriented relative d-critical structure there is a natural isomorphism
    \[\TS\colon \pi_1^\varphi(-)\boxtimes \pi_2^\varphi(-)\xrightarrow{\sim}(\pi_1\times \pi_2)^\varphi(-\boxtimes -).\] \label{item-micropullstk-TS}
\end{enumerate}
These natural isomorphisms satisfy natural compatibility relations as in \cref{prop:perversesmoothbasechangecompatibility,prop:perversefinitepushforwardcompatibility,prop:perversesmoothpushforwardcompatibility,prop:perverseverdiercompatibility,prop:perverseproductscompatibility}. Moreover, if $R$ is a field, $\pi^\varphi$ (along with the natural isomorphisms $\alpha,\beta,\gamma,\delta,\TS$ and their compatibility relations) extends to a colimit-preserving $t$-exact functor
\[\pi^\varphi\colon \bD(B)\longrightarrow \bD(X).\]
\end{theorem}
\begin{proof}
Denote by $\DCrit(\Art^{\lft})$ the $\infty$-category of triples $(X \to B, s, o)$ where $X \to B$ is a morphism in $\Art^{\lft}$ equipped with a relative d-critical structure $s$ and an orientation $o$.
A morphism $(X_1 \to B_1, s_1, o_1) \to (X_2 \to B_2, s_2, o_2)$ is a morphism $(X_1 \to B_1) \to (X_2 \to B_2)$ in $\Fun(\Delta^1, \Art^{\lft})_{0\smooth,1\smooth}$ such that $s_1$ is the pullback of $s_2$, together with an isomorphism of orientations $o_1 \cong (X_1 \to X_2)^*(o_2)$.
Consider also the full subcategory $\DCrit(\Sch^{\sep\ft})$ consisting of $(X \to B, s, o)$ where $X$ and $B$ lie in $\Sch^{\sep\ft}$.
We consider the functor
\[
  \Theta \colon
  \DCrit(\Sch^{\sep\ft})^\op \to \Fun(\Delta^1, \Cat^{\mathrm{ab}}),
\]
where $\Cat^{\mathrm{ab}}$ is the category of $R$-linear abelian categories and $R$-linear exact functors, determined by the following data:
\begin{itemize}
    \item For every $(\pi \colon X \to B, s, o)$ in $\DCrit(\Sch^{\sep\ft})$, the perverse pullback functor $\pi^\varphi \colon \Perv(B) \to \Perv(X)$ defined in \cref{thm:micropullsch}.
    \item For every commutative square in $\Sch^{\sep\ft}$
    \[\begin{tikzcd}
        X' \ar{r}{\pi'}\ar{d}{\overline{p}}
        & B' \ar{d}{p}
        \\
        X \ar{r}{\pi}
        & B
    \end{tikzcd}\]
    where $B' \to B$ and $X' \to X \times_B B'$ are smooth, the commutative square in $\Cat$
    \[\begin{tikzcd}
        \Perv(B) \ar{r}{\pi^\varphi}\ar{d}{p^\dag}
        & \Perv(X) \ar{d}{\overline{p}^\dag}
        \\
        \Perv(B') \ar{r}{\pi'^\varphi}
        & \Perv(X')
    \end{tikzcd}\]
    determined by the invertible natural transformation $\alpha_{p,\overline{p}}$ of \cref{prop:perversesmoothbasechangecompatibility}, where $\pi'$ is equipped with the pullback relative d-critical structure $\overline{p}^*(s)$ and orientation $\overline{p}^*(o)$.
\end{itemize}
That this defines a functor is guaranteed by \cref{prop:perversesmoothbasechangecompatibility}(1,2).
Since $\Theta$ satisfies \'etale descent, it extends uniquely to an \'etale sheaf
\[\Theta \colon \DCrit(\Art^{\lft})^\op \to \Fun(\Delta^1, \Cat)\]
which by \cite[Corollary 3.2.6]{KhanWeavelisse} is given by the formula (compare \eqref{eq:ShvArtrel}):
\[
    \Theta(X\rightarrow B, s,o) = \lim_{(X'\rightarrow B',s',o', p)} \Theta(X'\rightarrow B', s',o'),
\]
where the limit is taken over $(X'\rightarrow B',s',o', p) \in \DCrit^{\ori}(\Sch^{\sep\ft})_{/(X \to B, s,o)}$ with $p \colon (X'\rightarrow B',s',o') \to (X\to B, s,o)$ a morphism in $\DCrit^{\ori}(\Art^{\lft})$ and $(X'\rightarrow B',s',o') \in \DCrit^{\ori}(\Sch^{\sep\ft})$.
As the forgetful functor $\DCrit^{\ori}(\Sch^{\sep\ft})_{/(X \to B, s,o)} \to (\Fun(\Delta^1, \Sch^{\sep\ft})_{0\smooth, 1\smooth})_{/(X \to B)}$ is cofinal, it can equivalently be taken over $(X' \to B', p)$ with $p \colon (X'\rightarrow B') \to (X\to B)$ in $\Fun(\Delta^1, \Art^{\lft})_{0\smooth, 1\smooth}$ and $X',B' \in \Sch^{\sep\ft}$.

Unwinding, we have for every morphism $\pi\colon X\to B$ in $\Art^\lft$ and every oriented relative d-critical structure $(s,o)$ a functor $\Theta(\pi\colon X \to B, s,o)$, which is by definition the limit of the functors
\[
    \pi'^\varphi \colon \Perv(B') \to \Perv(X')
\]
over morphisms $p \colon (\pi' \colon X' \to B') \to (\pi \colon X\to B)$ in $\Fun(\Delta^1, \Art^{\lft})_{0\smooth,1\smooth}$ where $X',B' \in \Sch^{\sep\ft}$ and $\pi'$ is equipped with the pullback oriented relative critical structure.
A simple cofinality argument shows that we have the natural equivalences
\[
    \lim_{(X' \to B', p)} \Perv(B') \cong \Perv(B),
    \qquad \lim_{(X' \to B', p)} \Perv(X') \cong \Perv(X),
\]
under which $\Theta_{X/B,s,o}$ is identified with a natural functor
\[\pi^\varphi \colon \Perv(B) \to \Perv(X).\]
Moreover, $\Theta$ encodes the (associative) natural isomorphisms $\alpha_{p,\overline{p}}$ of (1).

The natural isomorphisms $\beta$, $\gamma$, $\delta$ and $\TS$ defined for separated schemes of finite type in \cref{prop:perversefinitepushforwardcompatibility,prop:perversesmoothpushforwardcompatibility,prop:perverseverdiercompatibility,prop:perverseproductscompatibility} extend to natural isomorphisms defined for higher Artin stacks locally of finite type using the commutation of $\beta$, $\gamma$, $\delta$ and $\tau$ with $\alpha$.
For example, in the situation of (3) let $S_0 \twoheadrightarrow S$, $B_0 \twoheadrightarrow B \times_S S_0$, and $X_0 \twoheadrightarrow X \times_B B_0$ be smooth surjections from schemes in $\Sch^{\sep\ft}$, and denote by $S_\bullet$, $B_\bullet$ and $X_\bullet$ the respective \v{C}ech nerves of $S_0 \twoheadrightarrow S$, $B_0 \twoheadrightarrow B$, and $X_0 \twoheadrightarrow X$.
Denote by $\pi_\bullet \colon X_\bullet \to B_\bullet$ and $p_\bullet \colon B_\bullet \to S_\bullet$ the induced morphisms and by $(s_\bullet, o_\bullet)$ the induced oriented relative d-critical structures.
Note that by \cref{prop:dcritpushforwardpullback}, the induced morphism $p_{0,*}(X_0,s_0) \twoheadrightarrow p_*(X,s)$ is smooth surjective and its \v{C}ech nerve is identified with the d-critical push-forward $p_{\bullet,*}(X_\bullet, s_\bullet)$.
Thus, we may conclude by repeatedly applying the following observation: if we are given natural isomorphisms $\gamma_\bullet \colon \pi_\bullet^\varphi p_\bullet^\dag \cong i_{\bullet,*} \overline{\pi}_\bullet^\varphi$ compatible with smooth pullbacks, then by totalization we obtain a natural isomorphism $\gamma \colon \pi^\varphi p^\dag \cong i_* \overline{\pi}^\varphi$ compatible with smooth pullbacks.
We first apply this in the case when $X$, $B$ and $S$ have schematic diagonal, so that the $X_\bullet$, $B_\bullet$, and $S_\bullet$ are all schemes and we have the $\gamma_\bullet$ by \cref{prop:perversesmoothpushforwardcompatibility}.
Next, if $X$, $B$ and $S$ are $1$-Artin then $X_\bullet$, $B_\bullet$, and $S_\bullet$ are all algebraic spaces (hence a fortiori have schematic diagonal), so we have the $\gamma_\bullet$ by the previous case.
Similarly, the case of $n$-Artin stacks reduces to that of $(n-1)$-Artin stacks for each $n>1$.

Moreover, using the same commutation with $\alpha$ the proofs of the compatibility relations between the natural isomorphisms $\beta$, $\gamma$, $\delta$ and $\TS$ for morphisms of higher Artin stacks reduce to the compatibility relations between these natural isomorphisms for morphisms of separated schemes of finite type.

Finally, suppose that $R$ is a field, so that $\bD(X) \cong \Ind(\Db(\Perv(X)))$ for any $X \in \Sch^{\sep\ft}$ by \cref{prop:constructible6functor}(7).
Applying \cref{prop:Catab}, $\Theta \colon \DCrit(\Sch^{\sep\ft})^\op \to \Fun(\Delta^1, \Cat^{\mathrm{ab}})$ promotes to a functor
\[
  \widetilde{\Theta} \colon \DCrit(\Sch^{\sep\ft})^\op \to \Fun(\Delta^1, \mathrm{Pr}_{\mathrm{L}}^{\mathrm{St},t})
\]
valued in the $\infty$-category of presentable stable $R$-linear $\infty$-categories equipped with a t-structure, encoding the t-exact perverse pullbacks $\pi^\varphi \colon \bD(B) \to \bD(X)$ for $(\pi \colon X \to B, s, o)$ in $\DCrit(\Sch^{\sep\ft})$, together with the natural isomorphisms $\alpha_{p,\overline{p}}$, associative up to coherent homotopy.
As above, we now apply right Kan extension to obtain a functor
\[
  \widetilde{\Theta} \colon \DCrit(\Art^{\lft})^\op \to \Fun(\Delta^1, \mathrm{Pr}_{\mathrm{L}}^{\mathrm{St},t})
\]
encoding the t-exact perverse pullbacks $\pi^\varphi \colon \bD(B) \to \bD(X)$ for $(\pi \colon X \to B, s, o)$ in $\DCrit(\Art^{\lft})$, together with the associative natural isomorphisms $\alpha_{p,\overline{p}}$.
The compatibilities between $\alpha$ and the various natural isomorphisms $\beta$, $\gamma$, $\delta$, and $\TS$ can be encoded as invertible $2$-morphisms in the $\infty$-category $\Fun(\Delta^1, \mathrm{Pr}_{\mathrm{L}}^{\mathrm{St},t})$, so these also extend to $\Art^\lft$.
\end{proof}

\section{Perverse pullbacks for exact \texorpdfstring{$(-1)$}{-1}-shifted symplectic fibrations}

In this section we translate the construction of perverse pullbacks from the setting of morphisms of algebraic stacks equipped with a relative d-critical structure to the setting of morphisms of derived algebraic stacks equipped with a relative exact $(-1)$-shifted symplectic structure.

\subsection{D-critical and shifted symplectic structures}

The main source of relative d-critical structures is given by the theory of shifted symplectic structures \cite{PTVV}. Namely, given a morphism of derived stacks $X\rightarrow B$ the references \cite{PTVV,CPTVV,CHS,ParkSymplectic} define the spaces $\cA^{2, \cl}(X/B, n)$ ($\cA^{2, \ex}(X/B, n)$) of relative closed (exact) two-forms of degree $n$ using the cotangent complex of derived stacks together with a natural morphism $\cA^{2, \ex}(X/B, n)\rightarrow \cA^{2, \cl}(X/B, n)$ and a forgetful map $\cA^{2, \cl}(X/B, n)\rightarrow \cA^2(X/B, n)$ to the space of two-forms of degree $n$.

When $X \rightarrow B$ is a morphism of derived stacks with a perfect cotangent complex (e.g. an lfp geometric morphism), an (exact) $n$-shifted symplectic structure is an element $\omega\in \cA^{2, \cl}(X/B, n)$ ($\omega\in \cA^{2, \ex}(X/B, n)$) such that the underlying $2$-form induces an isomorphism $\cdot\omega\colon \bL_{X/B}\dual \cong \bL_{X/B}[n]$. We have the following natural source of exact $n$-shifted symplectic structures.

\begin{proposition}\label{prop:Gmcanonicallyexact}
Let $\pi\colon X\rightarrow B$ be a morphism of derived stacks together with a $\Gm$-action on $X$ such that $\pi$ is invariant. Let $w\in\Z$ be an integer invertible in $k$. Let $\cA^{2, \ex}(X/B, n)(w)$ and $\cA^{2, \cl}(X/B, n)(w)$ be the spaces of exact and closed relative $n$-shifted two-forms on $X\rightarrow B$ of weight $w$ with respect to the $\Gm$-action. Then the natural morphism
\[\cA^{2, \ex}(X/B, n)(w)\longrightarrow \cA^{2, \cl}(X/B, n)(w)\]
is an isomorphism. In particular, every $n$-shifted symplectic structure of weight $w$ is canonically exact.
\end{proposition}
\begin{proof}
We freely use the notation from \cite{PTVV,CPTVV}. Recall that for a morphism of derived stacks $X\rightarrow B$ one has the relative de Rham complex
\[\DR_B(X)\in\CAlg(\QCoh(B)^{\gr, \epsilon}),\]
a graded mixed commutative algebra in $\QCoh(B)$ defined as in \cite[Section B.11]{CHS} and \cite{ParkSymplectic}. Let
\[\DR(X/B) = \Gamma(B, \DR_B(X))\in\CAlg^{\gr, \epsilon}_k\]
be the underlying graded mixed cdga over $k$.

If $X$ is equipped with a $\Gm$-action, let $\overline{X}=[X/\Gm]$ and denote by
\[\DR(X/B)(w) = \Gamma(B\times \B\Gm, \DR_{B\times\B\Gm}(\overline{X})\otimes \cO(-w))\]
the graded mixed complex of forms of weight $w$, where $\cO(-w)$ is the line bundle on $\B\Gm$ corresponding to the one-dimensional $\Gm$-representation of weight $-w$. We claim that under the assumptions the complex $|\DR(X/B)(w)|\in\Mod_k$ is zero.

By construction the functor $\DR_{B\times\B\Gm}(-)$ sends colimits of stacks over $B\times\B\Gm$ to limits. Thus, the claim is reduced to the case $X$ is affine. The action map $a\colon \Gm\times X\rightarrow X$ gives rise to the coaction map
\[a^\ast\colon \DR(X/B)\longrightarrow \DR(X\times \Gm/B)\cong \DR(X/B)\otimes \DR(\Gm),\]
where the last isomorphism is the K\"unneth isomorphism. The action map $a$ is $\Gm\times\Gm$-equivariant, where on the left we act by the respective copy of $\Gm$ on the corresponding factor and on the right the two copies of $\Gm$ act in the same way on $X$. Therefore, the coaction map restricts to a morphism
\[a^\ast\colon \DR(X/B)(w)\longrightarrow \DR(X/B)(w)\otimes \DR(\Gm)(w).\]
If we denote the standard coordinate on $\Gm$ by $z$, then $\DR(\Gm)(w)$ is spanned by $z^w$ and $z^{w-1} dz$. But then $|\DR(\Gm)(w)|=0$ since the de Rham differential is an isomorphism. In particular,
\[|\DR(X/B)(w)\otimes \DR(\Gm)(w)|\cong |\DR(X/B)(w)|\otimes |\DR(\Gm)(w)| = 0.\]
Let $\epsilon\colon \DR(\Gm)(w)\rightarrow k$ be the counit map given by $z^w\mapsto 1$ and $dz\mapsto 0$. The composite
\[\DR(X/B)(w)\xrightarrow{a^\ast} \DR(X/B)(w)\otimes \DR(\Gm)(w)\xrightarrow{\id\otimes \epsilon} \DR(X/B)(w)\]
is equivalent to the identity. Therefore, $|\DR(X/B)(w)|$ is a retract of the zero object and hence it is the zero object itself. But since $\cA^{2, \ex}(X/B, n)(w)\rightarrow \cA^{2, \cl}(X/B, n)(w)$ is the fiber of
\[\cA^{2, \cl}(X/B, n)(w)\longrightarrow \cA^{0, \cl}(X/B, n+2)(w)=\Map_{\Mod_k}(k[n+2], |\DR(X/B)(w)|)\]
at the zero form, the claim follows.
\end{proof}

Given a derived stack $B$, consider the $\infty$-category $\mathrm{Symp}^{\mathrm{ex}}_{B,-1}$ whose objects are lfp geometric morphisms $\pi \colon X \to B$ equipped with a relative exact $(-1)$-shifted symplectic structure $\omega$, and whose morphisms $(X_1, \omega_1) \dashrightarrow (X_2, \omega_2)$ are exact Lagrangian correspondences $X_2 \gets L \to X_1$.
Given an lfp geometric morphism $p \colon B_1 \to B_2$, it is shown in \cite[Theorem A]{ParkSymplectic} that the base change functor $p^* \colon \mathrm{Symp}^{\mathrm{ex}}_{B_2,-1} \to \mathrm{Symp}^{\mathrm{ex}}_{B_1,-1}$ admits a right adjoint $p_*$.
Given $(X, \omega) \in \mathrm{Symp}^{\mathrm{ex}}_{B_1,-1}$, we call
\[p_*(X,\omega) \in \mathrm{Symp}^{\mathrm{ex}}_{B_2,-1}\]
the \defterm{symplectic pushforward} of $(X,\omega)$.
It can be described explicitly as the zero locus of the \emph{moment map}, a canonical Lagrangian $\mu_X \colon X \to \T^*(B_1/B_2)$ (see \cite[Proposition 2.3.1]{ParkSymplectic}).

\begin{example}\label{exam:dCrit}
Let $g \colon Y\rightarrow B$ be an lfp geometric morphism between derived stacks and $f$ a function on $Y$.
As explained in \cite[Eq.~(13)]{ParkSymplectic}, the function $f$ determines a relative exact $(-1)$-shifted symplectic structure on the identity map $Y\rightarrow Y$, which we also denote by $f$.
The \defterm{derived relative critical locus} of $f$ is its symplectic pushforward
\[
\left( \R\Crit_{Y / B}(f) \to B \right) \coloneqq g_* (Y, f).
\]
When $g$ is a smooth geometric morphism between classical stacks,
$\R\Crit_{Y / B}(f)$ is a derived enhancement of the relative critical locus defined in \cref{def-relative-crit}.
When $f$ is zero, the derived critical locus recovers the $(-1)$-shifted cotangent stack $\T^*[-1](Y / B)$.
\end{example}

Any morphism $X\rightarrow B$ of derived stacks induces on classical truncations a morphism $X^{\cl}\rightarrow B^{\cl}$ fitting in a commutative square
\[
\xymatrix{
X^{\cl} \ar[r] \ar[d] & X \ar[d] \\
B^{\cl} \ar[r] & B.
}
\]
In particular, we obtain a pullback morphism $\cA^{2, \ex}(X/B, n)\rightarrow \cA^{2, \ex}(X^{\cl}/B^{\cl}, n)$.

\begin{theorem}\label{thm:shiftedsymplecticdcritical}
Let $X\rightarrow B$ be an lfp geometric morphism of derived stacks equipped with a relative exact $(-1)$-shifted symplectic structure $\omega\in\cA^{2, \ex}(X/B, -1)$. Let $i\colon X^{\cl}\rightarrow X$ be the inclusion of the classical truncation. Then $i^\ast \omega$ is a relative d-critical structure on $X^{\cl}\rightarrow B^{\cl}$.
\end{theorem}
\begin{proof}
Let us first prove the claim when the base $B$ is a classical affine scheme. By \cite[Corollary 4.2.2]{ParkSymplectic}, there exists an LG pair $(U, f)$ over $B$ and a $(-1)$-shifted exact Lagrangian correspondence
\[\xymatrix{ & L\ar[ld]_a \ar[rd]^b & \\ \R\Crit_{U/B}(f) && X,}\]
where $\R\Crit_{U/B}(f)$ is the derived critical locus, $a^\cl$ is an isomorphism and $b$ is smooth surjective.
In particular, we have a smooth surjective morphism
$b^\cl\circ (a^{\cl})^{-1}\colon \Crit_{U/B}(f) \to X^\cl$
such that 
\[(b^\cl\circ (a^{\cl})^{-1})^\ast (i^\ast(\omega)) \cong s_f \in \cA^{2, \ex}(\Crit_{U/B}(f)/B, -1).\]
By \cref{prop:dcritbasechangestacks}(2), $i^\ast\omega$ is a d-critical structure.

For a general derived stack $B$ we consider the commutative diagram
\[
\xymatrix{
X' \ar[d] \ar[r] & X^{\cl} \ar[r] \ar[d] & X \ar[d] \\
B' \ar[r] & B^{\cl} \ar[r] & B,
}
\]
where $X'\rightarrow B'$ is a morphism of classical schemes and $X'\rightarrow X^{\cl}\times_{B^{\cl}} B'$ is smooth, where we consider the fiber product in the $\infty$-category $\Stk$. We have to prove that the pullback of $\omega$ to $X'\rightarrow B'$ defines a relative d-critical structure. By \cref{prop:dcritbasechangestacks}(2) it is enough to prove the claim when $B'$ is a classical affine scheme. By the previous argument we know that the pullback of $\omega$ to $(X\times^\R_B B')^{\cl}\rightarrow B'$ defines a relative d-critical structure, where $X\times^\R_B B'$ is the fiber product in the $\infty$-category $\dStk$. But $X'\rightarrow (X\times^\R_B B')^{\cl}\cong X^{\cl}\times_{B^{\cl}} B'$ is smooth by assumption, so by \cref{cor:dcritbasechange}(1) the pullback of $\omega$ to $X'\rightarrow B'$ defines a relative d-critical structure.
\end{proof}

We will next discuss orientations of such relative d-critical structures by describing the relative canonical bundle $K_{X/B}=\det(\bL_{X/B})$. For a morphism of derived stacks $X\rightarrow B$ equipped with a relative $(-1)$-shifted symplectic structure and a point $x\in X$ we construct an isomorphism
\[\kappa^{\der}_x\colon K_{X/B, x}\cong \det(\tau^{\geq 0}\bL_{X/B, x})^{\otimes 2}\]
by the composition
\begin{align*}
K_{X/B, x}\cong \det(\tau^{\leq -1}\bL_{X/B, x})\otimes \det(\tau^{\geq 0}\bL_{X/B, x})&\cong \det((\tau^{\geq 0}\bL_{X/B, x})^\vee[1])\otimes \det(\tau^{\geq 0}\bL_{X/B, x}) \\
&\cong \det(\tau^{\geq 0}\bL_{X/B, x})^{\otimes 2},
\end{align*}
where the first isomorphism is induced by the fiber sequence
\[\tau^{\leq -1}\bL_{X/B, x}\longrightarrow \bL_{X/B, x}\longrightarrow \tau^{\geq 0}\bL_{X/B, x}\]
and the second isomorphism is induced by $(-\cdot\omega)^{-1}|_x\colon \tau^{\leq -1}\bL_{X/B, x}\cong (\tau^{\geq 0}\bL_{X/B, x})^{\vee}[1]$.

\begin{example}\label{ex:canonicalcriticallocus}
Let $B$ be a scheme and $(U, f)$ an LG pair over $B$. Let $X=\R\Crit_{U/B}(f)$ be the derived critical locus which carries a relative $(-1)$-shifted symplectic structure. 
We have a fiber sequence
\[
\mathbb{L}_{U/B} |_X \to \mathbb{L}_{X / B} \to \mathbb{L}_{X / U}
\]
and an isomorphism $\mathbb{L}_{X / U} \cong \mathbb{L}_{U / \T^*(U /B)}|_{X} \cong  \mathbb{L}_{U / B}^{\vee}[1] |_{X}$, hence we obtain an isomorphism
\[\Lambda_{(U, f)}\colon K_{X/B}|_{X^{\red}}\cong K_{U/B}^{\otimes 2}|_{X^{\red}}.\]
For $x\in X$ the diagram
\begin{equation}\label{eq:Kcritkappa}
\xymatrix{
K_{X/B, x} \ar^-{\Lambda_{(U, f)}|_x}[r] \ar^{\kappa^{\der}_x}[d] & K_{U/B}^{\otimes 2} \ar^{\kappa_x}[d] \\
\det(\tau^{\geq 0}\bL_{X/B, x})^{\otimes 2} \ar^{\sim}[r] & \det(\Omega^1_{X^{\cl}/B}, x)^{\otimes 2}
}
\end{equation}
commutes up to the sign $(-1)^{\dim \Omega^{1}_{X / B, x}}$.
This follows from the commutativity of the diagram \cite[(3.25)]{KPS} together with the discrepancy of the sign convention for $\kappa_x$ explained in \cref{rem:kappa-discrepancy-KPS}.
\end{example}

Next we describe the behavior of the canonical bundle with respect to Lagrangian correspondences. Consider a Lagrangian correspondence $Y\xleftarrow{q_Y} L\xrightarrow{q_X} X$ of relative exact $(-1)$-shifted symplectic stacks over $B$. As in \cite[(3.5)]{KPS} we have a natural isomorphism
\begin{equation}\label{eq-Upsilon-derived}
\Upsilon^{\der}_{(q_X, q_Y)}\colon K_{X/B}|_L\otimes K_{L/X}^{\otimes 2}\xrightarrow{\sim} K_{Y/B}|_L.
\end{equation}
Explicitly, this isomorphism is constructed as follows. First, there is an obvious isomorphism 
\[
K_{X/B}|_L\otimes K_{L/X} \otimes K_{L/Y}^{\vee} \cong K_{Y / B} |_{L}.
\]
On the other hand, the Lagrangian structure induces an isomorphism $\mathbb{L}_{L / X}^{\vee}[1] \cong \mathbb{L}_{L / Y}[-1]$, which in turn induces $K_{L/X} \cong K_{L / Y}^{\vee}$, and hence we obtain the desired isomorphism.
Note that if $q_X\colon L\rightarrow X$ is smooth, using the isomorphism $\bL_{L/X}^\vee\cong \bL_{L/Y}[-2]$, we get that $\bL_{L/Y}$ is 2-connective. The following statement for $Y$ and $L$ schemes and $q_X$ schematic in \cite[Theorem 3.18(b)]{BBBBJ} and \cite[Theorem 4.9]{KinjoDimred} (see also \cite[Proposition 6.9]{KPS} for a closely related statement). The proofs use the standard commutative diagrams involving the isomorphisms $i(\Delta)$ and work verbatim for higher Artin stacks.

\begin{lemma}\label{lm:kappaLagrangiancorrespondence}
Let $B$ be a scheme and $Y\xleftarrow{q_Y} L\xrightarrow{q_X} X$ be a Lagrangian correspondence of relative exact $(-1)$-shifted symplectic stacks over $B$. Assume that $q_X$ is smooth. For a point $l\in L$ the diagram
\[
\xymatrix{
K_{X/B, q_X(l)} \otimes K_{L/X, l}^{\otimes 2} \ar^-{\Upsilon^{\der}_{(q_X, q_Y)}|_l}[r] \ar^{\kappa^{\der}_{q_X(l)}\otimes \id}[d] & K_{Y/B, q_Y(l)} \ar^{\kappa^{\der}_{q_Y(l)}}[d] \\
\det(\tau^{\geq 0} \bL_{X/B, q_X(l)})^{\otimes 2} \otimes K_{L/X, l}^{\otimes 2} \ar^-{i(\Delta)^2}[r] & \det(\tau^{\geq 0} \bL_{Y/B, q_Y(l)})^{\otimes 2}
}
\]
commutes up to the sign $(-1)^{\dim (L / X)}$, where the bottom isomorphism is induced by the fiber sequence
\[\Delta\colon \tau^{\geq 0}\bL_{X/B, q_X(l)}\longrightarrow \tau^{\geq 0}\bL_{L/B, l}\longrightarrow \bL_{L/X, l}\]
as well as the isomorphism $q_Y^\ast\colon \tau^{\geq 0}\bL_{Y/B, q_Y(l)}\xrightarrow{\sim}\tau^{\geq 0}\bL_{L/B, l}$.
\end{lemma}

We are now ready to compare the virtual canonical bundles of a relative d-critical locus and its derived enhancement.

\begin{proposition}\label{prop:symplecticorientations}
Let $X\rightarrow B$ be an lfp geometric morphism of derived stacks equipped with a relative exact $(-1)$-shifted symplectic structure $\omega\in\cA^{2, \ex}(X/B, -1)$. Let $i\colon X^{\cl}\rightarrow X$ be the embedding of the classical truncation. Then there is an isomorphism
\[\Lambda_X\colon K_{X/B}|_{X^{\red}}\xrightarrow{\sim} K^{\vir}_{X^{\cl}/B^{\cl}}\]
such that for every $x\in X$ the diagram
\begin{equation}\label{eq:Lambdakappa}
\xymatrix{
K_{X/B, x}\ar^{\Lambda_X|_x}[r] \ar^{\kappa^{\der}_x}[d] & K^{\vir}_{X^{\cl}/B^{\cl}, x} \ar^{\kappa_x}[d] \\
\det(\tau^{\geq 0}\bL_{X/B, x})^{\otimes 2}\ar^{i^\ast}[r] & \det(\tau^{\geq 0}\bL_{X^{\cl}/B^{\cl}, x})
}
\end{equation}
commutes up to the sign $(-1)^{\rank (\Omega^1_{X /B, x})}$.
\end{proposition}
\begin{proof}
\eqref{eq:Lambdakappa} uniquely determines the isomorphism $\Lambda_X$ if it exists. The isomorphism $\kappa^{\der}_x$ is natural with respect to base change since it involves the isomorphisms $i(\Delta)$ which are natural for isomorphisms.  The isomorphism $\kappa_x$ is natural with respect to base change by \cref{prop:virtualcanonicalstacks}(4). Therefore, it is enough to establish the existence of the isomorphism $\Lambda_X$ fitting into a commutative diagram \eqref{eq:Lambdakappa} for $B$ a classical affine scheme.

As in \cref{thm:shiftedsymplecticdcritical} we may find an LG pair $(U, f)$ over $B$ and a $(-1)$-shifted exact Lagrangian correspondence
\[\xymatrix{ & L\ar[ld]_a \ar[rd]^b & \\ \R\Crit_{U/B}(f) && X,}\]
$a^\cl$ is an isomorphism and $b$ is smooth surjective. In particular, we get a smooth surjective morphism $c\colon R=\Crit_{U/B}(f)\rightarrow X^{\cl}$ under which $c^\ast i^\ast \omega = s_f$ and such that the restriction $i_L^\ast\colon \bL_{L/X}\rightarrow \bL_{R/X^{\cl}}$ is an isomorphism. Consider the unique isomorphism
\[\Lambda_{X, U}\colon K_{X/B}|_{R^{\red}}\xrightarrow{\sim} K^{\vir}_{X^{\cl}/B}|_{R^{\red}}\]
which fits into the diagram
\[
\xymatrix{
K_{X/B}|_{R^{\red}}\otimes K_{L/X}^{\otimes 2}|_{R^{\red}} \ar^-{\Upsilon^{\der}_{(b, a)}}[r] \ar^{\Lambda_{X, U}\otimes i_L^\ast}[d] & K_{\R\Crit_{U/B}(f)/B}|_{R^{\red}} \ar^{\Lambda_{(U, f)}}[d] \\
K^{\vir}_{X^{\cl}/B}|_{R^{\red}}\otimes K_{R/X^{\cl}}^{\otimes 2}|_{R^{\red}} \ar^-{\Upsilon_{R\rightarrow X^{\cl}}}[r] & K^{\vir}_{R/B}.
}
\]

Using the commutative diagram (up to sign) \eqref{eq:Kcritkappa}, the compatibility of $\kappa^{\der}$ with $\Upsilon^{\der}$ given by \cref{lm:kappaLagrangiancorrespondence} and the compatibility of $\kappa$ with $\Upsilon$ given by \cref{prop:virtualcanonicalstacks}(4), for every $r\in R$ the diagram
\[
\xymatrix{
K_{X/B, c(r)}\ar^{\Lambda_{X, U}|_r}[r] \ar^{\kappa^{\der}_{c(r)}}[d] & K^{\vir}_{X^{\cl}/B^{\cl}, c(r)} \ar^{\kappa_{c(r)}}[d] \\
\det(\tau^{\geq 0}\bL_{X/B, c(r)})^{\otimes 2}\ar^{i^\ast}[r] & \det(\tau^{\geq 0}\bL_{X^{\cl}/B^{\cl}, c(r)})
}
\]
commutes up to the sign $(-1)^{\rank (\Omega_{X/B, x})}$. Therefore, $\Lambda_{X, U}$ descends along $c$ to an isomorphism $\Lambda_X$ independent of choices.
\end{proof}

We also note that the symplectic pushforward in \cite{ParkSymplectic} is compatible with the d-critical pushforward along \emph{smooth} morphisms.

\begin{proposition}\label{prop:symp push vs dcrit push along smooth}
Consider lfp geometric morphisms of derived stacks $X\rightarrow B_1\xrightarrow{p} B_2$, where $p$ is smooth, equipped with a relative exact $(-1)$-shifted symplectic structure $\omega\in\cA^{2, \ex}(X/B_1, -1)$. Then we have a canonical isomorphism $p_*(X,\omega)^\cl \cong p^{\cl}_*(X^\cl, i^\ast \omega)$ compatibly with the relative d-critical structures over $B^{\cl}_2$.
\end{proposition}

We now describe the behavior of the canonical bundle under the symplectic pushforward.

\begin{proposition}\label{prop-symp-push-ori}
    Consider lfp geometric morphisms of derived stacks  $X\xrightarrow{\pi}B\xrightarrow{p} S$, where $\pi$ is equipped with a relative exact $(-1)$-shifted symplectic structure $\omega\in\cA^{2, \ex}(X/B, -1)$.
    Let $\overline{\pi} \colon  p_*(X, \omega) \to S$ be the symplectic pushforward, which fits into a commutative diagram
    \[
    \xymatrix{
    p_*(X, \omega) \ar^-{r}[r] \ar^{\overline{\pi}}[d] & X \ar^{\pi}[d] \\
    S & B. \ar_{p}[l]
    }
    \]
    Then there is a canonical isomorphism
        \begin{equation}\label{eq-symp-push-can}
        \Sigma^{\mathrm{der}}_p \colon
           K_{X / B} |_{p_*(X, \omega)} \otimes K_{B/ S}^{\otimes 2} |_{p_*(X, \omega)} \cong K_{p_*(X, \omega) / S},
        \end{equation}
    which satisfies the following:
    \begin{enumerate}
          \item \label{item-push-canonical-functorial}
          It is functorial for compositions: $\Sigma^{\mathrm{der}}_{\id} = \id$ and given another  morphism $q\colon S \rightarrow T$ with $R =(q\circ p)_*(X, \omega)$ the diagram
    \begin{equation}\label{eq-symp-push-canonical-composition}
    \xymatrix@C=1.5cm{
    K_{X/B}|_R\otimes K^{\otimes 2}_{B/S}|_R\otimes K^{\otimes 2}_{S/T}|_R \ar^-{\id\otimes i(\Delta)^2}[r] \ar^{\Sigma^{\mathrm{der}}_p\otimes \id}[d] 
    & K_{X/B}|_R\otimes K^{\otimes 2}_{B/T}|_R \ar^{\Sigma^{\mathrm{der}}_{q\circ p}}[d] \\
    K_{p_*(X, \omega)/S}|_R\otimes K^{\otimes 2}_{S/T}|_R \ar^-{\Sigma^{\mathrm{der}}_q}[r] & K_{(q \circ p)_*(X, \omega)/R}
    }
    \end{equation}
    commutes, where the top horizontal morphism is induced by the fiber sequence
    \[
    \Delta\colon p^*\bL_{S/T}\longrightarrow \bL_{B/T}\longrightarrow \bL_{B/S}.
    \]
    \item \label{item-push-canonical-base-change}
    Consider a commutative diagram of derived stacks
    \begin{equation}\label{eq-symp-push-canonical-base-change-diagram}
    \xymatrix{
    X' \ar[r] \ar[d] & B' \ar^{p'}[r] \ar[d]^{a_1} & S' \ar[d]^{a_2} \\
    X \ar[r] & B \ar^{p}[r] & S
    }
    \end{equation}
    with all morphisms lfp, 
    the left square being Cartesian,
    equipped with a relative exact $(-1)$-shifted symplectic structure $\omega\in\cA^{2, \ex}(X/B, -1)$ and let $\omega' \in\cA^{2, \ex}(X'/B', -1)$ be the base change. 
    Assume that the map $B' \to B \times_{S} S'$ is smooth.
    We set $R' \coloneqq p'_*(X', \omega')^{\cl}$. 
    Then the diagram
    \begin{equation}\label{eq-item-push-canonical-base-change-diagram}
    \begin{tikzcd}
    	{K_{X / B} |_{R'} \otimes K_{B / S}^{\otimes 2} |_{R'}} \otimes K_{B' / B \times_{S} S'}^{\otimes 2}|_{R'} & {K_{p_*(X, \omega) / S} |_{R'}  \otimes K_{B' / B \times_{S} S'}^{\otimes 2}|_{R'}} \\
    	{K_{X' / B'} |_{R'} \otimes K_{B / S}^{\otimes 2} |_{R'} \otimes K_{B' / B \times_{S} S'}^{\otimes 2} |_{R'}} & {K_{p_*(X, \omega) \times_{S} S' / S'} |_{R'}  \otimes K_{B' / B \times_{S} S'}^{\otimes 2}|_{R'}} \\
    	{K_{X' / B'} |_{R'} \otimes K_{B' / S'}^{\otimes 2}|_{R'}} & {K_{p'_*(X', \omega') / S'}|_{R'}}
    	\arrow["{\Sigma^{\mathrm{der}}_p \otimes \id}", from=1-1, to=1-2]
    	\arrow["\sim", from=1-1, to=2-1]
    	\arrow["\sim", from=1-2, to=2-2]
    	\arrow["{\id \otimes i(\Delta)}", from=2-1, to=3-1]
    	\arrow["{\Upsilon^{\mathrm{der}}_{(t, s)}}", from=2-2, to=3-2]
    	\arrow["{\Sigma^{\mathrm{der}}_{p'}}"', from=3-1, to=3-2]
    \end{tikzcd}\end{equation}
    commutes, where
    \[\Delta \colon \bL_{B/S}|_{B'}\longrightarrow \bL_{B'/S'}\longrightarrow \bL_{B'/B\times_{S} S'}\]
    and
    \begin{equation}\label{eq-Lag-corresp-in-push-canonical}
    \begin{tikzcd}
    	& {X' \times_{\T^*(B / S) }B} \\
    	  {p'_*(X', \omega')} &&  {p_*(X, \omega) \times_{S} S'}
    	\arrow["s"', from=1-2, to=2-1]
    	\arrow["t", from=1-2, to=2-3]
    \end{tikzcd}
    \end{equation}
    is the Lagrangian correspondence given by the Beck--Chevalley map.

    \item \label{item-push-canonical-smooth-pullback}
    Assume that we are given an exact Lagrangian correspondence
        $(Y, \omega') \xleftarrow[]{q_{Y}} L \xrightarrow[]{q_{X}} (X, \omega) $ over $B$,
        where $q_X$ is a smooth morphism.
        We let
        $p_*(Y, \omega') \xleftarrow{q_{p_*(Y, \omega')}} p_*L \xrightarrow[]{q_{p_*(X, \omega)}} (X, \omega)$ denote the induced Lagrangian correspondence.
        Then the following diagram commutes:
        \[\begin{tikzcd}
        	{K_{X / B} |_ {p_*L} \otimes K_{B / S} ^{\otimes 2}|_ {p_*L}  \otimes  K_{{L} /X}^{\otimes 2}|_ {p_*L} } & {K_{p_*(X, \omega) / S} |_ {p_*L} \otimes  K_{{p_*L} /p_*(X, \omega)}^{\otimes 2} } \\
        	{K_{Y / B} |_ {p_*L}\otimes  K_{B / S }^{\otimes 2} |_{p_*L}} & {K_{p_*(Y, \omega') / S} |_ {p_*L}.}
        	\arrow["{{{\Sigma^{\mathrm{der}}_p}}}", from=1-1, to=1-2]
        	\arrow["{{{\Upsilon^{\mathrm{der}}_{(q_{X}, q_{Y})}}}}", from=1-1, to=2-1]
        	\arrow["{{{\Upsilon^{\mathrm{der}}_{(q_{p_*(X, \omega)}, q_{p_*(Y, \omega')})}}}}", from=1-2, to=2-2]
        	\arrow["{{{\Sigma^{\mathrm{der}}_p}}}", from=2-1, to=2-2]
        \end{tikzcd}\]
        
    \item \label{item-push-canonical-d-critical-comparison}
        Assume that $p$ is smooth and let $s$ be the underlying d-critical structure on $p^{\cl} \colon X^{\cl} \to B^{\cl}$.
        Then the following diagram commutes:
        \[\begin{tikzcd}
    	{K_{X / B} |_{p_*(X, \omega)^{\mathrm{red}}} \otimes K_{B / S}^{\otimes 2} |_{p_*(X, \omega)^{\mathrm{red}}}} & {K_{p_*(X, \omega) / S} |_{p_*(X, \omega)^{\mathrm{red}}}} \\
    	{K^{\vir}_{X^{\cl} / B^{\cl}} |_{p^{\cl}_*(X^{\cl}, s)^{\mathrm{red}}} \otimes K_{B^{\cl} / S^{\cl}}^{\otimes 2} |_{p^{\cl}_*(X^{\cl}, s)^{\mathrm{red}}}} & {K^{\vir}_{p_*(X^{\cl}, s) / S^{\cl}}  |_{p^{\cl}_*(X^{\cl}, s)^{\mathrm{red}}}.}
    	\arrow[from=1-1, to=1-2]
    	\arrow["{{{\Sigma^{\mathrm{der}}_p}}}"{description}, shift left=3, draw=none, from=1-1, to=1-2]
    	\arrow["{{{\Lambda_{X}}}}", from=1-1, to=2-1]
    	\arrow["{{{\Lambda_{p_*(X, \omega)}}}}", from=1-2, to=2-2]
    	\arrow[from=2-1, to=2-2]
    	\arrow["{\Sigma_{p^{\cl}}}"{description}, shift left=3, draw=none, from=2-1, to=2-2]
        \end{tikzcd}\]
    \end{enumerate}
    
\end{proposition}

\begin{proof}
    We first construct the isomorphism \eqref{eq-symp-push-can}.
    Recall from \cite[Proposition 2.3.1]{ParkSymplectic} that the stack $p_* (X, \omega)$ fits into the Cartesian diagram
\[\begin{tikzcd}
	{p_* (X, \omega)} & X \\
	{B} & {\T^*(B / S).}
	\arrow[from=1-1, to=1-2]
	\arrow[from=1-1, to=2-1]
	\arrow[from=1-2, to=2-2]
	\arrow["0", from=2-1, to=2-2]
\end{tikzcd}\]
    In particular, there is a canonical isomorphism
    \[
    \mathbb{L}_{p_* (X, \omega) / X} \cong \mathbb{L}_{B / \T^*{(B / S)}} |_{p_* (X, \omega)} \cong \mathbb{L}_{B / S}^{\vee}[1] |_{p_* (X, \omega)}.
    \]
    Therefore we obtain a fiber sequence
    \[
    \Delta \colon  \mathbb{L}_{X / S} |_{p_* (X, \omega)} \to 
    \mathbb{L}_{p_* (X, \omega) / S} \to 
    \mathbb{L}_{B / S}^{\vee}[1] |_{p_* (X, \omega)}
    \]
    which induces an isomorphism
    \begin{align}\label{eq-det-pushforward-1}
       \begin{aligned}
    K_{p_* (X, \omega) / S}
    &\xrightarrow[\cong]{i(\Delta)}
    K_{X / S} |_{p_* (X, \omega)} \otimes
    \det(\mathbb{L}_{B / S}^{\vee}[1]) |_{p_* (X, \omega)} \\
    &\xrightarrow[\cong]{\id \otimes \theta_{\mathbb{L}_{B / S}^{\vee}}}
    K_{X / S} |_{p_* (X, \omega)} \otimes \det(\mathbb{L}_{B / S}^{\vee})^{\vee} |_{p_* (X, \omega)} \\
    &\xrightarrow[\cong]{\id \otimes \iota_{\mathbb{L}_{B / S}}}
    K_{X / S} |_{p_* (X, \omega)} \otimes K_{B / S} |_{p_* (X, \omega)}.
    \end{aligned}
        \end{align}
    Now consider the fiber sequence
    \[
    \Delta' \colon \mathbb{L}_{B / S} |_{X} \to \mathbb{L}_{X / S} \to \mathbb{L}_{X / B}.
    \]
    It induces an isomorphism
    \begin{equation}\label{eq-det-pushforward-2}
        i(\Delta) \colon K_{X / S} \cong K_{B / S} |_X \otimes K_{X / B}. 
    \end{equation}
    Combining \eqref{eq-det-pushforward-1} and \eqref{eq-det-pushforward-2},
    we obtain the desired isomorphism.

    The property (\ref{item-push-canonical-functorial}) is obvious from the construction.
    The property (\ref{item-push-canonical-base-change}) is obvious when the right square in the diagram \eqref{eq-symp-push-canonical-base-change-diagram} is Cartesian.
    Therefore it is enough to prove the case when $p$ and $a_2$ are identity maps.
    Since we have $p'_*(X \times_{B} B', \omega |_{X \times_{B} B'}) = a_{1, *} (X \times_{B} B', \omega |_{X \times_{B} B'}) = (X, \omega) \times \T^*[-1](B / S)$, we may further assume that $X = B$.
    In this case, the correspondence \eqref{eq-Lag-corresp-in-push-canonical} is identified with the following Lagrangian correspondence
    \[
    \begin{tikzcd}
        	& {B'} \\
    	  {\T^*[-1](B' / B)} &&  {B}
    	\arrow["0"', from=1-2, to=2-1]
    	\arrow["a_1", from=1-2, to=2-3]
    \end{tikzcd}
    \]
    given by the zero section.
    The map $\mathbb{L}^{\vee}_{B' / B}[1] \cong \mathbb{L}_{B' / \T^*[-1](B' / B)}[-1] $ induced from the Lagrangian correspondence is equivalent to the natural map.
    In particular, the isomorphism $\Upsilon^{\mathrm{der}}_{(a_1, 0)}$ is identified with the following isomorphism
    \begin{align*}
    K_{\T^*[-1](B' / B) /B} |_{B'}  
    \cong K_{B' / B} \otimes K_{B' / \T^*[-1](B' / B) }^{\vee} 
    &\cong K_{B' / B} \otimes \det(\mathbb{L}_{B' / \T^*[-1](B' / B)}[-1]) \\
    \cong &K_{B' / B} \otimes \det(\mathbb{L}_{B' / B}^{\vee}[1]) \cong K_{B' / B}^{\otimes 2}. 
    \end{align*}
    On the other hand, $\Sigma^{\mathrm{der}}_{p'}$ is identified with the isomorphism 
    \[
    K_{B' / B}^{\otimes 2} 
    \cong K_{B' / B} \otimes K_{\T^*[-1](B' / B) / B'} |_{B'}  \cong K_{\T^*[-1](B' / B) /B'} |_{B'},
    \]
    where the first isomorphism is induced by the isomorphism 
    $\mathbb{L}_{\T^*[-1](B' / B) / B'} |_{B'} \cong \mathbb{L}_{B' / \T^*(B' / B)} \cong \mathbb{L}_{B' / B}^{\vee}[1]$.
    Therefore we obtain the commutativity of the diagram \eqref{eq-symp-push-canonical-base-change-diagram}.
    
    The property (\ref{item-push-canonical-smooth-pullback}) is obvious from the construction.
    To prove the property (\ref{item-push-canonical-d-critical-comparison}), using (\ref{item-push-canonical-base-change}) and \cref{prop:canonicalpushforward}(2), we may assume that $B$ and $S$ are classical affine schemes. Further, using  (\ref{item-push-canonical-smooth-pullback}), \cref{prop:canonicalpushforward}(3) and \cite[Corollary 4.2.2]{ParkSymplectic}, 
    we may assume that there exists a LG pair $(U, f)$ over $B$ and $X = \R \Crit_{U / B}(f)$. Using the functoriality of the isomorphism $\Sigma^{\mathrm{der}}_{p}$, we may assume $X = B = U$ and the relative exact $(-1)$-shifted symplectic structure is induced from $f$.
    In this case, we have $p_* (X, \omega) = \R \Crit_{U / S}(f)$ and the isomorphism $\Sigma^{\mathrm{der}}_p$ is identified with the isomorphism
    \[
    K_{\R \Crit_{U / S}(f) / S} \cong K_{U / S} |_{\R \Crit_{U / S}(f)} \otimes K_{\R \Crit_{U / S}(f) / U}  \cong K_{U / S}^{\otimes 2}|_{\R \Crit_{U / S}(f)}
    \]
    where the latter isomorphism is induced from the isomorphism $\mathbb{L}_{\R \Crit_{U / S}(f) / U} \cong \mathbb{L}_{U / \T^*{(U / S)}} |_{\R \Crit_{U / S}(f)} \cong \mathbb{L}^{\vee}_{U / S}[-1]$.
    On the other hand, $\Sigma_{p^{\cl}}$ is given by the natural isomorphism $K_{\Crit_{U / S}(f) / S}^{\vir} \cong K_{U / S}^{\otimes 2} |_{\Crit_{U / S}(f)^{\red}}$.
    By the construction of the map $\Lambda_{\R \Crit_{U / S}(f)}$, we obtain the desired claim.  
\end{proof}

\begin{definition}
Let $\pi\colon X\rightarrow B$ be an lfp geometric morphism of derived stacks equipped with a relative exact $(-1)$-shifted symplectic structure $\omega$. An \defterm{orientation} of $(\pi\colon X\rightarrow B, \omega)$ is a pair $(\cL, o)$ consisting of a graded line bundle $\cL$ on $X$ together with an isomorphism $o\colon \cL^{\otimes 2}\xrightarrow{\sim} K_{X/B} = \det(\bL_{X/B})$.
\end{definition}

By \cref{thm:shiftedsymplecticdcritical} the induced morphism on classical truncations $\pi^\cl \colon X^\cl \to B^\cl$ inherits a canonical relative d-critical structure,
and \cref{prop:symplecticorientations} implies that an orientation $o$ of the relative exact $(-1)$-shifted symplectic structure on $X\rightarrow B$ naturally induces an orientation $o^{\cl}$ of the relative d-critical structure on $X^{\cl} \rightarrow B^{\cl}$.

In the setting of \cref{prop-symp-push-ori}, suppose that we are given an orientation $(\mathcal{L}, o)$ for $(X \to B, \omega)$. Then we define an orientation $p_*o$ on $p_*(X, \omega) \to S$ as the composite
\[
p_*o \colon (\mathcal{L}|_{p_*(X, \omega)} \otimes K_{B / S}|_{p_*(X, \omega)})^{\otimes 2}  \xrightarrow[\sim]{o \otimes \id} K_{X / B} |_{p_*(X, \omega)} \otimes K_{B /S}^{\otimes 2} |_{p_*(X, \omega)}\xrightarrow[\sim]{\Sigma_{p}} K_{p_*(X, \omega) / S}.
\]
As a special case, when $X = B$ and $\omega$ is induced from a function $f$ on $X$,
we equip $\pi \colon X \xrightarrow{=} X$ with the obvious orientation $o_{X /X}^{\mathrm{can}} \colon \mathcal{O}_X^{\otimes 2} \cong K_{X / X}$ and define the \defterm{canonical orientation} $o^{\can}_{\R \Crit_{X / B}(f) / B} \coloneqq p_* o_{X /X}^{\mathrm{can}}$ on $\R \Crit(f) = p_* (X, \omega) \to B$. As a special case when $f = 0$, we obtain a canonical orientation on $o^{\can}_{\T^*[-1](X / B) /B}$ on the $(-1)$-shifted cotangent stack $\T^*[-1](X / B) \to B$.
When $X$ is smooth over $B$, it is clear from the construction that the canonical orientation of the derived critical locus is compatible with that canonical orientation of the classical critical locus with the natural d-critical structure. Namely, there exists natural isomorphism
\[
(o^{\can}_{\R \Crit_{X / B}(f) / B})^{\cl} \cong o^{\can}_{\Crit_{X^{\cl} / B^{\cl}}(f) / B^{\cl}}.
\]

\subsection{Perverse pullbacks}

It is useful to express perverse pullbacks constructed in \cref{thm:micropullstk} in the language of shifted symplectic geometry. For this, let $R$ be a commutative ring and for a derived stack $X$ define $\Shv(X; R) = \Shv(X^{\cl}; R)$ and similarly for $\bD(-), \Perv(-)$.
Let $X\rightarrow B$ denote an lfp morphism between derived Artin stacks equipped with a relative exact $(-1)$-shifted symplectic structure $\omega \in \cA^{2, \ex}(X/B, -1)$ and an orientation $o$.
Then \cref{prop:symplecticorientations} implies that the morphism induced on the classical truncations $X^{\cl} \to B^{\cl}$ is naturally equipped with an relative exact d-critical structure, and inherits an orientation $o^{\cl}$.
Therefore, by \cref{thm:micropullstk} we obtain a perverse pullback functor
\[\pi^\varphi\colon \Perv(B) \to \Perv(X).\]
When $R$ is a field, it extends to a functor $\pi^\varphi \colon \bD(B) \to \bD(X)$.
The perverse pullback functor for morphisms with oriented relative exact $(-1)$-shifted symplectic structures satisfies the following properties, as a direct consequence of the corresponding properties in the d-critical setting (\cref{thm:micropullstk}):
\begin{enumerate}
    \item Assume that we are given a finite morphism $c \colon \tilde{B} \to B$ and form the Cartesian diagram 
    \[
    \begin{tikzcd}
        {\tilde{X}} & X \\
        {\tilde{B}} & {B.}
        \arrow["{\tilde{c}}", from=1-1, to=1-2]
        \arrow["{\tilde{\pi}}", from=1-1, to=2-1]
        \arrow["\pi", from=1-2, to=2-2]
        \arrow["c", from=2-1, to=2-2]
    \end{tikzcd}
    \]
    Equip $\tilde{\pi} \colon \tilde{X} \to \tilde{B}$ with the pullback relative exact $(-1)$-shifted symplectic structure and orientation.
    Then there exists a natural isomorphism  
    \begin{equation}\label{eq-perverse-pullback-symplectic-finite-base-change-compatibility}
        \beta_{c} \colon \pi^{\varphi} c_* \xrightarrow[]{\sim} \tilde{c}_* \tilde{\pi}^{\varphi}.
    \end{equation}
    
    \item Assume that we are given a smooth morphism $p \colon B \to S$. 
    Let $p_*(X, \omega)$ be the symplectic pushforward which fits into a commutative diagram
    \[\begin{tikzcd}
    	{p_*(X, \omega)} & X \\
    	S & {B.}
    	\arrow["i", from=1-1, to=1-2]
    	\arrow["{\bar{\pi}}", from=1-1, to=2-1]
    	\arrow["\pi", from=1-2, to=2-2]
    	\arrow["p", from=2-2, to=2-1]
    \end{tikzcd}\]
    Equip $p_*(X, \omega)$ with the pushforward orientation $p_* o$. Then there exists a natural isomorphism
    \begin{equation}\label{eq-symp-push-perverse-pullback-smooth}
    \gamma_{p} \colon \pi^{\varphi} p^{\dagger} \xrightarrow[]{\sim} i_* \bar{\pi}^{\varphi}.
    \end{equation}
    \item Let $\bar{o}$ be the orientation for $(X \to B, - \omega)$ defined in a similar manner as $\eqref{eq:orientationreversedcritical}$.
    We let $\pi^{\varphi, -}$ denote the perverse pullback for $(X \to B, - \omega)$ with respect to $\bar{o}$. Then there is a natural isomorphism
    \[
     \delta \colon \pi^{\varphi}\mathbb{D} \xrightarrow[]{\sim} \mathbb{D} \pi^{\varphi, -}.
    \]
    \item Assume that we are given lfp morphisms between derived Artin stacks $\pi_i \colon X_i \to B_i$ for $i = 1, 2$ equipped with relative exact $(-1)$-shifted symplectic structures $\omega_i \in \mathcal{A}^{2, \ex}(X_i/B_i, -1)$ and orientations $o_i$. Equip $\pi_1 \times \pi_2 \colon X_1 \times X_2 \to B_1 \times B_2$ with the relative exact $(-1)$-shifted symplectic structure $\omega_1 \boxplus \omega_2$ and orientation $o_1 \boxtimes o_2$.
    Then there exists a natural isomorphism
    \[
    \TS \colon \pi_1^{\varphi} (-) \boxtimes \pi_2^{\varphi} (-) \xrightarrow[]{\sim} (\pi_1 \times \pi_2)^{\varphi}(- \boxtimes -).
    \]
\end{enumerate}

The isomorphism $\alpha$ constructed for perverse pullbacks along morphisms equipped with a relative d-critical structure has the following meaning in terms of shifted symplectic structures.

\begin{proposition}\label{prop:Lagrangian functoriality of perverse pullback} 
    Let $\pi \colon X \to B$ be an lfp morphism of derived Artin stacks equipped with a relative exact $(-1)$-shifted symplectic structure and orientation.
    Let $p \colon B' \to B$ be a smooth geometric morphism, $\pi' \colon X' \to B'$ an lfp morphism of derived Artin stacks equipped with a $(-1)$-shifted symplectic structure,
    \[ X' \xleftarrow{q'} L \xrightarrow{\tilde{q}} X \times_B B' \]
    a Lagrangian correspondence over $B'$ with $\tilde{q}$ smooth, and $q \coloneqq \pr_1 \circ \tilde{q} \colon L \to X$ the composite.
    Regard $\pi'$ with the induced orientation \eqref{eq-Upsilon-derived}.
    Then there is a natural isomorphism
    \begin{equation}\label{eq-Lagrangian functoriality of perverse pullback}
    \alpha_{p, (q, q')} \colon (\pi')^\varphi p^{\dagger}
    \xrightarrow{\sim} q'_{*} q^\dag \pi^\varphi
    \end{equation}
    of functors $\Perv(B) \to \Perv(X')$.
\end{proposition}
\begin{proof}
    On classical truncations, $q$ induces a smooth morphism $f \colon (X')^{\cl} \cong L^\cl \to (X \times_B B')^{\cl} \to X^{\cl}$.
    Consider the relative d-critical structures $s$ and $s'$ on $\pi^{\cl} \colon X^{\cl} \to B^\cl$ and $\pi'^{\cl} \colon (X')^{\cl} \to (B')^\cl$ induced by the relative $(-1)$-shifted symplectic structures on $\pi$ and $\bar{\pi}$ (see \cref{thm:shiftedsymplecticdcritical}).
    Using the given Lagrangian correspondence, we obtain $f^\ast(s_2) = s_1$.
    Applying \cref{thm:micropullstk}(1) thus yields the isomorphism $(\pi'^{\cl})^\varphi p^{\dagger} \cong f^\dag (\pi^{\cl})^\varphi$. Translating from the language of relative d-critical structures to the language of relative exact $(-1)$-shifted symplectic structures, this becomes the isomorphism asserted.
\end{proof}

\begin{remark}
    When $p = \id$ (resp. $(q, q') = (\id, \id)$), we will write
    $\alpha_{(q, q')} \coloneqq \alpha_{p, (q, q')}$ (resp. $\alpha_{p} \coloneqq \alpha_{p, (q, q')}$).
\end{remark}


While we have defined d-critical pushforwards only along smooth morphisms, symplectic pushforwards are defined along arbitrary lfp geometric morphisms. We have the following additional compatibility between perverse pullbacks and symplectic pushforwards with respect to closed immersions.

\begin{proposition}\label{prop:symp push compatibility}
Suppose given lfp morphisms of derived Artin stacks  $X\xrightarrow{\pi}B\xrightarrow{p} S$, where $p$ is a closed immersion and $\pi$ is equipped with a relative exact $(-1)$-shifted symplectic structure $\omega\in\cA^{2, \ex}(X/B, -1)$.
Let $\overline{\pi} \colon  p_*(X, \omega) \to S$ be the symplectic pushforward, which fits into a commutative diagram
\[
\xymatrix{
p_*(X, \omega) \ar^-{r}[r] \ar^{\overline{\pi}}[d] & X \ar^{\pi}[d] \\
S & B. \ar_{p}[l]
}
\]
Assume further that $\pi \colon X \to B$ is equipped with an orientation $o$, and equip $\bar{\pi} \colon p_*(X, \omega) \to S$ with the orientation $p_* o$.
Then $r$ is smooth, and there is a canonical isomorphism
\begin{equation}\label{eq-symp-push-perverse-pullback}
\varepsilon_p\colon \overline{\pi}^\varphi p_* \xrightarrow{\sim} r^\dag \pi^\varphi,
\end{equation}
functorial for compositions in $p$.
Moreover, we have the following compatibilities between $\varepsilon_p$ and the isomorphisms $\alpha$, $\beta$, $\gamma$, $\delta$, and $\TS$ of \cref{thm:micropullstk}:
\begin{enumerate}
\item \label{item:symp-push-commutes-smooth-descent}
Let $h \colon S' \to S$ be a smooth morphism and set $B' \coloneqq B \times_{S} S'$ with projections $h' \colon B' \to B$ and $p' \colon B' \to S'$.
Let $\pi' \colon X' \to B'$ be an lfp geometric morphism equipped with a relative exact $(-1)$-shifted symplectic structure $\omega'$, and
\[ X' \xleftarrow{q'} L \xrightarrow{\tilde{q}} X \times_B B' \]
a Lagrangian correspondence over $B'$ with $\tilde{q}$ smooth.
Form the symplectic pushforward $S' \xleftarrow{\bar{\pi}'} p'_*(X',\omega') \xrightarrow{r'} X'$ and consider the induced Lagrangian correspondence
\[
  p'_*(X',\omega') \xleftarrow{\bar{q}'} \bar{L} \xrightarrow{\bar{\tilde{q}}} p'_*(X \times_B B', h^*\omega) \cong p_*(X,\omega) \times_S S'.
\]
Then $\bar{\tilde{q}}$ is smooth, and the following diagram commutes:
\[\begin{tikzcd}
    {(\bar{\pi}')^{\varphi} p'_* (h')^{\dagger}} & {(\bar{\pi}')^{\varphi}h^{\dagger}p_*} & {\bar{q}'_* \bar{q}^{\dagger} \bar{\pi}^{\varphi}p_*} \\
    {(r')^{\dagger}(\pi')^{\varphi}  (h')^{\dagger}} & {(r')^{\dagger} q'_* q^{\dagger} \pi^{\varphi} } & {\bar{q}'_* \bar{q}^{\dagger}r^{\dagger} \pi^{\varphi},}
    \arrow["{{\varepsilon_{p'}}}"', from=1-1, to=2-1]
    \arrow["{{\Ex^!_*}}"', "\sim", from=1-2, to=1-1]
    \arrow["{{\alpha_{h, (\bar{q}, \bar{q}')}}}", from=1-2, to=1-3]
    \arrow["{{\varepsilon_p}}", from=1-3, to=2-3]
    \arrow["{{\alpha_{h', (q, q')}}}"', from=2-1, to=2-2]
    \arrow["{{\Ex^!_*}}", "\sim"', from=2-3, to=2-2]
\end{tikzcd}\]
where $q \coloneqq \pr_1 \circ \tilde{q} \colon L \to X$, $\bar{q} \coloneqq \pr_1 \circ \tilde{\bar{q}} \colon \bar{L} \to p_*(X,\omega)$, and $\alpha$ is as in \cref{prop:Lagrangian functoriality of perverse pullback}.

\item \label{item-symp-finite-push-base-change}
Assume that we are given a finite morphism from a derived Artin stack $c_2 \colon S' \to S$, $c_1 \colon B' \to B$, $p' \colon S' \to B'$ and form the following commutative diagram:
\[\begin{tikzcd}
    {X' \times_{\T^*(B/S) } B} \\
    {p_*(X, \omega) \times_{S} S'} & {p'_*(X', \omega')} && {X'} \\
    {p_*(X, \omega)} && X \\
    & {S'} && {B'} \\
    S && B
    \arrow["i"', from=1-1, to=2-1]
    \arrow["h", from=1-1, to=2-2]
    \arrow["{c_2'}"', from=2-1, to=3-1]
    \arrow["{\bar{\pi} \times_{S} \id_{S'}}"{pos=0.3}, from=2-1, to=4-2]
    \arrow["{{r'}}"{pos=0.7}, from=2-2, to=2-4]
    \arrow["{{{\bar{\pi}'}}}"'{pos=0.2}, from=2-2, to=4-2]
    \arrow["q"', from=2-4, to=3-3]
    \arrow["{{c_2}}", from=4-2, to=5-1]
    \arrow["{{{p'}}}"'{pos=0.2}, from=4-4, to=4-2]
    \arrow["{{c_1}}", from=4-4, to=5-3]
    \arrow["p"'{pos=0.2}, from=5-3, to=5-1]
    \arrow["{{{\pi'}}}"'{pos=0.2}, from=2-4, to=4-4]
    \arrow["r"{pos=0.8}, from=3-1, to=3-3, crossing over]
    \arrow["{{{\bar{\pi}}}}"'{pos=0.2}, crossing over,  from=3-1, to=5-1]
    \arrow["\pi"'{pos=0.2}, crossing over, from=3-3, to=5-3]
\end{tikzcd}\]
Here, the right square is Cartesian, $\omega' \in \cA^{2, \ex}(X'/B', -1)$ is the restriction of $\omega$, the front and back squares are symplectic pushforward squares,  $c_2'$ is the base change of $c_2$, and $i$ and $h$ are the natural morphisms. 
Then $h$ is smooth, $i$ induces an isomorphism on the classical truncation and the following diagram commutes:
\[\begin{tikzcd}
    {\bar{\pi}^{\varphi}  p_* c_{1, *}} & {\bar{\pi}^{\varphi}  c_{2, *} p'_*} & {c_{2, *}' (\bar{\pi} \times_S \id_{S'})^{\varphi} p'_*} & {c_{2, *}' i_* h^{\dagger} (\bar{\pi}')^{\varphi} p'_*} \\
    {r^{\dagger} \pi^{\varphi} c_{1, *}} & {r^{\dagger}  q_* (\pi')^{\varphi} } && {c_{2, *}' i_* h^{\dagger} (r')^{\dagger} (\pi')^{\varphi}.}
    \arrow["\sim", from=1-1, to=1-2]
    \arrow["{{{{\varepsilon_p}}}}"', from=1-1, to=2-1]
    \arrow["{{{{\beta_{c_2}}}}}", from=1-2, to=1-3]
    \arrow["{\alpha_{(h, i)}}", from=1-3, to=1-4]
    \arrow["{{{{\varepsilon_{p'}}}}}"', from=1-4, to=2-4]
    \arrow["{{{{\beta_{c_1}}}}}"', from=2-1, to=2-2]
    \arrow["{{\Ex^!_*}}"', from=2-2, to=2-4]
\end{tikzcd}\]

\item \label{item-push-finite-smooth}
Assume that $p$ fits in the following Cartesian diagram:
\[\begin{tikzcd}
    R & {S'} \\
    S & {B}
    \arrow["{\bar{p}}", from=1-2, to=1-1]
    \arrow["{\bar{q}}"', from=2-1, to=1-1]
    \arrow["q"', from=2-2, to=1-2]
    \arrow["p"', from=2-2, to=2-1]
\end{tikzcd}\]
where $\bar{p}$ is a closed immersion and $\bar{q}$ is smooth. Form the following commutative diagram:
\[\begin{tikzcd}
    & {\bar{p}_*q_*(X, \omega)} && {q_*(X, \omega)} \\
    {p_*(X, \omega)} && X \\
    & R && {S'} \\
    S && B.
    \arrow["{\bar{r}}"{pos=0.7}, from=1-2, to=1-4]
    \arrow["{{{\bar{\pi}'}}}"'{pos=0.2}, from=1-2, to=3-2]
    \arrow["{{{\pi'}}}"'{pos=0.2}, from=1-4, to=3-4]
    \arrow["{\bar{s}}"', from=1-2, to=2-1]
    \arrow["{{{\bar{\pi}}}}"'{pos=0.2}, from=2-1, to=4-1]
    \arrow["s"', from=1-4, to=2-3]
    \arrow["{\bar{p}}"'{pos=0.2}, from=3-4, to=3-2]
    \arrow["{\bar{q}}"', from=4-1, to=3-2]
    \arrow["q"', from=4-3, to=3-4]
    \arrow["p"'{pos=0.2}, from=4-3, to=4-1]
    \arrow["{{{r}}}"{pos=0.8}, crossing over, from=2-1, to=2-3]
    \arrow["\pi"'{pos=0.2}, crossing over, from=2-3, to=4-3]
\end{tikzcd}\]
Here, the top and bottom squares are Cartesian and the other four squares are symplectic pushforward squares.
Then the following diagram commutes:
\[\begin{tikzcd}
    { \bar{\pi}^{\varphi} p_* q^{\dagger}} & { \bar{\pi}^{\varphi} \bar{q}^{\dagger} \bar{p}_*} & { \bar{s}_*  (\bar{\pi}')^{\varphi} \bar{p}_*} \\
    {r^{\dagger} \pi^{\varphi} q^{\dagger}} & {r^{\dagger} s_* (\pi')^{\varphi}} & { \bar{s}_* \bar{r}^{\dagger} (\pi')^{\varphi}.}
    \arrow["{{{\varepsilon_{p}}}}"', from=1-1, to=2-1]
    \arrow["{\Ex^!_*}"', from=1-2, to=1-1]
    \arrow["{{\gamma_{\bar{q}}}}", from=1-2, to=1-3]
    \arrow["{{{\varepsilon_{\bar{p}}}}}", from=1-3, to=2-3]
    \arrow["{{\gamma_q}}"', from=2-1, to=2-2]
    \arrow["{\Ex^!_*}"', from=2-2, to=2-3]
\end{tikzcd}\]     

\item \label{item-symp-push-perverse-pullback-dual}
We let $\pi^{\varphi, -}$ denote the perverse pullback for $(X\to B, - \omega)$ with respect to the orientation $\bar{o}$ and define $\bar{\pi}^{\varphi, -}$ in a similar manner. Then the following diagram commutes:
\[\begin{tikzcd}
    {r^{\dagger} \pi^{\varphi} \mathbb{D}} & {r^{\dagger}  \mathbb{D} \pi^{\varphi, -}} & {\mathbb{D} r^{\dagger}  \pi^{\varphi, -}} \\
    { \bar{\pi}^{\varphi}p_* \mathbb{D}} & { \bar{\pi}^{\varphi}  \mathbb{D} p_*} & { \mathbb{D} \bar{\pi}^{\varphi, -}  p_*.}
    \arrow["\delta", from=1-1, to=1-2]
    \arrow["{{\varepsilon_p}}"',  from=1-1, to=2-1]
    \arrow["{{\Ex^{!, \mathbb{D}}}}", "\sim"', from=1-2, to=1-3]
    \arrow["{{\Ex_{*, \mathbb{D}}}}"', "\sim", from=2-1, to=2-2]
    \arrow["\delta"', from=2-2, to=2-3]
    \arrow["{{\mathbb{D}\varepsilon_{p}}}"', from=2-3, to=1-3]
\end{tikzcd}\]

\item \label{item-symp-push-perverse-pullback-product}
Let $X_i \xrightarrow[]{\pi_i} B_i \xrightarrow[]{p_i} S_i$ be lfp morphisms between derived Artin stacks for $i = 1, 2$, where $\pi_i$ is equipped with a relative exact $(-1)$-shifted symplectic structure $\omega_i \in \mathcal{A}^{2, \ex}(X_i / B_i, -1)$ and an orientations $o_i$.
We let $\bar{\pi}_i \colon p_{i, *} (X_i, \omega_i) \to S_i$ be the symplectic pushforward and $r_i \colon p_{i, *}(X_i, \omega_i) \to X_i$ be the natural map. 
Equip $\pi_1 \times \pi_2$, $\bar{\pi}_1$, $\bar{\pi}_2$ and $\bar{\pi}_1 \times \bar{\pi}_2$ with orientations $o_1 \boxtimes o_2$, $p_{1, *} o_1$, $p_{2, *} o_2$ and $p_{1, *} o_1 \boxtimes p_{2, *} o_2$ respectively.
Then the following diagram commutes:
\[\begin{tikzcd}
    {\bar{\pi}_{1}^{\varphi} p_{1, *} (-) \boxtimes \bar{\pi}_{2}^{\varphi} p_{2, *} (-) } & {(\bar{\pi}_1 \times \bar{\pi}_2)^{\varphi}(p_{1, *}(-) \boxtimes  p_{2, *} (-)) } & {(\bar{\pi}_1 \times \bar{\pi}_2)^{\varphi} (p_1 \times p_2)_*((-) \boxtimes  (-)) } \\
    {r_1^{\dagger} \pi_{1}^{\varphi} (-) \boxtimes r_2^{\dagger} \pi_{2}^{\varphi} (-)} & {(r_1 \times r_2)^{\dagger} ( \pi_{1}^{\varphi} (-) \boxtimes  \pi_{2}^{\varphi} (-))} & {(r_1 \times r_2)^{\dagger} (\pi_1 \times \pi_2)^{\varphi} ((-) \boxtimes  (-)).}
    \arrow["\TS", from=1-1, to=1-2]
    \arrow["{\varepsilon_{p_1} \boxtimes \varepsilon_{p_2}}"', from=1-1, to=2-1]
    \arrow["\sim", from=1-2, to=1-3]
    \arrow["{\varepsilon_{p_1 \times p_2}}", from=1-3, to=2-3]
    \arrow["\sim"', from=2-1, to=2-2]
    \arrow["\TS"', from=2-2, to=2-3]
\end{tikzcd}\]        
\end{enumerate}
\end{proposition}
\begin{proof}
    To see that $r$ is smooth, recall that $p_*(X, \omega)$ can be described as the zero locus of the moment map $\mu_{X} \colon X \to \T^*(B/S)$ (see \cite[Proposition 2.3.1]{ParkSymplectic}).
    Since $p$ is unramified, $\bL_{B/S}$ is of Tor-amplitude $\leq -1$.
    Hence the zero section of $\T^*(B/S)$ has cotangent complex of Tor-amplitude $\geq 0$ and is smooth.
    
    To construct the isomorphism \eqref{eq-symp-push-perverse-pullback}, form the Cartesian square of derived stacks
    \[\begin{tikzcd}
        p_*(X, \omega) \times_{S} B \ar{r}{p'}\ar{d}{\overline{\pi}'}
        & p_*(X, \omega) \ar{d}{\overline{\pi}}
        \\
        B \ar{r}{p}
        & S
    \end{tikzcd}\]
    and regard $\overline{\pi}'$ with its induced oriented relative exact $(-1)$-shifted symplectic structure.
    Denote by $i \colon p_*(X, \omega) \to p_*(X, \omega) \times_{S} B$ the unique morphism equipped with identifications $p' \circ i \cong \id$ and $\pi \circ r \cong \overline{\pi}' \circ i$.
    Then the diagram
    \[\begin{tikzcd}
        & p_*(X, \omega) \ar[swap]{ld}{r}\ar{rd}{i} & 
        \\
        X & & p_*(X, \omega) \times_{S} B
    \end{tikzcd}\]
    defines an oriented Lagrangian correspondence between $\pi$ and $\overline{\pi}'$.
    Since $\overline{\pi}$ factors through $p$ and $p$ is a closed immersion, $p'$ induces an isomorphism on classical truncations and hence so does its section $i$.
    Applying \cref{prop:Lagrangian functoriality of perverse pullback} we deduce the canonical isomorphism $ i_* r^\dag \pi^\varphi \cong \overline{\pi}'^\varphi$, hence $r^\dag \pi^\varphi \cong p'_* \overline{\pi}'^\varphi$.
    Combining this with the canonical isomorphism $\overline{\pi}^\varphi p_* \cong p'_* \overline{\pi}'^\varphi$ of \cref{thm:micropullstk}(2), using that $p$ is finite, we define:
    \[
    \varepsilon_{p} \colon \overline{\pi}^\varphi p_* \cong p'_* \overline{\pi}'^\varphi \cong r^\dag \pi^\varphi.
    \]
    Functoriality for compositions in $p$ follows from the corresponding functoriality statements in \cref{thm:micropullstk}(1,2). 

    The claims (\ref{item:symp-push-commutes-smooth-descent}, \ref{item-symp-finite-push-base-change}, \ref{item-push-finite-smooth}, \ref{item-symp-push-perverse-pullback-dual}, \ref{item-symp-push-perverse-pullback-product}) follow from the fact that the map $\varepsilon_{p}$ is constructed as a composite of the isomorphisms $\alpha$ of \cref{thm:micropullstk}(1) and $\beta$ of \cref{thm:micropullstk}(2), and these are both individually compatible with $\alpha$, $\beta$, $\gamma$, $\delta$, and $\TS$ by \cref{thm:micropullstk}.
    To illustrate this, we give a detailed proof of (\ref{item-push-finite-smooth}); the proofs of (\ref{item:symp-push-commutes-smooth-descent}, \ref{item-symp-finite-push-base-change}, \ref{item-symp-push-perverse-pullback-dual}, \ref{item-symp-push-perverse-pullback-product}) are similar.

    Consider the following commutative diagram:
    \[\begin{tikzcd}
    	{\bar{p}_* q_*(X, \omega)} & {q_*(X, \omega)} & {\bar{p}_* q_*(X, \omega) \times_{R}S'} \\
    	{p_*(X,  \omega)} & X & {p_*(X,  \omega) \times_S B.}
    	\arrow["{\bar{r}}", from=1-1, to=1-2]
    	\arrow["{{\bar{t}}}", shift left=2, curve={height=-24pt}, from=1-1, to=1-3]
    	\arrow["{\bar{p}_X}", curve={height=24pt}, from=1-3, to=1-1]
    	\arrow["{\bar{s}}"', from=1-1, to=2-1]
    	\arrow["{{r}}"', from=2-1, to=2-2]
    	\arrow["{{t}}"', shift right=2, curve={height=18pt}, from=2-1, to=2-3]
    	\arrow["s", from=1-2, to=2-2]
    	\arrow["{{\bar{s} \times_{R} \id_{S'}}}"', from=1-3, to=2-3]
    	\arrow["p_X"', curve={height=-18pt}, from=2-3, to=2-1]
    \end{tikzcd}\]
    Here, $p_X$ and $\bar{p}_X$ are base change of $p$ and $\bar{p}$, and $t$ and $t'$ are naturally defined morphisms.
    By construction, $t$ and $p_X$ are mutually inverse after passing to the classical truncation, and similarly for $\bar{t}$ and $\bar{p}_X$.
    Consider the following diagram:
    \[\adjustbox{scale=0.8,center}{
    \begin{tikzcd}
    	{ \bar{\pi} ^{\varphi} p_* q^{\dagger}} &[-15pt] &[-15pt] { \bar{\pi}^{\varphi} \bar{q}^{\dagger} \bar{p}_* } &[-10pt] &[-15pt] { \bar{s}_* (\bar{\pi}')^{\varphi} \bar{p}_*} \\
    	{p_{X, *}(\bar{\pi} \times_{S} \id_B)^{\varphi}  q^{\dagger}} && {p_{X, *} (\bar{s} \times_R \id_{S'})_*  (\bar{\pi}' \times_{R} \id_{S'})^{\varphi}} && { \bar{s}_*  \bar{p}_{X, *} (\bar{\pi}' \times_{R} \id_S)^{\varphi} } \\
    	{ p_{X, *} t_* r^{\dagger} \pi^{\varphi} q^{\dagger}} & { p_{X, *} t_* r^{\dagger} s_* (\pi')^{\varphi} } & { p_{X, *} t_*  \bar{s}_* \bar{r}^{\dagger}  (\pi')^{\varphi} } & { p_{X, *} (\bar{s} \times_R \id_{S'})_* \bar{t}_*  \bar{r}^{\dagger}  (\pi')^{\varphi} } & { \bar{s}_*  \bar{p}_{X, *} \bar{t}_* \bar{r}^{\dagger} (\pi')^{\varphi} } \\
    	{r^{\dagger} \pi^{\varphi} q^{\dagger}} && { {r^{\dagger} s_* (\pi')^{\varphi} }} && {{ \bar{s}_* \bar{r}^{\dagger} (\pi')^{\varphi} }.}
    	\arrow["{{\beta_{p}}}", from=1-1, to=2-1]
    	\arrow[""{name=0, anchor=center, inner sep=0}, "{{\varepsilon_p}}"', curve={height=60pt}, from=1-1, to=4-1]
    	\arrow["\sim", "\Ex^!_*"', from=1-3, to=1-1]
    	\arrow["{{\gamma_{\bar{q}}}}", from=1-3, to=1-5]
    	\arrow["{(B)}"{description}, draw=none, from=1-3, to=2-3]
    	\arrow["{{\beta_{\bar{p}}}}"', from=1-5, to=2-5]
    	\arrow[""{name=1, anchor=center, inner sep=0}, "{{\varepsilon_{\bar{p}}}}", curve={height=-65pt}, from=1-5, to=4-5]
    	\arrow[""{name=2, anchor=center, inner sep=0}, "{{\gamma_q}}", from=2-1, to=2-3]
    	\arrow[""{name=3, anchor=center, inner sep=0}, "{{\alpha_{(r, t)}}}", from=2-1, to=3-1]
    	\arrow[""{name=4, anchor=center, inner sep=0}, "\sim", from=2-3, to=2-5]
    	\arrow["{{\alpha_{(\bar{r}, \bar{t})}}}"{pos=0.6}, from=2-3, to=3-4]
    	\arrow[""{name=5, anchor=center, inner sep=0}, "{{\alpha_{(\bar{r}, \bar{t})}}}"', from=2-5, to=3-5]
    	\arrow[""{name=6, anchor=center, inner sep=0}, "{{\gamma_q}}", from=3-1, to=3-2]
    	\arrow["\sim", from=3-1, to=4-1]
    	\arrow[""{name=7, anchor=center, inner sep=0}, "{{\Ex^!_*}}", from=3-2, to=3-3]
    	\arrow["\sim", from=3-2, to=4-3]
    	\arrow[""{name=8, anchor=center, inner sep=0}, "\sim", from=3-3, to=3-4]
    	\arrow[""{name=9, anchor=center, inner sep=0}, "\sim", from=3-4, to=3-5]
    	\arrow["\sim"', from=3-5, to=4-5]
    	\arrow[""{name=10, anchor=center, inner sep=0}, "{{\gamma_q}}"', from=4-1, to=4-3]
    	\arrow[""{name=11, anchor=center, inner sep=0}, "{{\Ex^!_*}}"', from=4-3, to=4-5]
    	\arrow["{(A)}"{description}, draw=none, from=0, to=3]
    	\arrow["{(C)}"{description}, draw=none, from=2, to=7]
    	\arrow["{(D)}"{description}, draw=none, from=4, to=9]
    	\arrow["{{(G)}}"{description}, draw=none, from=5, to=1]
    	\arrow["{(E)}"{description}, draw=none, from=6, to=10]
    	\arrow["{(F)}"{description}, draw=none, from=8, to=11]
    \end{tikzcd}}\]
    It is enough to prove the commutativity of the outer square.
    The commutativity of the diagrams (A) and (G) follows from the construction of the map $\varepsilon_{p}$ and $\varepsilon_{\bar{p}}$.
    The commutativity of the diagrams (D), (E) and (F) are obvious.
    The commutativity of the diagram (B) follows from the compatibility between the $\beta_{-}$ and $\gamma_{-}$ proved in \cref{thm:micropullstk}.
    Similarly, the commutativity of the diagram (C) follows from the compatibility relation between $\alpha_-$ and the $\gamma_{-}$, which is also proved in \cref{thm:micropullstk}.
    Hence we conclude the commutativity of the outer square as desired.
\end{proof}

\printbibliography

\end{document}